\documentclass[12pt]{amsart}
\usepackage[margin=1.2in,includeheadfoot]{geometry}
\usepackage{enumerate}
\usepackage{comment}
\usepackage{bm}
\usepackage{tikz-cd}
\usepackage{stmaryrd}
\usepackage{amsfonts}
\usepackage{amssymb}
\usepackage{amsmath,amsthm}
\usepackage{latexsym}
\usepackage{amscd}
\usepackage{esint}
\usepackage{mathrsfs}
\usepackage{dsfont}
\usepackage{indentfirst}
\usepackage{enumitem}
\newlist{Aenumerate}{enumerate}{1}
\setlist[Aenumerate]{label=(S\arabic*)}
\newcommand{\subscript}[2]{$#1 _ #2$}
\usepackage{enumerate}
\usepackage{slashed}
\usepackage{fancyhdr}
\usepackage{aliascnt}
\usepackage{pdfsync}
\usepackage{setspace}
\usepackage{color}
\usepackage{xcolor}
\definecolor{citeblue}{RGB}{0,0,139}
 \usepackage[colorlinks = true,
            linkcolor = citeblue,
            urlcolor  =  citeblue,
            citecolor =  citeblue,
            anchorcolor = citeblue]{hyperref}       
\newtheorem{theorem}{Theorem}[section]
\newtheorem{coro}[theorem]{Corollary}
\newtheorem{lemma}[theorem]{Lemma}
\newtheorem{prop}[theorem]{Proposition}

\newtheorem*{op}{\bf Open Problem}
\newtheorem*{hp}{\bf Homotopy Problem for \bm{$H$}-surfaces}
\newtheorem*{pro}{\bf Problem}
\newtheorem*{gpro}{\bf Generalized Yau's Open Problem}
\theoremstyle{definition}
\newtheorem{defi}[theorem]{Definition}
\theoremstyle{remark}
\newtheorem{rmk}[theorem]{Remark}
\numberwithin{equation}{subsection}
\numberwithin{theorem}{subsection}
\newcommand{\norm}[1]{\left\Vert#1\right\Vert}
\newcommand{\abs}[1]{\left\vert#1\right\vert}
\newcommand{\inner}[1]{\left\langle#1\right\rangle}
\newcommand{\set}[1]{\left\{#1\right\}}

\newcommand{\Vol}{\mathrm{Vol}}

\newcounter{stepnum}[section]
\newcommand{\step}{%
	\par
	\noindent
	\refstepcounter{stepnum}%
	\textbf{Step \arabic{stepnum}}.\enspace\ignorespaces
}

\newcounter{claimnum}[subsection]
\newcommand{\claim}{%
	\par
	\refstepcounter{claimnum}%
	\noindent  
	\textbf{Claim \arabic{claimnum}}.\enspace\ignorespaces
}

\newcounter{casenum}[section]

\usepackage{etoolbox}

\makeatletter
\newcommand{\reduceoperator}[1]{%
  \@for\next:=#1\do{\expandafter\reduceoperator@\expandafter{\next}}%
}
\newcommand{\reduceoperator@}[1]{%
  \csletcs{normal@#1@}{#1@}%
  \csedef{#1@}{\noexpand\reduceoperator@@\csname normal@#1@\endcsname}%
}
\newcommand{\reduceoperator@@}[1]{%
  \mathop{\mathpalette\reduceoperator@@@{#1}}%
}
\newcommand{\reduceoperator@@@}[2]{%
  \ifx#1\displaystyle\textstyle\fi#2%
}
\makeatother

\reduceoperator{bigcup,bigcap,bigotimes}

\newcommand{\p}[1]{\left(#1\right)}
\newcommand{\dbl}[1]{\widetilde{#1}}
\def\R{\mathbb R}
\def\S{\mathbb S}

\def\L{\mathscr{L}}
\def\T{\mathcal{T}}
\selectfont
\DeclareMathOperator{\supp}{\text{supp}\,}
\DeclareMathOperator{\re}{\text{Re}\,}
\DeclareMathOperator{\im}{\text{Im}\,}
\DeclareMathOperator{\diver}{\mathrm{div}\,}

\DeclareMathOperator{\nablap}{\nabla^\perp\hspace{-0.5mm}}

\allowdisplaybreaks[4]

\begin{document}
	\title[Min-max theory and existence of H-spheres]{Min-max theory and existence of H-spheres \\ with arbitrary codimensions}

	\author[Gao]{Rui Gao}
	\address{School of Mathematical Sciences, Shanghai Jiao Tong University\\ 800 Dongchuan Road \\ Shanghai, 200240 \\ P. R. China}%
	\email{gaorui0416@sjtu.edu.cn}
	
	\author[Zhu]{Miaomiao Zhu}
	\address{School of Mathematical Sciences, Shanghai Jiao Tong University\\ 800 Dongchuan Road \\ Shanghai, 200240 \\ P. R. China}
	\email{mizhu@sjtu.edu.cn}
	
	\thanks{We would like to thank Professor Xin Zhou for valuable conversations and helpful comments.}
	\date{\today}
	\dedicatory{Dedicated to Stefan Oscar Walter Hildebrandt (1936-2015)}
	%
	
	\begin{abstract}
		We demonstrate the existence of branched immersed 2-spheres with prescribed mean curvature, with controlled Morse index and with arbitrary codimensions in closed Riemannian manifold $N$ admitting finite fundamental group, where $\pi_k(N) \neq 0$ and $k \geq 2$, for certain generic choice of prescribed mean curvature vector. Moreover, we enhance this existence result to encompass all possible choices of prescribed mean curvatures under certain Ricci curvature condition on $N$ when $\dim{N} = 3$. When $\dim{N} \geq 4$, we establish a Morse index lower bound while $N$ satisfies some isotropic curvature condition. As a consequence, we can leverage latter strengthened result to construct 2-spheres with parallel mean curvature when $N$ has positive isotropic curvature and $\dim{N} \geq 4$. At last, we partially resolve the homotopy problem concerning the existence of a representative surface with prescribed mean curvature type vector field in some given homotopy classes.
	\end{abstract}

	\maketitle
    \vskip0.5cm
	\tableofcontents
	
	\section{Introduction}\label{intro}
	\vskip15pt
	\subsection{Open Problems and Main Results}
	\
	\vskip5pt
	Surfaces with \textit{Constant Mean Curvature} (CMC) and with \textit{Prescribed Mean Curvature} (PMC) are important in mathematics, physics and biology. They arise naturally in partitioning problems, isoperimetric problems, general relativity, two phases interface problems and tissue growth etc.
	
	Around the early 1980's, Yau imposed the following
	\begin{op}[{\cite[Problem 59]{yau1982seminar}}]\label{yau conj}
		Let $h$ be a real valued function on $\R^3$. Find reasonable conditions on $h$ to insure that one can find a closed surface with prescribed genus in $\R^3$ whose mean curvature is given by $h$.
	\end{op}
	
	At almost the same time, Yau also posed an open problem for the existence of closed PMC surfaces in closed 3-manifolds\footnote{This was confirmed by communication between Yau and Zhou-Zhu, see comments in \cite[page. 312]{Zhou-Zhu2020}.}. It is natural to impose the following extension to the higher codimensional setting:
	\begin{gpro}
		Let $H \in \Gamma(\wedge^2(N)\otimes TN)$ be a tensor field on a closed $n$-dimensional Riemannian manifold $N$. Find reasonable conditions on $H$ to insure that one can construct a closed surface with prescribed genus in $N$ whose mean curvature vector is given by $H$.
	\end{gpro}
	
	The general existence for PMC surfaces (also called $\mu$-bubbles) with all codimensions is a challenging problem, see for instance comments by Gromov in a recent series of four lectures \cite[pp.116, footnote 194.]{gromov2021lectures}.
	
	In this paper, we derive a resolution to the generalized Yau’s open problem for 
	branched immersed 2-spheres in closed Riemannian manifolds having finite fundamental groups with arbitrary codimensions, with controlled Morse index and with certain generic choice of prescribed mean curvature vector field $H$.

	Now, we state our first result:
	\begin{theorem}\label{main theorem 1}
		Let $(N,h)$ be a closed $n$-dimensional ($n \geq 3$) Riemannian manifold with finite fundamental group, then given any $\omega \in C^{2}(\wedge^2(N))$ with induced mean curvature type tensor field $H \in \Gamma(\wedge^2(N)\otimes TN)$ determined by \eqref{eq: defi H by omega}, for almost every constant $\lambda > 0$, there exists a non-trivial branched immersed 2-sphere in $N$ with prescribed mean curvature vector $\lambda H$ and Morse index at most $k-2$, where $k$ is the least integer such that $\pi_k(N)\neq 0$.
	\end{theorem}

	By imposing some curvature constraint in the target $N$, we can improve the Morse index estimates to get the following:
	
	\begin{theorem}\label{main theorem 2}
		Let $(N,h)$ be a closed $n$-dimensional Riemannian manifold with finite fundamental group and given $\omega \in C^{2}(\wedge^2(N))$ with induced mean curvature type  vector field $H \in \Gamma(\wedge^2(N)\otimes TN)$  determined by \eqref{eq: defi H by omega}, the following holds:
		\begin{enumerate}
			\item \label{main theorem 2 part 1} If $\dim{N} = 3$ with $\pi_3(N) \neq 0$ and the Ricci curvature of $N$ satisfies
			\begin{equation}\label{eq:intro ricci condi}
				|H|^2h + \frac{\mathrm{Ric}_h}{2} > |\nabla H|h,
			\end{equation}
			then there exists a non-trivial branched immersed 2-sphere in $N$ with prescribed mean curvature vector field $H$ determined by \eqref{eq: defi H by omega} and Morse index exactly $1$.
			\item \label{main theorem 2 part 2} If $\dim{N} \geq 4$ and $N$ has positive isotropic curvature (PIC), namely, on any totally isotropic two plane $\sigma \subset TN\otimes \mathbb{C}$ the complex sectional curvature fulfils $\mathcal{K}(\sigma) > 0$, then, for $\mathcal{K}$ being the isotropic curvature and for almost every constant $\lambda > 0$ satisfying  
			\begin{equation}\label{eq:intro isotropic}
				\mathcal{K} - \lambda|\nabla H|h > 0, 
			\end{equation}
			there exists a non-trivial branched immersed 2-sphere in $N$ with prescribed mean curvature vector field $\lambda H$ and Morse index satisfying $[(n-2)/2] \leq \mathrm{Index} \leq k-2$, where $k$ is the least integer such that $\pi_k(N)\neq 0$.
		\end{enumerate}
	\end{theorem}
	
	Consequently, for the existence of surfaces with parallel mean curvature vector field (see Definition \ref{defi parallel} below)  we have
	\begin{theorem}\label{main coro}
		Let $(N,h)$ be a closed $n$-dimensional ($n\geq 4$) Riemannian manifold with positive isotropic curvature and finite fundamental group, then given any $\omega \in C^{2}(\wedge^2(N))$ with induced parallel mean curvature type  vector field $H \in \Gamma(\wedge^2(N)\otimes TN)$ determined by \eqref{eq: defi H by omega}, for almost every constant $\lambda > 0$, there exists a non-trivial branched immersed 2-sphere in $N$ with prescribed parallel mean curvature vector field $\lambda H$ and with Morse index satisfying $[(n-2)/2] \leq \mathrm{Index} \leq k-2$, where $k$ is the least integer such that $\pi_k(N)\neq 0$.
	\end{theorem}

     \begin{rmk}
        Indeed, stronger existence results than Theorem \ref{main theorem 1}, Theorem \ref{main theorem 2} and Theorem \ref{main coro} still hold true, namely, we can remove the assumption that $N$ has finite fundamental group, see step \ref{outline 3} in Subsection \ref{section: outline} and Proposition \ref{non constant limit} below for more details. Considering the completeness of our min-max theory and technical issues arising from the compactness in Section \ref{section 4}, we added the assumption that $N$ has finite fundamental group in these Theorems.
    \end{rmk}
	In general, CMC 2-spheres in Riemannian 3-manifold need not being embedded, see comments by Meeks-Mira-P\'{e}rez-Ros \cite{Meeks-Mira-Perez-Ros2013}. In fact, for certain ambient Berger 3-spheres with positive sectional curvature and for some specific choice of $H$, Torralbo \cite{torralbo2010rotationally} proved that every immersed CMC 2-sphere with  mean curvature $H$ always exists self-intersecting points. See also comments by Zhou \cite[pp. 2700]{Zhouicm}. In a recent work \cite[Theorem 1.1, Theorem 1.2, Remark 1.3]{Cheng2020ExistenceOC}, Cheng-Zhou settled the existences of branched immersed CMC 2-spheres in 3-manifolds with constant $H>0$ by developing a min-max theory for a fourth-order perturbation of the action functional. In the present paper, we develop a new min-max theory for a perturbed functional of Sacks-Uhlenbeck-Moore type, it achieves a natural extension to the more general setting of PMC 2-spheres in Riemannian manifolds with arbitrary codimensions.
	\begin{rmk}
		Theorem~\ref{main theorem 1} and Theorem \ref{main theorem 2} can be viewed as a natural generalization of the existence results about minimal 2-spheres by Sacks-Uhlenbeck \cite{sacks1981existence} and Micallef-Moore \cite{Micallef-Moore-1988}. As an application of Morse index estimates of minimal 2-spheres described as in Theorem \ref{main theorem 2},  Micallef-Moore \cite[pp.201 Main Theorem]{Micallef-Moore-1988} proved that any closed $n$-dimensional ($n \geq 4$) PIC manifold admits $\pi_i(N) = 0$ for $2 \leq i \leq [n/2]$. In particular, if $N$ is simply connected, then $N$ is homeomorphic to a sphere.
	\end{rmk}
	In 1964, Eells-Sampson proposed a fundamental homotopy problem for the existence of harmonic maps:
	\begin{pro}[{\cite[Section 6, pp.133]{eells1964harmonic}}\footnote{See also \cite[Section (1.3), p.1 and (2.3), p.72]{eells1995two} and \cite[Section 3.2, pp.49]{struwe2023bubbling}}]
		Is it possible for every homotopy class of smooth maps $u: M \rightarrow N$ between closed Riemannian manifolds $M$ and $N$ have a  harmonic representative? 
	\end{pro}
	Eells-Sampson \cite{eells1964harmonic} answered this question when $M$ is an oriented closed $m$-dimensional Riemannian manifold and $N$ has non-positive sectional curvature. For $M$ a closed surface and $N$ a general target with $\pi_2(N) = 0$, Sacks-Uhlenbeck \cite{sacks1981existence} proved that there exists a minimizing harmonic map in every homotopy class of the mapping space $C^0(M,N)$ and moreover, the generating set of $\pi_2(N,p)$ can be represented by area minimizing 2-spheres by viewing $\pi_2(N,p)$ as a $\mathbb{Z}(\pi_1(N,p))$-module for some base point $p \in N$. Transforming into the $H$-surface setting, it is natural to impose the following:
	\begin{hp}
		Can every homotopy class of smooth map $u$ from a closed surface $M$ into a closed Riemannian manifold $N$ be represented by a $H$-surface?
	\end{hp}
	
	In this paper, we answer this homotopy problem by imposing 
	\begin{equation}\label{eq:omega < 1}
		\|\omega\|_{L^\infty(N)} : = \max_{p \in N} \| \omega(p)\|_{T_pN} < 1.
	\end{equation}
	\begin{theorem}\label{thm: homotopy1}
		Let $N$ be a closed Riemannian manifold with $\pi_2(N) = 0$ and $\norm{\omega}_{L^\infty(N)}$ $ < 1$, then there exists a minimizing $H$-surface $u:M\rightarrow N$ in every homotopy class of maps in $C^0(M, N).$
	\end{theorem}
	Note that the assumption $\pi_2(N) = 0$ is necessary as claimed in \cite{sacks1981existence}. Moreover, by considering $ \pi _2(N,p) $ as a $ \mathbb { Z } ( \pi _1(N,p)) $-module, whose action orbits represent the free homotopy classes of maps from $ \S ^2 $ into $ N $, we obtain:
	\begin{theorem}\label{thm: homotopy2}
		Let $N$ be a closed Riemannian manifold and $\norm{\omega}_{L^\infty(N)} < 1$. Then, each generating set for $\pi_2(N,p)$ acted by $\pi_1(N,p$) can be represented by minimizing $H$-spheres. In particular, if $N$ is simply connected, then every set of generators of $H_2(N,\mathbb{Z})$ can be represented by minimizing $H$-spheres.
	\end{theorem}
	
	Theorem \ref{thm: homotopy1} and Theorem \ref{thm: homotopy2} extend the existence theory of minimizing harmonic maps with arbitrary codimensions in \cite{sacks1981existence} to the $H$-surface setting. Moreover, when $\dim(N) = 3$, by Micallef-White's descriptions \cite{micallef1995structure} on local behavior of branch points for minimizing almost conformal $H$-surfaces, the $H$-surfaces obtained in Theorem \ref{thm: homotopy1} and Theorem \ref{thm: homotopy2} are immersed and free from branch points if they are almost conformal.
	\
	\vskip5pt
	\subsection{Backgrounds}
	\
	\vskip5pt
	The searching for CMC surfaces and PMC surfaces is a long standing problem. There has been extensive and substantial works dedicated to this problem. The construction of CMC surfaces with boundary started at Plateau's experiment with soap films and soap bubbles enclosed by various contours, and the natural existence questions arose from these experiments are called Plateau problem. The classical Plateau problem for disk-type minimal surfaces in $\R^3$, that is, with mean curvature $H=0$, was solved independently by Douglas \cite{douglas1931solution} and Rad\'{o} \cite{rado1930plateau}.
	In 1954, Heinz \cite{Heinz1954berDE} investigated the Plateau problem for CMC surfaces in $\R^3$ and obtained an existence result provided that $|H| < \frac{1}{8}(\sqrt{17} - 1)$. Werner \cite{werner1957problem} optimized the configuration of Heinz to improve the existence result for $|H| < \frac{1}{2}$. Later, Hildebrandt \cite{hildebradt1970} proved the existence of CMC surface for $|H| \leq 1$ and subsequently extended this result to variable mean curvature $H$ satisfying $||H||_{L^\infty} \leq 1$ in \cite{hildebrandt1969randwertprobleme}. It is worth mentioning that the upper bound $|H|\leq 1$ is optimal in the sense that there is no solution when $|H| > 1$ for the Plateau contour $\Gamma = \{(\cos{\theta},\sin{\theta},0)\,:\, 0\leq \theta < 2\pi\}\subset \R^3$, see comments by Heinz \cite[pp. 250]{Heinznonexistence1969} and Jost \cite[pp.45]{jost-memor-hildebradt2016}. For further developments of the Plateau problem for CMC surfaces, see e.g. Wente \cite{wente1969}, Hildebrandt-Kaul \cite{Hildebrandt1972two}, Steffen \cite{Steffen1976}. 
	Br{e}zis-Coron \cite{Brezis1984} and Struwe \cite{Struwe1985,Struwe1986} proved the existence of at least two solution for the Plateau problem of CMC surfaces which confirmed the Rellich's conjecture, see also Steffen \cite{steffen1986nonuniqueness}. For existence of PMC surfaces with Plateau boundary, see \cite{hildebrandt1969randwertprobleme,gulliver1971plateau,gulliver1972existence,Steffen1976,duzaar1993existence}. The study for the existence of CMC surfaces or PMC surfaces with free boundary initiated from Struwe \cite{Struwe-1988} by applying heat flow method, and also studied by \cite{burger2008area,li2021min,cheng2022existence}. The Dirichlet problem for the minimal surface equation was studied by Jenkins-Serrin \cite{jekins-serrin1966} for domain contained in $\R^2$ and Spruck \cite{spruck1972} considered the same problem when mean curvature $H > 0$, see also \cite{rey1991,Laurent-Rosenberg2009,collin-Rosenberg2010,Mazet-Rosenberg2011,Folha-Rosenberg2012} for recent developments.
	
	For closed CMC surfaces, Hopf \cite{Hopf1951} proved that the round sphere is the only CMC surface with genus zero in $\R^3$. Later, Barbosa-Do Carmo \cite{barbosa2012stability} showed that the standard sphere is the only closed stable CMC hypersurfaces in $\R^{n+1}$.  Wente \cite{wente1986counterexample} constructed an immersed CMC torus in $\R^3$ which also gives a counterexample of Hopf's conjecture \cite{Hopf1983}. Moreover, Kapouleas \cite{Kapouleas1990,kapouleas1991compact,kapouleas1992constant} constructed a series of immersed closed CMC surfaces in $\R^3$ with arbitrary genus, see also Briener-Kapouleas \cite{breiner2021complete} for higher dimensional CMC hypersurfaces setting. The Almgren-Pitts min-max theory, which was introduced in \cite{Almgren1962, Almgren1965, Pitts1981, Schoen-Simon1981}, is a significant breakthrough in the field of constructing closed minimal hypersurfaces. In recent years, this theory has been further developed and refined, starting with the confirmation of the Willmore conjecture by Marques-Neves \cite{Marques-Neves2014}, followed by the resolution of Yau's conjecture \cite[Problem 88]{yau1982seminar} on the existence of infinitely many closed minimal surfaces in any closed 3-manifold, also by Marques-Neves \cite{Marques-Neves2017} under the assumption of positive Ricci curvature. Later, Zhou \cite{zhou2020multiplicity} confirmed the multiplicity one conjecture of Marques-Neves \cite[Section 1.2]{Marques-Neves2021adv} and Song \cite{song2023existence}  proved the Yau's conjecture \cite[Problem 88]{yau1982seminar} in the general case where the dimension of ambient manifold is relaxed from 3 to 7. Additionally, there have been several recent works on this topic, such as \cite{Chodosh-Mantoulidis2020,chodosh2021minimal,wangzhichao2022,Liyangyang2023jdg}.  In the context of $H\neq0$, closed CMC hypersurfaces in an ambient manifold were initially constructed by minimizing the area functional among all volume-preserving variations. For a more detailed understanding, we refer to \cite{Almgren1976, Morgan2003, ros2001isoperimetric}. However, this approach provides little information about the mean curvature value or the topology of CMC surfaces in ambient manifolds. There is another series of deformation approaches to construct CMC hypersurfaces, one can generate foliations by closed CMC hypersurfaces with small mean curvature $H$ from a closed non-degenerate minimal surface. Moreover, Ye \cite{Ye1991}, Mahmoudi-Mazzeo-Pacard \cite{Mahmoudi-Mazzeo-Pacard2006}, and others have constructed foliations by closed CMC hypersurfaces from minimal submanifolds of strictly lower dimensions, see also \cite{Pacard2005,Pacard-Xu2009}. The obtained CMC hypersurfaces by this method have a mean curvature that tends to be either very small or very large. Besides, a degree theory established by Rosenberg-Smith \cite{Rosenberg-Smith2020} constructed many important examples of CMC hypersurface when mean curvature is greater than some constant. Zhou-Zhu's  construction \cite{Zhou-Zhu2019} on the min-max method has led to the establishment of a comprehensive existence theory for closed CMC hypersurfaces in closed Riemannian manifolds of dimension between 3 and 7.  Furthermore, Zhou-Zhu \cite{Zhou-Zhu2020} extended their min-max theory for CMC hypersurfaces into the PMC setting, allowing for any prescribed mean curvature $h$ lying in an open dense subset of smooth function space. This was later generalized to higher dimensions by Dey \cite{Dey2019}, allowing a singular set of codimension 7, see also \cite{mazurowski2022cmc,Pacard-Sunpre} for more recent development of min-max theory developed from \cite{Zhou-Zhu2019}. Very recently, Mazurowski-Zhou \cite{Mazurowski-Zhou} introduced the half-volume spectrum $\set{\dbl{w}_p}_{p \in N}$ which also satisfies a Weyl law and they developed an Almgren-Pitts type min-max theory for finding closed  CMC hypersurfaces associated to the half-volume spectrum in \cite{mazurowski2024infinitelyhalfvolumeconstantmean}, hence showed that there exists infinitely many geometrically distinct closed CMC hypersurfaces in closed manifold $M^n$ of dimension $3 \leq n \leq 5$.
	
	The existence theory of CMC and PMC surfaces with prescribed topology in general closed Riemannian manifold is less understood.
	For minimal 2-spheres, Simon-Smith \cite{Smith1982} explored the existence of embedded minimal 2-spheres in any Riemannian 3-sphere, utilizing the Almgren-Pitts min-max theory \cite{Almgren1962}. 
	The existence of branched immersed minimal 2-spheres in general Riemannian manifold with arbitrary codimensions was firstly studied in the pioneering work by Sacks-Uhlenbeck \cite{sacks1981existence} and then explored in greater depth by Micallef-Moore \cite{Micallef-Moore-1988}, who obtained branched immersed minimal 2-spheres with controlled Morse index. 
	For CMC 2-spheres in homogeneous 3-spaces, Meeks-Mira-P\'{e}rez-Ros proved the existence and uniqueness of immersed CMC 2-spheres with any prescribed mean curvature in homogeneous 3-spheres \cite{Meeks-Mira-Perez-Ros2013} and later in homogeneous 3-manifolds \cite{Meeks-Mira-Perez-Ros2020}.
	
	For CMC 2-spheres in closed Riemannian 3-manifolds, a recent breakthrough
	was made by Cheng-Zhou \cite[Theorem 1.1]{Cheng2020ExistenceOC} who established the existence of branched immersed CMC 2-spheres in arbitrary Riemannian 3-spheres $(\mathbb{S}^3,h)$ for almost every positive constant mean curvature $H>0$ and with Morse index at most 1. Moreover, if $(\mathbb{S}^3,h)$ has positive Ricci curvature, then the existence result can be enhanced to encompass any choice of constant mean curvature $H>0$ and guarantee the Morse index exactly 1 \cite[Theorem 1.2]{Cheng2020ExistenceOC}. 
    \begin{rmk}
        Here, we need to point out that, due to the third De Rham cohomology group  $H_{dR}^3(\S^3) \cong \R \neq 0$, by our choices of functional $E^\omega$ we can not obtain the existence of branched immersed CMC 2-spheres in $3$-manifolds, see Theorem \ref{main coro}. However, we believe that, if we modify our functional $E^\omega$ as following 
        \begin{equation*}
            E_H = \frac{1}{2}\int_M |\nabla u |^2 dV_g + H\cdot V(f_u)
        \end{equation*}
         and adapt a similar Sacks-Uhlenbeck type perturbation
        \begin{equation*}
            E_{\alpha,H} := \frac{1}{2}\int_M \p{ 1 + |\nabla u |^2 }^\alpha dV_g + H\cdot V(f_u)
        \end{equation*}
        to study the existence of branched immersed 2-spheres with constant mean curvature in Riemannian 3-spheres $(\S^3,h)$, utilizing a similar proof as Theorem \ref{main theorem 1} and Theorem \ref{main theorem 2} we will provide an alternative proof for Cheng-Zhou's results \cite[Theorem 1.1, Theorem 1.2]{Cheng2020ExistenceOC}. Here, $H>0$ is some positive constant and  $V(f_u)$ is the enclosed volume of the map $u : (\S^2,g) \rightarrow (\S^3,h)$,  see \cite[Section 2]{Struwe-1988} and \cite[Section 2.1]{Cheng2020ExistenceOC} for more detailed descriptions about $E_H$.
    \end{rmk}
	
	Very recently, Sarnataro-Stryker \cite{Lorenzo-Stryker} constructed an embedded PMC 2-sphere in the round 3-sphere for generic set of prescribed mean curvature functions $h$ with $L^\infty$ norm at most 0.547 and obtained an embedded 2-sphere with constant mean curvature $H$ when the metric on $(\mathbb{S}^3,h)$ is sufficiently close to the round metric and $H$ is below some threshold.
	
	As a higher codimensional generalization of CMC surfaces, the study of surfaces with parallel mean curvature vectors (see definition \ref{defi parallel} below) can be traced back to approximately 1940s, see e.g. Schouten-Struik \cite{Schouten-Struik1938}, Coburn \cite{Coburn1939} and Wong \cite{Wong1946}. From the early 1970s, numerous studies have been conducted on the rigidity theory of such surfaces in homogeneous ambient manifolds, such as, the characterization of the spherical immersion of  4-dimensional space form by Ferus \cite{Ferus1971}, 
	rigidity results for submanifolds with parallel mean curvature vector in spaces of constant sectional curvature by
	Yau \cite{yau1974}  which involves the classification of surfaces with parallel mean curvature vector in 4-dimensional real space form, see also independent and related works by Hoffman \cite{Hoffman-thesis1972,hoffman1973jdg} and Chen \cite{chen-Chern72}, a structure theorem into higher dimensional space form by Alencar-do Carmo-Tribuzy \cite{Alencar-doCarmo-Tribuzy2010}. And for more classification results see Kenmotsu-Zhou \cite{Kenmotsu-Zhou2000} in 2-dimensional complex space form, Fetcu \cite{Fetcu2012} in general complex space forms and Fectu-Rosenberg \cite{Fetcu-Rosenberg2015} in Sasakian space forms.
	
	In contrast with the theory of CMC and PMC  hypersurfaces, the existence of closed surfaces with parallel mean curvature vector, or more generally, with prescribed mean curvature vector, admitting prescribed topology and controlled Morse index in general $n$-dimensional compact Riemannian manifold is not widely understood. In this paper, our existence theory of 
	CMC spheres  (Theorem \ref{main coro}), PMC spheres (Theorem \ref{main theorem 1} and Theorem \ref{main theorem 2}), and partial resolution of homotopy problem (Theorem \ref{thm: homotopy1} and Theorem \ref{thm: homotopy2}) with arbitrary codimensions serves a supplement in such area.

	\subsection{Settings}
	\
	\vskip5pt
	Before describing our main ideas about proving main results, we recall some basic notions about $H$-surfaces with arbitrary codimensions, and introduce some notations, see Gr\"{u}ter \cite[Section 2]{gruter1984conformally}, Jost \cite[Section 1.2]{jost1991two} and Rivi\`{e}re \cite[Section III]{riviere2012conformally} for more details.
	
	Let $(M,g)$ be a closed surface and $(N,h)$ a closed Riemannian manifold of dimension $n$ that is isometrically embedded into some $\mathbb{R}^K$.
	Take a $C^2$ 2-form $\omega$ on $N$ and consider the functional
	\begin{equation}\label{el2}
		E^\omega(u) = \frac{1}{2}\int_M |\nabla u|^2 dV_g + \int_M u^* \omega
	\end{equation}
	acting on maps $u\in C^2(M, N)$. It is easy to see that the functional $E^\omega$ is conformally invariant. Surprisingly enough, Gr\"{u}ter \cite[Theorem 1]{gruter1984conformally} showed that any coercive conformally invariant functional with quadratic growth has the form of  \eqref{el2} for some appropriately chosen metric on $M$ and $\omega$ on $N$. Critical points of the functional $E^\omega$ are called $H$-surfaces and one can verify that $H$-surface satisfies the Euler-Lagrange equation
	\begin{equation}\label{eq H-surface intro}
		\Delta u + A(u)(\nabla u, \nabla u) = H(u)(\nablap u, \nabla u).
	\end{equation}
	Here, $A$ is the second fundamental form of embedding $N\subset \mathbb{R}^K$, and the mean curvature type  vector field $H \in \Gamma(\wedge^2(N)\otimes TN)$ is determined by  
	\begin{equation}\label{eq: defi H by omega}
		\forall\, U, V, W\in \Gamma(TN),\quad d\omega(U, V, W):= \inner{ U, H(V, W)}_{TN} = U\cdot H(V,W)
	\end{equation}
	where $``\cdot"$ is the standard scalar product on $\R^K$. We write
	\begin{equation*}
		H(u)(\nablap u,\nabla u) := H(u_{x^1},u_{x^2}) - H(u_{x^2},u_{x^1}) = 2H(u_{x^1},u_{x^2})
	\end{equation*}
	for notation simplicity. If a solution $u$ to \eqref{eq H-surface intro} is conformal, that is,
	\begin{equation}\label{eq: H surface 2}
		|u_{x^1}|^2 - |u_{x^2}|^2 = \inner{u_{x^1},u_{x^2}}_{u^*(TN)} = 0,
	\end{equation}
	then $H(u)$ is the mean curvature vector of the surface determined by $u:M \rightarrow N$, see \cite[pp.10]{jost1991two}. In particular, due to the uniqueness of conformal structure on $ \S^2 $, all $H$-spheres are conformal automatically. Naturally, we assume $H$ is non-degenerate in the sense that $H(u)(u_{x^1},u_{x^2}) \not \equiv 0$, otherwise, the problem is reduced to a harmonic map setting. Then, we define $\nabla H \in \Gamma(\wedge^3(N)\otimes TN)$ to be the convariant differential of $H$ with respect to vector field component of $H$, more precisely, $\nabla H$ is determined by
	\begin{equation*}
		\inner{ \nabla H(U,V), W }_{TN} = (\nabla_W H)(U,V) \in \Gamma(T(N)) \quad \text{for all } U,V, W \in TN.
	\end{equation*}
	In particular, when $\dim N = 3$, the 3-form $d\omega$ defined on $N$ can be identified with a function on $N$, that is, there exists $H \in C^1(N,\R)$ such that $d\omega = H dz^1\wedge dz^2 \wedge dz^3$ where $(z^1,z^2,z^3)$ is some local coordinates of $N$. In this case, the equation \eqref{eq H-surface intro} can be written as
	\begin{equation}\label{eq: H-surface N = 3}
		\Delta u + A(u)(\nabla u, \nabla u) = 2H(u) (u_{x^1}\wedge u_{x^2}).
	\end{equation}
	A solution to \eqref{eq: H surface 2} and \eqref{eq: H-surface N = 3} determines a PMC surface, which is a CMC surface when $H$ is constant.
	
	A natural extension of CMC surfaces in higher codimensions is the concept of surfaces with parallel mean curvature vector field. For detailed description of surfaces with parallel mean curvature vector field in some homogeneous spaces, see e.g. classical works \cite{Hoffman-thesis1972}, \cite[Section 1.]{hoffman1973jdg}, \cite[pp.655]{chen-Chern72}. For surfaces in general Riemannian manifolds with prescribed parallel mean curvature vectors and with arbitrary codimensions, based on our previous settings it is natural to define the following:
	\begin{defi}[Parallel $H$-surfaces] \label{defi parallel}
		We call a $C^2$ map $u : M \rightarrow N$ a \textit{parallel $H$-surface} if $u$ is a solution to \eqref{eq H-surface intro}  with the mean curvature type  vector field $H$ satisfying $\nabla H \equiv 0$.
	\end{defi}
	From the perspective of submanifold theory, it is important to note that the mean curvature vector $H$ is a section of the normal bundle and parallelism is referred to the mean curvature vector $H$ is parallel in the normal bundle. We would like to mention that the mean curvature of the parallel conformal $H$-surface, as described in Definition \ref{defi parallel}, is parallel in the usual sense (see e.g. \cite[Section 1.]{hoffman1973jdg}).
	\vskip5pt
	\subsection{Basic Ideas of Seeking \texorpdfstring{{$H$}}{Lg}-spheres}
	\
	\vskip5pt
	
	Seeking critical points of $ E^ \omega $ by directly applying methods from calculus of variations is a challenging task due to several technical difficulties:
	\begin{enumerate}
		\item The conformally invariant functional $E^\omega$ does not satisfy the Palais-Smale condition.
		\item  Due to the appearance of the term  involving $\omega$ in $E^\omega$, some classical methods developed for harmonic maps can not be applied. 
	\end{enumerate}
	
	To this end, we consider a  perturbation of $E^\omega$, denoted by $$E^\omega_{\alpha}:W^{1,2\alpha}(M,N)\rightarrow \R,$$
	as follows, called the {\bf Sacks-Uhlenbeck-Moore approximation:}
	\begin{equation*}
		E^\omega_{\alpha}(u) = \frac{1}{2}\int_M \p{1 + |\nabla u|^2}^\alpha dV_g + \int_{M} u^*\omega
	\end{equation*}
	where $\alpha > 1$ and $\omega$ is a $C^2$ 2-form on $N$. 
	
	In his book \cite[Section 4.5]{moore2017}, Moore wrote down the above perturbed functional and pointed out that it satisfies Palais-Smale condition and indicated the regularity of critical points for $E^\omega_{\alpha}$. 
    Sacks-Uhlenbeck type perturbations have been effectively employed in various other settings, for instance, by utilizing these approximations, Cheng-Zhou \cite{Cheng-Zhou-curve} established the existence of curves with constant geodesic curvature in Riemannian 2-spheres, and Cheng \cite{cheng2022existence} demonstrated the existence of free boundary disks with constant mean curvature in $\R^3$.
	
	In this paper, we demonstrate that this perturbed functional $E^\omega_{\alpha}$ is a feasible one to derive the existence of branched immersed $H$-spheres in Riemannian manifolds with arbitrary codimensions. More precisely, we develop a min-max theory for the functional $E^\omega_{\alpha}$, then deduce a compactness theory for non-trivial critical points of $E^\omega_{\alpha}$ as $\alpha \searrow 1$, and finally achieve the desired existence results.
	\
	\vskip5pt
	\subsection{Outline of Proof}\label{section: outline}
	\
	\vskip5pt
	
	Perturbing the functional $E^\omega$ also brings in many new challenges. To this end, we develop a method to construct sequences of non-trivial critical points $\{u_{\alpha_j}\}_{j \in \mathbb N}$ of $E^\omega_{\alpha_j}$ with uniformly bounded $\alpha_j$-energy and uniform Morse index upper bound. Also, we implement a convergence scheme to produce non-constant $H$-spheres and we related the existence of prescribed mean curvature sphere with the continuity of $\alpha$-energy $E_\alpha$ as $\alpha \searrow 1$.
	
	\begin{enumerate}
		\item Firstly, thanks to the fact that the functional $E^\omega_{\alpha}$ satisfies the Palais-Smale condition on Banach manifold $W^{1,2\alpha}(\S^2,N)$, inspired by a monotonicity technique by Struwe \cite[Section 4]{Struwe-1988} and an argument by
		Colding-Minicozzi \cite[Section 1.6]{Colding-Minicozzi2008b}, for almost every $\lambda \in \R_+$, we exploit the notion of \textit{Width} with higher dimensional parameter spaces in our setting to construct a sequence of non-constant critical points $\{u_{\alpha_j}\}$ of the functional $E^{\lambda\omega}_{\alpha_j}$ with uniformly bounded $\alpha_j$-energy $E_{\alpha_j}(u_{\alpha_j})$. 
		Here, we utilized the monotonicity technique in \cite{Struwe-1988} with respect to the parameter $\lambda \in \R_+$ to obtain some $\alpha$-energy upper bound which depends on the choice of $\lambda$ but is uniform for the sequence $\{u_{\alpha_j}\}$ as $\alpha_j \searrow 1$, see Proposition \ref{prop energy bound}. To establish the Morse index upper bound for our approximated sequence, we draw inspiration from the Morse index upper bound estimates within the framework of Almgren-Pitts min-max theory as explored by Marques-Neves \cite{Marques-Neves2016}, Song \cite{Song-2023-Morse}, and Li \cite{Liyangyang-2023}. Additionally, we refer to the works of Cheng-Zhou \cite{Cheng2020ExistenceOC} and Cheng \cite{cheng2022existence} for insights into a newly devised min–max theory setting. Leveraging a homotopical deformation approach for the min-max sequences of sweepouts, we construct a sequence $\{u_{\alpha_j}\}_{j \in \mathbb N}$ with prescribed mean curvature type vector field $\lambda H$ that simultaneously satisfies the desired Morse index upper bound and uniformly $\alpha_j$-energy bound, see Theorem \ref{prop deformation} and Theorem \ref{thm:Morse index k-2} for detailed descriptions.
  
		\item \label{outline 3} Next, we investigate the limit $u_{\alpha_j}$ as $\alpha_j \searrow 1$. Standard tools for blow-up analysis for $\alpha$-harmonic maps developed in \cite{sacks1981existence} still hold in our case. We also have an alternative: 
		\begin{itemize}
			\item If Dirichlet energy $E(u_{\alpha_j})$ is nowhere concentrated as $\alpha_j \searrow 1$, then $u_{\alpha_j}$ converges strongly in $C^\infty(\S^2, N)$ to some $(\lambda H)$-sphere with same Morse index upper bound $k-2$.
			
			\item If Dirichlet energy $E(u_{\alpha_j})$ concentrates somewhere $x_1 \in \S^2$, then the rescaled sequence $v_{\alpha_j}(x) :=u_{\alpha_j}(x_1 + \lambda_{\alpha_j} x)$ for some $\lambda_{\alpha_j} \searrow 0$ also converges smoothly to a limit $v \in W^{1,2}(\S^2,N)$. But due to the absence of conformally invariance for functional $E^\omega_{\alpha}$, $v$ solves a new equation
			\begin{equation*}
				\Delta v  + A(v)\left(\nabla v, \nabla v\right) = \frac{1}{\mu} \lambda H(v)(\nablap v, \nabla v)
			\end{equation*}
			where $\mu = \liminf_{\alpha_j\searrow 1} \lambda_{\alpha_j}^{2 - 2\alpha_j} \in [1,\infty)$. We call it the \textit{blow-up spectrum} of $v_{\alpha_j}$ which characterizes the competition between the extent of energy dissipation $|\nabla u_{\alpha_j}| \nearrow \infty$ and the speed of $(\alpha_j - 1) \searrow 0$ as $j \rightarrow \infty$ during blow-up process. For $\alpha$-harmonic maps, such type of quantity was introduced by Li-Wang \cite{li2010} to investigate the generalized energy identity. The second challenge in our paper is to establish $\mu = 1$. An intriguing observation is that $\mu = 1$ if and only if there is no energy loss during the blow-up process for a general sequence $\{ u_{\alpha_j} \}_{\alpha_j \searrow 1}$ of critical points of $E^\omega_{\alpha}$ around each energy concentration point. Through a meticulous neck analysis and leveraging Gromov's \cite{Gromov1978} estimates on the length of geodesics by its Morse index, we demonstrate that the energy identity holds, hence $\mu = 1$, for sequences of min-max type critical points $\{ u_{\alpha_j} \}_{\alpha_j \searrow 1}$, see Section \ref{section 4} for more details. Therefore, when energy concentrates at a particular point, a non-constant $(\lambda H)$-sphere with Morse index bounded from above by $k-2$ is also obtained.
		\end{itemize}
		
		In general, as a consequence of Theorem \ref{thm convergence}, $u_{\alpha_j}$ converges to some $(\lambda H)$-sphere $u$ weakly in $W^{1,2}(M,N)$ and strongly in $C^2(\S^2\backslash\set{x_1,x_2,\cdots, x_l},N)$ for some $l \geq 0$. Surprisingly, we observe that the weak limit $u$ of $u_{\alpha_j}$ is always non-constant, as shown in Proposition \ref{non constant limit}. This suggests the possibility of a second non-trivial $H$-sphere being produced when the bubbling phenomenon occurs. As a result, the proof of Theorem \ref{main theorem 1} is completed in both scenario.
	\end{enumerate}
	
	To prove the part \eqref{main theorem 2 part 1} of Theorem \ref{main theorem 2}, 
	we firstly modify the calculation of Ejiri-Micallef \cite{Ejiri-Micallef2008}, which is also applied in the proof of \cite[Theorem 1.2]{Cheng2020ExistenceOC} in CMC setting, to obtain a new bi-linear form whose Morse index is controlled by $E^\omega$. Then, we combine the Ricci curvature assumption on target $N$ with the conformal balance argument (see for instance Li-Yau \cite{Li-Yau1982}) to get a uniform energy upper bound and exclude the possibility of existence of non-trivial stable $H$-sphere. Then Theorem \ref{main theorem 2} when $\dim (N) = 3$ follows by a convergence argument for sequences of $H$-spheres. For the part \eqref{main theorem 2 part 2} of Theorem \ref{main theorem 2}, by adapting the calculation and counting argument of \cite[Theorem 1]{Micallef-Moore-1988}, we can get the Morse index lower bound of non-constant $H$-sphere.
	
	The proofs of Theorem \ref{thm: homotopy1} and Theorem \ref{thm: homotopy2} follow by adapting the scheme of Sacks-Uhlenbeck's \cite{sacks1981existence} resolutions on homotopy problem of harmonic map setting and an observation that the assumption \eqref{eq:omega < 1} implies the lower boundedness of functional $E^\omega$.
	\ 
	\vskip5pt
	\subsection{Organizations}
	\
	\vskip5pt
    The paper is structured as follows:

    In Section \ref{section 2}, we provide the necessary notations and discuss the variational properties of the perturbed functional, $E^\omega_{\alpha}:W^{1,2\alpha}(M,N) \rightarrow \R$. Note that the results presented in this section are applicable to general closed Riemann surfaces $(M,g)$.

    Section \ref{sec: 3 non-constant u alpha} is dedicated to the construction of a sequence of non-trivial critical points $\{u_{\alpha_j}\}_{j \in \mathbb{N}}$ of $E^\omega_{\alpha_j}$ with uniformly bounded $\alpha_j$-energy $E_{\alpha_j}$ and a Morse index bounded from above by $k-2$. 
	 
    Section \ref{section 4} focuses on investigating the limits of $u_{\alpha_j}$ as $\alpha_j \searrow 1$. More precisely, we establish a generalized energy identity and unveil a direct convergence relationship between our blow-up spectrum and the energy identity for a general sequence $u_{\alpha_j}$ from a closed Riemann surface $M$ to a compact $n$-manifold $N$.
	
	In Section \ref{section 5}, we prove the main results Theorem \ref{main theorem 1}, Theorem \ref{main theorem 2}, Theorem \ref{thm: homotopy1} and Theorem \ref{thm: homotopy2}. 
	\vskip2cm
	\section{Variational Properties of Perturbed Functional \texorpdfstring{$E^\omega_{\alpha}$}{Lg}}\label{section 2}
	\vskip10pt
	In this section, we shall review some notations and variational properties of the perturbed functional $E^\omega_{\alpha}$. It is worth noting that the restriction of $M=\S^2$ is not a necessary condition for the results to hold and consequences presented in this section can apply to general Riemann surfaces $(M,g)$.
	\subsection{Some Preliminaries}\label{section 2.1}
	\
	\vskip5pt
	Recall that we assumed that $N$ is isometrically embedded into $\R^K$ for some $K \in \mathbb{N}$. In order to utilize the coordinate of $\mathbb{R}^K$ to locate the point of $N$, we choose a tubular neighborhood $\dbl{N}$ of $N$ equipped with canonical Euclidean coordinate $(y^1, y^2,\dots, y^K)$ in $\R^K$. 
	Furthermore, $\dbl{N}$ can be chosen to be close to $N$ enough such that $\omega$ can be extended to a $C^2$ 2-form defined on $\dbl{N}$ which is also denoted by $\omega$. Hence, utilizing this local coordinate of $\dbl{N}$,  we can write
	\begin{equation*}
		\omega = \omega_{ij}\,dy^i\wedge dy^j
	\end{equation*}
	and
	\begin{equation*}
		H =  H^i_{kl}dy^k\wedge dy^l\otimes \frac{\partial}{\partial y^i} \in \Gamma\left(\wedge^2(N)\otimes TN\right).
	\end{equation*}
	Similar to the extension procedure of $\omega$, we also extend $H$ to a small neighborhood $\dbl{N}$ of $N$ and using the coordinate of $\dbl{N}$ to represent $H$. In the following content of paper, we will always use coordinate of $\dbl{N}$ to represent $\omega$ and $H$ defined on $N$ unless giving other specific convention. In particular, taking $U = \frac{\partial}{\partial y^k}$, $V = \frac{\partial}{\partial y^i}$ and $W = \frac{\partial}{\partial y^j}$, by the correspondence \eqref{eq: defi H by omega} we can write 
	\begin{equation*}
		d\omega \left(\frac{\partial}{\partial y^k}, \frac{\partial}{\partial y^i}, \frac{\partial}{\partial y^j}\right) = \frac{\partial \omega_{ij}}{\partial y^k} + \frac{\partial \omega_{jk}}{\partial y^i} + \frac{\partial \omega_{ki}}{\partial y^j} := H^k_{ij}\,.
	\end{equation*}
	Then coefficients of $H$ are anti-symmetric  in the indices $i,j$ and $k$, i.e. 
	\begin{equation}\label{eq:H anti-symmetric}
		H^k_{ij} = - H^k_{ji} \quad \text{and}\quad  H^k_{ij} = - H^i_{kj}\,. 
	\end{equation}
	Let $u:(M,g) \rightarrow (N,h)$ be a critical point of $E^\omega_{\alpha}$ in  $W^{1,2\alpha}(M,N)$, which is called the $\alpha$-$H$-surface.  By Sobolev embedding
	\begin{equation*}
		W^{1,2\alpha}(M,\R^K) \hookrightarrow C^0(M,\R^K),
	\end{equation*}
	the mapping space  $W^{1,2\alpha}(M,N)$ is a smooth, closed, infinite dimensional submanifold of $W^{1,2\alpha}(M,\R^K)$. For each $u \in W^{1,2\alpha}(M,N)$, the tangent space $\mathcal{T}_u$ of Banach manifold $W^{1,2\alpha}(M,N)$ at $u$ can be identified with 
	\begin{equation*}
		\mathcal{T}_u := \left\{V \in W^{1,2\alpha}(M,\R^K)\,:\, V(x) \in T_{u(x)}N \, \text{ for all }x\in M\right\} 
	\end{equation*}
	which is a closed subspace of $W^{1,2\alpha}(M,\R^K)$. 

	\subsection{First and Second Variation of Perturbed Functional \texorpdfstring{{$E^\omega_{\alpha}$}}{Lg}} \label{section 2.2}
	\
	\vskip5pt
	The differential of $E_{\alpha}^\omega$ at $u$ or the first variation of $E_{\alpha}^\omega$ at $u$ , denoted by $\delta E^\omega_\alpha(u)\,:\, \mathcal{T}_u \rightarrow \R$, is defined as following 
	\begin{equation*}
		\delta E^\omega_{\alpha}(u)(V) := \left.\frac{d}{dt}\right|_{t = 0} E^\omega_{\alpha}(u_{V,\,t}) \quad \text{for all } u_{V,\,t} = \exp_{u(x)}{t V(x)} \text{ and } V \in \mathcal{T}_u.
	\end{equation*}
	Moreover, if $\delta E^\omega_{\alpha}(u) = 0$, then the Hessian of $E^\omega_{\alpha}$ at $u$ or the second variation of  $E^\omega_{\alpha}$ at $u$, denoted by $\delta^2 E^\omega_{\alpha}(u)\, :\, \mathcal{T}_u \times \mathcal{T}_u \rightarrow \R$ is defined by 
	\begin{equation*}
		\delta^2 E^\omega_{\alpha}(u)(V,V) := \left.\frac{d^2}{dt}\right|_{t = 0} E^\omega_{\alpha}(u_{V,\,t}) \quad \text{for all } u_{V,\,t} = \exp_{u(x)}{t V(x)} \text{ and } V \in \mathcal{T}_u.
	\end{equation*}
	And by parallelogram law, for any $V,W\in \mathcal{T}_u$ we have
	\begin{equation*}
		\delta^2 E^\omega_{\alpha}(u)(V,W) = \frac{1}{4}\p{\delta^2 E^\omega_{\alpha}(u)(V + W,V + W) - \delta^2 E^\omega_{\alpha}(u)(V - W,V - W)}
	\end{equation*}

	To begin, we compute the first  and second variations of $E^\omega_{\alpha}$. Although we will carry out the computation in choosing a local version, i.e. for compactly supported variations, the choice turns out to be not crucial and the outcome makes sense globally.
	\begin{lemma}\label{variation formula}
		Let $u \in W^{1,2\alpha}(M,N)$ and $V\in \mathcal{T}_u$. Then, the first variation formula of $E^\omega_\alpha$ is 
		\begin{align*}
			\delta E^\omega_{\alpha}(u)(V) &= \int_M \alpha \p{1 + \abs{\nabla u}^2}^{\alpha - 1}\inner{ \nabla u, \nabla V} dV_g\\
			&\quad + \int_M \inner{ H(\nablap u, \nabla u), V} dV_g,
		\end{align*}
		where $\mathcal{P}_u : T_uW^{1,2\alpha}(M,\R^K) \cong W^{1,2\alpha}(M,\R^K) \rightarrow \mathcal{T}_u$ is the orthogonal projection, and  the second variation formula of $E^\omega_{\alpha}$ is 
		\begin{align*}
			\delta^2 E^\omega_{\alpha}(u)(V,V) &= \alpha \int_M \p{1 + \abs{\nabla u}^2}^{\alpha - 1} \Big( \inner{ \nabla V, \nabla V } - R(V,\nabla u, V, \nabla u) \Big) d V_g\\
			& \quad + 2\alpha(\alpha - 1)\int_M \p{1 + \abs{\nabla u}^2}^{\alpha - 2}\langle \nabla u, \nabla V\rangle^2 dV_g\\
			&\quad + 2 \int_M \langle H(\nablap u, \nabla V), V\rangle dV_g \\
			& \quad + \int_M \langle \nabla_V H(\nablap u, \nabla u), V\rangle dV_g,    
		\end{align*}
		where $A$ is  the second fundamental form of embedding $N \hookrightarrow \R^K$.
	\end{lemma}
	\begin{proof}
		As mentioned before, we only need to compute the variation formula for compact supported section $V \in \mathcal{T}_u$. On the one hand, for the first variation of $\alpha$-energy $E_\alpha$, it is well known that
		\begin{equation*}
			\left.\frac{d}{dt}\right|_{t = 0} E_{\alpha}(u_{V,\,t}) = \int_{M}\alpha\p{1 + \abs{\nabla u}^2}^{\alpha - 1} \langle \nabla u, \nabla V\rangle d V_g.
		\end{equation*}
		On the other hand,  in local coordinate $\{y^1, y^2, \dots, y^K\}$ of $\Tilde{N}$  and integration by parts we compute that 
		\begin{align*}
			\left.\frac{d}{dt}\right|_{t = 0}\int_{M} (u_{V,\, t})^*\omega &= \left.\frac{d}{dt}\right|_{t = 0} \int_M \omega_{ij}(u_{V,\, t}) \nablap u_{V,\, t}^i \nabla u_{V,\, t}^j dx\\
			& = \int_{M} \frac{\partial \omega_{ij}}{\partial y^k} \frac{d u_{V,\, t}^k}{dt} \nablap u^i \nabla u^j dx + \int_{M} \omega_{ij}(u) \nablap V^i \nabla u^j dx \\
			&\quad + \int_M \omega_{ij} \nablap u^i \nabla V^j dx\\
			& = \int_M \p{ \frac{\partial \omega_{ij}}{\partial y^k} + \frac{\partial \omega_{jk}}{\partial y^i} + \frac{\partial \omega_{ki}}{\partial y^j} } \nablap u^i \nabla u^j V^k dx\\
			& = \int_M \langle H(\nablap u, \nabla u), V  \rangle dV_g.
		\end{align*}
		Consequently, we obtain the following first variation formula for $E_{\alpha}^\omega$
		\begin{align*}
			\delta E^\omega_{\alpha}(u)(V) &= \int_M \alpha \p{1 + \abs{\nabla u}^2}^{\alpha - 1}\langle \nabla u, \nabla V\rangle dV_g\\
			&\quad + \int_M \langle H(\nablap u, \nabla u), V\rangle dV_g .
		\end{align*}
		Next, we turn to compute the Hessian of $\delta^2 E^\omega_{\alpha}$,  from the definition we compute the $\delta^2 E^\omega_{\alpha}(u)(V,V)$ by taking the derivative of the following expression with respect to $t$ at $t = 0$:
		\begin{align}\label{eq:second vari eq 1}
			&\int_{M}\alpha\p{1 + \abs{\nabla u_{V,\, t}}^2}^{\alpha - 1} \inner{ \nabla u_{V,\, t}, \nabla \frac{d u_{V,\, t}}{dt}} d V_g \nonumber\\
            &\quad + \int_{M} \inner{ H(\nablap u_{V,\,t}, \nabla u_{V,\,t}), \frac{d u_{V,t}}{dt} } dV_g 
		\end{align}
		To begin with, we can differentiate the first term in \eqref{eq:second vari eq 1} to obtain
		\begin{align}\label{eq:second vari eq 2}
			\left.\frac{d}{dt}\right|_{t = 0} & \int_{M}\alpha\p{1 + \abs{\nabla u_{V,\, t}}^2}^{\alpha - 1} \inner{ \nabla u_{V,\, t}, \nabla \frac{d u_{V,\, t}}{dt}} d V_g\nonumber \\
			&\quad = 2\alpha(\alpha - 1)\int_{M} \p{1 + \abs{\nabla u_{V,\, t}}^2}^{\alpha - 2} \langle \nabla u, \nabla V\rangle^2 d V_g\nonumber \\
			&\quad\quad +  \int_{M}2\alpha\p{1 + \abs{\nabla u}^2}^{\alpha - 1} \p{\langle \nabla V, \nabla V\rangle + \inner{ \nabla u, \nabla_V \nabla_{\nabla u} V}}d V_g\nonumber \\
			&\quad = 2\alpha(\alpha - 1)\int_{M} \p{1 + \abs{\nabla u}^2}^{\alpha - 2} \langle \nabla u, \nabla V\rangle^2 d V_g\nonumber\\
			&\quad\quad +  \int_{M}\alpha\p{1 + \abs{\nabla u}^2}^{\alpha - 1} \p{\langle \nabla V, \nabla V\rangle - R(V,\nabla u, V, \nabla u) } d V_g,
		\end{align}
		where $R$ is the Riemann curvature tensor on $N$. Next, we consider the mean curvature type  vector field part in \eqref{eq:second vari eq 1}. Before penetrating into details, we note that 
		\begin{equation*}
			\left.\frac{d^2 u_{V,\, t}}{dt^2}\right|_{t = 0} = \left.\frac{d^2}{dt^2}\right|_{t = 0} \exp_{u(x)}{t V(x)} = A(V,V),
		\end{equation*}
		which is perpendicular to the tangent bundle $TN$. Utilizing this observation and integrating by parts, we can compute
		\begin{align}\label{eq: symmetric of D H}
			\left.\frac{d}{dt}\right|_{t = 0} &  \int_{M} \inner{ H(\nablap u_{V,\,t}, \nabla u_{V,\,t}), \frac{d u_{V,t}}{dt} } dV_g \nonumber\\
			&= \int_M \frac{\partial H^k_{ij}}{\partial y^l} V^l V^k \nablap u^i \nabla u^j dx + \int_M H^k_{ij} V^k \nablap V^i \nabla u^j dx\nonumber\\
			&\quad +\int_M H^k_{ij} V^k \nablap u^i \nabla V^j dx\nonumber \\
			& =  \int_M \langle \nabla_{V} H(\nablap u, \nabla u), V\rangle dV_g + 2 \int_M \langle H(\nablap u, \nabla V), V\rangle dV_g.
		\end{align}
		Combining computations \eqref{eq:second vari eq 2} and \eqref{eq: symmetric of D H} will yield the second variation formula of $E^\omega_{\alpha}$ formulated in Lemma \ref{variation formula}.
	\end{proof}
	For $\alpha > 1$, by a similar computations as Lemma \ref{variation formula}, the Euler-Lagrange equation of critical points of $E^\omega_{\alpha}$, which are called $\alpha$-$H$-surfaces, can be written as
	\begin{align}
		\label{el}
		\Delta  u_\alpha + (\alpha - 1)\frac{\nabla |\nabla  u_\alpha|^2\cdot \nabla  u_\alpha}{1+|\nabla  u_\alpha|^2} &+ A(u_\alpha)\left(\nabla  u_\alpha, \nabla  u_\alpha\right)\nonumber\\
		&= \frac{H(u_\alpha)(\nablap  u_\alpha, \nabla  u_\alpha) }{\alpha \left(1 + |\nabla  u_\alpha|^2\right)^{\alpha - 1}},
	\end{align}
	or equivalently in divergence form
	\begin{align}
		\label{el1}
		\diver\left(\left(1+|\nabla  u_\alpha|^2\right)^{\alpha - 1}\nabla  u_\alpha\right) &+ \left(1 + |\nabla  u_\alpha|^2\right)^{\alpha -1}A(u_\alpha)(\nabla  u_\alpha,\nabla  u_\alpha)\nonumber\\
		&= \frac{1}{\alpha}H(u_\alpha)(\nablap  u_\alpha, \nabla  u_\alpha).
	\end{align}
	Using parallelogram law, from Lemma \ref{variation formula}, we can write the Hessian of $E^\omega_{\alpha}$ as following.
	\begin{coro}\label{coro hessian}
		For $V,W \in \T_u$, we can compute the Hessian for $E^\omega_{\alpha}(u)$ at some critical point $u \in W^{1,2\alpha}(M,N)$
		\begin{align*}
			\delta^2 E^\omega_{\alpha}(u)(V,W) &= \alpha \int_M \p{1 + \abs{\nabla u}^2}^{\alpha - 1} \Big( \langle \nabla V, \nabla W\rangle - R(V,\nabla u, W, \nabla u) \Big) d V_g\\
			& \quad + 2\alpha(\alpha - 1)\int_M \p{1 + \abs{\nabla u}^2}^{\alpha - 1}\langle \nabla u, \nabla V\rangle \langle \nabla u, \nabla W\rangle  dV_g\\
			&\quad + \int_M \Big( \inner{ H(\nablap u, \nabla V), W} + \inner{ H(\nablap u, \nabla W), V} \Big)  dV_g\\
			&\quad + \frac{1}{2} \int_M \Big( \inner{  (\nabla_V H)(\nablap u, \nabla u), W} + \inner{(\nabla_W H)(\nablap u, \nabla u), V} \Big) dV_g 
		\end{align*}
		where $A(\cdot,\cdot)$ is the second fundamental form of embedding $N\hookrightarrow \R^K$.
	\end{coro}
	In order to derive a priori estimate that will be used in later sections, especially when it comes to proving Lemma \ref{P-S condition}, we need a formula for $\delta E^\omega_{\alpha}(\mathcal{P}_u(\varphi))$ with $\varphi \in W^{1,2\alpha}(M,\R^K)$. To simplify the notation, we set
	\begin{equation*}
		G_{\alpha}^\omega(u) : = \delta E^\omega_{\alpha} \circ \mathcal{P}_u
	\end{equation*}
	and define the norm of $G_{\alpha}^\omega(u) : W^{1,2\alpha}(M,\R^K) \rightarrow \R$ by
	\begin{equation*}
		\norm{G_{\alpha}^\omega(u)} := \sup\left\{\abs{G_{\alpha}^\omega(u)(\varphi)}\,:\, \varphi \in W^{1,2\alpha}(M,\R^K)\, \text{ with } \norm{\varphi}_{W^{1,2\alpha}(M,\R^K)}\leq 1\right\}.
	\end{equation*}
	Based on above conventions, we can obtain the following estimates:
	\begin{lemma}\label{property G}
		The following properties for $G^\omega_{\alpha}$ holds:
		\begin{enumerate}
			\item \label{proper G part 1} Given $u \in W^{1,2\alpha}(M,N)$, for any $\varphi \in W^{1,2\alpha}(M,\R^K)$ we have that 
			\begin{align}\label{eq 4}
				G^\omega_{\alpha}(u)(\varphi) &= \int_{M} \alpha \p{1 + \abs{\nabla u}^2}^{\alpha - 1}\Big(\nabla u \cdot \nabla \varphi  - A(u)(\nabla u, \nabla u)\cdot \varphi \Big) dV_g \nonumber\\
				&\quad +\int_M H(\nablap u,\nabla u)\cdot \varphi dV_g ;
			\end{align}
			\item \label{proper G part 2} For all $L > 0$, there exists constant $C_L > 0$ such that 
			\begin{equation*}
				\norm{G^\omega_{\alpha}(u)} \leq C_L
			\end{equation*}
			whenever $\alpha > 1$ and $u \in W^{1,2\alpha}(M,N)$ satisfying $\norm{u}_{W^{1,2\alpha}(M,N)} \leq L$;
			\item \label{proper G part 3} For all $L > 0$, there exists constant $C_L > 0$ such that 
			\begin{equation*}
				\norm{G^\omega_{\alpha}(u_1) - G^\omega_{\alpha}(u_2)} \leq C_L \norm{u_1 - u_2}_{W^{1,2\alpha}(M,N)}
			\end{equation*}
			whenever $\alpha > 1$, and $\norm{u_1}_{W^{1,2\alpha}(M,\R^K)}$,  $\norm{u_2}_{W^{1,2\alpha}(M,\R^K)} \leq L$.
		\end{enumerate}
	\end{lemma}
	\begin{proof}
		For the part \eqref{proper G part 1}, using the first variation formula in Lemma \ref{variation formula}, we get
		\begin{align}\label{eq: first vari to G}
			\delta E^\omega_{\alpha}(u)(\mathcal{P}_{u}(\varphi)) &=  \int_M \alpha \p{1 + \abs{\nabla u}^2}^{\alpha - 1}\inner{ \nabla u, \nabla \mathcal{P}_{u}(\varphi) } dV_g \nonumber\\
			&\quad + \int_M \inner{ H(\nablap u, \nabla u), \mathcal{P}_{u}(\varphi)} dV_g\nonumber\\
			& = - \alpha \int_M \left\langle \diver\p{ \p{1 + \abs{\nabla u}^2}^{\alpha - 1} \nabla u}, \mathcal{P}_{u}(\varphi) \right\rangle dV_g\nonumber\\
			&\quad + \int_M \inner{ H(\nablap u, \nabla u), \mathcal{P}_{u}(\varphi)} dV_g.
		\end{align}
		Recalling that 
		\begin{align*}
			\mathcal{P}_u\set{\diver\p{ \p{1 + \abs{\nabla u}^2}^{\alpha - 1} \nabla u}} &= \diver\p{ \p{1 + \abs{\nabla u}^2}^{\alpha - 1} \nabla u}\\
			&\quad - \p{1 + \abs{\nabla u}^2}^{\alpha - 1}A(\nabla u, \nabla u),
		\end{align*}
		we decompose $\varphi = \mathcal{P}_{u}(\varphi) + \varphi^\perp$ where $\varphi^\perp$ is normal component of $\varphi$ in $TN^\perp$ and plug this into \eqref{eq: first vari to G} to obtain 
		\begin{align*}
			&\delta G^\omega_{\alpha}(u)(\varphi)=\delta E^\omega_{\alpha}(u)(\mathcal{P}_{u}(\varphi))\nonumber\\
			&= - \alpha \int_M  \diver\p{ \p{1 + \abs{\nabla u}^2}^{\alpha - 1} \nabla u} \cdot \varphi dV_g + \int_M  H(\nablap u, \nabla u)\cdot \varphi dV_g\\
			&\quad - \alpha\int_M \p{1 + \abs{\nabla u}^2}^{\alpha - 1} A(\nabla u, \nabla u) \cdot \varphi\, dV_g\\
			&= \alpha \int_M  \p{1 + \abs{\nabla u}^2}^{\alpha - 1} \left\langle  \nabla u, \nabla\varphi \right\rangle dV_g + \int_M  H(\nablap u, \nabla u)\cdot \varphi dV_g\\
			&\quad - \alpha\int_M \p{1 + \abs{\nabla u}^2}^{\alpha - 1}A(\nabla u, \nabla u) \cdot \varphi dV_g,
		\end{align*}
		which is exactly the desired of part \eqref{proper G part 1}.
		For part \eqref{proper G part 2}, using the formula \eqref{eq 4} we straightforward estimate that 
		\begin{align*}
			\abs{G^\omega_{\alpha}(u)(\varphi)} &\leq \int_M \alpha \p{1 + \abs{\nabla u}^2}^{\alpha - 1} \abs{\nabla \varphi} \cdot \abs{\nabla u}+ \alpha \p{1 + \abs{\nabla u}^2}^{\alpha}\cdot\abs{\varphi} d V_g\\
			&\quad + \int_M C \abs{\nabla u}^2\cdot \abs{\varphi} dV_g\\
			&\leq C_L^\prime\Big( \norm{\nabla \varphi}_{L^{2\alpha}(M,\R^K)} + \norm{\varphi}_{L^\infty(M,\R^K)} \Big)\\
            &\leq C_L \norm{\varphi}_{W^{1,2\alpha}(M,\R^K)}.
		\end{align*}
		Here, in the last inequality we used the Sobolev embedding 
		$$W^{1,2\alpha}(M,\R^K) \hookrightarrow C^0(M,\R^K).$$
		
		Next, we consider the part \eqref{proper G part 3}. To begin, we use formula \eqref{eq 4} obtained in part \eqref{proper G part 1} to get
		\begin{align}\label{eq 6}
			&G^\omega_{\alpha}(u_1)(\varphi) - G^\omega_{\alpha}(u_2)(\varphi)\nonumber\\
			&= \int_{M} \alpha \Big(\p{1 + \abs{\nabla u_1}^2}^{\alpha - 1}\nabla u_1 \cdot \nabla \varphi - \p{1 + \abs{\nabla u_2}^2}^{\alpha - 1}\nabla u_2 \cdot \nabla \varphi\Big)dV_g\nonumber \\
			& \quad - \int_M\alpha \p{1 + \abs{\nabla u_1}^2}^{\alpha - 1} \Big( A(u_1)(\nabla u_1, \nabla u_1) -A(u_2)(\nabla u_2, \nabla u_2)\Big)\cdot \varphi dV_g \nonumber\\
			&\quad -\int_M\alpha \left(\p{1 + \abs{\nabla u_2}^2}^{\alpha - 1} - \p{1 + \abs{\nabla u_1}^2}^{\alpha - 1}\right)A(u_2)(\nabla u_2, \nabla u_2)\cdot \varphi dV_g \nonumber\\
			& \quad + \int_M \Big(H(\nablap u_1,\nabla u_1)- H(\nablap u_2,\nabla u_2)\Big)\cdot \varphi dV_g .
		\end{align}
		Based on part \eqref{proper G part 2} and Sobolev embedding  $W^{1,2\alpha}(M,\R^K) \hookrightarrow C^0(M,\R^K)$, we can simplify the estimates by focusing on the case when
		\begin{equation*}
			\norm{u_1 - u_2}_{C^0(M,N)} \leq \delta_0
		\end{equation*}
		for some small $\delta_0 > 0$. Here, we choose small enough $\delta_0 >0$ to ensure that 
		$$tu_1 + (1 - t)u_2 \in W^{1,2\alpha}(M,\Tilde{N}),$$
		hence the formula obtained in part \eqref{proper G part 1} can be applied to such convex combination. we can estimate the integrand of the first integral on the right hand side of \eqref{eq 6} to obtain
		\begin{align*}
			&\alpha \p{1 + \abs{\nabla u_1}^2}^{\alpha - 1}\nabla u_1 \cdot \nabla \varphi - \alpha\p{1 + \abs{\nabla u_2}^2}^{\alpha - 1}\nabla u_2 \cdot \nabla \varphi \\
			& \quad = \int_{0}^1  \alpha(\alpha - 1)\p{1 + |\nabla (t u_1 + (1 - t)u_2)|^2}^{\alpha - 2}\cdot\\
			&\quad \quad \quad \quad\quad \quad\quad \quad \quad\quad \quad\quad\quad \quad [\nabla\p{t u_1 + (1 - t)u_2} \cdot \nabla \varphi] [\nabla(u_1 - u_2) \cdot \nabla \varphi]dt\\
			& \quad \quad + \int_{0}^1 \alpha \Big(1 +|t u_1 + (1 - t)u_2|^2\Big)^{\alpha - 1}\nabla(u_1 - u_2)\cdot \nabla \varphi dt.
		\end{align*}
		From this identity and H\"{o}lder's inequality, the first integral on the right-hand side will be bounded by 
		$$C_L \norm{u_1 - u_2}_{W^{1,2\alpha}(M,\R^K)}\cdot\norm{\varphi}_{W^{1,2\alpha}(M,\R^K)}.$$
		The remaining terms in \eqref{eq 6} can be estimated in a complete similar manner by using Sobolev embedding $W^{1,2\alpha}(M,\R^K) \hookrightarrow C^0(M,\R^K)$ and H\"{o}lder's inequality.
	\end{proof}
	
	\subsection{Palais-Smale Condition and Regularity of Critical Points for Functional \texorpdfstring{{ $E^\omega_{\alpha}$}}{Lg}} \label{section 2.3}
	\ 
	\vskip5pt
	The Palais-Smale conditions are a set of compactness conditions that are essential in variational analysis. Our first objective in this subsection is to confirm that the functional $E^\omega_{\alpha}: W^{1,2\alpha}(M,N) \rightarrow \mathbb{R}$ satisfies a version of the Palais-Smale condition, which was also pointed out by Moore \cite[pp. 212]{moore2017}.
	\begin{lemma}\label{P-S condition}
		For $\alpha > 1$, the functional $E^\omega_{\alpha}: W^{1,2\alpha}(M,N)\rightarrow \R$ satisfies the Palais-Smale condition, more precisely, let $\{u_j\}$ be a sequence in $W^{1,2\alpha}$ $(M,N)$ satisfying
		\begin{enumerate}
			\item $E^\omega_{\alpha}(u_j) \leq C$ for some universal constant $C$ independent of $j \in \mathbb{N}$;
			\item $\norm{\delta E^\omega_{\alpha}(u_j)} \rightarrow 0$ as $j \rightarrow \infty$,
		\end{enumerate}
		then, after passing to a subsequence if necessary, $u_j$ converges strongly in $W^{1,2\alpha}(M,N)$ to a critical point of $E^\omega_{\alpha}$ as $j \rightarrow \infty$.
	\end{lemma}
	\begin{proof}
		Let $\{u_j\}_{j \in \mathbb N}$ be a  sequence in $W^{1,2\alpha}(M,N) \subset W^{1,2\alpha}(M,\R^K)$ such that $E^\omega_{\alpha}(u_j)$ is uniformly bounded and $\norm{\delta E^\omega_{\alpha}(u_j)} \rightarrow 0$ for fixed $\alpha > 1$. It follows that the $\alpha$-energy of $u_j$, $E_\alpha(u_j)$, is also uniformly bounded. Otherwise if, after choosing a subsequence, $E_\alpha(u_j) \rightarrow \infty$ as $j\rightarrow \infty$, then 
		\begin{align*}
			\abs{E^\omega_{\alpha}(u_j)} &\geq \frac{1}{2} \int_M\p{1 + \abs{\nabla u_j}^2}^\alpha dV_g - \frac{1}{2}\norm{\omega}_{L^\infty(N)} \int_M \abs{\nabla u_j}^2 dV_g\\
			& \geq  \frac{1}{4} \int_M\p{1 + \abs{\nabla u_j}^2}^\alpha dV_g \rightarrow \infty,\quad \text{as }j \rightarrow \infty,
		\end{align*}
		which contradicts to the uniformly boundedness of  $E^\omega_{\alpha}$. From this observation, by Sobolev embedding $W^{1,2\alpha}(M,N) \hookrightarrow W^{1,2\alpha}(M,\R^K) \hookrightarrow C^{0,1 - 1/\alpha}(M,\R^K)$,  every $u_j$ is H\"{o}lder continuous. And together with the compactness of  $N$, it follows that $\{u_j\}_{j \in \mathbb N}$ is equi-continuous.  Therefore, thanks to the Arzela-Ascoli Theorem, there exists a subsequence of $\{u_j\}$, also denoted by $\{u_j\}$, converges uniformly to a continuous map $u_0 : M \rightarrow N$. To complete the proof of the Lemma \ref{P-S condition}, it is sufficient to demonstrate that $ \{ \nabla u_j \} $ is a Cauchy sequence in $ L^ { 2 \alpha } (M,N) $. 
		
		Recall that $\mathcal{P} : \R^K \times \R^K \cong T\R^K\rightarrow TN$ is the pointwise orthogonal projection from $\R^K \times \R^K \cong T\R^K$ onto $TN$ and simplify the notation of $(\mathcal{P}\circ u) (V)$ to $\mathcal{P}_u(V
		)$ for any $V\in T_uW^{1,2\alpha}(M,\mathbb{R}^K)$. This allows us to make the following estimates
		\begin{align*}
			\norm{d \mathcal{P}_u(V)}_{L^{2\alpha}(M,\R^K)} &\leq \norm{d \p{ \mathcal{P}\circ u}(V)}_{L^{2\alpha}(M,\R^K)} + \norm{\mathcal{P}_u\left( \nabla V\right)}_{L^{2\alpha}(M,\R^K)}\\
			&\leq C \norm{V}_{L^\infty(M, \R^K)} \norm{\nabla u}_{L^{2\alpha}(M,\R^K)} + C \norm{\nabla V}_{L^{2\alpha}(M,\R^K)}
		\end{align*}
		which further implies that 
		\begin{align}\label{eq bounded P}
			\norm{\mathcal{P}_u(V)}_{W^{1,2\alpha}(M,\R^K)} &\leq C\p{\norm{V}_{L^\infty(M, \R^K)} E_\alpha(u) +  \norm{\nabla V}_{L^{2\alpha}(M,\R^K)} + \norm{V}_{L^{2\alpha}(M,\R^K)} } \nonumber\\
			& \leq C \p{1 + E_\alpha(u)} \norm{V}_{W^{1,2\alpha}(M,\R^K)}. 
		\end{align}
		Since $E_\alpha (u_j)$ is uniformly bounded and $u_j$ is also uniformly bounded, then $\{u_j\}$ is uniformly bounded in $W^{1,2\alpha}(M, \R^K)$ and \eqref{eq bounded P} tells us that $\mathcal{P}_u(u_i - u_j)$ is bounded in $W^{1,2\alpha}(M, \R^K)$ for each $ u \in W^{1,2\alpha}(M,\R^K)$. Because $\norm{\delta E^\omega_{\alpha}(u_j)} \rightarrow 0$ and recall part \ref{proper G part 3} Lemma \ref{property G} , we have
		\begin{equation}\label{E 52}
			\abs{\delta G^\omega_{\alpha}(u_i)(u_i - u_j) - \delta G^\omega_{\alpha}(u_j)(u_i - u_j)} \longrightarrow 0 \quad \text{as }i,\, j \rightarrow \infty.
		\end{equation}
		By part \eqref{proper G part 1} of Lemma \ref{property G}, replacing $\varphi$ with $u_i - u_j$ to obtain
		\begin{align*}
			&G^\omega_{\alpha}(u)(u_i - u_j)\\
			&= \int_{M} \alpha \p{1 + \abs{\nabla u}^2}^{\alpha - 1}\Big(\nabla u \cdot \nabla (u_i - u_j)  - A(u)(\nabla u, \nabla u)\cdot (u_i - u_j) \Big) dV_g \nonumber\\
			&\quad +\int_M H(\nablap u,\nabla u)\cdot (u_i - u_j) dV_g .
		\end{align*}
		Plugging this identity into \eqref{E 52}, we obtain
		\begin{align}\label{E 53}
			&\left|\int_M  \alpha \left[ \p{1 + \abs{\nabla u_i}^2}^{\alpha - 1}  \nabla u_i\cdot \nabla(u_i - u_j)\right.\right. \nonumber\\
			&\quad \quad-  \left.\left.\p{1 + \abs{\nabla u_j}^2}^{\alpha - 1}   \nabla u_j\cdot \nabla(u_i - u_j)\right] dV_g \right.\nonumber\\
			& \quad \quad + \int_M \left[ H(\nablap u_i, \nabla u_i)\cdot (u_i - u_j)  -  H(\nablap u_j, \nabla u_j)\cdot(u_i - u_j) \right]dV_g\nonumber\\
			&\quad \quad - \int_M\p{1 + \abs{\nabla u_i}^2}^{\alpha - 1} A(\nabla u_i, \nabla u_i) \cdot (u_i - u_j) dV_g\nonumber\\
			& \quad \quad \left.- \int_M \p{1 + \abs{\nabla u_j}^2}^{\alpha - 1} A(\nabla u_j, \nabla u_j) \cdot (u_i - u_j) dV_g\right|\nonumber\\
			&\quad \quad\quad \quad \longrightarrow 0 \quad \text{as } i, j\rightarrow\infty.
		\end{align}
		We consider the above asymptotic quantity \eqref{E 53} term by term, first we note that
		\begin{align*}
			&\left| \int_M \p{1 + \abs{\nabla u_i}^2}^{\alpha - 1} A(\nabla u_i, \nabla u_i) \cdot (u_i - u_j) dV_g\right|\\
			&\quad \quad\quad \quad\leq 2\norm{u_i - u_j}_{L^\infty(M,\R^K)} \norm{A}_{L^\infty(N)} E_\alpha(u_i) \longrightarrow 0 \quad \text{as } i,j\rightarrow\infty.
		\end{align*}
		And similarly, 
		\begin{align*}
			\abs{\int_M H(\nablap u_i, \nabla u_i)\cdot (u_i - u_j) dV_g} &\leq 2\norm{u_i - u_j}_{L^\infty(M,\R^K)} \norm{H}_{L^\infty(N)} E(u_i)\\
			&\quad\longrightarrow 0\quad  \text{as } i,j\rightarrow\infty.
		\end{align*}
		Thus, \eqref{E 53} is equivalent to 
		\begin{multline*}
			\left|\int_M  \alpha \left[ \p{1 + \abs{\nabla u_i}^2}^{\alpha - 1}   \nabla u_i\cdot \nabla(u_i - u_j) \right.\right. \\  \left.\left. -  \p{1 + \abs{\nabla u_j}^2}^{\alpha - 1}  \nabla u_j\cdot \nabla(u_i - u_j) \right] dV_g \right| \longrightarrow 0 \quad \text{as } i,j\rightarrow\infty,
		\end{multline*}
		and this further implies that 
		\begin{align*}
			&\int_M \abs{\nabla u_i - \nabla u_j}^{2\alpha} d V_g\\
			&\quad \leq C \int_M \int_0^1 \abs{\nabla u_j + t (\nabla u_i - \nabla u_j)}^{2\alpha - 2} \abs{\nabla u_i - \nabla u_j}^2 dt d V_g\\
			&\quad\leq C \int_{M}\int_0^1 \int_0^1 d^2\left(1 + \abs{\nabla u}^2\right)^\alpha(\nabla u_j + t (\nabla u_i - \nabla u_j))(\nabla u_i - \nabla u_j) dt dV_g \\
			&\quad\leq C\left|\int_M  \alpha \left[ \p{1 + \abs{\nabla u_i}^2}^{\alpha - 1}  \nabla u_i\cdot \nabla(u_i - u_j) \right.\right.\\
			&\quad\quad \quad  \left.\left. -  \p{1 + \abs{\nabla u_j}^2}^{\alpha - 1}  \nabla u_j\cdot \nabla(u_i - u_j) \right] dV_g \right| \longrightarrow 0 \quad \text{as } i,j\rightarrow\infty,
		\end{align*}
		that is, $\set{\nabla u_j}$ is Cauchy in $L^{2\alpha}(M,\R^K)$. In summary, by combining the fact that $u_j \rightarrow u_0$ in $C^0(M, \mathbb{R}^K)$ with that $\{\nabla u_j\}_{j \in \mathbb N}$ is a Cauchy sequence in $L^{2\alpha}(M, \mathbb{R}^K)$ and the completeness of $W^{1,2\alpha}(M, \mathbb{R}^K)$ , we can conclude the desired assertion of Lemma \ref{P-S condition}.
	\end{proof}
	Based on the beginning part of the proof for Lemma \ref{P-S condition}, it can be argued that the functional $E^{\omega}_{\alpha}$ is bounded from below, although the lower bound may be dependent on the choice of $\alpha > 1$. Thus, by a classical consequence of variational analysis in \cite{palais1966foundations}, we have 
	\begin{coro}\label{coro palais}
		For functional $E^\omega_{\alpha} : W^{1,2\alpha}(M,N) \rightarrow \R$, the followings hold
		\begin{enumerate}
			\item $E^\omega_{\alpha}$ attains its minimum value in every component of $W^{1,2\alpha}(M,N)$;
			\item If there are no critical points of $E^\omega_{\alpha}$ in the interval $[a,b]$, then there exists a deformation retraction
			\begin{equation*}
				\varrho : (E^{\omega}_{\alpha})^{-1}(-\infty, b] \rightarrow (E^{\omega}_{\alpha})^{-1}(-\infty, a].
			\end{equation*}
		\end{enumerate}
	\end{coro}
	
	We will now examine the smoothness of the critical points $u \in W^{1,2\alpha}(M,N)$ with respect to the functional $E^\omega_{\alpha}$ by following the argument in \cite[Proposition 2.3]{sacks1981existence}, which was also indicated by Moore \cite[Remark 4.4.8.]{moore2017}.  To ensure the comprehensiveness of our content, we provide an outline of the proof here.
	\begin{lemma}\label{lem regularity}
		For sufficiently small  $ \alpha_0 - 1 $, all critical points of $ E^ \omega_{\alpha}: W^{1,2\alpha}(M,N) \rightarrow \R $ are smooth. 
	\end{lemma}
	\begin{proof}
		By Sobolev embedding $W^{1,2\alpha}(M,N) \subset W^{1,2\alpha}(M,\R^K) \hookrightarrow C^{0,1 - 1/\alpha}(M,\R^K)$, the critical point $u\in W^{1,2\alpha}(M,N)$ is H\"{o}lder continuous. The Euler-lagrange equation for $u$ can be written as
		\begin{align}
			\label{eq:lem regu 1}
			\diver\left(\left(1+|\nabla  u|^2\right)^{\alpha - 1}\nabla  u\right) &+ \left(1 + |\nabla  u|^2\right)^{\alpha -1}A(u)(\nabla  u,\nabla  u)\nonumber \\
			&= \frac{1}{\alpha}H(u)(\nablap  u, \nabla  u)
		\end{align}
		in weak sense. This is a quasi-linear uniformly elliptic system when $\alpha - 1$ is small enough. Based on standard variational analysis consequences, such as in \cite[Theorem 1.11.1]{morrey1966} or \cite[Chapter 8.3, Theorem 1]{Evans2010}, it can be concluded that $u$ belongs to the space $W^{1,2\alpha}(M,N)$. Thus, the Euler-Lagrange equation \eqref{eq:lem regu 1} can be written pointwisely as
		\begin{equation*}
			\Delta u + (\alpha - 1)\frac{\nabla_{}|\nabla_{} u|^2\cdot \nabla_{} u}{1+|\nabla_{} u|^2}  = \frac{H(u)(\nablap_{} u, \nabla_{} u)}{\alpha \left(1 + |\nabla_{} u|^2\right)^{\alpha - 1}} -  A(u)\left(\nabla_{} u, \nabla_{} u\right). 
		\end{equation*}
		The right hand side of the above equation belongs to $L^{\alpha}(M,N)$, $\alpha > 1$, for $u \in W^{1,2\alpha}(M,N)$. From the $L^p$ theory of uniformly elliptic equations, we can conclude that $u$ belongs to  $W^{2,2\alpha}(M,N)$.  Then, the conclusion of Lemma \ref{lem regularity} follows by applying standard elliptic PDE's bootstrapping argument.
	\end{proof}
	
	
	\vskip2cm
	\section{Existence of non-trivial Critical Points of the Perturbed Functional \texorpdfstring{$E^{\lambda\omega}_{\alpha}$}{Lg}}\label{sec: 3 non-constant u alpha}
	\vskip10pt
	
	In this section, we combine the analytic preliminaries established in previous Section \ref{section 2} with the min-max theory for perturbed functional $E^\omega_\alpha$, that will be described in this section, to establish the existence of a sequences of non-trivial critical points $u_{\alpha_j}$ of $E^\omega_{\alpha_j}$. More precisely, for each fixed $\omega \in C^2(\wedge^2(N))$, we can find a  generic choice of $\lambda \in \R_+$ to construct a sequence of non-constant critical points ${u_{\alpha_j}}$ of $E^{\lambda\omega}_{\alpha_j}$ (see \eqref{eq:E lambda alpha psi} below for the definition of $E^{\lambda \omega}_\alpha$) with bounded $\alpha_j$-energy that is uniformly with respect to $j \in \mathbb N$, see Proposition \ref{prop energy bound}, and with bounded Morse index from above, see Theorem \ref{thm:Morse index k-2}. The main result of this section is summarized in Corollary \ref{coro:section 3 summary}.
	\ 
	\vskip5pt
	\subsection{Construction of Min-Max Type Critical Value.} \label{section 3.1}
	\ 
	\vskip5pt
	In this subsection, we construct the min-max type critical value by introducing a higher dimensional version of \textit{width}, in the spirit of \cite{Colding-Minicozzi2008b} for finding minimal 2-sphere in Riemannian 3-sphere. Furthermore, we always assume that $N$ is a closed Riemannian manifold with $\pi_k(N) \neq 0$ for some $k \geq 2$ and $M$ is the standard 2-sphere $\S^2$ which is same as the assumption of our main Theorem \ref{main theorem 1} and Theorem \ref{main theorem 2}. 
	
	Then, we can choose a least integer $k \geq 2$ such that $\pi_k(N) \neq 0$. Let 
 $$I^{k-2} = \{t=(t^1,t^2,\dots,t^{k-2})\,:\, 0\leq t^i \leq 1, 1\leq i\leq k-2\}$$ 
 be the $k-2$ dimensional unit complex cube. We define a sweepout as a continuous map $\sigma : I^{k-2} \rightarrow W^{1,2\alpha}(\S^2,N)$, which also can be viewed as a map $\sigma(\cdot,t):\S^2\times I^{k-2} \rightarrow N$, such that $\sigma$ assigns $\partial I^{k-2}$ to constant maps which can be identified with the points in $N$ and the map $f_\sigma : \S^k \rightarrow N$ induced by $\sigma$ represents a non-trivial free homotopy class in $ \pi_{k}(N)$. Considering that $W^{1,2\alpha}(\mathbb{S}^2,N)\hookrightarrow C^0(M,N)$ when $\alpha>1$, it can be inferred that the induced map $f_\sigma$ is continuous. Then
	take a non-trivial homotopy class $[\iota] \in \pi_k(N)$ and we define the set of  admissible sweepouts as below
	\begin{equation*}
		\mathscr{S} = \left\{\sigma \in C^0(I^{k-2},W^{1,2\alpha}(\S^2,N))\,:\, \sigma(t) \text{ is constant for } t\in \partial I^{k-2} \text{ and } f_\sigma \in [\iota]\right\}
	\end{equation*}
	Thus, the corresponding min-max value for $E^\omega_{\alpha}$ is defined as following
	\begin{equation*}
		\mathcal{W}_{\alpha,\omega} := \inf_{\sigma \in \mathscr{S}} \sup_{t \in I^{k-2}} E^\omega_{\alpha} (\sigma(t))
	\end{equation*}
	which is called \textit{width} depicted by the functional $E^\omega_{\alpha}$ upon $W^{1,2\alpha}(\S^2 , N)$. It is important to note that for every admissible sweepout $\sigma$, the boundary of the complex cube $\partial I^{k-2}$ is mapped to constant maps. This implies that
	\begin{equation*}
		0\leq \mathcal{W}_{\alpha, \omega} < \infty
	\end{equation*}
	which represents a well-defined real-valued number depending on the choice of $\alpha > 1$ and $\omega$.
	\ 
	\vskip5pt
	\subsection{\texorpdfstring{{$\alpha$}}{Lg}-Energy Estimates for Min-Max Type Critical Points}\label{section 3.2}
	\ 
	\vskip5pt
	In this subsection, we apply the idea of Struwe's monotonicity technique  \cite{Struwe-1988} to obtain uniformly $\alpha_j$-energy estimates for some specific subsequence of critical points $\{u_{\alpha_j}\}_{j\in \mathbb N}$ of the perturbed functional $E^{\lambda\omega}_{\alpha_j}$ for generic choice of $\lambda > 0$, see \eqref{eq:E lambda alpha psi} below for the definition of $E^{\lambda \omega}_\alpha$. This is essential in the process of obtaining the existence of $H$-sphere by letting $\alpha_j \searrow 1$ for suitable choice of sequences $\alpha_j \searrow 1$.
	
	In order to apply the monotonicity technique, we need to bring in a scalar parameter in our perturbed functional as follows
	\begin{equation}\label{eq:E lambda alpha psi}
		E^{\lambda\omega}_{\alpha}(u) : = E_\alpha(u) + \lambda \int_{\S^2} u^*\omega
	\end{equation}
	for given $\omega \in C^2(\wedge^2(N))$ and positive number $\lambda > 0$. And we use abbreviated notation $\mathcal{W}_{\alpha,\lambda}$ to denote the min-max value $\mathcal{W}_{\alpha,\lambda\omega}$ constructed in previous Section \ref{section 3.1} and  $u_\alpha$ is denoted to be the critical points of $E^{\lambda\omega}_{\alpha}$. Equipped with these notations for functional $E^{\lambda \omega}_{\alpha}$ and corresponding min-max value $\mathcal{W}_{\alpha,\lambda}$ we have
	\begin{lemma}\label{lem monotone}
		Viewing $\mathcal{W}_{\alpha,\lambda}$ as a two variable function of $\alpha \in (1,\infty)$ and $\lambda \in (0,\infty)$, the following properties hold
		\begin{enumerate}
			\item \label{lem monotone part 1} For each $\alpha > 1$, the function $\lambda \mapsto {\mathcal{W}_{\alpha,\lambda}}/{\lambda}$ is non-increasing;
			\item \label{lem monotone part 2} For each $\lambda > 0$, the function $\alpha \mapsto \mathcal{W}_{\alpha,\lambda}$ is non-decreasing;
		\end{enumerate}
		Moreover, given any sequences $\alpha_j \searrow 1$, then, for almost every $\lambda > 0$, there exists a subsequence of $\{\alpha_j\}_{j \in \mathbb N}$, also denoted by $\alpha_j$, and a constant $C > 0$ which is independent of $j$, such that 
		\begin{equation*}
			0 \leq \frac{d}{d\lambda}\p{- \frac{\mathcal{W}_{\alpha_j,\lambda}}{\lambda}} \leq C, \quad \text{for all } j \in \mathbb{N}.
		\end{equation*}
	\end{lemma}
	\begin{proof}
		We first consider part \eqref{lem monotone part 1}, for any $u \in W^{1,2\alpha}(\S^2,N)$ and $0 < \lambda_1 < \lambda_2$, by the expression \eqref{eq:E lambda alpha psi} of functional $E^{\lambda\omega}_{\alpha}$ we have 
		\begin{equation}\label{eq 1}
			\frac{E_{\alpha}^{\lambda_1\omega}(u)}{\lambda_1} - \frac{E^{\lambda_2\omega}_{\alpha}(u)}{\lambda_2} = \frac{\lambda_2 - \lambda_1}{\lambda_1\cdot \lambda_2} {E_\alpha(u)} \geq 0.
		\end{equation}
		By the construction of min-max value $\mathcal{W}_{\alpha,\lambda}$, for any $\varepsilon > 0$, there exists $\sigma \in \mathscr{S}$ such that 
		\begin{equation*}
			\mathcal{W}_{\alpha,\lambda_1} \leq \max_{t \in I^{k-2}} E^{\lambda_1\omega}_{\alpha}(\sigma(t)) \leq \mathcal{W}_{\alpha,\lambda_1} + \varepsilon.
		\end{equation*}
		Thus, using the monotone formula \eqref{eq 1} we can estimate
		\begin{equation*}
			\frac{\mathcal{W}_{\alpha,\lambda_2}}{\lambda_2} \leq \max_{t \in I^{k-2}} \frac{E^{\lambda_2\omega}_{\alpha}(\sigma(t))}{\lambda_2} \leq \max_{t \in I^{k-2}} \frac{E^{\lambda_1\omega}_{\alpha}(\sigma(t))}{\lambda_1} \leq \frac{\mathcal{W}_{\alpha,\lambda_1}}{\lambda_1} + \frac{\varepsilon}{\lambda_1}.
		\end{equation*}
		Since the choice of $\varepsilon > 0$ is arbitrary, we can obtain the conclusion of \eqref{lem monotone part 1}.
		
		For part \eqref{lem monotone part 2}, it is worth noting that the function $(1 + |x|^2)^\alpha$ is increasing with respect to $\alpha > 1$, hence we have 
		\begin{equation*}
			E_{\alpha_2}^{\lambda\omega}(u) - E_{\alpha_1}^{\lambda\omega}(u) = \int_{\S^2} \p{1 + |\nabla u|^2}^{\alpha_2} - \p{1 + |\nabla u|^2}^{\alpha_1} dV_g \geq 0
		\end{equation*}
		for any $u\in W^{1, 2\alpha}(\S^2, N)$ and $1 < \alpha_1 < \alpha_2$. The remaining argument is essentially identical to the process outlined in part \eqref{lem monotone part 1} and therefore leads to the conclusion as stated in \eqref{lem monotone part 2}.
		
		For the last statement, since $\mathcal{W}_{\alpha_j,\lambda}/\lambda$ is a monotone function with respect to $\lambda$ by part \eqref{lem monotone part 1},  the derivative 
		\begin{equation*}
			\frac{d}{d \lambda} \p{- \frac{\mathcal{W}_{\alpha_j,\lambda}}{\lambda}}
		\end{equation*}
		exists for almost every $\lambda \in (0,\infty)$ and is non-negative. Furthermore, given any $0 < \lambda_1 < \lambda_2 < \infty$, we have 
		\begin{align*}
			\int_{\lambda_1}^{\lambda_2} \frac{d}{d\lambda} \p{- \frac{\mathcal{W}_{\alpha_j,\lambda}}{\lambda}} d\lambda \leq \frac{\mathcal{W}_{\alpha_j, \lambda_1}}{\lambda_1} - \frac{\mathcal{W}_{\alpha_j,\lambda_2}}{\lambda_2} < \infty
		\end{align*}
		Then, by Fatou's Lemma, we have 
		\begin{align*}
			\int_{\lambda_1}^{\lambda_2} \liminf_{j\rightarrow \infty} \frac{d}{d\lambda} \p{- \frac{\mathcal{W}_{\alpha_j,\lambda}}{\lambda}}d\lambda
			&\leq \liminf_{j\rightarrow \infty}\int_{\lambda_1}^{\lambda_2} \frac{d}{d\lambda} \p{- \frac{\mathcal{W}_{\alpha_j,\lambda}}{\lambda}} d\lambda \\
			&\leq \liminf_{j\rightarrow \infty} \p{\frac{\mathcal{W}_{\alpha_j, \lambda_1}}{\lambda_1} - \frac{\mathcal{W}_{\alpha_j,\lambda_2}}{\lambda_2}}\\
			&\leq \frac{\mathcal{W}_{2, \lambda_1}}{\lambda_1} - \frac{\mathcal{W}_{1,\lambda_2}}{\lambda_2} < \infty,
		\end{align*}
		where in the last inequality we used the monotonicity of $\mathcal{W}_{\alpha,\lambda}$ with respect to $\alpha$ obtained in part \eqref{lem monotone part 2}. It can be inferred that for almost every $ \lambda $ within the range of $ [\lambda _1, \lambda _2]$ there holds 
		\begin{equation*}
			\liminf_{j\rightarrow \infty} \frac{d}{d\lambda} \p{- \frac{\mathcal{W}_{\alpha_j,\lambda}}{\lambda}} < \infty.
		\end{equation*}
		Therefore, the last assertion of Lemma \ref{lem monotone} follows by the arbitrary choices of $\lambda_1$ and $\lambda_2$.
	\end{proof}

Next, in Lemma \ref{admissable h} below, we show that based on the conclusion of Lemma \ref{lem monotone} for certain choice $t_0 \in I^{k-2}$ there exists a $\alpha$-energy control for some sweepouts valued at $t_0$, with upper bounds depending on the constant $C$ obtained in last assertion of Lemma \ref{lem monotone}.

	\begin{lemma}\label{admissable h}
		Let $\lambda > 0$ and $\alpha_0 >\alpha > 1$ for some small enough $\alpha_0 - 1$. And assume there exists a constant $C > 0$, independent of $\alpha \searrow 1$, such that
		\begin{equation*}
			0 \leq \frac{d}{d\lambda}\p{- \frac{\mathcal{W}_{\alpha,\lambda}}{\lambda}} \leq C.
		\end{equation*}
		Then, there exists a sequence of  sweepouts $\sigma_j:I^{k-2}\rightarrow W^{1,2\alpha}(\S^2,N)$ such that 
		\begin{equation*}
			\max_{t \in I^{k-2}} E_{\alpha}^{\lambda\omega}(\sigma_j(t)) \leq \mathcal{W}_{\alpha,\lambda} + \lambda\varepsilon_j,
		\end{equation*}
		and 
		\begin{equation*}
			E_\alpha(\sigma_j(t_0)) \leq 8 \lambda^2 C 
		\end{equation*}
		as long as $t_0 \in I^{k-2}$ satisfying $E_{\alpha}^{\lambda\omega}(\sigma_j(t_0)) \geq \mathcal{W}_{\alpha,\lambda} - \lambda\varepsilon_j$, for any sequence $\varepsilon_j \searrow 0$ with $0 < \varepsilon_j \leq \lambda /2$.
	\end{lemma}
	\begin{proof}
		Given any sequence $\varepsilon_j \searrow 0$ with $0 < \varepsilon_j \leq \lambda /2$ for each $j \in \mathbb{N}$, we set $\lambda_j = \lambda - \varepsilon_j/(4C)$. By the assumption of Lemma \ref{admissable h}, there exists a large $j_0 \in \mathbb{N}$ such that for all $j \geq j_0$ there holds 
		\begin{equation*}
			\frac{1}{\lambda - \lambda_j} \p{\frac{\mathcal{W}_{\alpha,\lambda_j}}{\lambda_j} - \frac{\mathcal{W}_{\alpha,\lambda}}{\lambda}} \leq 2C.
		\end{equation*}
		This is equivalent to 
		\begin{equation*}
			\frac{\mathcal{W}_{\alpha,\lambda_j}}{\lambda_j} \leq \frac{\mathcal{W}_{\alpha,\lambda}}{\lambda} + \frac{\varepsilon_j}{2},\quad  \text{for all } j\geq j_0. 
		\end{equation*}
		Next, recalling that the functional $E^{\lambda\omega}_{\alpha}$ has min-max value $\mathcal{W}_{\alpha,\lambda}$, there exists a sequence of sweepouts $\sigma_j \in \mathscr{S}$ such that
		\begin{equation}\label{eq 2}
			\frac{1}{\lambda_j} \max_{t \in I^{k-2}} E_{\alpha}^{\lambda_j\omega}(\sigma_j(t)) \leq \frac{1}{\lambda_j}\mathcal{W}_{\alpha,\lambda_j} + \frac{\varepsilon_j}{2}.
		\end{equation}
		Thus, combining with the monotone formula \eqref{eq 1} we can obtain
		\begin{equation*}
			\max_{t \in I^{k-2}} E_{\alpha}^{\lambda\omega}(\sigma_j
			(t)) \leq \lambda\p{\frac{\mathcal{W}_{\alpha,\lambda_j}}{\lambda_j}} + \frac{\lambda \varepsilon_j}{2}\leq \mathcal{W}_{\alpha,\lambda} + \lambda \varepsilon_j.
		\end{equation*}
		For the another side of inequality, we pick $t_0 \in I^{k-2}$ satisfying
		\begin{equation}\label{eq 3}
			\frac{1}{\lambda} E_{\alpha}^{\lambda\omega}(\sigma_j(t_0)) \geq 
			\frac{1}{\lambda} \mathcal{W}_{\alpha,\lambda} - \varepsilon_j,
		\end{equation}
		then, for $j \geq j_0$, by subtracting \eqref{eq 3} from \eqref{eq 2} and utilizing \eqref{eq 1}, we can obtain
		\begin{align*}
			\frac{1}{\lambda\cdot \lambda_j} E_\alpha(\sigma_j(t_0)) &= \frac{1}{\lambda - \lambda_j}\p{\frac{E_{\alpha}^{\lambda_j\omega}(\sigma_j(t_0))}{\lambda_j} - \frac{E^{\lambda\omega}_{\alpha}(\sigma_j(t_0))}{\lambda}}\\
			& \leq \frac{1}{\lambda - \lambda_j} \p{\frac{\mathcal{W}_{\alpha,\lambda_j}}{\lambda_j} - \frac{\mathcal{W}_{\alpha,\lambda}}{\lambda} + \frac{3\varepsilon_j}{2}}\\
			&\leq \frac{1}{\lambda - \lambda_j} \p{\frac{\mathcal{W}_{\alpha,\lambda_j}}{\lambda_j} - \frac{\mathcal{W}_{\alpha,\lambda}}{\lambda}} + 6C \leq 8 C.
		\end{align*}
		Therefore, we have 
		\begin{equation*}
			E_\alpha(\sigma_j(t_0)) \leq 8 \lambda^2 C 
		\end{equation*}
		provided $t_0 \in I^{k-2}$ satisfying \eqref{eq 3}.
	\end{proof}
	
	Next, we show that the min-max value $\mathcal{W}_{\alpha,\lambda}$ is a critical value and construct a critical point $\{u_{\alpha}\}$ for the functional $E^{\lambda \omega}_{\alpha}$ with uniformly bounded $\alpha$-energy, where the $\alpha$-energy upper bound depends on the value of derivatives 
	\begin{equation*}
		\frac{d}{d\lambda}\p{- \frac{\mathcal{W}_{\alpha,\lambda}}{\lambda}},
	\end{equation*}
	the choice of $\lambda > 0$ in view of Lemma \ref{admissable h}.
	\begin{prop}\label{prop energy bound}
		Given $\lambda > 0$ and $\alpha > 1$ such that there exists a sequence of sweepouts $\sigma_j \in \mathscr{S}$ satisfying
		\begin{equation*}
			\max_{t \in I^{k-2}} E_{\alpha}^{\lambda\omega}(\sigma_j(t)) \leq \mathcal{W}_{\alpha,\lambda} + \lambda\varepsilon_j,
		\end{equation*}
		and 
		\begin{equation*}
			E_\alpha(\sigma_j(t_0)) \leq 8 \lambda^2 C
		\end{equation*}
		for any $t_0 \in I^{k-2}$ admitting $E_{\alpha}^{\lambda\omega}(\sigma_j(t_0)) \geq 
		\mathcal{W}_{\alpha,\lambda} - \lambda\varepsilon_j$ and for any sequence $\varepsilon_j \searrow 0$ with $0 < \varepsilon_j \leq \lambda /2$. Then, after passing to a subsequence, there exists $t_j \in I^{k-2}$ so that the following holds:
		\begin{enumerate}
			\item \label{prop energy bound part 1} $\abs{E^{\lambda\omega}_{\alpha}(\sigma_j(t_j)) - \mathcal{W}_{\alpha,\lambda}} \leq \lambda \varepsilon_j$, which further implies $E_{\alpha}^{\lambda\omega}(\sigma_j(t_j)) \rightarrow \mathcal{W}_{\alpha,\lambda}$ as $j \rightarrow \infty$;
			\item \label{prop energy bound part 2} $\sigma_j(t_j)$ converges strongly in $W^{1,2\alpha}(\S^2,N)$ to some $u_{\alpha}$ with uniformly bounded energy
			\begin{equation*}
				E_{\alpha}(u_{\alpha}) \leq 8 \lambda^2 C;
			\end{equation*}
			\item \label{prop energy bound part 3} The limiting map $u_{\alpha}$ obtained in part \eqref{prop energy bound part 2} is non-constant. Moreover, there exists a positive constant $\delta(\alpha,\lambda\omega) > 0$ depending on $\alpha > 1$, $\lambda > 1$ amd $\omega \in C^2(\wedge^2(N))$ such that $E_\alpha(u_{\alpha}) \geq \frac{1}{2}\mathrm{Vol}(\S^2) + \delta(\alpha,\lambda\omega)$.
		\end{enumerate}
	\end{prop}
	\begin{rmk}
		Given a sequence of admissible sweepouts $\{\sigma_j\}_{j \in \mathbb N} \subset \mathscr{S}$, we call $\sigma_j$ is a \textit{min-max sequence} if 
		\begin{equation*}
			\limsup_{j \rightarrow \infty} \max_{t \in I^{k-2}} E^{\lambda \omega}_\alpha(\sigma_j(t)) = \mathcal{W}_{\alpha,\lambda}.
		\end{equation*}
		Therefore, combining the Lemma \ref{lem monotone}, Lemma \ref{admissable h} and Proposition \ref{prop energy bound}, we can conclude that, for almost every choice of $\lambda \in \R_+$ given any min-max sequence $\{\sigma_j\}_{j \in \mathbb N} \subset \mathscr{S}$, after passing to certain subsequences,  there exists a sequence of $t_j \in I^{k-2}$ such that $\sigma_j(t_j)$ converges strongly in $W^{1,2\alpha}(\S^2, N)$ to a non-constant $\alpha$-$H$-surface $u_\alpha$ with
		\begin{equation*}
			E^{\lambda\omega}_\alpha(u_\alpha) = \mathcal{W}_{\alpha,\lambda}.
		\end{equation*}
		Moreover, Given any $\alpha_j \searrow 1$ there exists a subequence of $\alpha_j \searrow 1$ such that the $\alpha_j$-energy of $u_{\alpha_j}$ is uniformly bounded.
	\end{rmk}
        \begin{rmk}
            In the Lemma \ref{blow 2} below in Section \ref{section 4.1 pre of blow up}, we can actually show that when $\alpha_0 - 1$ is small enough there exists a constant $\delta(\lambda\omega) > 0$ independent of $\alpha \in (1, \alpha_0)$ such that $E_\alpha(u_{\alpha}) \geq \frac{1}{2}\mathrm{Vol}(\S^2) + \delta(\lambda\omega)$ for the non-constant critical point $u_\alpha$ obtained in \eqref{prop energy bound part 3} of Proposition \ref{prop energy bound}. 
        \end{rmk}
	\begin{proof}
		We first consider part \eqref{prop energy bound part 1} and define
		\begin{equation*}
			U_j = \left\{t \in I^{k-2} \, :\, E^{\lambda\omega}_{\alpha}(\sigma_j
			(t)) > \mathcal{W}_{\alpha,\lambda} - \lambda \varepsilon_j \right\}\subset I^{k-2}.
		\end{equation*}
		Since $E_{\alpha}^{\lambda\omega}$ satisfies the Palais-Smale condition, see Lemma \ref{P-S condition}, considering the assumption of Proposition \ref{prop energy bound},  it suffices to show that the first variation acting on $\sigma_j(U_j)$ is not bounded away from zero. More precisely, we claim that:
		
		\hspace{-1ex}\claim \label{claim 3} For any $\varepsilon > 0$, there exists $j_0 \in \mathbb{N}$ such that 
		\begin{equation}\label{eq claim 3}
			\inf_{t \in U_j} \norm{\delta E_{\alpha}^{\lambda\omega}(\sigma_j(t))} < \varepsilon,\quad \text{for all }j \geq j_0.
		\end{equation}
		\begin{proof}[\textbf{Proof of Claim \ref{claim 3}}]
			We prove the Claim \ref{claim 3} by contradiction, that is, suppose that there exists some $\delta > 0$ and a subsequence of $\sigma_j \in \mathscr{S}$, which is also denoted by $\sigma_j$, such that  
			\begin{equation*}
				\norm{\delta E_{\alpha}^{\lambda\omega}(\sigma_j(t))} \geq \delta,\quad \text{for all } t \in U_j \text{ and all } j \in \mathbb{N}.
			\end{equation*}
			The following existence of pseudo-gradient vector field is essential for us and the detailed proof can be founded in \cite[Chapter II. Lemma 3.2, Lemma 3.9]{Struwe-1988}.
			\begin{lemma}\label{lemma vector}
				There exists a locally Lipschitz continuous map 
				$$X: \dbl{V} \rightarrow T\p{W^{1,2\alpha}(\S^2,N)}\subset T\p{W^{1,2\alpha}(\S^2,\R^K)},$$
				where 
				\begin{equation*}
					\dbl{V} = \left\{ u \in W^{1,2\alpha}(\S^2,N)\, : \, \delta E^{\lambda\omega}_{\alpha}(u)\neq 0\right\},
				\end{equation*}
				such all the following holds:
				\begin{enumerate}
					\item \label{lemma vector part 1} $X(u) \in \mathcal{T}_u$ for each $u \in \dbl{V}$;
					\item \label{lemma vector part 2}$\norm{X(u)}_{W^{1,2\alpha}(\S^2,\R^K)} < 2 \min\set{\norm{\delta E^{\lambda \omega}_{\alpha}(u)}, 1}$;
					\item \label{lemma vector part 3} $\delta E^{\lambda\omega}_{\alpha}(u)(X(u)) < - \min\set{\norm{{E^{\lambda \omega}_{\alpha}(u)}}, 1} \cdot \norm{{E^{\lambda \omega}_{\alpha}}(u)}.$
				\end{enumerate}
			\end{lemma}
			Then, we consider the continuous  1-parameter family of homeomorphisms associated with $X$, denoted by
			\begin{equation*}
				\Phi : \left\{(u,s) \,:\, u \in \dbl{V},\, 0 \leq s < T(u) \right\} \rightarrow W^{1,2\alpha}(\S^2, N) \subset W^{1,2\alpha}(\S^2,\R^K),  
			\end{equation*}
			where $T(u)$ is the maximal existence time of the integral curve from $u$ along $X$. Next, we show that $T(u)$ has a uniformly positive lower bound which is independent of $u \in \dbl{V}$, following the outline of the proof presented in \cite[Lemma 3.8]{Cheng2020ExistenceOC}.
			\begin{lemma}\label{lem existence time}
				For all $L > 0$ and $0 < \delta < 1$, there exists $T = T(\delta,L) > 0$ such that  if $\norm{\delta E^{\lambda \omega}_{\alpha}(u)} \geq \delta$ and $E_\alpha(u) \leq L$, then the maximal existence time $T(u)$ satisfies
				\begin{equation*}
					T(u) \geq T(\delta, L).
				\end{equation*}
				In particular, when $s \leq T(\delta , L)$ there holds
				\begin{equation*}
					\norm{\delta E^{\lambda \omega}_{\alpha}(\Phi(u,s))} \geq \frac{\delta}{2}.
				\end{equation*}
			\end{lemma}
			\begin{proof}
				By part \eqref{lemma vector part 2} of Lemma \ref{lemma vector} and general ODE theory on Banach manifold, we see that if $T(u) < \infty$, then
				\begin{equation*}
					\liminf_{s \nearrow \,\,T(u)} \norm{\delta E^{\lambda \omega}_{\alpha}(\Phi(u,s)) } = 0.
				\end{equation*}
				Thus, it suffices to obtain a lower bound for $\norm{E^{\lambda \omega}_{\alpha}(\Phi(u,s)) }$ when $s \in [0,T(u)]$ in order to obtain a lower bound of $T(u)$. To this end, given $s  < \min\{1/2, T(u)\}$, we use property \eqref{lemma vector part 2} in Lemma \ref{lemma vector} to see 
				\begin{align}
					\norm{\Phi(u,s) - u}_{W^{1,2\alpha}(\S^2,N)} &\leq \int_0^s \norm{X(\Phi(u,t))}_{W^{1,2\alpha}(\S^2,\R^K)} dt\nonumber\\
					&\leq \int_0^s  2 \min\set{\norm{\delta E^{\lambda \omega}_{\alpha}(\Phi(u,t))}, 1}dt\leq 2s  < 1.
				\end{align}
				This further implies
				\begin{equation*}
					\norm{\Phi(u,s)}_{W^{1,2\alpha}(\S^2,N)} \leq C_N(L + 1), \quad \text{whenever }s <  \min\{1/2, T(u)\},
				\end{equation*}
				for some constant depending only on geometries of $N$. By the estimates
				\begin{equation*}
					||u||_{W^{1,2\alpha}(\S^2,N)} \leq C_N E_\alpha(u) \leq C_N L \quad \text{and} \quad  ||\Phi(u,s)||_{W^{1,2\alpha}(\S^2,N)} \leq  C_N(L + 1),
				\end{equation*}
				we can apply part \eqref{proper G part 3} of Proposition \ref{property G} to get
				\begin{equation*}
					\norm{\delta E^{\lambda \omega}_{\alpha}(\Phi(u,s)) - \delta E^{\lambda \omega}_{\alpha}(u)} \leq C_{L} \norm{\Phi(u,s) - u}_{W^{1,2\alpha}(\S^2,N)} \leq 2C_{L} s.
				\end{equation*}
				Utilizing this inequality, we can get 
				\begin{equation*}
					\norm{\delta E^{\lambda \omega}_{\alpha}(\Phi(u,s)) }\geq \frac{\delta}{2}, \quad \text{whenever }s \leq \min\left\{T(u), \frac{\delta}{4(C_{L} + 1)}\right\}
				\end{equation*}
				which implies that $T(u) \geq \frac{\delta}{4(C_{L} + 1)} : = T(\delta , L)$ and the second conclusion is also followed.
			\end{proof}
			Then we come back to the proof of Claim \ref{claim 3}. Recalling that we assumed  by contradiction that 
			\begin{equation}
				\norm{\delta E^{\lambda \omega}_{\alpha}(\sigma_j(t))} \geq \delta, \quad \text{for all } t \in U_j,
			\end{equation}
			and by the assumption of Proposition \ref{prop energy bound}, there exists a universal constant $C_\lambda:= 8\lambda^2 C > 0$ such that 
			\begin{equation*}
				E_\alpha(\sigma_j(t)) \leq C_\lambda, \quad \text{when } t\in U_j.
			\end{equation*}
			So, we can apply Lemma \ref{lem existence time} to obtain a lower bound  of $T(\sigma_j(t)) \geq T(\delta ,C_\lambda)$ for all $t \in U_j$ and 
			\begin{equation}\label{eq 9}
				\norm{\delta E^{\lambda \omega}_{\alpha}(\Phi(\sigma_j(t), s))} \geq \frac{\delta}{2} \quad \text{for all }(t,s) \in U_j \times [0,T(\delta,C_\lambda)].
			\end{equation}
			In order to construct a new sweepout from $\sigma_j$ and $\Phi$, we define a compact subset $V_j$ of $U_j$ as follows 
			\begin{equation*}
				V_j = \left\{t \in I^{k-2}\, :\, E^{\lambda\omega}_{\alpha}(\sigma_j(t)) \geq \mathcal{W}_{\alpha,\lambda} - \frac{\lambda \varepsilon_j}{2} \right\}.
			\end{equation*}
			By the continuity of $t \mapsto E^{\lambda\omega}_{\alpha}(\sigma_j(t))$, $V_j$ is a compact subset of $U_j$. So, there exists a smooth cut-off function $\varphi_j : I^{k-2}\rightarrow \R$ such that $\varphi_j \equiv 1$ on $V_j$ and vanishes outside of $U_j$. Then, we set 
			\begin{equation*}
				\dbl{\Phi}_j(t,s) := \Phi\big(\sigma_j(t), \varphi_j(t) T(\delta, C_\lambda)s\big) \quad \text{for } (t,s) \in I^{k-2} \times [0,1].
			\end{equation*}
			We observe that, when $t \in \partial I^{k-2}$ and $j$ is large enough, $\varphi_j(t) = 0$ and $\dbl{\Phi}_j(t,s) = \sigma_j(t)$ is a constant map for all $s \in [0, 1]$. Hence, if we let $\dbl{\sigma}_j(t) = \dbl{\Phi}_j(t,1)$, $\dbl{\sigma}_j \in \mathscr{S}$ is an admissible sweepout.  Then differentiating $E^{\lambda \omega}_{\alpha}(\dbl{\Phi}_j(t,s))$ with respect to $s$ at $(t_0,s_0)\in I^{k-2}\times[0,1]$ yields that 
			\begin{equation*}
				\left.\frac{d}{ds}\right|_{s = s_0} E^{\lambda \omega}_{\alpha}\p{\dbl{\Phi}_j(t_0,s)} =   \varphi_j(t_0)T(\delta,C_\lambda) \delta E^{\lambda \omega}_{\alpha}\p{\dbl{\Phi}_j(t_0,s_0)}\p{X\p{\dbl{\Phi}_j(t_0,s_0)}}. 
			\end{equation*}
			Then, we integrate the above identity with respect to $s$ from 0 to 1 by changing variables to get
			\begin{equation}\label{eq 8}
				E^{\lambda \omega}_{\alpha}(\dbl{\sigma}_j(t)) = E^{\lambda \omega}_{\alpha}\p{\sigma_j(t)} + \int_0^{  \varphi_j(t) T(\delta,C_\lambda)} \delta E^{\lambda \omega}_{\alpha}\p{\dbl{\Phi}_j(x,s)}\p{ \,X\p{\dbl{\Phi}_j(x,s)}} ds.
			\end{equation}
			Next,  combining the identity \eqref{eq 8} with  estimate \eqref{eq 9} and part \eqref{lemma vector part 3} of Lemma \ref{lemma vector}, we can get
			\begin{equation*}
				E^{\lambda \omega}_{\alpha}(\dbl{\sigma}_j(t)) <  E^{\lambda \omega}_{\alpha}(\sigma_j(t)) -  \frac{\delta^2}{4}  T(\delta,C_\lambda) < \mathcal{W}_{\alpha,\lambda} + \lambda \varepsilon_j - \frac{\delta^2}{4}  T(\delta,C_\lambda)
			\end{equation*}
			for all $t \in V_j$. 
			Thus, when $t \in V_j$ and  $j$ is large enough, we have
			\begin{equation*}
				\max_{t \in I^{k-2}} E^{\lambda \omega}_{\alpha}(\dbl{\sigma}_j(t)) \leq  \mathcal{W}_{\alpha,\lambda} - \frac{\delta^2   T(\delta,C_\lambda)}{8} < \mathcal{W}_{\alpha,\lambda},
			\end{equation*}
			which is a contradiction to the definition of min-max value $\mathcal{W}_{\alpha,\lambda}$. Therefore Claim \ref{claim 3} holds.
		\end{proof}
		Let us now return to the proof of the Proposition \ref{prop energy bound}. Consequently, by Claim \ref{claim 3}, there exists a subsequence of $\sigma_j(t_j)$ for $t_j \in U_j$, which are still denoted by same symbols, such that 
		\begin{equation*}
			E_\alpha(\sigma_j(t_j)) \leq C_\lambda \quad \text{and}\quad  \norm{\delta E^{\lambda \omega}_{\alpha}(\sigma_j(t_j)) } \rightarrow 0 \quad \text{as } j \rightarrow \infty.
		\end{equation*}
		In view of Lemma \ref{P-S condition}, after passing to a subsequence, $\sigma_j(t_j)$ converges strongly in $W^{1,2\alpha}(\S^2, N)$ to some $u_{\alpha}$ satisfying $\delta E^{\lambda \omega}_{\alpha}(u_{\alpha}) = 0$ and $E_\alpha(u_{\alpha}) \leq C_\lambda$. The conclusion of part \eqref{prop energy bound part 1} and \eqref{prop energy bound part 2} of Proposition \ref{prop energy bound} follows directly.  For part \eqref{prop energy bound part 3} of Proposition \ref{prop energy bound}, we first note that the $\alpha$-energy for critical points $u_{\alpha}$ is strictly larger than $\frac{1}{2}\mathrm{Vol}(\S^2)$, that is, there exists a $\delta(\alpha, \lambda\omega) > 0$ such that 
		\begin{equation*}
			E_\alpha(u_{\alpha})  \geq \frac{1}{2}\mathrm{Vol}(\S^2) +  \delta(\alpha, \lambda\omega).
		\end{equation*}
		Otherwise, assume for any $\varepsilon  > 0$ there exists a sweepout $\sigma_\varepsilon \in \mathscr{S}$ such that 
		\begin{equation*}
			\max_{t \in I^{k-2}} E_\alpha(\sigma_\varepsilon(t)) < \frac{1}{2}\mathrm{Vol}(\S^2) + \varepsilon.
		\end{equation*}
		Then, the map $f_{\sigma_\varepsilon}: \S^{k} \rightarrow N$, induced by $\sigma_\varepsilon$, is homotopy to some constant map, by directly applying Poinc{a}r\'{e}'s inequality and Sobolev Embedding $W^{1,2\alpha}(\S^k,N) \hookrightarrow C^0(\S^k,N)$,
		which contradicts to the choice of $\sigma_j \in [\iota]\neq 0$. Thus, $u_{\alpha}$ is a non-constant critical point for $E^{\lambda\omega}_{\alpha}$ and the proof of Proposition \ref{prop energy bound} is now complete.
	\end{proof}

	\subsection{Morse Index Upper Bound for Min-Max Critical Points \texorpdfstring{$u_\alpha$}{Lg}}
	\ 
	\vskip5pt
	In this subsection, we are devoted to construct a sequence of non-constant critical points $\{u_{\alpha_j}\}_{j \in \mathbb N}$ of $E^{\lambda \omega}_{\alpha_j}$ for $\alpha_j \searrow 1$ as $j \rightarrow \infty$ that admits an uniformly $\alpha_j$-energy upper bound together with a Morse index upper bound: $\mathrm{Ind}_{E^{\lambda \omega}_{\alpha_j}}(u_{\alpha_j}) \leq k-2$. The main obstruction in constructing the critical points $u_{\alpha_j}$ with desired Morse index upper bound is the dependence of $\alpha_j$-energy upper bound obtained in Proposition \ref{prop energy bound} with the choices of sequence $\{u_{\alpha_j}\}_{j \in \mathbb N}$ and $\lambda \in \R_{+}$. Such dependence prevents us to apply Morse theory to obtain the Morse index upper bound estimates by perturbing the functional $E^{\lambda \omega}_{\alpha_j}$ further to a Morse one. To overcome this obstacle, inspired by the Morse index upper bound estimates in the setting of Almgren-Pitts min-max theory by Marques-Neves \cite{Marques-Neves2016}, Song \cite{Song-2023-Morse}, and Li \cite{Liyangyang-2023}, see also Cheng-Zhou \cite{Cheng2020ExistenceOC} and Cheng \cite{cheng2022existence} for the setting of a newly devised min–max theory, we design a homotopical deformation for the min-max sequences of sweepouts $\sigma_l: I^{k-2} \rightarrow W^{1,2\alpha}(\S^2, N)$ obtained in Proposition \ref{prop energy bound} to construct a sequence $\{u_{\alpha_j}\}_{j \in \mathbb N}$ with the desired  Morse index upper bound and $\alpha_j$-energy bound simultaneously, for more details see Theorem \ref{prop deformation} and Theorem \ref{thm:Morse index k-2} below. Note that the main result---Theorem \ref{prop deformation} in  this subsection holds for all choice of $\lambda \in \R_+$ and any $\alpha > 1$ in the definition of functional $E^{\omega}_\alpha$, so we simply write $\alpha$ for $\alpha_j$ when $j \in \mathbb N$ fixed, $\omega$ for $\lambda \omega$ and $E^{\omega}_\alpha$ for $E^{\lambda \omega }_\alpha$. 
	
	Before penetrating into the detailed description of homotopical deformation Theorem \ref{prop deformation} and Theorem \ref{thm:Morse index k-2}, we prepare some essential notions and estimates. To begin, recall that the second variation formula of $E^{\omega}_{\alpha}:W^{1,2\alpha}(\S^2, N) \rightarrow \R$ is written as following, for more details see Lemma \ref{variation formula},
	\begin{align}\label{eq:second variation formula section 3}
		\delta^2 E^{\omega}_{\alpha}(u)(V,V) &= \alpha \int_{\S^2} \p{1 + \abs{\nabla u}^2}^{\alpha - 1} \Big( \langle \nabla V, \nabla V \rangle - R(V,\nabla u, V, \nabla u) \Big) d V_g \nonumber\\
		& \quad + 2\alpha(\alpha - 1)\int_{\S^2} \p{1 + \abs{\nabla u}^2}^{\alpha - 2}\langle \nabla u, \nabla V\rangle^2 dV_g \nonumber\\
		&\quad + 2 \int_{\S^2} \inner{ H(\nablap u, \nabla V), V} dV_g +  \int_{\S^2} \inner{ (\nabla_V H)(\nablap u, \nabla u), V} dV_g,
	\end{align}
	for $V \in \mathcal{T}_u$. 
	
	\begin{defi}\label{defi: Morse index}
		For any $\omega \in C^2(\wedge^2(N))$, the \textit{Morse index} of a critical point $u \in W^{1,2\alpha}(\S^2,N)$ for $E^\omega_{\alpha}$ is the maximal dimension of linear subspace of $\mathcal{T}_u$ on which $\delta^2 E^\omega_{\alpha}(u)$ restricted to be a negative definite symmetric bilinear form. 
	\end{defi}
 By Lemma \ref{lem regularity}, every critical point $u \in W^{1,2\alpha}(\S^2, N)$ of $E^\omega_{\alpha}$ is smooth for small enough $\alpha > 1$. Then, we can extend $\delta^2 E^\omega_{\alpha}(u): \mathcal{T}_u \times \mathcal{T}_u \rightarrow \R$ to a bounded symmetric bilinear form on the Hilbert space 
	\begin{equation}
		\label{eq:extended Tu}
		\dbl{\mathcal{T}}_u := \set{V \in W^{1,2}(\S^2, \R^K) \, :\, V(x) \in T_{u(x)}N\, \text{ for a.e. } x \in \S^2}.
	\end{equation}
	The Morse index of critical point $u$ of $E^\omega_\alpha$ on $\dbl{\mathcal{T}}_u$ is defined exactly the same manner with Definition \ref{defi: Morse index}.
	At each $\dbl{\mathcal{T}}_u$, it follows from the Riesz representation theorem that there exists a bounded linear operator $\L_u$ such that 
	\begin{equation}\label{jacobi operator}
		\delta^2 E^\omega_{\alpha}(u)(V,W) = \langle \L_u(V),W\rangle_{\dbl{\mathcal{T}}_u }\quad \text{for } V,\,W \in \dbl{\mathcal{T}}_u,
	\end{equation}
	which is called the \textit{Jacobi operator} of $E^\omega_{\alpha}$ at $u$. Here, the inner product $\inner{\cdot, \cdot}_{\dbl{\mathcal{T}}_u}$ is induced from inclusion $\dbl{\mathcal{T}}_u \subset W^{1,2}(M, \R^K)$.
	Furthermore, in the Lemma \ref{lem:spectral decom} below, we demonstrate that the Morse index defined on $\mathcal{T}_u$ is equivalent to the one extended on $\dbl{\mathcal{T}}_u$ and we establish a spectral decomposition on $\dbl{\mathcal{T}}_u$ using the standard uniformly elliptic operator theory.
	\begin{lemma}\label{lem:spectral decom}
		Given a critical point $u \in W^{1,2\alpha}(\S^2, N)$ of $E^\omega_{\alpha}$, when $\alpha  - 1$ is small enough, the following properties holds:
		\begin{enumerate}[label=(\subscript{L}{{\arabic*}})]
			\item\label{lem:spectral decom 1} The Jacobi operator $\L_u$ is a self-adjoint second order elliptic differential operator, hence a Fredholm operator on $\dbl{\mathcal{T}}_u$.
			\item\label{lem:spectral decom 2} There exists a sequence of real eigenvalues $\lambda_j \nearrow \infty$ of $\L_u$ and a sequence of corresponding eigenfunctions $\{\phi_j\}_{j \in \mathbb N}$ which forms a basis of $\dbl{\mathcal{T}}_u$ such that 
			\begin{equation*}
				\delta E^\omega_{\alpha}(u)(\phi_i ,\phi_j) = \inner{\L_u(\phi_i), \phi_j}_{\dbl{\mathcal{T}}_u} = \lambda_i \inner{\phi_i ,\phi_j}_{L^2} = \lambda_i\delta_{ij}.
			\end{equation*}
			\item\label{lem:spectral decom 3} The Morse index defined on $\dbl{\mathcal{T}}_u$ is finite and is identical with the standard one defined as in Definition \ref{defi: Morse index}.
		\end{enumerate}
	\end{lemma}
	\begin{proof}
		It suffices to show \ref{lem:spectral decom 1}, \ref{lem:spectral decom 2} follows directly from the Sobolev compact embedding
		\begin{equation*}
			\dbl{\mathcal{T}}_u \hookrightarrow  \set{V \in L^{2}(\S^2, \R^K) \, :\, V(x) \in T_{u(x)}N\, \text{ for a.e. } x \in \S^2},
		\end{equation*}
		and the application of standard compact operator theory to $\L_u$. And \ref{lem:spectral decom 3} follows from the observation that each eigenfunction $\phi_j$ actually is smooth by applying the elliptic bootstrapping to the eigenequations of $\phi_j$. Note that the self-adjointness of $\L_u$ follows from the symmetry of bilinear form $\delta^2 E^\omega_\alpha$, and that by the smothness of critical point $u$, when $\alpha - 1$ is small enough, $\L_u$ is uniformly elliptic. To verify $\L_u$ is a Fredholm operator on $\dbl{\mathcal{T}}_u$, it is enough to establish the following G$\mathrm{\mathring{a}}$rding-type inequality:
		\begin{equation}\label{eq:garding ineq}
			\delta^2 E^\omega_{\alpha}(u)(V,V) \geq C_1(u, H, N)\int_M |\nabla V|^2 d V_g - C_2(u,H,N) \int_{M} |V|^2 d V_g
		\end{equation}
		for some constants $C_1(u, H, N)$, $C_2(u, H, N)$ depending on $u$, mean curvature type  vector field $H$ and geometries of $N$. In fact, since the integrand in second line of \eqref{eq:second variation formula section 3} is positive, it is not difficult to see
		\begin{align}\label{lem:spectral decom eq 1}
			\delta^2 E^\omega_{\alpha}(u)(V,V) &\geq  \int_M \p{  \abs{\nabla V}^2 - C(u,N) \abs{V}^2 }d V_g \nonumber\\
			&\quad - C(u,H)  \int_M \abs{\nabla V} \cdot \abs{V} dV_g - C(u,H)\int_M  |V|^2 dV_g.
		\end{align}
		Then, applying the Cauchy-Schwartz inequality with $\varepsilon$ to the first integrand in the second line of \eqref{lem:spectral decom eq 1} will yield \eqref{eq:garding ineq}, hence completes the proof of Lemma \ref{lem:spectral decom}.
	\end{proof}
	Utilizing the conclusions \ref{lem:spectral decom 2} and \ref{lem:spectral decom 3} of Lemma \ref{lem:spectral decom}, the Morse index of critical point $u$ of $\delta^2 E^\omega_\alpha$ can also be defined as 
	\begin{defi}
		The $\mathrm{Ind}_{E^\omega_\alpha}(u)$ equals to the number of negative eigenvalues of $\L_u$ on $\dbl{\mathcal{T}}_u$ counted with multiplicity. Furthermore, by the spectral decomposition \ref{lem:spectral decom 2} of Lemma \ref{lem:spectral decom} and regularity of $\phi_j$, we have
		\begin{equation*}
			\mathcal{T}_u = \mathcal{T}_u^- \oplus \mathcal{T}_u^+
		\end{equation*}
		where $\mathcal{T}_u^-$ is the direct sum of negative eigenspaces of $\L_u$ with $\dim(\mathcal{T}_u^-) = \mathrm{Ind}_{E^\omega_\alpha}(u)$ and $\mathcal{T}_u^+$ is the $L^2$-orthogonal complement of $\mathcal{T}_u^-$ in $\mathcal{T}_u$. Based on this decomposition, we write $V = V^- + V^+$ for each $V \in \mathcal{T}_u$.
	\end{defi}
	For $\alpha > 1$, $W^{1,2\alpha}(\S^2,N)$ is a Banach manifold. Then, for each $u \in W^{1,2\alpha}(\S^2,N)$, taking a small enough ball $$\mathcal{B}_u(0,r_u) = \set{V \in \mathcal{T}_u \, :\, \norm{V}_{\mathcal{T}_u} < r_u} \subset \mathcal{T}_u$$
	center at the origin of $\mathcal{T}_u$ such that
        $$\left.\delta^2 E^{\omega}_\alpha(\Phi_{u}(w))\right|_{\mathcal{T}_{u}^-}: \mathcal{T}_{w}^- \times \mathcal{T}_{w}^- \cong \mathcal{T}_{u}^- \times \mathcal{T}_{u}^- \rightarrow \R $$ is negatively definite for all $w \in \mathcal{B}_u(0,r_u)$, we define the coordinate map as below
	\begin{equation*}
	   \Phi_u : \mathcal{B}_u(0,r_u) \rightarrow W^{1,2\alpha}(\S^2,N) \quad \text{by } [\Phi_u(V)](x) = \exp_{u(x)}(V(x)), \quad \text{for } x \in \S^2.
	\end{equation*}
	Here, the norm $\norm{\cdot}_{\mathcal{T}_u}$ is induced from inclusion $\mathcal{T}_u \subset W^{1,2\alpha}(\S^2, \R^K)$.
	The collection 
	$$\set{\mathcal{B}_u(0,r_u), \Phi_u}_{u \in W^{1,2\alpha}(\S^2,N)}$$ 
	consists of a smooth structure for $W^{1,2\alpha}(\S^2,N)$. Note that in the sequel, we use 
	\begin{equation*}
		\mathcal{B}^-_u(0,r_u) \quad \text{and}\quad  \mathcal{B}^+_u(0,r_u)
	\end{equation*}
	to represent the balls in $\mathcal{T}_u^-$ and $\mathcal{T}_u^+$, respectively.
	
	In Lemma \ref{lem:local estimates index} below, around each critical point $u$ of $E^{\omega}_\alpha$ we establish some local estimates of $E^{\omega}_\alpha\circ\Phi_u$ on $\mathcal{B}_u(0,r_0(u))$ for small enough $0 < r_0(u) < r_u /3$. 
	\begin{lemma}[See also {\cite[Proposition 4.5]{cheng2022existence}}]
	   \label{lem:local estimates index}
	   With the same notations as above, given $\omega \in  C^2(\wedge^2 (N))$ and $\alpha > 1$, for each critical point $u$ of $E^{\omega}_\alpha$ there exists $0 < r_0 := r_0(u) < {r_u}/{3}$ such that the following holds:
	   \begin{enumerate}[label=(\subscript{I}{{\arabic*}})]
		       \item\label{lem:local estimates index 1} There exists a constant $0 < \kappa := \kappa(u) < 1$ and a constant $C := C(u) > 0$ such that for all $V \in \mathcal{T}_u$ with
		       \begin{equation*}
			           V \in \mathcal{B}_u(0,r_0) \quad \text{and} \quad \norm{V^+}_{\mathcal{T}_u}\leq \kappa \norm{V^-}_{\mathcal{T}_u},
			       \end{equation*}
		        we have 
		        \begin{equation*}
			           E^{\omega}_{\alpha}(\Phi_u(V)) -  E^{\omega}_{\alpha}(\Phi_u(0)) \leq - C \norm{V^-}_{\mathcal{T}_u}^2. 
			        \end{equation*}
		        \item\label{lem:local estimates index 2} There exists a  constant $C := C(u) > 0$ such that for all $V, W \in \mathcal{T}_u$ with
		       \begin{equation*}
			           V \in \mathcal{B}_u(0,r_0) \quad \norm{W^-}_{\mathcal{T}_u} = 1 \quad  \text{and} \quad \delta E^{\omega}_\alpha(\Phi_u(V))(W^-) \leq 0,
			       \end{equation*}
		        we have
		        \begin{equation*}
			            E^{\omega}_{\alpha}(\Phi_u(V + r W^-))  -  E^{\omega}_{\alpha}(\Phi_u(V)) \leq - C r^2, \quad \text{for }\, 0 \leq r \leq r_0.
			        \end{equation*}
		   \end{enumerate}
	\end{lemma}
	\begin{proof}
	   Since $\delta^2 E^{\omega}_\alpha(\Phi_u(0)) = \delta^2  E^{\omega}_{\alpha}(u): \mathcal{T}_u \times \mathcal{T}_u \rightarrow \R$ is a bounded bilinear form, there exists $C_1(u) := \norm{\delta^2  E^{\omega}_{\alpha}(u)} > 0$ such that 
	   \begin{equation}\label{eq:local estimates 1}
		       \abs{\delta^2 E^{\omega}_\alpha(\Phi_u(0))(V,V)} \leq C_1(u) \norm{V}_{\mathcal{T}_u}^2, \quad \text{for all  } V \in \mathcal{T}_u.
		   \end{equation}
	   In particular, observe that $\dim(\mathcal{T}_u^-) = \mathrm{Ind}_{E^{\omega}_\alpha}(u) < \infty$, the induced norms $\norm{\cdot}_{\mathcal{T}_u}$ and $\norm{\cdot}_{\dbl{\mathcal{T}}_u}$ restricted on $\mathcal{T}_u^-$ are equivalent. There exists a constant $C_2(u) > 0$ such that 
	   \begin{equation} \label{eq:local estimates 2}
		       \delta^2 E^{\omega}_\alpha(\Phi_u(0))(V^-,V^-) \leq -C_2(u) \norm{V^-}^2_{\mathcal{T}_u} \quad \text{for all } V^- \in \mathcal{T}_u^-.
		   \end{equation}
	   By the continuity of $\delta^2  E^{\omega}_{\alpha}\circ \Phi_u$ on $\mathcal{B}_{u}(0,r_u)$, we can choose small enough $0 < r_0 < r_u/3$ such that for any $W \in \mathcal{B}_{u}(0,2r_0)$ and any $V \in \mathcal{T}_u$ there holds
	   \begin{equation}\label{eq:local estimates 3}
		       \abs{\delta^2 E^{\omega}_\alpha(\Phi_u(W))(V,V) - \delta^2 E^{\omega}_\alpha(\Phi_u(0))(V,V)} \leq \frac{\kappa^2 C_1(u)}{4} \norm{V}^2_{\mathcal{T}_u},
		   \end{equation}
	   where $\kappa > 0$ is a constant that will be determined later.
	   Therefore, for $V \in \mathcal{B}_u(0,r_0)$ with $\norm{V^+}_{\mathcal{T}_u}\leq \kappa \norm{V^-}_{\mathcal{T}_u}$, utilizing the Taylor formula with integral remainder at critical point $u$ and keeping in mind \eqref{eq:local estimates 1} and \eqref{eq:local estimates 2} we obtain
	   \begin{align}
		        E^{\omega}_{\alpha}(\Phi_u(V)) &-  E^{\omega}_{\alpha}(\Phi_u(0)) \nonumber\\
                &= \frac{1}{2}\p{ \delta^2  E^{\omega}_{\alpha}(\Phi_u(0))(V^-,V^-) + \delta^2  E^{\omega}_{\alpha}(\Phi_u(0))(V^+,V^+)}\nonumber\\
		       &\quad + \int_{0}^1 (1-s) \p{\delta^2 E^{\omega}_\alpha(\Phi_u(s V))(V,V) - \delta^2 E^{\omega}_\alpha(\Phi_u(0))(V,V)} ds \nonumber\\
		       & \leq - \frac{C_2(u)}{ 2} \norm{V^-}^2_{\mathcal{T}_u} + \frac{C_1(u)}{2} \norm{V^+}^2_{\mathcal{T}_u} + \frac{\kappa^2 C_1(u)}{8} \norm{V}^2_{\mathcal{T}_u}\nonumber\\
		       & \leq - \frac{1}{2} \p{C_2(u) -  \kappa^2 C_1(u) - {\kappa^2 C_1(u)}} \norm{V^-}_{\mathcal{T}_u}^2,
		   \end{align}
	   where in the first equality we used $\delta^2 E^{\omega}_\alpha(\Phi_u(0))(\mathcal{T}_u^-, \mathcal{T}_u^+) = 0$ in viewing of part \ref{lem:spectral decom 2} in Lemma \ref{lem:spectral decom} and we used the inequality $\norm{V}_{\mathcal{T}_u} \leq 2 \norm{V^-}_{\mathcal{T}_u}$ in the last inequality which comes from our assumption $\norm{V^+}_{\mathcal{T}_u}\leq \kappa \norm{V^-}_{\mathcal{T}_u}$ for $0 < \kappa < 1$. Then, letting $\kappa > 0$ small enough such that 
	   \begin{equation*}
		       \kappa^2 < \min \p{\frac{C_2(u)}{4 C_1(u)},1}
		   \end{equation*}
	   will yield the conclusion of \ref{lem:local estimates index 1} by taking $C(u) := \frac{1}{4} C_2(u)$.
	
	   Next, in order to prove \ref{lem:local estimates index 2} we apply \eqref{eq:local estimates 3} and the Taylor formula with integral remainder for $E^\omega_\alpha \circ \Phi_u$ at $u$ acting on  $V \in \mathcal{B}_u(0,r_0)$ and $V + t W^- \in \mathcal{B}_u(0, 2r_0)$ satisfying  $\norm{W^-}_{\mathcal{T}_u} = 1$ \text{and}  $\delta E^{\omega}_\alpha(\Phi_u(V))(W^-) \leq 0,$  to obtain
	   \begin{align}
		       E^{\omega}_\alpha(\Phi_u(V &+ t W^-)) - E^{\omega}_\alpha(\Phi_u(V))\nonumber \\
		       &= t \delta E^{\omega}_\alpha(\Phi_u(V))(W^-) + \frac{t^2}{2}\delta^2 E^{\omega}_\alpha(\Phi_u(V))(W^-,W^-)\nonumber\\
		       &\quad +\int_0^t (t-s) \Big(\delta^2 E^{\omega}_\alpha(\Phi_u(V+ s W^-))(W^-, W^-) \nonumber\\
         &\quad \quad \quad \quad \quad \quad \quad \quad  - \delta^2 E^{\omega}_\alpha(\Phi_u(0))(W^-, W^-)\Big) ds\nonumber\\
		       &\leq - \frac{t^2}{2} \p{C_2(u) - \frac{\kappa^2 C_1(u)}{8} }\norm{W^-}_{\mathcal{T}_u}^2 \leq -\frac{31 C_2(u)}{64} t^2.
		   \end{align}
	   This leads to the conclusion of \ref{lem:local estimates index 2} in Lemma \ref{lem:local estimates index} by letting
	   $$C(u) := \frac{31 C_2(u)}{64}. $$
	\end{proof}

	In order to describe the main result in this subsection more precisely, given $\lambda >0$ and $\alpha > 1$, 
 writing $\omega$ as $\lambda \omega$ we define $\mathcal{A}_{\varepsilon, C}$ to be the collection of admissible sweepouts $\sigma \in \mathscr{S}$ satisfying
	\begin{equation*}
		\max_{t \in I^{k-2}} E^{\omega}_\alpha(\sigma(t)) \leq \mathcal{W}_{\alpha,\lambda} + \varepsilon
	\end{equation*}
	and 
	\begin{equation*}
		E_\alpha(\sigma(t))\leq C \quad \text{as long as $t \in I^{k-2}$ satisfying}\quad E_\alpha^{\omega}(\sigma(t)) \geq \mathcal{W}_{\alpha,\lambda} - \varepsilon.
	\end{equation*}
Note that $\mathcal{A}_{\varepsilon, C}$ is exactly the set of sweepouts that fulfills the assumptions of Proposition \ref{prop energy bound} by replacing constant $C > 0$ by $8\lambda^2 C$. 
 Moreover, for $C > 0$ we define $\mathcal{U}_C \subset W^{1,2\alpha}(\S^2, N)$ to be the set of critical points $u \in W^{1,2\alpha}(\S^2, N)$ for functional $E^{\omega}_\alpha$ satisfying
	\begin{equation*}
		\delta  E^{\omega}_{\alpha}(u) = 0 \quad \text{with}\quad  E_\alpha(u) \leq C \text{ and } E_{\alpha}^{\omega}(u) = \mathcal{W}_{\alpha,\lambda}.
	\end{equation*}
	Since $E^{\omega}_\alpha$ satisfies the Palais-Smale condition, see Lemma \ref{P-S condition}, it is not difficult to see that $\mathcal{U}_C \subset W^{1,2\alpha}(\S^2, N)$ is a compact set. Equipped with these notations at our disposal, we are prepared to present the main result in this subsection:
	\begin{theorem}[Deformations of Sweepouts]\label{prop deformation}
		Given $\omega \in C^2(\wedge^2 (N))$, $\lambda > 0$ writing $\omega$ as $\lambda \omega$, $\alpha > 1$ and $C > 0$,  let $\mathcal{U}_0$ be a closed subset of $\mathcal{U}_{C+1}$. If $\mathrm{Ind}_{ E^{\omega}_{\alpha}}(u) \geq k-1$ for all $u \in \mathcal{U}_0$, then for each sequence of sweepouts $\{\sigma_j\}_{j \in \mathbb{N}} \subset \mathcal{A}_{\varepsilon_j, C}$ with $\varepsilon_j \searrow 0$ there exists another sequence of sweepouts $\{\dbl{\sigma}_j\}_{j \in \mathbb{N}} \subset \mathscr{S}$ such that
		\begin{enumerate}[label=(\subscript{\textit{T}}{{\arabic*}})]
			\item\label{prop deformation 1} $\dbl{\sigma}_j \in \mathcal{A}_{\varepsilon_j, C + 1}$ when $j$ is large enough.
			\item \label{prop deformation 2}For any $u \in \mathcal{U}_0$, there exists $j_0(u) \in \mathbb{N}$, $\varepsilon_0(u) > 0$ such that for all $j \geq j_0(u)$,
			\begin{equation*}
				\inf_{t \in I^{k-2},\,\, j \geq j_0}\set{\norm{\dbl{\sigma}_j(t) - u}_{W^{1,2\alpha}(\S^2, N)}\,:\, E^{\omega}_\alpha(\dbl{\sigma}_j(t)) \geq \mathcal{W}_{\alpha,\lambda}- \varepsilon_j} \geq \varepsilon_0(0).
			\end{equation*}
		\end{enumerate}
	\end{theorem}
 	\begin{proof}
		Firstly, we observe that there exists $\delta_0 > 0$ such that if $\sigma \in \mathscr{S}$ is an admissible sweepout then
		\begin{equation}\label{eq:the definition of delta}
			\max_{t \in I^{k-2}} E_\alpha(\sigma(t)) \geq \frac{1}{2} \mathrm{Vol}(\S^2) + \delta_0.
		\end{equation}
		In fact, suppose that for any $\delta > 0$ there exists a $\sigma \in \mathscr{S}$ such that 
		\begin{equation*}
			\max_{t \in I^{k-2}} E_\alpha(\sigma(t)) < \frac{1}{2} \mathrm{Vol}(\S^2) +  \delta.
		\end{equation*}
		By Poincar\'{e}'s inequality, we see that there exists a $C(\alpha) > 0$ such that the oscillation of $\sigma(t) \in W^{1,2\alpha}(\S^2, N)$ satisfies
		\begin{equation*}
			Osc_{\S^2}(\sigma(t)) \leq C(\alpha) \p{E_\alpha(\sigma(t)) - \frac{1}{2} \mathrm{Vol}(\S^2)}^{\frac{1}{2\alpha}} \leq C(\alpha) \delta^{\frac{1}{2\alpha}}, \quad \text{for all }t \in I^{k-2},
		\end{equation*}
		which means $\sigma:I^{k-2} \rightarrow W^{1,2\alpha}(\S^2, N)$ can be deformed onto a map that assigns each points in $(k-2)$-dimensional complex cube $I^{k-2}$ to constant maps, hence the induced map  $f_\sigma: \S^k \rightarrow N$ is null-homotopic, contradicting to the definition of $\mathscr{S}$.

	Equipped with this observation, the construction of new sweepouts $\dbl{\sigma}_j \in \mathscr{S}$ stated in Theorem \ref{prop deformation} splits into  four steps.
		\step\label{step1:deformation}We construct a finite subset $\set{u_i}_{i = 1}^m \subset \mathcal{U}_0$ and their associated open balls 
  \begin{equation*}
      {B}^{1,2\alpha}({u_i, r_1(u_i)}) \subset W^{1,2\alpha}(\S^2, N) \quad \text{for }1\leq i \leq m
  \end{equation*}
		satisfying the following:
		\begin{enumerate}[label=(\subscript{A}{{\arabic*}})]
			\item \label{step1:item 1}$\mathcal{U}_0 \subset \bigcup_{i = 1}^m B^{1,2\alpha}({u_i, r_1(u_i)})\subset \bigcup_{i = 1}^m\mathcal{D}_{u_i}(1)$ where $\mathcal{D}_u(1)$ is defined in \eqref{eq: defi Du};
                \item \label{step1:item 2} $ E^{\omega}_{\alpha}(v) \leq  E^{\omega}_{\alpha}(u) -  \frac{(r_0(u))^2}{4}C(u)$ \text{for any } $v \in \partial^-\mathcal{D}_u(2)$ where $r_0(u)$ and $C(u)$ are obtained in \ref{lem:local estimates index 1} of Lemma \ref{lem:local estimates index} and $\partial^-\mathcal{D}_u(2)$ is defined in \eqref{eq: defi boundary Du}.
			\item \label{step1:item 3}For each $1\leq i \leq m$ and any $v$, $w \in B^{1,2\alpha}({u_i, 2r_1(u_i)})$, there holds 
			\begin{equation}
				|E_\alpha(u) - E_\alpha(v)| \leq \frac{\delta_0}{4}
			\end{equation}
        where $\delta_0 >0$ is obtained in \eqref{eq:the definition of delta}.
		\end{enumerate}
  Here, 
  $$B^{1,2\alpha}(u_i,r_1(u_i)) := \set{v \in W^{1,2\alpha}(\S^2, N)\, :\, \norm{v-u_i}_{W^{1,2\alpha}(\S^2, N)} < r_1(u_i)}$$
  are balls defined with respect to the topology induced by the Finsler structure of $W^{1,2\alpha}(\S^2, N)$.
		\begin{proof}[\textbf{Proof of Step \ref{step1:deformation}}]
			For each $u \in \mathcal{U}_0$, by Lemma \ref{lem:local estimates index} we can find constants $0 < r_0(u) < r_u/3 $, $0 < \kappa(u) < 1$ and $C(u) > 0$ such that the conclusions of \ref{lem:local estimates index 1} and \ref{lem:local estimates index 2} in Lemma \ref{lem:local estimates index} can be applied in the neighborhood $\Phi_u(\mathcal{B}_u(0,r_u))$ of $u$. 
			By the continuity of $E_\alpha:W^{1,2\alpha}(\S^2, N) \rightarrow \R$, after shrinking $r_0(u)$ if necessary, we can further assume that 
			\begin{equation}\label{eq:step1 deformation 1}
				|E_\alpha(v) - E_\alpha(w)| \leq \frac{\delta_0}{4} \quad \text{for any }\, v, w \in \Phi_u\big(\mathcal{B}_u(0,r_0(0))\big).
			\end{equation}
				Then, we define
                    \begin{align}
                        \label{eq: defi Du}
                        \mathcal{D}_u(\rho):= \Phi_u\p{ \set{V \in \mathcal{T}_u \,:\, \norm{V^-}_{\mathcal{T}_u} \leq \frac{r_0(u)}{4} \rho, \norm{V^+}_{\mathcal{T}_u} \leq \frac{\kappa(u) r_0(u)}{4} \rho}},
                    \end{align}
                    and 
                    \begin{equation}\label{eq: defi boundary Du}
                        \partial^-\mathcal{D}_u(\rho) := \Phi_u\p{ \set{V \in \mathcal{T}_u \,:\, \norm{V^-}_{\mathcal{T}_u} = \frac{r_0(u)}{4} \rho, \norm{V^+}_{\mathcal{T}_u}\leq \frac{\kappa(u) r_0(u)}{4} \rho }}
                    \end{equation}
                    for $\rho \in [1,4]$. Thus, by the estimates obtained in \ref{lem:local estimates index 1} of Lemma \ref{lem:local estimates index} we see that 
                    \begin{equation}
                        \label{eq:step1 deformation decrease}
                         E^{\omega}_{\alpha}(v) \leq  E^{\omega}_{\alpha}(u) -  \frac{\p{r_0(u)}^2}{4}C(u) \quad \text{for any } v \in \partial^-\mathcal{D}_u(2)
                    \end{equation}
				Then, for each $u \in \mathcal{U}_0$ we can choose small enough $ 0 < r_1(u) < r_0(u)$ such that 
				\begin{align}\label{step1 the definition ball}
					B^{1,2\alpha}(u, 2r_1(u))  \subset \mathcal{D}_u(1).
				\end{align}
                 Since the collection $\set{B^{1,2\alpha}(u, r_1(u))}_{u \in \mathcal{U}_0}$ consists of an open covering of compact set $\mathcal{U}_0$, there exists a finite subcovering $\set{B^{1,2\alpha}(u,r_1(u))}_{i =1}^m$ satisfying all the assertions of Step \ref{step1:deformation} by the choice of $r_1(u)$.
			\end{proof}

			In the following, for the notation simplicity, we write 
			\begin{equation*}
				r_i := r_1(u_i) \quad  \text{and} \quad  b_i := \frac{\p{r_0(u_i)}^2}{4} C(u_i),
			\end{equation*}
			and  denote 
			\begin{equation*}
				\underline{r} := \min_{1 \leq i \leq m} r_i, \quad \underline{b} := \min_{1\leq i\leq m}  b_i, \quad \underline{C}:= \min_{1 \leq i \leq m} C(u_i).
			\end{equation*}
			\step \label{step2:deformation} We construct a constant $\eta > 0$ such that the following properties hold
			\begin{enumerate}[label=(\subscript{B}{{\arabic*}})]
				\item $\mathcal{N}_{\eta}^{1,2\alpha} := \bigcup_{u \in \mathcal{U}_0} B^{1,2\alpha}(u,\eta) \subset \bigcup_{i = 1}^m B^{1,2\alpha}(u_i,r_i)$.
				\item\label{step2 item 2} For any $u \in \mathcal{U}_0$ and any $v \in B(u,\eta)$, we have 
				\begin{equation}
					|E^{\omega}_\alpha(v) - \mathcal{W}_{\alpha,\lambda}| \leq \frac{1}{4}\underline{b}.
				\end{equation}
				\item\label{step2 item 3} For any $p \in \set{0,1,\cdots,k-2}$, any $\vartheta \in (0,\underline{b}/4)$,  any $1\leq i \leq m$ and any continuous map $\varsigma : I^p \rightarrow \mathcal{N}_\eta^{1,2\alpha} \cap B^{1,2\alpha}(u_i, r_i)$, there exists a continuous homotopy $H_{p,i}^{\varsigma,\vartheta}:I^p \times [0,1] \rightarrow \mathcal{D}_{u_i}(3)$ such that the following holds
				\begin{enumerate}[label=(\subscript{3b}{{\arabic*}})]
					\item\label{step2:deformation item 3 a} $H_{p,i}^{\varsigma,\vartheta}(\tau,0) = \varsigma (\tau)$ for $\tau \in I^p$.
					\item \label{step2:deformation item 3 b} $E^{\omega}_\alpha\p{H_{p,i}^{\varsigma,\vartheta}(\tau,t)} - E^{\omega}_\alpha\p{\varsigma(\tau)} \leq \vartheta$, for all $\tau \in I^p$ and  $t \in [0,1]$.
					\item \label{step2:deformation item 3 c} $\norm{H_{p,i}^{\varsigma,\vartheta}(\tau,t) - \varsigma(\tau)}_{W^{1,2\alpha}(\S^2,N)} \leq \vartheta$, for all $\tau\in I^p$ and  $t \in [0,1/2]$.
					\item \label{step2:deformation item 3 d}$H_{p,i}^{\varsigma,\vartheta}(\tau,1) \notin \mathcal{N}_\eta^{1,2\alpha}$, for all $\tau \in I^{p}$. 
				\end{enumerate}
			\end{enumerate}
			\begin{proof}[\textbf{Proof of Step \ref{step2:deformation}}]
				Given $\eta > 0$ we define 
				\begin{equation*}
					\mathcal{N}_\eta^{1,2\alpha} : = \bigcup_{u \in \mathcal{U}_0} B^{1,2\alpha}(u,\eta)
				\end{equation*}
				and by the compactness of $\mathcal{U}_0$ we can choose small enough $\eta > 0$ such that 
				\begin{equation*}
					\mathcal{N}_\eta^{1,2\alpha} \subset  \bigcup_{i = 1}^{m} B^{1,2\alpha}(u_i, r_i)
				\end{equation*}
				and the second assertion \ref{step2 item 2} of Step \ref{step2:deformation} can also be satisfied by the continuity of $w \mapsto |E^{\omega}_\alpha(w) - \mathcal{W}_{\alpha,\lambda}|$. For the part \ref{step2 item 3} of Step \ref{step2:deformation}, firstly, we are devoted to construct a  continuous homotopy $\hat{H}^{\varsigma,\vartheta}_{p,i}$ inductively on the $l$-dimensional skeleton $I^p_{(l)}$ of $I^p$ for $0 \leq l \leq p$ such that 
				\begin{enumerate}[label=(\subscript{\hat{H}}{{\arabic*}})]
					\item \label{first homotopy 1}$\hat{H}^{\varsigma,\vartheta}_{p,i}(\tau,0) = \varsigma(\tau)$, for $\tau \in I^p$.
					\item \label{first homotopy 2}For $t \in [0,1]$ and $\tau \in I^p$,
					$$\norm{\hat{H}^{\varsigma,\vartheta}_{p,i}(\tau,t) - \varsigma(\tau)}_{W^{1,2\alpha}(\S^2,N)} \leq \vartheta.$$
					\item \label{first homotopy 3}There exists some $\delta(\vartheta,\varsigma,p,i) > 0$ satisfying
					\begin{equation*}
					    \inf_{\tau \in I^p}\p{ \sup_{V^- \in \mathcal{T}_{u_i}^- \text{ with } \norm{V^-}_{\mathcal{T}_{u_i}} = 1} \abs{\delta E^{\omega}_\alpha\p{ \Phi_{u_i}^{-1}\p{\hat{H}^{\varsigma,\vartheta}_{p,i}(\tau,1)}}(V) }} \geq  \delta(\vartheta,\varsigma,p,i) > 0.
					\end{equation*}
				\end{enumerate}
				If 
				\begin{equation*}
				 \left.	\delta E^{\omega}_\alpha\p{\Phi_{u_i}^{-1}(\varsigma(\tau))}\right|_{\mathcal{T}^-_{u_i}} \not\equiv 0 \quad \text{for each } \tau \in I^p,
				\end{equation*}
				then by the continuity of $\varsigma$ on compact complex cube $I^p$ we can simply choose $\hat{H}^{\varsigma,\vartheta}_{p,i} \equiv \varsigma$ for $t \in [0,1]$ being the constant homotopy and let $\delta(\vartheta,\varsigma,p,i)$ to be
                \begin{equation*}
                     \inf_{\tau \in I^p}\p{ \sup_{V^- \in \mathcal{T}_{u_i}^- \text{ with } \norm{V^-}_{\mathcal{T}_{u_i}} = 1} \abs{\delta E^{\omega}_\alpha\p{ {\Phi_{u_i}^{-1}(\varsigma(\tau))}}(V) }} > 0.
                \end{equation*}
                Thus, we assume that 
                $$\left.\delta E^{\omega}_\alpha \p{{\Phi_{u_i}^{-1}(\varsigma(\tau))}}\right|_{\mathcal{T}^-_{u_i}} \equiv 0$$
                for some $\tau \in I^p$. Observe that $\delta E^{\omega}_\alpha(\Phi_{u_i}(0)) = 0$ and $\delta^2 E^{\omega}_\alpha(\Phi_{u_i}(w))|_{\mathcal{T}_{u}^-}$ is negatively definite for all $w \in \mathcal{B}_{u_i}(0,r_{u_i})$ by the choice of $r_{u_i} > 0$, which implies that $\delta E^{\omega}_\alpha(\Phi_{u_i}(V^-)) \not\equiv 0$ for $V^- \in \mathcal{B}_{u_i}^-(0, r_i)\backslash \set{0}$. Since  $\mathcal{B}^-_{u_i}(0,r_{i})$ is a convex set of dimension at least $k-1$ by the assumption of Theorem \ref{prop deformation}, it suffices to define $\hat{H}^{\varsigma,\vartheta}_{p,i}(\tau,1)$ through the homeomorphism $\Phi_{u_i}$ such that 
				\begin{equation}\label{step2: eq 3}
					\norm{\Phi_{u_i}^{-1}\p{ \hat{H}^{\varsigma,\vartheta}_{p,i}(\tau,1)} - \Phi_{u_i}^{-1}\p{\varsigma(\tau)}}_{\mathcal{T}_{u_i}} \leq \vartheta^\prime, \quad \text{for all } \tau \in I^{p}
				\end{equation}
				and 
				\begin{equation}\label{step2: eq 4}
					\p{\Phi_{u_i}^{-1}\p{\hat{H}^{\varsigma,\vartheta}_{p,i}(\tau,1)}}^- \neq 0, \quad \text{for all } \tau \in I^p.
				\end{equation}
				Here, by the continuity of $\Phi_{u_i}$, $\vartheta^\prime > 0$ is chosen to satisfy the property \ref{first homotopy 2} in Step \ref{step2:deformation} of $\hat{H}^{\varsigma,\vartheta}_{p,i}$ as long as \eqref{step2: eq 3} holds. Then, for each $\tau \in I^p_{(0)}$ which is a finite discrete set,  we can choose $\big(\Phi_{u_i}^{-1}(\hat{H}^{\varsigma,\vartheta}_{p,i}(\tau,1))\big)^- \neq 0$ satisfying \eqref{step2: eq 3}. Suppose we have defined the homotopy $\hat{H}^{\varsigma,\vartheta}_{p,i}(\tau,1)$ on $I^p_{(l)}$ for some $0 \leq l\leq p-1$ such that \eqref{step2: eq 3} and \eqref{step2: eq 4} are satisfied. Let $\mathcal{O}\big(\Phi_{u_i}^{-1}\p{\varsigma(\tau)}, \vartheta^\prime\big)$ be the $\vartheta^\prime$-neighborhood of $\Phi_{u_i}^{-1}\big(\varsigma(I^p_{(l)})\big)$, that is, 
				\begin{align*}
			\mathcal{O}&\p{\Phi_{u_i}^{-1}\p{\varsigma\p{I^p_{(l)}}}, \vartheta^\prime} \\
                        &:= \set{w \in \mathcal{T}_{u_i} : \min_{\tau \in I^p_{(l)}}\norm{w - \Phi_{u_i}^{-1}\p{\varsigma(\tau)}}_{\mathcal{T}_{u_i}}< \vartheta^\prime}
				\end{align*}    
				which is also equal to 
				$$\bigcup_{ \tau \in I^{q}} \mathcal{B}_{u_i}\p{\Phi_{u_i}^{-1}\p{\varsigma(\tau)},\vartheta^\prime}.$$
				And we use 
                    $$\mathcal{O}^-\p{\Phi_{u_i}^{-1}\p{\varsigma(\tau)}, \vartheta^\prime}$$
                    to denote the $L^2$-orthogonal projection of $\mathcal{O}(\Phi_{u_i}^{-1}\p{\varsigma(\tau)}, \vartheta^\prime)$ into $\mathcal{T}_{u_i}^-$. Since $\dim{\mathcal{T}_{u_i}^-} \geq k-1 > p \geq l+1$, we have that 
				\begin{equation*}
					\pi_{l}\p{\mathcal{O}^-\p{\Phi_{u_i}^{-1}\p{\varsigma(I^p_{(l)})}, \vartheta^\prime} \backslash \set{0}} = \pi_l\p{\mathcal{T}_{u_i}^- \backslash \set{0}} =  0,
				\end{equation*} 
				there exists a continuous extension of $H^{\varsigma,\vartheta}_{p,i}(\tau,1)$ from $I^p_{(l)}$ onto $I^p_{(l+1)}$ such that \eqref{step2: eq 3} and \eqref{step2: eq 4} are also satisfied. Therefore, we complete the induction construction of the homotopy $\hat{H}^{\varsigma,\vartheta}_{p,i}$.
				
				Now, we construct the desired homotopy $H^{\varsigma,\vartheta}_{p,i}$ satisfying the properties stated in \ref{step2 item 3} of Step \ref{step2:deformation}. For each $w \in \mathcal{B}_{u_i}(0,r_i) \backslash \set{0}$, we pick up a $\xi_{u_i}(w) \in \mathcal{T}_{u_i}^-$ with $\norm{\xi_{u_i}(w)}_{\mathcal{T}_{u_i}} = 1$ satisfying
                \begin{equation*}
                    \delta E^{\omega}_\alpha\p{\Phi_{u_i}} (\xi_{u_i}(w)) := \inf_{V^- \in \mathcal{T}_{u_i}^- \text{ with } \norm{V^-}_{\mathcal{T}_{u_i}} = 1} \delta E^{\omega}_\alpha\p{\Phi_{u_i}(w^-)}(V^-) < 0.
                \end{equation*} 
                Because $$\left.\delta E^{\omega}_\alpha\p{\Phi_{u_i}}\right|_{\mathcal{T}_{u_i}^-}: \mathcal{T}_{u_i}^- \rightarrow \R$$ is a linear function defined on a finite dimensional vector space $\mathcal{T}_{u_i}^-$, 
                the minimum point $\xi_{u_i}(w) \in \mathcal{T}_{u_i}^-$ on unit sphere of $\mathcal{T}_{u_i}^-$ is unique and well defined. Then, we pick $0 < \varrho_{u_i}(w) \leq r_{u_i}$ such that 
                \begin{equation*}
                    \norm{w^- + \varrho_{u_i}(w) \xi_{u_i}(w)}_{\mathcal{T}_{u_i}} = \frac{r_0(u_i)}{2}
                \end{equation*}
                and that 
                \begin{equation*}
                    \norm{w^- + t \xi_{u_i}(w)}_{\mathcal{T}_{u_i}} \leq \frac{r_0(u_i)}{2} \quad \text{for any  } 0\leq t \leq \varrho_{u_i}(w).
                \end{equation*}
                Thus, for $w \in \mathcal{B}^-_{u_i}(0,r_i) \backslash \set{0}$, we can define a path $\gamma_w : [0,1]\rightarrow \mathcal{D}_{u_i}(3)$ as below
				\begin{equation*}
					\gamma_w(t) = \Phi_{u_i}\p{ w^0 + \p{w^- + t\frac{\varrho_{u_i}(w)}{2\norm{\xi_{u_i}(w)}_{\mathcal{T}_{u_i}}} \xi_{u_i}(w) } + w^+}
				\end{equation*}
				which starts at $\Phi_{u_i}(w)$ and  terminates at 
				$$\Phi_{u_i}\p{ w^0 + \p{w^- + \frac{\varrho_{u_i}(w)}{2} \cdot\frac{\xi_{u_i}(w)}{2\norm{\xi_{u_i}(w)}_{\mathcal{T}_{u_i}}}}  + w^+} \in \partial^- \mathcal{D}_{u_i}(2) .$$
                
			    By the \ref{lem:local estimates index 1} in Lemma \ref{lem:local estimates index} for $E^{\omega}_\alpha\circ \Phi_{u_i}$ and the choice of $\xi_{u_i}(w)$, we see that 
				\begin{align*}
				     E^{\omega}_{\alpha}(\gamma_w(1)) - \mathcal{W}_{\alpha,\lambda} &=   E^{\omega}_{\alpha}(\gamma_w(1)) -  E^{\omega}_{\alpha}(\gamma_w(0)) +  E^{\omega}_{\alpha}(\gamma_w(0)) - \mathcal{W}_{\alpha,\lambda}\\
                    &\leq -\frac{\p{r_0(u_i)}^2}{4} C(u_i) + \frac{1}{4} \underline{b} \leq - \frac{3}{4}\underline{b}.
				\end{align*}
				Therefore, we define
				\begin{equation*}
					H^{\varsigma,\vartheta}_{p,i}(\tau,t) := \left\{
					\begin{aligned}
						&\hat{H}^{\varsigma,\vartheta}_{p,i}(\tau,2t), &\text{for }& t \in \left[0,\frac{1}{2} \right],\\
						&\gamma_{\hat{H}^{\varsigma,\vartheta}_{p,i}(\tau,1)}(2t-1),\quad  &\text{for }& t \in \left[\frac{1}{2},1 \right].
					\end{aligned}
					\right.
				\end{equation*}
                By the construction of $\hat{H}^{\varsigma,\vartheta}_{p,i}$, we see that 
                \begin{equation*}
                    \left.\delta E^{\omega}_\alpha\p{{\Phi_{u_i}^{-1}\p{\hat{H}^{\varsigma,\vartheta}_{p,i}(\tau,1)}}}\right|_{\mathcal{T}_{u_i}^-} \not\equiv 0 \quad \text{for each } \tau \in I^p,
                \end{equation*}
                which means that $\xi_{u_i}(\hat{H}^{\varsigma,\vartheta}_{p,i}(\tau,1))$ depends continuously on $\tau \in I^p$. Hence, $H^{\varsigma,\vartheta}_{p,i}(\tau,t)$ is a well-defined continuous homotopy.
                
			Next, we show that $H^{\varsigma,\vartheta}_{p,i}$ satisfies the properties stated in \ref{step2 item 3} of Step \ref{step2:deformation}. When $t \in [0,1/2]$, the \ref{step2:deformation item 3 a}, \ref{step2:deformation item 3 b} and \ref{step2:deformation item 3 c} follow straightforwardly from the construction of $\hat{H}^{\varsigma, \vartheta}_{p,i}$ corresponding to properties \ref{first homotopy 1} and \ref{first homotopy 2}. When $t \in [1/2,1]$, thanks to \ref{lem:local estimates index 2} of Lemma \ref{lem:local estimates index} and the definition of $H^{\varsigma,\vartheta}_{p,i}$, the \ref{step2:deformation item 3 b} of Step \ref{step2:deformation} also holds. At last, by the choice of $\vartheta \leq \underline{b}/4$ and the construction of $\hat{H}^{\varsigma,\vartheta}_{p,i}$, we see that for all $\tau \in I^p$
                \begin{equation*}
                    E^{\omega}_\alpha\p{{H^{\varsigma,\vartheta}_{p,i}(\tau,1)}} - E^{\omega}_\alpha(\varsigma(\tau)) \leq -\frac{3}{4} \underline{b} + \vartheta \leq -\frac{1}{2} \underline{b}.
                \end{equation*}
                Thanks to the choice of $\eta > 0$ in \ref{step2 item 2} of Step \ref{step2:deformation}, we see that $H_{p,i}^{\varsigma,\vartheta}(s,1) \notin \mathcal{N}_\eta^{1,2\alpha}$, for all $s \in I^{k-2}$.  
			\end{proof}
   
			\step\label{step3:deformation}We would like to show that for any $u \in \mathcal{U}_0$, there exists finite many positive numbers $\set{e_{p}(u)}_{p = 1}^{k-1} \subset \R_+$ and $\set{\theta_{p}(u)}_{p = 0}^{k-2} \subset \R_+$ such that for any $p \in \set{1,\cdots,k-1}$, any $i \in \set{1,\cdots, m}$ with $B^{1,2\alpha}(u,\eta) \subset B^{1,2\alpha}(u_i, r_i) $ and any $\tau \in I^p$ with 
			\begin{equation*}
				\varsigma(\tau) \notin B^{1,2\alpha}(u,\eta/e_p(u)) \quad \text{and}\quad E^{\omega}_\alpha(\varsigma(\tau)) - \mathcal{W}_{\alpha,\lambda} \leq \theta_p(u),
			\end{equation*}
			we have 
			\begin{equation*}
				H^{\varsigma, \vartheta}_{p,i}(\tau,t) \notin B^{1,2\alpha}(u,\eta/e_{p + 1}(u))
			\end{equation*}
			for all $t \in [0,1]$ and $\vartheta < \min \p{\eta/(4e_p(u)), \theta_p(u)}$.
   
                Furthermore, viewing $e_p(u)$ and $\theta_p(u)$ as functions 
                $$e_p(u) : \mathcal{U}_0 \rightarrow \R_+ \quad \text{and}\quad \theta_p(u) : \mathcal{U}_0 \rightarrow \R_+, $$
                $e_p(u)$  has an uniform upper bound $\overline{e}_p$ on $\mathcal{U}_0$ and $\theta_p(u)$ has a positive lower bound $\underline{\theta}_p > 0$ on $\mathcal{U}_0$.
			
			\begin{proof}[\textbf{Proof of Step \ref{step3:deformation}}]
				We construct $e_p(u)$ and $\theta_{p}(u)$ by induction on $p$. For $u \in \mathcal{U}_0$, we take $e_1(u) = 2$, $\theta_0(u) = 0$ and suppose that we have defined $e_p(u)$ and $\theta_{p-1}(u)$ for some $p \in \set{2,\cdots,k-1}$.
				
				For any natural number $1 \leq i \leq m$ with $B^{1,2\alpha}(u,\eta) \subset B^{1,2\alpha}(u_i, r_i)$, we define 
				\begin{equation*}
					d_{i,p}(u):= \mathrm{dist}\p{\mathcal{T}_{u_i}^- \bigcap \Phi_{u_i}^{-1}\p{\overline{B^{1,2\alpha}\left(u, \frac{\eta}{2 e_p(u)}\right)}},\mathcal{T}_{u_i}^- \bigcap\Phi_{u_i}^{-1}\p{\partial B^{1,2\alpha}\left(u, \frac{3\eta}{4 e_p(u)}\right)} }.
				\end{equation*}
				Note that $\dim(\mathcal{T}_{u_i}^-) < \infty$ which means that 
				\begin{equation*}
					\mathcal{T}_{u_i}^- \bigcap \Phi_{u_i}^{-1}\p{\overline{B^{1,2\alpha}\left(u, \frac{\eta}{2 e_p(u)}\right)}} \quad \text{and} \quad \mathcal{T}_{u_i}^- \bigcap\Phi_{u_i}^{-1}\p{\partial B^{1,2\alpha}\left(u, \frac{3\eta}{4 e_p(u)}\right)}
				\end{equation*}
				are two disjoint compact set, hence, $d_{i,p}(u): B^{1,2\alpha}(u_i, r_i - \eta) \rightarrow \R_+$ is a positive continuous function. Moreover, we see that 
				\begin{equation*}
					\underline{d}_{i,p}:= \inf\set{d_{i,p}(u) \,:\, u \in \mathcal{U}_0 \text{ with } B^{1,2\alpha}(u,\eta) \subset B^{1,2\alpha}(u_i, r_i)} > 0.
				\end{equation*}
				Otherwise, suppose that $\underline{d}_{i,p} = 0$, we can find a sequence $\{u_j\}_{j \in \mathbb N} \subset \mathcal{U}_0$ such that 
				\begin{equation*}
					\lim_{j \rightarrow \infty} d_{i,p}(u_j) = 0.
				\end{equation*}
				Since $E^{\omega}_\alpha : W^{1,2\alpha}(\S^2,N) \rightarrow \R$ satisfies the Palais-Smale condition, see Lemma \ref{P-S condition}, by the definition of $\mathcal{U}_0$ after passing to certain subsequence we can assume $u_j$ converges to some $u_0 \in \mathcal{U}_0$ in $W^{1,2\alpha}(\S^2,N)$. This leads to the contradiction
				\begin{equation*}
					0 < d_{i,p}(u_0) =  \lim_{j \rightarrow \infty} d_{i,p}(u_j)  = 0.
				\end{equation*}
				Then we define 
				\begin{equation*}
					\theta_p(u) := \min \set{\frac{C(u_i)}{4} \underline{d}_{i,p}^2\, :\, 1 \leq  i \leq m \quad \text{with }  B^{1,2\alpha}(u,\eta) \subset B^{1,2\alpha}(u_i, r_i)} > 0
				\end{equation*}
				and take $e_{p + 1}(u) \geq 2 e_{p}(u) + 1$ to be the smallest number such that for any $$w \in \overline{B^{1,2\alpha}(u, {\eta}/{e_{p + 1}(u)})},$$ there holds
				\begin{equation}
					E^{\omega}_\alpha(w) \geq \mathcal{W}_{\alpha,\lambda} - \theta_p(u).
				\end{equation}
				Note that 
                    $$\theta_p(u) \geq \underline{\theta}_p := \min_{1 \leq i \leq m} \frac{C(u_i)}{4}\underline{d}^2_{i,p}>0,$$
                    for any $u \in \mathcal{U}_0$, has a positive uniform lower bound, where $C(u_i) > 0$ is a constant determined in Step \ref{step1:deformation} and Lemma \ref{lem:local estimates index}.  Moreover, we claim that
				\claim \label{Morse index claim 1} There exists $\overline{e}_{p + 1} > 0$ such that 
				\begin{equation}
					\sup_{v \in \mathcal{U}_0} e_{p + 1}(v) \leq \overline{e}_{p + 1}.
				\end{equation}
				\begin{proof}[\textbf{Proof of Claim \ref{Morse index claim 1}}]
					In fact, since $\eta \leq 1/2$, it suffices to show there exists a constant $c_p > 0$ such that for any $u \in \mathcal{U}_0$ and $w \in \overline{B^{1,2\alpha}(u, {\eta}/{c_{p}})}$ there holds
					\begin{equation*}
						E^{\omega}_\alpha(w) \geq \mathcal{W}_{\alpha,\lambda} - \min_{1 \leq i \leq m} \frac{C(u_i)}{4}\underline{d}^2_{i,p}.
					\end{equation*}
					Then, we can conclude that $e_{p + 1}(u) \leq c_p$ for any $ u \in \mathcal{U}_0$. By contradiction, suppose that there exists two sequences 
                        $$\set{u_j}_{j \in \mathbb{N}} \subset \mathcal{U}_0\quad  \text{and}\quad  \set{w_j \, :\, w_j \in \overline{B^{1,2\alpha}(v_j, 1/j)}}_{j \in \mathbb{N}}$$
                        such that 
					\begin{equation}
						E^{\omega}_\alpha(w_j) <  \mathcal{W}_{\alpha,\lambda} - \min_{1 \leq i \leq m} \frac{C(u_i)}{4}\underline{d}^2_{i,p}.
					\end{equation}
					Then similarly to the previous argument, by the compactness of $\mathcal{U}_0$, after passing to a subsequence,  we can conclude that $u_j$ converges strongly in $W^{1,2\alpha}(\S^2, N)$ to some $u_0 \in \mathcal{U}_0$ and $w_j$ converges strongly in $W^{1,2\alpha}(\S^2, N)$ to $u_0$. However, the identity $E^{\omega}_\alpha(u_0) = \mathcal{W}_{\alpha,\lambda}$ leads to the contradiction
					\begin{equation*}
						\mathcal{W}_{\alpha,\lambda} < \mathcal{W}_{\alpha,\lambda} - \min_{1 \leq i \leq m} \frac{C(u_i)}{4}\underline{d}^2_{i,p}.
					\end{equation*}
					Therefore,  we complete the proof of Claim \ref{Morse index claim 1}.
				\end{proof}
				Now, the sequences $\set{e_{p}(u)}_{p = 1}^{k-1} \subset \R_+$ and $\set{\theta_{p}(u)}_{p = 1}^{k-2} \subset \R_+$ are defined by our induction argument. In the end, we prove the following Claim to finish the proof of Step \ref{step3:deformation}.
				\claim \label{Morse index claim 2} $H^{\varsigma, \vartheta}_{p,i}(\tau, t) \notin B^{1,2\alpha}(u,\eta/e_{p + 1}(u))$ for all $t \in [0,1]$ provided that 
				$$\varsigma(\tau) \notin B^{1,2\alpha}(u, \eta/e_{p}(u)), \quad  \quad  E^{\omega}_\alpha(\varsigma(\tau)) - \mathcal{W}_{\alpha,\lambda} \leq \theta_p(u)$$
				and $\vartheta < \min \p{\eta/(4e_p(u)), \theta_p(u)}$.
				\begin{proof}[\textbf{Proof of Claim \ref{Morse index claim 2}}]
					For $t\in [0,1/2]$, since $\vartheta < \min \p{\eta/(4e_p(u)), \theta_p(u)}$ and $\varsigma(\tau) \notin B^{1,2\alpha}(u, \eta/e_{p}(u))$, consulting the \ref{step2:deformation item 3 c} of item \ref{step2 item 3} in Step \ref{step2:deformation} we see that 
					\begin{align*}
						\norm{H^{\varsigma,\vartheta}_{p,i}(\tau,t) - u}_{W^{1,2\alpha}(\S^2,N)} &\geq \norm{\varsigma(\tau) - u}_{W^{1,2\alpha}(\S^2,N)} -  \norm{H^{\varsigma,\vartheta}_{p,i}(\tau,t) - \varsigma(\tau)}_{W^{1,2\alpha}(\S^2,N)} \\
						&\geq \frac{3\eta}{4 e_{p}(u)} \geq \frac{\eta}{e_{p + 1}(u)},
					\end{align*}
					which implies $H^{\varsigma, \vartheta}_{p,i}(\tau, t) \notin B(u,\eta/e_{p + 1}(u))$ for all $t \in [0,1/2]$
					
					Next, we show that 
					\begin{equation*}
						\norm{H^{\varsigma,\vartheta}_{p,i}(\tau,t) - u}_{W^{1,2\alpha}(\S^2,N)} \geq \frac{\eta}{e_{p + 1}(u)}\quad \text{when } t\in (1/2,1].
					\end{equation*}
					To see this, we note that 
					\begin{align*}
						E^{\omega}_\alpha\p{H^{\varsigma,\vartheta}_{p,i}(\tau,t)} &= \p{E^{\omega}_\alpha(H^{\varsigma,\vartheta}_{p,i}(\tau,t)) - E^{\omega}_\alpha(\varsigma(\tau))} + E^{\omega}_\alpha(\varsigma(\tau))\\
						& \leq \vartheta + \mathcal{W}_{\alpha,\lambda} + \theta_p(u) \leq \mathcal{W}_{\alpha,\lambda} + 2\theta_p(u)
					\end{align*}
					for $\vartheta < \theta_p(u)$. Next, we proceed the proof by contradiction. Suppose that there exists some $t \in (1/2,1]$ such that $H^{\varsigma,\vartheta}_{p,i}(\tau,t) \in {B}^{1,2\alpha}(u,\eta/e_{p+1}(u))$. Observe that $H^{\varsigma,\vartheta}_{p,i}(\tau,t) \in \mathcal{D}_{u_i}(3) \subset \Phi_{u_i}(\mathcal{B}_{u_i}(0,r_{u_i}))$ and by the definitions of $e_{p+1}(u)$ and $d_{i,p}(u)$ we have that
					\begin{align*}
						&\norm{ \p{\Phi^{-1}_{u_i}\p{H^{\varsigma,\vartheta}_{p,i}(\tau,t)}}^-  - \p{\Phi^{-1}_{u_i}\p{H^{\varsigma,\vartheta}_{p,i}(\tau,1/2)}}^-}_{\mathcal{T}_{u_i}}\\
						&\quad \geq \mathrm{dist}\p{\mathcal{T}_{u_i}^- \bigcap \Phi_{u_i}^{-1}\p{\overline{B^{1,2\alpha}\left(u, \frac{\eta} {e_{p + 1}(u)}\right)}},\mathcal{T}_{u_i}^- \bigcap\Phi_{u_i}^{-1}\p{\partial B^{1,2\alpha}\left(u, \frac{3\eta}{4 e_p(u)}\right)} }\\
						& \quad \geq d_{i,p}(u) \geq \min_{1 \leq i \leq m} \underline{d}_{i,p} > 0.
					\end{align*}
					Next, by the construction of $H^{\varsigma,\vartheta}_{p,i}(\tau,t)$ for $t \geq 1/2$ in Step \ref{step2:deformation}, we see that 
					\begin{align*}
						E^{\omega}_\alpha\p{H^{\varsigma,\vartheta}_{p,i}(\tau,t)} - E^{\omega}_\alpha\p{H^{\varsigma,\vartheta}_{p,i}(\tau,1/2)} &= E^{\omega}_\alpha\p{\gamma_{H^{\varsigma,\vartheta}_{p,i}(\tau,1/2)}(2t-1)} - E^{\omega}_\alpha\p{H^{\varsigma,\vartheta}_{p,i}(\tau,1/2)}\\
						&\leq E^{\omega}_\alpha\p{H^{\varsigma,\vartheta}_{p,i}(\tau,1/2)} - C(u_i)\min_{1 \leq i \leq m} \underline{d}_{i,p}^2\\
						&\leq \mathcal{W}_{\alpha,\lambda} + 2\theta_p(u) - 4 \theta_p(u)\\
						& = \mathcal{W}_{\alpha,\lambda} - 2 \theta_p(u),
					\end{align*}
					which contradicts to the definition of $e_{p+1}(u)$. This completes the proof of Claim \ref{Morse index claim 2}.
				\end{proof}
			Therefore, we finish the proof of Step \ref{step3:deformation}.
			\end{proof}
			Before penetrating to the detailed description of next step, for the notation simplicity we take
			\begin{equation}\label{eq:define underline d}
				\underline{d}:= \min_{1 \leq p \leq k-1} \inf_{ u \in \mathcal{U}_0} \min\p{\frac{\eta}{4 e_p(u)}, \theta_p(u)} > 0.
			\end{equation}
			\step\label{step4:deformation}After passing to some subsequence, we construct a sequence of desired continuous homotopies $H_j:I^{k-2} \times [0,1] \rightarrow W^{1,2\alpha}(\S^2, N)$ such that $H_j(\cdot, 0) = \sigma_j$ and $H_j(\cdot, 1) = \dbl{\sigma}_j$ satisfies all the properties asserted in Theorem \ref{prop deformation}. 
			\begin{proof}[\textbf{Proof of Step \ref{step4:deformation}}]
				To begin, we choose a subsequence $\sigma_{j_l} \in \mathcal{A}_{\varepsilon_{j_l},C}$ of $\sigma_j$ such that 
				\begin{equation*}
					\max_{\tau \in I^{k-2}} E^{\omega}_\alpha(\sigma_{j_l}(\tau)) \leq \mathcal{W}_{\alpha,\lambda} + \underline{d}/{2}.
				\end{equation*}
				where $\underline{d}$ is defined in \eqref{eq:define underline d}. For notation simplicity, we still write $\{\sigma_{j}\}_{j \in \mathbb N}$ to represent $\{\sigma_{j_l}\}_{l \in \mathbb N}$. For fixed $\sigma_j: I^{k-2} \rightarrow W^{1,2\alpha}(\S^2, N)$, to show Step \ref{step4:deformation}, recalling that 
				\begin{equation*}
					\mathcal{U}_0 \subset \mathcal{N}^{1,2\alpha}_{\eta} = \bigcup_{u \in \mathcal{U}_0} B^{1,2\alpha}(u,\eta) \quad 
					\text{for some fixed } \eta > 0,
				\end{equation*}
				it suffices to construct the homotopy $H_j$ from $\sigma_j$ to $\dbl{\sigma}_j$ such that
				\begin{equation*}
					\dbl{\sigma}_j(I^{k-2}) \bigcap \mathcal{N}^{1,2\alpha}_{\frac{\eta}{2 \overline{e}_{k-1}}} = \emptyset
				\end{equation*}
				where the constant $\overline{e}_{k-1} > 0$ is defined in Step \ref{step3:deformation} and $\eta$ is obtained in Step \ref{step2:deformation}.
    
				To this end, firstly we denote $I(1,n)$ to be the cell complex on the unit interval $I$ whose 1-cells are the intervals $[0,1\cdot 3^{-n}]$, $[1\cdot 3^{-n}, 2\cdot 3^{-n}]$,\dots, $[1-3^{-n},1]$ and whose 0-cells are the end points $\set{0}$, $\set{3^{-n}}$, $\set{2\cdot 3^{-n}}$,\dots, $\set{1}$. Then, $I(k-2,n)$ denotes the cell complex of $I^{k-2}$ as below
				\begin{equation*}
					I(k-2, n) = I(1,n)\otimes I(1,n) \otimes \cdots \otimes I(1,n).
				\end{equation*}
				Then, for each fixed $j \in \mathbb N$ we take $n$ large enough to obtain a sufficiently fine subdivision of $I^{k-2}$ such that for each closed face $F$ of $I(k-2,n)$ the followings are fulfilled:
				\begin{enumerate}[label=(\subscript{C}{{\arabic*}})]
					\item\label{step4:deformation 1} If $\sigma_j(F) \bigcap \mathcal{N}^{1,2\alpha}_{\eta/2} \neq 
					\emptyset$, then $\sigma_j(F) \subset \mathcal{N}^{1,2\alpha}_\eta$.
					\item\label{step4:deformation 2} Each $F \in \mathcal{F}$ can be covered by single $B^{1,2\alpha}(u_i,r_i)$ for certain $1 \leq i \leq m$. Here, $\mathcal{F}$ is the set of faces $F$ satisfying $\sigma_j(F) \bigcap \mathcal{N}^{1,2\alpha}_{\eta/2} \neq 
					\emptyset$. 
				\end{enumerate}
			
			To see \ref{step4:deformation 1}, since $I(k-2,n)$ is compact and $\sigma_j:I(k-2,n) \rightarrow W^{1,2\alpha}(\S^2,N)$ is continuous, we can take finitely many times barycentric subdivision upon $I^{k-2}$ such that the oscillation of $\sigma_j$ on each face $F \in \mathcal{F}$ is less than ${\eta}/{4}$. Moreover, by Step \ref{step2:deformation}, $\mathcal{N}^{1,2\alpha}_\eta$ is covered by the union of finite collection $\set{B^{1,2\alpha}(u_i, r_i)}_{i = 1}^m$, after further taking finitely many times barycentric subdivision, by the notion of Lebesgue's number, we can arrange that each maximal $F  \in  \mathcal{F}$ (here, $F  \in  \mathcal{F}$ is maximal means that there is no $F^\prime \in \mathcal{F}$ with $F \subsetneqq F^\prime$ ) $\sigma_{j}(F)$ can be covered by exactly one $B^{1,2\alpha}(u_i,r_i)$. Then \ref{step4:deformation 2} follows by an induction argument by decreasing dimensions of $F \in \mathcal{F}$.
			
			In the following, we are devoted to construct the desired continuous homotopy $H_j$ inductively on dimension $1 \leq l \leq k-2$ of the $l$-skeletons $I^{k-2}_{(l)}$ for $I^{k-2}$.
			
			For $l = 0$, we can apply the homotopy constructed in \ref{step2 item 3} of Step \ref{step2:deformation}. More precisely, for the $0$-cells outside $\mathcal{N}^{1,2\alpha}_{\eta/e_1}$, we simply take the  $H^{(0)}_j$ to be the constant homtopy on them. For the $0$-cells $X$ in $\mathcal{N}^{1,2\alpha}_{\eta/e_1}$, we choose the homotopy 
            $$H^{(0)}_j := \left. H^{\sigma_j, \vartheta}_{0,i}\right|_{X}$$
            defined on $X \times [0,1]$, where $H^{\sigma_j, \vartheta}_{0,i}$ is constructed
            in Step \ref{step2:deformation}, $i \in \set{1,2,\dots, m}$ is chosen to be the smallest positive integer such that $\sigma_{j}(X) \subset B(u_i,r_i)$ and $\vartheta$ is chosen to satisfies $\vartheta < \underline{d}/4$.
			In summary, by Step \ref{step2:deformation} and Step \ref{step3:deformation} we obtain a homotopy $H_j^{(0)}$ such that 
			\begin{equation*}
				H_j^{(0)}\left(I^{k-2}_{(0)} \times \set{1}\right)\bigcap \mathcal{N}_{\eta/e_1} = \emptyset
			\end{equation*}
		 and 
		 \begin{equation*}
		 	H_j^{(0)}(X,t) \in \mathcal{D}_{u_i}(3) \quad \text{for any } X \in I^{k-2}_{(0)} \text{ and any } t\in [0,1].
		 \end{equation*}
			
			Now, suppose that we have constructed $H_j^{(l-1)}$ on $I^{k-2}_{(l-1)} \times [0,1]$ for some $l \geq 1$ such that 
			\begin{equation*}
				H_j^{(l-1)} \left(I^{k-2}_{(l-1)} \times \set{1}\right)\bigcap \mathcal{N}_{\eta/e_{l-1}} = \emptyset, \quad H_j^{(l-1)}\left(I^{k-2}_{(l-1)} \times [0,1]\right)\bigcap \mathcal{N}_{\eta/e_{l}} = \emptyset 
			\end{equation*}
			by consulting the conclusion of Step \ref{step3:deformation} and such that
			\begin{equation*}
				H_j^{(l-1)}(X,t) \in \mathcal{D}_{u_i}(3) \quad \text{for any } X \in I^{k-2}_{(l-1)} \text{ and any } t\in [0,1].
			\end{equation*}
		Then, we consider the $l$-cells in $I^{k-2}$. For $F_l \in I^{k-2}_{(l)}\backslash I^{k-2}_{(l-1)}$, we see that $\partial F_ l \in I^{k-2}_{(l-1)}$ and $\dbl{F}_l:= F_l \cup\p{\partial F_l \times [0,1]}$ is homeomorphic to $F_l$ by concatenating $F_l$ and $\partial F_l \times [0,1]$ along the $\partial F_l$. This implies that we can construct the continuous map
		\begin{equation*}
			\varsigma : F_l \cong\dbl{F}_l \rightarrow W^{1,2\alpha}(\S^2,N)
		\end{equation*}
	by gluing
	 $$H^{(l-1)}_j:\partial F_l \times [0,1] \rightarrow W^{1,2\alpha}(\S^2,N) \quad \text{and}\quad  \sigma_j: F_l \rightarrow W^{1,2\alpha}(\S^2,N)$$
	 along the $\partial F_l$.
	 
	  Next, we construct the homotopy $\dbl{H}_j^{(l)}$ from $\varsigma$ on $\dbl{F}_l \cong F_l$ conditionally depending on whether $\dbl{F}_l$ belongs to $\mathcal{F}$ or not.
	  \begin{itemize}
	  	\item If $\dbl{F}_l \in \mathcal{F}$, then by the definition of $\mathcal{F}$ there is also no cells of $\partial \dbl{F}_l$ belongs to $\mathcal{F}$. Therefore, we define 
	  	$\dbl{H}^{(l)}_j: \dbl{F}_l \times [0,1] \rightarrow W^{1,2\alpha}(\S^2,N) \equiv \varsigma$
	  	 to be the constant homotopy. Note that in this case $\varsigma|_{\partial\dbl{F}_l}$ satisfies the assumption of Step \ref{step3:deformation}, so we have that 
	  	 \begin{equation*}
	  	 	\dbl{H}^{(l)}_j\p{\dbl{F}_l\times \set{1} \bigcup \partial\dbl{F}_l\times [0,1]}\bigcap \mathcal{N}_{\frac{\eta}{e_{l+1}}} = \emptyset.
	  	 \end{equation*}
	  	\item If $\dbl{F}_l \notin \mathcal{F}$, by the induction assumption on $H^{(l-1)}_j$ and the choice of fine subdivision on $I^{k-2}$, see \ref{step4:deformation 2}, there exists $1 \leq i_l \leq m$ such that $\varsigma (\dbl{F}_l) \subset B(u_{i_l},r_{i_l})$. However, it is important to point out that $\varsigma (\dbl{F}_l)$ may not be contained in $\mathcal{N}_\eta$ which means that the construction of continuous homotopy obtained in \ref{step2 item 3} of Step \ref{step2:deformation} can not be applied directly. Based on this consideration, we further take finer subdivision on $\dbl{F}_l$ such that each $l$-dimensional face $\dbl{f}_l$ of $\dbl{F}_l$ satisfying $\varsigma(\dbl{f}_l) \cap \mathcal{N}_{\eta/e_p} \neq \emptyset$ must fulfill $\varsigma(\dbl{f}_l) \subset \mathcal{N}_\eta$. Then, similarly, we denote $\dbl{\mathcal{F}}_l$ to be the union of all $l$-dimensional faces $\dbl{f}_l$ of $\dbl{F}_l$ with $\varsigma(\dbl{f}_l) \cap \mathcal{N}_{\eta/e_l} \neq \emptyset$. By induction assumption, we see that $\varsigma(\partial \dbl{F}_l) \cap \mathcal{N}_{\eta/e_l} = \emptyset$ which implies that  $\varsigma(\partial \dbl{\mathcal{F}}_l) \cap \mathcal{N}_{\eta/e_l} = \emptyset$. Then, we can construct a homotopy map 
	  	\begin{equation}\label{eq:definition hat H}
	  		\hat{H}_j^{(l)}: \p{\dbl{F}_l \times [0,1/2] } \bigcup\p{\dbl{\mathcal{F}}_l\times [1/2,1]} \rightarrow W^{1,2\alpha}(\S^2, N) 
	  	\end{equation}
  		such that $\hat{H}(x,t) = \varsigma(x)$ when  $t \in [0,1/2]$ and $x \in \dbl{F}_l$ and that $\hat{H}(x,t) = H^{\varsigma,\vartheta}_{l,i_l}(x,2t-1)$ for $t \in [1/2,1]$, where $H^{\varsigma,\vartheta}_{l,i_l}(x,2t-1)$ is defined in \ref{step2 item 3} of Step \ref{step2:deformation} with $\vartheta < \underline{d}/4$. By the definition of $\dbl{\mathcal{F}}_l$, we see that 
  		\begin{equation*}
  			\hat{H}_j^{(l)} \p{ \left(\dbl{F}_l \backslash \dbl{\mathcal{F}}_l\right)\times \set{\frac{1}{2}} \bigcup \p{\partial \dbl{F}_l \times[0,1/2]}} \bigcap \mathcal{N}_{\frac{\eta}{e_{l+1}}} = \emptyset.
  		\end{equation*}
  		And by the construction of $\hat{H}_j^{(l)}$ and the \ref{step2 item 3} in Step \ref{step2:deformation}, we have that
  		\begin{equation*}
  			\hat{H}_j^{(l)}\left(\dbl{\mathcal{F}}_l\times \set{1}\right) \bigcap \mathcal{N}_{\frac{\eta}{e_{l+1}}} = \emptyset.
  		\end{equation*}
  		Moreover, by Step \ref{step3:deformation}, we have 
  		\begin{equation*}
  			\hat{H}_j^{(l)}\left(\partial\dbl{\mathcal{F}}_l\times [1/2,1]\right) \bigcap \mathcal{N}_{\frac{\eta}{e_{l+1}}} = \emptyset.
  		\end{equation*}
  	Then by the homeomorphisms 
  	\begin{equation*}
  		\dbl{F}_l \times[0,1] \cong {\p{\dbl{F}_l \times[0,1/2]} \bigcap \p{\dbl{\mathcal{F}}_l \times [1/2,1]}},
  	\end{equation*}
   that is induced from the homeomorphism
  \begin{equation*}
  	\dbl{F}_l\times \set{1} \cong {\p{\p{\dbl{F}_l \backslash \dbl{\mathcal{F}}_l} \times \set{\frac{1}{2}}} \bigcup \p{\dbl{\mathcal{F}}_l \times \set{{1}}}
  	\bigcup \p{\partial \dbl{\mathcal{F}}_l \times [1/2,1]}},
  \end{equation*}
  and the identification
  \begin{equation*}
  	\partial\dbl{F}_l \times[0,1] = \partial\dbl{F}_l \times[0,1],
  \end{equation*}
 we can derive a continuous homotopy $\dbl{H}^{(l)}_j : \dbl{F}_l\times [0,1] \rightarrow W^{1,2\alpha}(\S^2, N)$ from $\hat{H}_j^{(l)}$ which satisfies 
 \begin{equation}\label{step4 desired empty}
 	\dbl{H}^{(l)}_j\p{\p{\dbl{F}_l\times \set{1}} \bigcup \p{\partial \dbl{F}_l \times [0,1]}} \bigcap \mathcal{N}_{\frac{\eta}{e_{l+1}}} = \emptyset.
  \end{equation}
	  \end{itemize}
  In summary, whenever $F_l \in \mathcal{F}$ or not, we can construct a continuous homotopy 
  $$H^{(l)}_j : F_l \times [0,1] \rightarrow W^{1,2\alpha}(\S^2, N)$$
  induced from the homeomorphisms 
  \begin{equation*}
  	F_l\times[0,1] \cong \dbl{F}_l \times [0,1]\quad \text{and}\quad F_l \times \set{1} \cong { \p{\dbl{F}_l \times \set{1} }\bigcup \p{ \partial\dbl{F}_l \times [0,1]}},
  \end{equation*}
which satisfies that 
\begin{equation*}
	H^{(l)}_j|_{\partial F_l \times [0,1]} = H^{(l-1)}_j|_{\partial F_l \times [0,1]} \quad \text{and}\quad H^{(l)}_j (x,0) = \sigma_j(x)  \text{ for all } x \in F_l.
\end{equation*}
Then, we glue all together such continuous homotopy $H^{(l)}_j$ defined on $F_l \times [0,1]$ when $F_k$ runs through $I^{k-2}_{(l)} \backslash I^{k-2}_{(l-1)}$ to obtain the desired homotopy
\begin{equation*}
	H^{(l)}_j : I^{k-2}_{(l)} \times [0,1] \rightarrow W^{1,2\alpha}(\S^2, N).
\end{equation*}
Keeping in mind that 
\begin{equation*}
	H^{(l)}_j \left(F_l \times \set{1}\right) = \dbl{H}^{(l)}_j\p{\p{\dbl{F}_l \times \set{1}} \bigcup  \p{\partial\dbl{F}_l \times [0,1] }}
\end{equation*}
 on each $F_l \in I^{k-2}_{(l)} \backslash I^{k-2}_{(l-1)}$ for any $1\leq l \leq k-2 $ and recalling \eqref{step4 desired empty}, we have that 
 \begin{equation*}
 	H^{(k-2)}_j\p{I^{k-2}_{(k-2)} \times \set{1} } \bigcap \mathcal{N}_{\frac{\eta}{e_{l+1}}} = \emptyset.
 \end{equation*}
To complete the construction, we let $\dbl{\sigma}_j:= H^{(k-2)}_j(\cdot,1)$ which satisfies that 
\begin{equation}\label{eq:step4 deformation 1}
	\dbl{\sigma}_j(I^{k-2}) \bigcap \mathcal{N}_{\frac{\eta}{e_{k-1}}} = \emptyset.
\end{equation}

In the following, we show that there exists large enough $j_0 \in \mathbb N$ such that $\dbl{\sigma}_j$ satisfies the first property \ref{prop deformation 1} when $j \geq j_0$, and for each $u \in \mathcal{U}_0$ there exists $j_0(u) \in \mathbb N$ such that $\dbl{\sigma}_j$ satisfies the second property \ref{prop deformation 2} asserted in Theorem \ref{prop deformation} when $j \geq j_0(u)$. 

To see \ref{prop deformation 1} of Theorem \ref{prop deformation}, by our choice of $\delta_0$ and covering $\set{B^{1,2\alpha}(u_i,r_i)}_{i=1}^m$, in particular see \eqref{eq:step1 deformation 1}, we find that
\begin{equation*}
	\sigma_{j}(\tau) \notin \bigcup_{i=1}^m B^{1,2\alpha}(u_i,r_i) \quad \text{for any  } \tau\in \partial I^{k-2}.
\end{equation*}
This implies that the continuous homotopy $H^{(k-2)}_j$ restricted to $\partial I^{k-2}$ is constant, hence $\sigma_{j}|_{\partial I^{k-2}} = \dbl{\sigma}_{j}|_{\partial I^{k-2}}$. Thus, $\dbl{\sigma}_j \in \mathscr{I}$ is an admissible sweepout. Moreover, to show $\dbl{\sigma}_j \in \mathcal{A}_{\varepsilon_j, C+1}$, we first observe that the homotopy $H^{(k-2)}_j$ is constructed from $\sigma_j$ by gluing $\hat{H}^{(l)}_{j}$ finitely many times and that $\hat{H}^{(l)}_j$ is constructed from the continuous homotopy obtained in \ref{step2 item 3} of Step \ref{step2:deformation} for $1 \leq l \leq k-2$. By the arbitrariness of $\vartheta$ in \ref{step2:deformation item 3 b} of Step \ref{step2 item 3}, choosing large enough $j_0 \in \mathbb{N}$ with $\varepsilon_j \leq 
{\underline{d}}/{8}$ and small enough $\vartheta$ we can conclude that 
\begin{equation*}
	\max_{\tau \in I^{k-2}} E^{\omega}_\alpha\left( \dbl{\sigma}_j(\tau)\right) = \max_{\tau \in I^{k-2}} E^{\omega}_\alpha\left(H^{(k-2)}_j(\tau,1)\right) \leq \mathcal{W}_{\alpha,\lambda} + \varepsilon_j
\end{equation*}
for all $j \geq j_0$. Next, we verify the second part of definition of $\mathcal{A}_{\varepsilon_j,C+1}$. Let $\tau \in I^{k-2}$ be the point such that
\begin{equation}\label{step4 the choice tau}
	E^{\omega}_{\alpha}(\dbl{\sigma}_j(\tau)) \geq \mathcal{W}_{\alpha,\lambda} - \varepsilon_j.
\end{equation}
If $\tau \notin \bigcup_{F \in \mathcal{F}} F$, then we have 
\begin{equation*}
  	E^{\omega}_{\alpha}(\dbl{\sigma}_j(\tau)) =	E^{\omega}_{\alpha}({\sigma}_j(x)) \geq \mathcal{W}_{\alpha,\lambda} - \varepsilon_j
\end{equation*}
which implies that 
\begin{equation*}
	E_\alpha(\dbl{\sigma}_j(x)) = E_\alpha({\sigma}_j(x)) \leq  C < C+1.
\end{equation*}
On the other hand, if $\tau \in \bigcup_{F \in \mathcal{F}} F$, then $\sigma_j(\tau) \in \mathcal{N}_{\eta}$ and there exists $1 \leq i \leq m$ such that $\sigma_j(\tau) \in B^{1,2\alpha}(u_i, r_i)$. By the construction of $\dbl{\sigma}_j$ and the choice of $\tau$ satisfying \eqref{step4 the choice tau}, we know that $\dbl{\sigma}_j(\tau) \in B^{1,2\alpha}(u_i, 2r_i)$ which implies that 
\begin{equation*}
	E_\alpha(\dbl{\sigma}_j(\tau)) \leq \abs{E_\alpha(\dbl{\sigma}_j(\tau)) - E_\alpha(u_i)} + E_\alpha(u_i) \leq C + 1.
\end{equation*}
thanks to \eqref{eq:step1 deformation 1} and the choice $u_i \in \mathcal{U}_0 \subset \mathcal{U}_C$. Therefore, we showed that $\dbl{\sigma}_j \in \mathcal{A}_{\varepsilon_j, C+1}$, that is, the first assertion \ref{prop deformation 1} of Theorem \ref{prop deformation}. 

To show \ref{prop deformation 2} of Theorem \ref{prop deformation}, we argue by contradiction. Suppose that there exists subsequences of $\dbl{\sigma}_j$, $\varepsilon_j$ (still denoted by $\dbl{\sigma}_j$, $\varepsilon_j$), a sequence of points $u_j \in \mathcal{U}_0$ and a sequence $\tau_j \in I^{k-2}$ such that 
\begin{equation*}
	E^{\omega}_{\alpha}(\dbl{\sigma}_j(\tau_j)) \geq \mathcal{W}_{\alpha,\lambda} - \varepsilon_j
\end{equation*}
and that 
\begin{equation*}
    \lim_{j \rightarrow \infty}\norm{\dbl{\sigma}_j(\tau_j) - u_j}_{W^{1,2\alpha}(\S^2, N)} = 0.
\end{equation*}
for any $j_0 \in \mathbb N$. Then, by the compactness of $\mathcal{U}_0$, after passing some subsequences we have 
\begin{equation*}
	\lim_{j \rightarrow \infty} \norm{\dbl{\sigma}_j(\tau_j) - u}_{W^{1,2\alpha}(\S^2, N) } = 0
\end{equation*}
for some $u \in \mathcal{U}_0$.
But this contradicts to \eqref{eq:step4 deformation 1}, hence we finish the proof of Step \ref{step4:deformation}.
\end{proof}
Therefore, we complete the proof of Theorem \ref{prop deformation}.
\end{proof}

Equipped with the deformation of sweepouts Theorem \ref{prop deformation}, we can construct the desired non-trivial critical points in $\mathcal{U}_C$ with Morse index bounded from above by $k-2$.
\begin{theorem}\label{thm:Morse index k-2}
    Given $\omega \in C^2(\wedge^2 (N))$, $\lambda > 0$ writing $\omega$ as $\lambda \omega$, $\alpha > 1$ and  a sequence of sweepouts $\sigma_j \in \mathcal{A}_{\varepsilon_j, C}$ for some constant $C > 0$ with $\varepsilon_j \searrow 0$. Then, there exists a non-trivial critical point $u \in \mathcal{U}_C$ satisfying 
    \begin{equation*}
        E_{\alpha}(u) \geq \frac{1}{2}\mathrm{Vol}(\S^2) + \delta(\alpha, \lambda\omega) \quad \text{and} \quad \mathrm{Ind}_{E^{\omega}_\alpha}(u) \leq k-2,
    \end{equation*}
    for some constant $\delta(\alpha, \lambda\omega) > 0$ depending on the choice of $\alpha > 0$, $\omega\in C^2(\wedge^2(N))$ and $\lambda >0$ which is obtained in \eqref{prop energy bound part 3} of Proposition \ref{prop energy bound}.
\end{theorem}
\begin{proof}
    Let
    \begin{equation*}
        \mathcal{U}_0 = \set{u \in \mathcal{U}_{C+1} \,:\, E^{\omega}_\alpha(u) \geq \frac{1}{2}\mathrm{Vol}(\S^2) + \delta(\alpha, \omega)} 
    \end{equation*}
    which is a  non-empty closed subset of $\mathcal{U}_{C+1}$ by Proposition \ref{prop energy bound}. By contradiction, if $\mathrm{Ind}_{E^{\omega}_\alpha}(u) \geq k-1$ for all $u \in \mathcal{U}_0$, then by \ref{prop deformation 1} of Theorem \ref{prop deformation} we can obtain a sequence of sweepouts $\dbl{\sigma_j} \in \mathcal{A}_{\varepsilon_j,C+1}$ constructing from $\sigma_j$. Then, combining the Theorem \ref{prop deformation} with the Proposition \ref{prop energy bound}, after passing to some subsequences of $\dbl{\sigma}_j$ and $\varepsilon_j$, there exists a sequence $\tau_j \in I^{k-2}$ such that 
\begin{equation*}
	E^{\lambda\omega}_{\alpha}(\dbl{\sigma}_j(\tau_j)) \geq \mathcal{W}_{\alpha,\lambda} - \varepsilon_j
\end{equation*}
and that 
\begin{equation*}
	\lim_{j \rightarrow \infty} \norm{\dbl{\sigma}_j(\tau_j) - u}_{W^{1,2\alpha}(\S^2, N) } = 0.
\end{equation*}
for some $u \in \mathcal{U}_0$ which contradicts to \ref{prop deformation 2} of Theorem \ref{prop deformation}. Therefore, there exists an $u \in \mathcal{U}_0$ with $\mathrm{Ind}_{E^{\omega}_\alpha}(u) \leq k-2$ completing the proof of Theorem \ref{thm:Morse index k-2}.
\end{proof}

As a summary of this section, for almost every $\lambda \in \R_+$ we constructed a sequence of non-trivial $\alpha_j$-$\lambda H$-spheres with desired properties described as below.
\begin{coro}\label{coro:section 3 summary}
    Given $\omega \in C^2(\wedge^2(N))$, for almost every $\lambda \in \R_+$, there exists a constant $C > 0$, a sequence $\alpha_j \searrow 1$ and a sequence of positive constant $\delta(\alpha_j, \lambda \omega) > 0$ such that for each $j \in \mathbb N$ there exists a $\alpha_j$-$\lambda H$-sphere $u_{\alpha_j} \in W^{1,2\alpha_j}(\S^2,N)$ satisfying 
    \begin{equation*}
        \delta E^{\lambda \omega}_{\alpha_j}(u_{\alpha_j}) = 0, \,\, \frac{1}{2}\mathrm{Vol}(\S^2) + \delta(\alpha_j, \lambda \omega) \leq E^{\lambda\omega}_{\alpha_j}(u) \leq C+1\,\,\text{and}\,\, \mathrm{Ind}_{E^{\lambda \omega}_{\alpha_j}}(u_{\alpha_j}) \leq k-2.
    \end{equation*}
\end{coro}
\begin{proof}
    Given a sequence $\alpha_j \searrow 1$, combining Lemma \ref{lem monotone} and Lemma \ref{admissable h}, for almost every $\lambda \in \R_+$ we can find a constant $C > 0$ and a subsequence of $\alpha_j \searrow 1$, still denoted by $\alpha_j$, such that for each $j \in \mathbb N$, there exists a sequences of sweepouts 
    $$\set{\sigma_{l}^j}_{l \in \mathbb{N}} \subset \mathcal{A}_{\lambda\varepsilon_l^j, 8\lambda^2 C}$$ for some $\set{\varepsilon_l^j}_{l \in \mathbb N}$ with $\varepsilon^j_l \searrow 0$ as $l \rightarrow \infty$ and large enough $l \in \mathbb N$. Then, for fixed $j \in \mathbb N$ we can apply the Theorem \ref{thm:Morse index k-2} to obtain a sequence $u_{\alpha_j}$ satisfying all the properties asserted in Corollary \ref{coro:section 3 summary}.
\end{proof}
	
		\vskip2cm
		\section{Compactness for Critical Points of Functional \texorpdfstring{$E^{\lambda\omega}_\alpha$}{Lg}}\label{section 4}
		\vskip10pt
		In Section \ref{sec: 3 non-constant u alpha}, in particular see Corollary \ref{coro:section 3 summary}, we constructed a sequence of non-trivial critical points $\{u_{\alpha_j}\}_{j \in \mathbb N}$ of the functional $E^{\lambda\omega}_{\alpha_j}$ with uniformly bounded $\alpha_j$-energy and with uniformly bounded Morse index  from above by $k-2$. Next, in order to obtain the existence of non-constant $H$-sphere, we need to study the behavior of sequence $u_{\alpha_j}$ as $\alpha_j \searrow 1$ which is the primary task in this section.
         It is crucial to point out that the compactness result developed in this section is applicable to a broad range of sequences of $ \alpha$-$ H $-surfaces $ \{ u_{\alpha} \}_{ \alpha > 1 } $ (being the critical points of ${E}_\alpha^\omega$) between closed Riemann surface $M$ and compact Riemannian $n$-manifold $ N $ with uniformly bounded $\alpha$-energy and with uniformly bounded Morse index  from above, where $H \in \Gamma\p{\wedge^2(N)\otimes TN}$ induced from any given $\omega \in C^2(\wedge^2(N))$ as in \eqref{eq: defi H by omega}. Therefore,  we simply write $\alpha$ for $\alpha_j$ to represent the general sequence $\alpha \searrow 1$ , $\omega$ for $\lambda \omega$ and $E^{\omega}_\alpha$ for $E^{\lambda \omega }_\alpha$, respectively. Considering that the asymptotic analysis for general sequences of $\alpha$-$H$-surfaces are involved and complicated, we summarize the main result of this section in Subsection \ref{section 4.1 main result}, with the detailed proofs provided in the subsequent subsections.
            \ 
		\vskip5pt
		\subsection{Main Results on Asymptotic Analysis for Sequences of \texorpdfstring{$\alpha$-$H$}{Lg}-surfaces}\label{section 4.1 main result}
		\ 
		\vskip5pt
        \subsubsection{Descriptions of Bubbling Procedure}
        \ 
        \vskip5pt
            	Given a sequence of $\alpha$-$H$-surfaces $\{u_\alpha\}_{\alpha \searrow 1}: M \rightarrow N$ with uniformly bounded $\alpha$-energy
		\begin{equation*}
			\sup_{\alpha > 1} {E}_\alpha(u_\alpha) \leq \Lambda < \infty.
		\end{equation*}
		Then, from Lemma \ref{lem4.1}, Lemma \ref{blow 2}, Lemma \ref{blow 3} below and adapting a rescaling argument by Sacks-Uhlenbeck \cite{sacks1981existence}, after passing to a subsequence, $u_\alpha$ converges to a $H$-surface $u_0\,:\, M \rightarrow N$ smoothly except at most finitely many singular points (that is, energy concentration points) $\{x^i\}_{i = 1}^{n_0}$ as $\alpha \searrow 1$. Around point each $x^i$ for $1 \leq i \leq n_0$, we assume that there are $n_i$ bubbles (that is, non-trivial $H$-spheres) arising at $x^i$ during bubbling. Therefore, there exists sequences of points $\{x^{ij}_\alpha\}_{\alpha \searrow 1}$ for $1 \leq i \leq n_0$ and $1 \leq j \leq n_i$, and sequences of positive numbers $\{\lambda^{ij}_\alpha\}_{\alpha \searrow 1}$ such that 
		\begin{equation*}
			x^{ij}_\alpha \rightarrow x^i\,\text{ for }1 \leq j \leq n_i \quad \text{and}\quad  \lambda^{ij}_\alpha \rightarrow 0 \quad \text{ as } \alpha \searrow 1.
		\end{equation*}
		By a standard scaling argument, see for instance \cite{Qing1995,li2010}, for $1 \leq j_1, j_2 \leq n_i$ and $j_1\neq j_2$ at least one alternative of the following two statements holds:
		\begin{Aenumerate}
			\item \label{A1} for any fixed $R > 0$, $B^M(x^{ij_1}_\alpha, \lambda^{ij_1}_\alpha R) \bigcap B^M(x^{ij_2}, \lambda^{ij_2}_{\alpha} R) = \emptyset$ whenever $\alpha - 1$ is sufficiently small.
			\item \label{A2}$\frac{\lambda^{ij_1}_\alpha}{\lambda^{ij_2}_\alpha} + \frac{\lambda^{ij_2}_\alpha}{\lambda^{ij_1}_\alpha} \rightarrow \infty$,  as $\alpha \searrow 1$.
		\end{Aenumerate}
		
		Moreover, after taking a conformal transformation $\R^2 \cup \set{\infty} \cong \S^2$ and applying the removability of isolated singularities Lemma \ref{blow 3} the rescaled maps
		\begin{equation*}
			v^{ij}_\alpha : =u_{\alpha}(x^{ij}_\alpha + \lambda^{ij}_\alpha x)\quad \text{for } 1\leq i \leq n_0, 1\leq j \leq n_i
		\end{equation*}
            converges strongly in $$C^\infty_{loc} \p{\R^2\backslash\{p^{ij}_1, \dots, p^{ij}_{s_{ij}}\}}$$ to a  non-trivial $H$-sphere $w^{ij}: \S^2 \rightarrow N$ for some finite energy concentration points $\{p^{ij}_1, \dots, p^{ij}_{s_{ij}}\} \subset \R^2$. Then we chosoe small enough $r_i > 0$ such that
			$$B(x_i,r_i) \bigcap \{x^1, x^2, \cdots, x^{i-1}, x^{i+1}, \cdots, x^{n_0}\} = \emptyset,$$ and hence $v_\alpha^{ij} : B(0, (\lambda^{ij}_\alpha)^{-1} r_i) \rightarrow N$ is a critical points of 
			\begin{align}\label{eq:bubble functional}
			    E^\omega_{\alpha, ij} := &\frac{1}{2} \int_{B(0, (\lambda^{ij}_\alpha)^{-1}r_i)} \p{(\lambda^{ij}_\alpha)^2 + |\nabla v^{ij}_\alpha|^2}^{\alpha}dV_{g_{\alpha,ij}} \nonumber\\
                &+ (\lambda_\alpha^{ij})^{2\alpha - 2} \int_{B(0, (\lambda_\alpha^{ij})^{-1} r_i)}(v_\alpha^{ij})^*\omega,
			\end{align}
			where 
            $$g_{\alpha,ij} := e^{\varphi(x_\alpha + \lambda_\alpha^{ij} x)}((dx^1)^2 + (dx^2)^2)$$ arising from the metric $g : = e^{\varphi(x)} ((dx^1)^2 + (dx^2)^2)$ on $M$ under a conformal coordinate $(x^1, x^2)$ and $v_\alpha^{ij}$ solves the Euler-Lagrange equation 
			\begin{align}\label{eq:e-l of bubble}
			    \Delta v_\alpha^{ij} + (\alpha - 1)\frac{\nabla|\nabla_{g_{\alpha,ij}} v_\alpha^{ij}|^2\cdot \nabla v_\alpha^{ij}}{(\lambda_\alpha^{ij})^2+|\nabla_{g_{\alpha,ij}} v_\alpha^{ij}|^2} &+ A(v_\alpha^{ij})\left(\nabla v_\alpha^{ij}, \nabla v_\alpha^{ij}\right) \nonumber\\
                &= (\lambda_\alpha^{ij})^{2\alpha - 2} \frac{H(v_\alpha^{ij})(\nablap v_\alpha^{ij}, \nabla v_\alpha^{ij})}{\alpha \left((\lambda_\alpha^{ij})^2 + |\nabla_{g_{\alpha,ij}} v_\alpha^{ij}|\right)^{\alpha - 1}}.
			\end{align}
        The deficiency of conformally invariance for functional $E^{\omega}_{\alpha}$ leads to distinct formulations of $E^{\omega}_{\alpha}$ and $E^\omega_{\alpha,ij}$, hence the distinct formulation of Euler Lagrange equations \eqref{el} and \eqref{eq:e-l of bubble}. Based on this consideration,  we will employ the following more general functional to show our main Theorem \ref{generalizd energy}, Theorem \ref{analysis on neck} and Theorem \ref{thm convergence} in this section
        \begin{equation}
            \dbl{E}^{\omega}_{\alpha} = \frac{1}{2}\int_M {\p{\tau_\alpha + \abs{\nabla_{g_\alpha} u_\alpha}^2}^{\alpha} } dV_{g_\alpha} + \tau^{\alpha - 1}_\alpha \int_{M} u_\alpha^* \omega
        \end{equation}
        where $0 < \tau_\alpha \leq 1$ satisfying $0 < \beta_0 \leq \liminf_{\alpha \searrow 1}\tau^{\alpha -1 }_\alpha \leq 1$ and $g_\alpha$ is a sequence of metrics on $(M,g)$ that is conformal to $g$ and converges smoothly to the standard $g$ as $\alpha \searrow 1$. Critical points $u_\alpha : M \rightarrow N$ of generalized functional $\dbl{E}^{\omega}_{\alpha}$ are also called $\alpha$-$H$-surfaces for simplicity, similarly to \eqref{el} and \eqref{el1}, it solves the following generalized Euler Lagrange equation
         \begin{align}
              \label{el general 1}
        \Delta_{g_\alpha} u_\alpha + (\alpha - 1)\frac{\nabla_{g_\alpha}|\nabla_{g_\alpha} u_\alpha|^2\cdot \nabla_{g_\alpha} u_\alpha}{\tau_\alpha+|\nabla_{g_\alpha} u_\alpha|^2} &+ A(u_\alpha)\left(\nabla_{g_\alpha} u_\alpha, \nabla_{g_\alpha} u_\alpha\right)\nonumber\\
        &= \tau_\alpha^{\alpha -1}\frac{H(u_\alpha)(\nablap_{g_\alpha} u_\alpha, \nabla_{g_\alpha} u_\alpha)}{\alpha \left(\tau_\alpha + |\nabla_{g_\alpha} u_\alpha|^2\right)^{\alpha - 1}}
         \end{align}
or equivalently in divergence form
   \begin{align}
       \label{el general 2}
        -\diver\left(\left(\tau_\alpha+|\nabla_{g_\alpha} u_\alpha|^2\right)^{\alpha - 1}\nabla_{g_\alpha} u_\alpha\right) &+ \left(\tau_\alpha + |\nabla_{g_\alpha} u_\alpha|^2\right)^{\alpha -1}A(u_\alpha)(\nabla_{g_\alpha} u_\alpha,\nabla_{g_\alpha} u_\alpha)\nonumber\\
        &= \frac{\tau_\alpha^{\alpha -1}}{\alpha}H(u_\alpha)\left(\nablap_{g_\alpha} u_\alpha, \nabla_{g_\alpha} u_\alpha\right)
   \end{align}

        Considering the \eqref{eq:bubble functional} and \eqref{eq:e-l of bubble}, the following quantities arise naturally in the process of studying the energy identity and asymptotic analysis of necks
		\begin{equation}\label{mu}
			\mu_{ij} := \liminf_{\alpha \searrow 1} \p{\lambda^{ij}_\alpha}^{2 - 2\alpha}
		\end{equation}
		and
		\begin{equation}\label{nu}
			\nu_{ij} := \liminf_{\alpha \searrow 1} \p{\lambda^{ij}_\alpha}^{- \sqrt{\alpha - 1}}
		\end{equation}
		indicating the comparison with expansion speed of blow-up radius and the speed of $\alpha \searrow 1$. It is easy to check that $\,\mu_{ij},\,\nu_{ij} \in [1,\infty]$ as $\lambda_\alpha^{ij} \rightarrow 0$ as $\alpha \searrow 1$. Moreover, we can see that all $\mu_{ij}$ are finite, that is, there exists a positive constant $1 \leq \mu_{max} < \infty$ such that $\mu_{ij} \in [1,\mu_{max}]$. Indeed, without loss generality we can assume there is only one blow-up point $x_1 \in M$ and there are $n_1$ bubbles arising at this point, which implies, there exists a sequence of points $\{x_\alpha^j\}_{\alpha > 1}$ and a sequence of positive numbers $\{\lambda^j_{\alpha}\}_{\alpha > 1}$ satisfying one of \ref{A1} and \ref{A2}. For simplicity, we assume 
		\begin{equation*}
			\limsup_{\alpha \searrow 1} \frac{\lambda^1_\alpha}{\lambda^j_\alpha} < \infty\quad \text{for all } 2\leq j \leq n_1
		\end{equation*}
		which implies $w^1_\alpha(x) := u_{\alpha}(x_\alpha^1 + \lambda^1_\alpha x)$ converges strongly to $w^1$ in $C^\infty_{loc}(\R^2)$, namely, $w^1$ is the first non-trivial $H$-bubble. Therefore, we have
		\begin{align*}
			\Lambda > \lim_{R\rightarrow \infty} \lim_{\alpha \searrow 1} \int_{B(x^1_\alpha, \lambda^1_\alpha R)} \abs{\nabla u_\alpha}^{2\alpha} dx = \lim_{R\rightarrow \infty} \lim_{\alpha \searrow 1} \p{\lambda_\alpha^1}^{2-2\alpha} \int_{B(0,R)}\abs{\nabla w^1_\alpha} dx = \mu_1 E(w^1).
		\end{align*}
		By the energy gap Lemma \ref{blow 2}, there is a positive constant 
        \begin{equation}\label{eq:gap constant}
            \varepsilon_0 := \inf \set{E(w) = \frac{1}{2}\int_{\S^2} |\nabla w|^2 dV_{\S^2}:\,\,\, \parbox{10em}{$w$ is a non-constant \\ harmonic sphere in $N$}} > 0
        \end{equation}
        such that  $E(w^1) \geq \varepsilon_0 $
        hence 
		\begin{equation*}
			\mu_j \leq \mu_1 \leq \frac{\Lambda_1}{E(w^1)} \leq \frac{\Lambda_1}{\varepsilon_0} := \mu_{max} < \infty. 
		\end{equation*}
		
		\subsubsection{Generalized Energy Identity}
          \
          \vskip5pt
        
		Now, we are in a position to state our first main compactness result of generalized energy identity for sequences of $\alpha$-$H$-surfaces.
		\begin{theorem}\label{generalized ide}
			Let $(M,g)$ be a closed Riemann surface, $(N,h)$ be a $n$-dimensional closed Riemannian manifold that is isometrically embedded in $\R^K$ for some $K \in \mathbb N$. Assume that $\{u_{\alpha}\}_{\alpha \searrow 1} \subset C^\infty(M,N)$ is a sequence of $\alpha$-$H$-surfaces  with uniformly bounded generalized $\alpha$-energy, that is, 
            $$\sup_{\alpha \searrow 1}E_{\alpha}(u_{\alpha})\leq \Lambda < \infty.$$ We define the blow-up set 
			\begin{equation*}
				\mathfrak{S}:=\left\{x\in M \,:\, \liminf_{k\rightarrow \infty} \frac{1}{2}\int_{B^M(x,r)}\abs{\nabla u_{\alpha}} dV_g \geq  \varepsilon_0^2, \quad \text{for all } r > 0\right\}
			\end{equation*}
			where $B^M(x,r) = \{y\in M\,:\,\mathrm{dist}^M(x,y) < r\}$ denotes the geodesic ball in $M$ and $\varepsilon_0$ is determined in \eqref{eq:gap constant}. Then $\mathfrak{S}$ is finite, written as $\mathfrak{S} = \{x^1, \cdots, x^{n_0}\}$. After choosing a subsequence, there exists a smooth $H$-surface $u_0:M\rightarrow N$ and finitely many bubbles, that is, a finite set of $H$-spheres $w^{ij}$, $1 \leq j \leq n_i$ such that $u_{\alpha}\rightarrow u_0$ weakly in $W^{1,2}(M,\R^K)$ and strongly in $C^\infty_{loc}(M\backslash \mathfrak{S}, N)$. Moreover, the following generalized energy identity holds
			\begin{equation}\label{generalizd energy}
				\lim_{k\rightarrow \infty}{E}_{\alpha}(u_{\alpha}) = E(u_0) + \frac{1}{2} \mathrm{Vol}(M) + \sum_{i=1}^{n_0}\sum_{j = 1}^{n_i} \mu^2_{ij} E(w^{ij}).
			\end{equation}
		\end{theorem}
    \subsubsection{Asymptotic Behavior on Necks}
    \ 
    \vskip5pt
		Now, we present our second main result about asymptotic analysis on neck, which provides a complete geometric picture of all possible limiting behaviors of the necks occurring in the blow-up process for sequences of $\alpha$-$H$-surfaces. We show that all necks between  bubbles and the base map converge to geodesics and we provide a scheme to calculate the length of these geodesics, see Remark \ref{rmk length formula} below.  More precisely, we have
		\begin{theorem}\label{analysis on neck}
			Let $(M,g)$ be a closed Riemann surface, $(N,h)$ be a $n$-dimensional closed Riemannian manifold that is isometrically embedded in $\R^K$ for some $K \in \mathbb N$. Assume that $\{u_{\alpha}\}_{\alpha \searrow 1} \subset C^\infty(M,N)$ is a sequence of $\alpha$-$H$-surfaces with uniformly bounded $\alpha$-energy, that there is only one blow up point $\mathfrak{S} = \{x_1\}$ and there is only one bubble in $B^M(x_1,r) \subset M$, for some $r>0$, denoted by $w^1:\S^2\rightarrow N$. Let
			\begin{equation*}
				\nu^1 = \liminf_{\alpha \searrow 1} \p{\lambda^1_\alpha}^{-\sqrt{\alpha - 1}}.
			\end{equation*}
			Then one of following statement holds
			\begin{enumerate}
				\item when $\nu^1 = 1$, the set $u_0\p{B^M(x_1,r)} \bigcup w^1(\S^2)$ is a connected subset of $N$ where $u_0$ is the weak limit of $u_\alpha$ in $W^{1,2}(M,N)$ as $\alpha\searrow 1$;
				\item when $\nu^1\in (1,\infty)$, the set $u_0\p{B^M(x_1,r)}$ and $w^1(\S^2)$ are connected by a geodesic $\Gamma \subset N$ with length
				\begin{equation*}
					L(\Gamma) = \sqrt{\frac{E(w^1)}{\pi}} \log \nu^1;
				\end{equation*}
				\item when $\nu^1 = \infty$, the neck contains at least a geodesic of infinite length.
			\end{enumerate}
		\end{theorem}
        \begin{rmk}\label{rmk length formula}
            It is important to note that, although we state the Theorem \ref{analysis on neck} under the assumption that there is only one bubble $w^1$ occurring the single blow up point $\set{x_1}$, it is not difficult to obtain a general version by an induction argument following the proofs in Section \ref{section neck length}. The length formula looks quite complicated and needs to be discussed by case splitting. For instance, if there are two $H$-spheres, $w^1$ and $w^2$, occurring the blow up point $\set{x_1}$, namely, there exists sequences of positvies numbers $\lambda_\alpha^1 \searrow 1$, $\lambda_\alpha^2 \searrow 1$ with $\lambda_\alpha^1/\lambda_\alpha^2 \rightarrow 0$ and sequences of points $x_\alpha^1 \rightarrow x_1$, $x_\alpha^2 \rightarrow x_1$ satisfying
            \begin{equation*}
                w^1 = \lim_{\alpha \searrow 1} u_\alpha(\lambda_\alpha^1 x  + x_\alpha^1) \quad \text{and}\quad w^2 = \lim_{\alpha \searrow 1} u_\alpha(\lambda_\alpha^2 x + x_\alpha^2).
            \end{equation*}
            Then, the length formula for geodesic connecting $u_0\p{B^M(x_1, r)}$ and $w^2(\S^2)$ is given by 
            \begin{equation*}
                L(u_0, w^2) = \sqrt{\frac{E(w^1) + E(w^2)}{\pi}} \log \nu^2.
            \end{equation*}
            And the length formula for geodesic connecting $w^2(\S^2)$ and $w^1(\S^2)$ is given by
            \begin{equation*}
                L(w^2, w^1) = \sqrt{\frac{E(w^1)}{\pi}}\log \frac{\nu^1}{\nu^2}.
            \end{equation*}
            Here, $$\nu^1 = \liminf_{\alpha \searrow 1} \p{\lambda^1_\alpha}^{-\sqrt{\alpha - 1}} \quad \text{and}\quad  \nu^2 = \liminf_{\alpha \searrow 1} \p{\lambda^2_\alpha}^{-\sqrt{\alpha - 1}}.$$
        \end{rmk}
        \subsubsection{Energy Identity Under Topological and Curvature Conditions}
        \
        \vskip5pt
        
		The topology and geometry of target manifold $(N,h)$ plays a critical role in investigating the convergence properties of $\alpha$-$H$-surfaces from some compact surface and moreover comparing the $\alpha - 1$ and the rate of scaling $\lambda_\alpha^{ij} \rightarrow 0$ as $\alpha\searrow 1$, that is, the value of $\mu_{ij}$ and $\nu_{ij}$. From the point view of differential geometry, it is natural and reasonable to find some geometric and topology condition on target  $(N, h)$ to ensure the energy identity holds, equivalently ensure the neck converges to a geodesic of finite length. To this end, utilizing Gromov's estimates \cite{Gromov1978} (See also \cite[Corollary 3.3.5]{moore2017}) on length of geodesics by its Morse index, we have the following
		\begin{theorem}\label{thm convergence}
			Let $(M,g)$ be a closed Riemann surface, $(N,h)$ be a $n$-dimensional compact Riemannian manifold, that is isometrically embedded in $\R^K$ for some $K \in \mathbb N$ and has finite fundamental group. Assume that $\{u_{\alpha}\}_{\alpha \searrow 1} \subset C^\infty(M,N)$ is a sequence of $\alpha$-$H$-surfaces with uniformly bounded $\alpha$-energy and  unifomly bounded Morse index, that is, $\mathrm{Ind}_{E^\omega_\alpha}(u_\alpha) \leq  C$ for some universal constant $C > 0$. Then $\mathfrak{S}$ is finite, written as $\mathfrak{S} = \{x^1, \dots, x^{n_0}\}$. After choosing a subsequence, there exists a smooth $H$-surface $u_0:M\rightarrow N$ and finitely many bubbles, that is, a finite set of $H$-spheres $w^{ij}$, $1 \leq j \leq n_i$ such that $u_{\alpha}\rightarrow u_0$ weakly in $W^{1,2}(M,\R^K)$ and strongly in $C^\infty_{loc}(M\backslash \mathfrak{S}, N)$.  Moreover, the limiting necks consists of some geodesics of finite length, and hence the following energy identity holds
			\begin{equation}\label{energy ide}
				\lim_{k\rightarrow \infty} E_{\alpha}(u_{\alpha}) = E(u_0) + \frac{1}{2} \mathrm{Vol}(M) + \sum_{i=1}^{n_0}\sum_{j = 1}^{n_i} E(w^{ij}).
			\end{equation}
		\end{theorem}
    
    By Myers' Theorem from Riemannian geometry, see \cite{Myers1941}, the diameter of complete Riemannian manifold $(N,h)$ with $\mathrm{Ric}(N) \geq \kappa > 0$ satisfies
		\begin{equation*}
			\mathrm{diam}(N) \leq \frac{\pi}{\sqrt{\kappa}}
		\end{equation*}
		and any geodesic $\Gamma \subset (N,h)$ with length
		\begin{equation*}
			L(\Gamma) \geq \frac{\pi}{\sqrt{\kappa}}
		\end{equation*}
		is unstable. Moreover, the fundamental group $\pi_1(N)$ is finite. Utilizing this fact, as a corollary of Theorem \ref{thm convergence} we can obtain the following consequence
        \begin{coro}\label{coro:ricci}
            If we assume $(N,h)$ be a $n$-dimensional complete Riemannian manifold with strictly positive Ricci curvature, that is, $\mathrm{Ric}(N) > \kappa > 0$ and keep the remaining assumption same as Theorem \ref{thm convergence}, then the energy identity stated as Theorem \ref{thm convergence} for sequences of $\alpha$-$H$-surfaces with bounded Morse index still holds.
        \end{coro}
        In the context of $\alpha$-harmonic maps, Moore \cite{Moore2007} (see also \cite[Theorem 4.9.2]{moore2017}) demonstrated bubble tree convergence, akin to the conditions described in Theorem \ref{thm convergence}. Similar result for $\alpha$-harmonic maps under Ricci curvature assumptions, as seen in Corollary \ref{coro:ricci}, was established by Li-Liu-Wang \cite{Li-Liu-Wang2017}.
            \subsubsection{Non-constancy of Weak Limit}\label{section 4.3}
		\ 
		\vskip5pt
		There is a key insight about sequences of non-trivial $\alpha$-$H$-spheres, as one of main advantages of $\alpha$-energy approximation to Dirichlet energy. More precisely, we can show that, if $u_\alpha$ is a sequence of non-trivial $\alpha$-$H$-spheres with uniformly bounded $\alpha$-energy, then the weak limit of $u_\alpha$ is non-constant. In the context of $\alpha$-harmonic maps, see \cite[Lemma 5.3]{sacks1981existence}. The following Lemma \ref{lem average zero} plays a crucial role in reaching the desired result.
		\begin{lemma}\label{lem average zero}
			Let $\iota : \S^2 \rightarrow \R^3$ be the standard 
			isometric embedding, that is, 
            $$\p{\iota^1(p)}^2 + \p{\iota^2(p)}^2 + \p{\iota^3(p)}^2 = 1,\quad \text{for } p \in \S^2$$ If $u_\alpha \in C^2(\S^2,N)$ is a critical point for $E_\alpha^\omega$ for $\alpha > 1$, then
			\begin{equation}\label{eq average zero}
				\int_{\S^2} \iota^i(x) \Psi_\alpha \p{|\nabla u_\alpha(x)|^2} dV_g = 0, \quad i = 1,2,3,
			\end{equation}
			where $\Psi_\alpha:[0,\infty) \rightarrow \R$ is a strictly increasing smooth function defined by 
			\begin{equation*}
				\Psi_\alpha(r) = \frac{\alpha(1 + r)^{\alpha- 1}r - (1 + r)^{\alpha} + 1}{\alpha - 1}.
			\end{equation*}
		\end{lemma}
		\begin{proof}
			Without loss of generality, we take the standard metric on $\S^2$ such that it admits constant curvature one. Moreover, by the rotational symmetric of $\S^2$ and \eqref{eq average zero}, it suffices to show
			\begin{equation*}
				\int_{\S^2} \iota^3(x) \Psi_\alpha\p{|\nabla u_\alpha(x)|^2} dV_g = 0.
			\end{equation*}
			Utilizing the stereographic projection from $\S^2$ to $\R^2$  we can write the metric on $(\S^2,ds^2)$ with the polar coordinate $(\rho,\theta)$ as
			\begin{equation*}
				ds^2 = \frac{4}{\p{1 + \rho^2}^2}\p{d\rho^2 + \rho^2 d\theta^2}.
			\end{equation*}
			Then, taking a conformal transformation $(\rho,\theta)\mapsto (\varphi,\eta)$ by $\rho = e^{-\varphi}$ and $\theta = \eta$, we can rewrite the metric $ds^2$ as 
			\begin{equation*}
				ds^2 = \frac{1}{\cosh^2{\varphi}}\p{d\varphi^2 + d\eta^2}.
			\end{equation*}
			Note that, since stereographic projection is a conformal coordinate, $(\varphi,\eta)$ is also a conformal coordinate of $\S^2$. Then, using this coordinate we can define a collection of conformal transformation $\{\phi_t\}_{t\in\R}$ with each $\phi_t :\S^2 \rightarrow \S^2$ expressed as 
			\begin{equation*}
				\varphi(\phi_t(x)) = \varphi(x) + t,\quad \eta(\phi_t(x)) = \eta(x).
			\end{equation*}
			Thus, the function $E_\alpha^\omega$ acts on $u\circ \phi_t$ can be expressed as
			\begin{align}\label{eq 19}
				E_\alpha^\omega(u\circ \phi_t) &= \frac{1}{2}\int_{\S^2} \p{1 + \p{\abs{\frac{\partial u}{\partial \varphi}}^2 + \abs{\frac{\partial u}{\partial \eta}}^2}\cosh^2{(\varphi + t)}}^\alpha dV_g + \int_{\S^2}(u\circ \phi_t)^*\omega \nonumber\\
				&= \frac{1}{2}\int_{\S^2} \p{1 + \p{\abs{\frac{\partial u}{\partial \varphi}}^2 + \abs{\frac{\partial u}{\partial \eta}}^2}\cosh^2{(\varphi + t)}}^\alpha \frac{d\varphi d\eta}{\cosh^2{(\varphi + t)}} + \int_{\S^2}u^*\omega.
			\end{align}
			Here, in the second identity we use the conformally invariance of the integral of $u^*\omega$. Then, we take the derivative in the identity \eqref{eq 19} with respect to $t$ at $t = 0$ to get
			\begin{align*}
				\frac{d}{dt}\bigg|_{t=0} E^\omega_{\alpha}(u\circ \phi_t) &= \alpha\int_{\S^2} \p{1 + |\nabla u|^2}^{\alpha - 1} \p{\abs{\frac{\partial u}{\partial \varphi}}^2 + \abs{\frac{\partial u}{\partial \eta}}^2} \tanh{\varphi} d\varphi d\eta\\
				&\quad - \int_{\S^2}\p{1 + |\nabla u|^2}^{\alpha} \frac{\tanh{\varphi}}{\cosh^2{\varphi}} d\varphi d\eta\\
				&=\alpha\int_{\S^2} \p{1 + |\nabla u|^2}^{\alpha - 1}\abs{\nabla u}^2 \tanh{\varphi}  dV_g -\int_{\S^2}\p{1 + |\nabla u|^2}^{\alpha} \tanh{\varphi} dV_g.
			\end{align*}
			If $u$ is a critical point of $E_\alpha^\omega$, then by the regularity Lemma \ref{lem regularity}  $u$ is stationary with respect to $\phi_t$. Hence, we have 
			\begin{equation}\label{eq 20}
				\int_{\S^2} \p{\alpha\p{1 + |\nabla u|^2}^{\alpha - 1}\abs{\nabla u}^2 - \p{1 + |\nabla u|^2}^{\alpha} } \tanh{\varphi} dV_g = 0.
			\end{equation}
			In the stereographic projection coordinate, we have 
			\begin{equation*}
				\tanh{\varphi} = \frac{\sinh{\varphi}}{\cosh{\varphi}} = \frac{\rho^2 - 1}{\rho^2 + 1} = \iota(\varphi,\eta)^3
			\end{equation*}
			and 
			\begin{equation*}
				\int_{\S^2} \tanh{\varphi} dV_g = \int_{\S^2}  \iota^3(\varphi,\eta) dV_g = 0.
			\end{equation*}
			Plugging these two identities into \eqref{eq 20}, we finally obtain 
			\begin{equation*}
				\int_{\S^2} \p{\alpha\p{1 + |\nabla u|^2}^{\alpha - 1}\abs{\nabla u}^2 - \p{1 + |\nabla u|^2}^{\alpha}  + 1}\iota^3(\varphi,\eta) dV_g = 0
			\end{equation*}
			which is exactly \eqref{eq average zero}.
		\end{proof}
		We need to mention that $\Psi_\alpha$ converges to a smooth function $\Psi_1$ as $\alpha \searrow 1$, more precisely, it can be expressed as 
		\begin{equation*}
			\Psi_1(r) = r- \log{(1 + r)}
		\end{equation*}
		which is also a strictly increasing smooth function. Now, we can prove the main consequence of this subsection.
		\begin{prop}\label{non constant limit}
			Let $u_\alpha:\S^2\rightarrow N$ be a sequence of non-constant critical points for $E^\omega_{\alpha}$ that converges strongly in $C^2\p{\S^2\backslash\{x_1,x_2,\dots,x_l\}}$ to $u$ for some $l \in \mathbb{N}$ as $\alpha \rightarrow 1$. Then the limit $u:\S^2\rightarrow N$ is also non-constant.
		\end{prop}
		\begin{proof}
			Let $(\varphi,\theta)$ be the geographic coordinates of $\S^2$ with $0\leq \varphi \leq \pi$ and $0\leq \theta \leq 2\pi$. And denote $\S^+ = \{(\varphi,\theta)\,:\,0\leq \varphi\leq \frac{\pi}{2}\}$ and $\S^- = \{(\varphi,\theta)\,:\,\frac{\pi}{2}\leq \varphi\leq \pi\}$. Since the set of points that fails to convergence is finite, after taking a fractional linear transformation of $\S^2$, we can assume  $\{x_1,x_2,\dots,x_l\} \subset \mathrm{int}(\S^+)$ such that $\varphi(x_i) < \frac{\pi}{3}$ for $1\leq i \leq l$. Then splitting the integral domain $\S^2$ into $\S^+$ and $\S^-$ in identity \eqref{eq average zero} obtained in Lemma \ref{lem average zero} gives
			\begin{equation}\label{eq 21}
				\int_{\S^+} \iota^3(x) \Psi_\alpha\p{|\nabla u_\alpha(x)|^2} dV_g =  \int_{\S^-} \abs{\iota^3(x)} \Psi_\alpha\p{|\nabla u_\alpha(x)|^2} dV_g.
			\end{equation}
			If the limit $u$ is constant, then by Theorem \ref{thm convergence} the energy must concentrate at some point, say $x_1$, and one can construct a  rescaling map $v_\alpha$ that converges strongly in $C^2_{loc}$ to a non-constant bubble $v:\S^2\rightarrow N$. Then,  utilizing \eqref{eq 21} we can estimate
			\begin{align}\label{eq: half estimate}
				0 < \frac{E(v)}{2}\leq  \frac{1}{2}\liminf_{\alpha\searrow 1}E(u_\alpha,\S^+) &\leq \liminf_{\alpha\searrow 1} \int_{\S^+} \iota^3(x) |\nabla u_\alpha|^2 dV_g\nonumber\\
				&\leq 2\liminf_{\alpha \searrow 1} \int_{\S^+} \iota^3(x) \Psi_\alpha\p{|\nabla u_\alpha(x)|^2} dV_g \nonumber\\
				&=2\liminf_{\alpha \searrow 1} \int_{\S^-} \abs{\iota^3(x)} \Psi_\alpha\p{|\nabla u_\alpha(x)|^2} dV_g = 0
			\end{align}
			which is a contradiction. Here, we note that $\Psi(r)/r \rightarrow 1$ as $r \rightarrow \infty$ which implies the second inequality of the above estimates \eqref{eq: half estimate}. Therefore, we reach the conclusion of Proposition \ref{non constant limit}.
		\end{proof}
		\vskip5pt
        \subsection{Preparations for the Proof of Main Theorem}\label{sec 2.2}
        \ 
		\vskip5pt
        In this subsection, we will derive some basic Lemmas for $\alpha$-$H$-surfaces, such as small energy regularity, energy gap and removability of isolated singularities of $H$-surfaces that will be described in Subsubection \ref{section 4.1 pre of blow up}, and we will establish several Pohozaev type identities see Lemma \ref{lem:pohozaev} in Subsubection \ref{section:pohozaev}.
        
        By Riemann mapping theorem, for each $p \in M$ there exists an isothermal coordinate system in a neighborhood $U(p)$ of $p$ such that the metric $g$ can be written as 
        \begin{equation*}
            g = e^\varphi \p{(dx^1)^2 + (dx^2)^2}
        \end{equation*}
        where $x = (x^1,x^2) \in {B}(0,1) \subset \R^2$ and $\varphi$ is a smooth function satisfying $\varphi(p) = 0$. Therefore, it suffices to restrict our analysis on unit ball $B(0,1) \subset \R^2$ equipped with the metric
        \begin{equation*}
            g_\alpha := e^{\varphi_\alpha}\left(\p{dx^1}^2 + \p{dx^2}^2\right) \quad \text{with } \varphi_\alpha(0) = 0 \text{ and } \varphi_\alpha \rightarrow \varphi \in C^{\infty}(\overline{B(0,1)})
        \end{equation*}
        in order to investigate the local bubbling behavior for $\alpha$-$H$-surfaces. Hence, under these isothermal coordinates  the Euler Lagrange equation \eqref{el general 1 equi} and \eqref{el general 2 equi} are equivalent to the following
        \begin{align}
              \label{el general 1 equi}
        \Delta u_\alpha + (\alpha - 1)\frac{\nabla|\nabla_{g_\alpha} u_\alpha|^2\cdot \nabla u_\alpha}{\tau_\alpha+|\nabla_{g_\alpha} u_\alpha|^2} &+ A(u_\alpha)\left(\nabla u_\alpha, \nabla u_\alpha\right)\nonumber\\
        &= \tau_\alpha^{\alpha -1}\frac{H(u_\alpha)(\nablap u_\alpha, \nabla u_\alpha)}{\alpha \left(\tau_\alpha + |\nabla_{g_\alpha} u_\alpha|^2\right)^{\alpha - 1}}
         \end{align}
 and in divergence form
   \begin{align}
       \label{el general 2 equi}
        -\diver\left(\left(\tau_\alpha+|\nabla_{g_\alpha} u_\alpha|^2\right)^{\alpha - 1}\nabla u_\alpha\right) &+ \left(\tau_\alpha + |\nabla_{g_\alpha} u_\alpha|^2\right)^{\alpha -1}A(u_\alpha)(\nabla u_\alpha,\nabla u_\alpha)\nonumber\\
        &= \frac{\tau_\alpha^{\alpha -1}}{\alpha}H(u_\alpha)\left(\nablap u_\alpha, \nabla u_\alpha\right)
   \end{align}
		\subsubsection{Small Energy Regularity, Energy Gap and Removability of Isolated Singularities}\label{section 4.1 pre of blow up}
		\ 
		\vskip5pt
            Similarly to blow-up phenomenon for sequences of $\alpha$-harmonic maps, that was developed by Sacks-Uhlenbeck \cite{sacks1981existence}, by showing the small energy regularity Lemma \ref{lem4.1}, energy gap Lemma \ref{blow 2} and the removability of isolated singularities Lemma \ref{blow 3} for $H$-surfaces, we can establish a similar convergence theory for general sequence of $\alpha$-$H$-surfaces $\{u_{\alpha}\}_{\alpha \searrow 1}$ (as critical points of generalized functional $\dbl{E}^{\omega}_\alpha$) with uniformly bounded $\alpha$-energy.
		In the $H$-surface context, compared with the case of harmonic maps, the following inequality 
		\begin{align}\label{eq : H equiva A}
			\left\|\frac{H(u)(\nablap u, \nabla u)}{\alpha \left(\tau_\alpha + |\nabla_{g_\alpha} u|^2\right)^{\alpha - 1}}\right\|_{L^1(M)} &\leq \frac{1}{\alpha\beta_0}\norm{H(u)(\nablap u, \nabla u) }_{L^1(M)} \nonumber\\
			&\leq \frac{1}{2\alpha \beta_0}\|H\|_{L^\infty(N)}\|\nabla u\|^2_{L^2(M)}
		\end{align}
		implies that the new quadratic growth part arising from the mean curvature type  vector field $H(\nablap u,\nabla u)$ actually plays a complete similar role with the original second fundamental form term $A(\nabla u, \nabla u)$ in the proof of small energy regularity for $\alpha$-harmonic maps, see \cite[Main Estimate 3.2]{sacks1981existence}. Based on this fact, it is not difficult to establish the following small energy regularity for $\alpha$-$H$-surfaces:
		\begin{lemma}[Small Energy Regularity]\label{lem4.1}
			Let $\{u_\alpha\}_{\alpha > 1}$ be a sequence of critical points of $\dbl{E}^\omega_\alpha$ in $W^{1,2\alpha}(B(0,1),N)$ where $B(0,1)$ is equipped with metric
            \begin{equation*}
                g_\alpha := e^{\varphi_\alpha}\left(\p{dx^1}^2 + \p{dx^2}^2\right) \quad \text{with } \varphi_\alpha(0) = 0 \text{ and } \varphi_\alpha \rightarrow \varphi \in C^{\infty}(\overline{B(0,1)})
            \end{equation*}
            as $\alpha \searrow 1$.
            Then, there exists  constants $\varepsilon_0 > 0$ and $\alpha_0 > 1$ such that if
			\begin{equation*}
				\sup_{1 <  \alpha < \alpha_0} E(u_\alpha, B) \leq \varepsilon_0^2,
			\end{equation*}
			where $B := B(0,1)$ for simplicity, then for any $B^\prime \subset\subset B$ we have
			\begin{equation}
				||\nabla u_\alpha(x)||_{W^{2,p}(B^\prime,N)} \leq C(p,B^\prime, N)\|\nabla u_\alpha\|_{L^2(B(0,1), N)},
			\end{equation}
			\text{for all} $1 < \alpha \leq \alpha_0 $ and  $1 < p < \infty$, where $C(p,B^\prime,N)$ is a constant depending only on $1 < p < \infty$, $B^\prime \subset B$ and geometries of $N$.
		\end{lemma}
		\begin{proof}
			Since the desired estimates holds locally and $g_\alpha \rightarrow g$ smoothly as $\alpha \searrow 1$, it suffices to prove the Lemma \ref{lem4.1} for sequence $u_\alpha : B\subset \R^2 \rightarrow N$ with Euclidean metric on $B$ by choosing small enough $\alpha_0 -1$. Let $\varphi$ be a smooth function which is $1$ on $B^\prime$ and supports in $B$, then multiplying the Euler-Lagrange equation \eqref{el general 1 equi} for $\dbl{E}^\omega_{\alpha}$ by $\varphi$ and writing the terms arising from the derivatives on $\varphi$ in the right-hand side yield
			\begin{align*}
				&\left|\Delta(\varphi u_\alpha) + (\alpha - 1)\frac{\inner{\nabla^2 (\varphi u_\alpha),\nabla u_\alpha }\cdot \nabla u_\alpha}{\tau_\alpha+|\nabla u_\alpha|^2} \right. \\
				&\quad \quad \quad \quad  + \left.A(u_\alpha)\left(\nabla(\varphi u_\alpha), \nabla u_\alpha\right) + \tau_\alpha^{\alpha -1}\frac{ H(u_\alpha)(\nablap u_\alpha,\nabla(\varphi u_\alpha))}{\alpha \p{\tau_\alpha + |\nabla u_\alpha|^2}^{\alpha - 1}}\right|\\
				&\quad \quad \leq C \big(\varphi,\nabla\varphi, N, ||A||_{L^\infty}, ||H||_{L^\infty}\big) \p{|u_\alpha| + |\nabla u_\alpha|},
			\end{align*}
			where $C(\varphi,\nabla\varphi, N, ||A||_{L^\infty}, ||H||_{L^\infty})$ is a constant that depends on the $\varphi$, $\nabla\varphi$, geometries of target $N$, second fundamental form $A$ and mean curvature type  vector field $H$. For notation simplicity, we denote it by $C_0$. Keeping in mind \eqref{eq : H equiva A} and applying $L^p$ estimates for Laplace operators, we obtain
			\begin{align}\label{eq: p estimates}
				(C_p)^{-1}\norm{\varphi u_\alpha}_{W^{2,p}(B,N)} &\leq (\alpha - 1) \norm{\varphi u_\alpha}_{W^{2,p}(B,N)}\nonumber\\
				&\quad + (||A||_{L^\infty(N)} + \norm{H}_{L^\infty(N)}) \norm{|\nabla (\varphi u_\alpha)|\cdot |\nabla u_\alpha|}_{L^p(B,N)}\nonumber\\
				& \quad + C_0 \norm{u_\alpha}_{W^{1,p}(B,N)},
			\end{align}
			where $C_p$ is the constant arising from operator norms of Laplace operator. Now, let $p = 4/3$ and take $2(\alpha_0 - 1)< (C_p)^{-1}$, using H\"{o}lder's inequality we have 
			\begin{align}
				\label{eq: p =4/3}
				&\p{(C_{\frac{4}{3}})^{-1} - 2(\alpha - 1)}\norm{\varphi u_\alpha}_{W^{2,4/3}(B,N)}\nonumber\\
				&\quad \quad \quad  \leq C(A,H)\norm{|\nabla (\varphi u_\alpha)|\cdot |\nabla u_\alpha|}_{L^{4/3}(B,N)} + C_0 \norm{u_\alpha}_{W^{1,4/3}(B,N)}\nonumber\\
				&\quad \quad \quad  \leq C(A,H) E(u_\alpha,B)\norm{\nabla(\varphi u_\alpha)}_{L^4}+ C_0 \norm{u_\alpha}_{W^{1,4/3}(B,N)}.
			\end{align}
			By Sobolev embedding $W^{2,4/3}(B,N) \hookrightarrow W^{1,4}(B,N)$, we conclude that from \eqref{eq: p =4/3}
			\begin{align}\label{eq:p=4/3 2}
				\p{(C_{\frac{4}{3}})^{-1} - 2(\alpha - 1) - C_e C(A,H) E(u_\alpha,B)}&\norm{\varphi u_\alpha}_{W^{2,4/3}(B,N)}\nonumber\\
				&\leq C_0 \norm{u_\alpha}_{W^{1,4/3}(B,N)}
			\end{align}
			where $C_e$ is the norm of the embedding $W^{2,4/3}(B,N) \hookrightarrow W^{1,4}(B,N)$ and 
   $$C(A,H) := ||A||_{L^\infty(N)} + \norm{H}_{L^\infty(N)}.$$ Note that, after replacing $u_\alpha$ with $u_\alpha -1/\mathrm{Vol}(B) \int_{B} u_\alpha$, we can assume $\int_{B} u_\alpha = 0$. So, the right-hand side of \eqref{eq:p=4/3 2} is controlled by $E(u_\alpha,B)$ by Poincar\'{e}'s inequality.  We take $\varepsilon_0$ is small enough such that
			\begin{equation*}
				(C_{\frac{4}{3}})^{-1} - 2(\alpha - 1) - C_e C(A,H) \varepsilon_0^2 > 0.
			\end{equation*}
			Then, in estimate \eqref{eq: p estimates}, we take $p = 2$ to obtain
			\begin{align}\label{eq: p =2}
				\p{(C_2)^{-1} - 2(\alpha - 1)}&\norm{\varphi u_\alpha}_{W^{1,2\alpha}(B,N)}\nonumber\\
				&\leq C(A,H)\norm{\varphi u_\alpha}_{W^{1,4}(B,N)} + C_0 \norm{u_\alpha}_{W^{1,2}(B,N)}.
			\end{align}
			By Sobolev embedding $W^{1,2\alpha}(B,N)\hookrightarrow W^{1, p}(B,N)$ for all $1 < p < \infty$. \eqref{eq: p =2} will give the estimates of $||\varphi u_\alpha||_{W^{1, p}(B,N)}$ and plugging this estimates into \eqref{eq: p estimates}
			gives
			\begin{equation*}
				||\varphi u_\alpha||_{W^{2, p}(B,N)} \leq C^\prime (\varphi,\nabla\varphi, N, ||A||_{L^\infty}, ||H||_{L^\infty}) \norm{\nabla u_\alpha}_{L^4(B,N)}.
			\end{equation*}
			Then, by $W^{2,4/3}(B,N) \hookrightarrow W^{1,4}(B,N)$, plugging \eqref{eq:p=4/3 2} into above inequality will give the desired estimates of the Lemma \ref{lem4.1}.
		\end{proof}
		By a similar argument to \cite[Theorem 3.3]{sacks1981existence}, we can obtain the following globally energy gap Lemma for $\alpha$-$H$-surfaces $u_\alpha$ from $M$ to $N$.
		\begin{lemma}[Energy Gap]\label{blow 2}
			There exists $\varepsilon_0 > 0$ and $\alpha_0 > 1$ such that if $E(u_\alpha)< \varepsilon_0^2$, $1\leq \alpha < \alpha_0$ and $u_\alpha : M \rightarrow N$ is a critical map of $\dbl{E}^\omega_\alpha$, then $u_\alpha$ is constant and $E(u_\alpha) = 0.$
		\end{lemma}
		\begin{proof}
			If we replace the smooth function $\varphi$ with $\varphi \equiv 1$ and do the estimates globally on $M$, then $C_0 \equiv 0$ arising in Lemma \ref{lem4.1}. Thus, \eqref{eq:p=4/3 2} becomes
			\begin{equation*}
				\p{(C_{\frac{4}{3}})^{-1} - 2(\alpha - 1) - C_e C(A,H) E(u_\alpha,B)}\norm{\varphi u_\alpha}_{W^{2,4/3}(M,N)}\leq 0.
			\end{equation*}
			Therefore, when $E(u_\alpha,M) < \varepsilon_0^2$ is small enough, every critical point $u_\alpha$ of $\dbl{E}^\omega_{\alpha}$ is constant.
        \end{proof}
		
  Moreover, when $\alpha = 1$, we have the following removability of isolated singularities for $H$-surfaces by combining the proof in \cite[Theorem 2.4.1]{jost1991two} and the regularity result in \cite[Theorem 1.2]{Riviere2007}.
		\begin{lemma}[Removability of Isolated Singularities]\label{blow 3}
			Suppose that $u \in C^2(B(0,1)$ $ \backslash \{0\},N) $ where $B(0,1)$ equipped with metric $g= e^\varphi \p{(dx^1)^2 + (dx^2)^2}$ for some smooth function $\varphi$, $E(u,B(0,1)) <\infty$ and that $u$ satisfies the Euler-Lagrange equation $\eqref{eq H-surface intro}$, then $u$ can extends to a smooth $H$-surface $u:B(0,1)\rightarrow N$.
		\end{lemma}
		
\subsubsection{Pohozaev type Identities}\label{section:pohozaev}
\ 
\vskip5pt

As a corollary of Lemma \ref{lem4.1}, we can establish  the following boundedness estimates
\begin{equation*}
    \limsup_{\alpha \searrow 1} \Big\|\p{\tau_\alpha + \abs{\nabla_{g_\alpha}u_\alpha}^2}^{\alpha - 1}\Big\|_{C^0(B(0,1))} \leq C < \infty.
\end{equation*}
\begin{lemma}\label{bounded F}
     Let $(B(0,1),g_\alpha)$ be a unit disk in $\R^2$ equipped with a metric 
     $$g_\alpha = e^{\varphi_\alpha}\big((dx^1)^2 + (dx^2)^2\big)$$
     where $\varphi_\alpha(0) = 0$ and $\varphi_\alpha$ is a sequence of smooth function such that $\varphi_\alpha \rightarrow \varphi$ strongly in $ C^\infty(\overline{B(0,1)})$. If $u_\alpha$ is a sequence of $\alpha$-$H$-surfaces with uniformly bounded generalized $\alpha$-energy $\sup_{\alpha > 1}\dbl{E}_\alpha(u_\alpha, B(0,1)) < \infty$ and $\lim_{\alpha \searrow 1} \tau_\alpha^{\alpha - 1} > \beta_0 > 0$, then there exists a positive $\beta_1 > 0$ which is independent of $\alpha \searrow 1$ such that 
     \begin{align}\label{eq:C0 estimates}
         \beta_0 &\leq \liminf_{\alpha \searrow 1}\Big\|\p{\tau_\alpha + \abs{\nabla_{g_\alpha}u_\alpha}^2}^{\alpha - 1}\Big\|_{C^0(B(0,1))} \nonumber\\
         &\leq \limsup_{\alpha \searrow 1}\Big\|\p{\tau_\alpha + \abs{\nabla_{g_\alpha}u_\alpha}^2}^{\alpha - 1}\Big\|_{C^0(B(0,1))} \leq \beta_1.
     \end{align}
\end{lemma}
\begin{proof}
    It suffices to prove the upper bound part of \eqref{eq:C0 estimates}. If the energy concentrate set 
    \begin{equation*}
        \mathfrak{S}:=\left\{x\in B(0,1) \,:\, \liminf_{k\rightarrow \infty}\frac{1}{2}\int_{B(x,r)}\abs{\nabla_{g_\alpha} u_{\alpha}}^2 dV_{g_\alpha} \geq  {\varepsilon_0^2}, \quad \text{for all } r > 0\right\}
    \end{equation*}
    is empty, then by Lemma \ref{lem4.1} $u_\alpha$ converges to some $H$-surface $u_0$ smoothly which implies 
    \begin{equation*}
        \limsup_{\alpha \searrow 1} \norm{\nabla_{g_\alpha} u_\alpha}_{C^0(B(0,1))} \leq C <\infty.
    \end{equation*}
    Hence, \eqref{eq:C0 estimates} follows directly. Thus, we assume that $\mathfrak{S}$ is non-empty. Without loss of generality, we further assume that $0 \in \mathfrak{S}$ is the only energy concentration point. Then, there exists finitely many bubbles occurring around $0$, hence there exists sequences of positive numbers $\lambda_\alpha^i\searrow 0$ and sequences of points $x_\alpha^i \searrow 0$ as $\alpha \searrow 1$, for $1 \leq i \leq n_0$ satisfying the alternative \ref{A1} or \ref{A2}. We choose the smallest $\lambda_\alpha^{i_0}$ satisfying
    \begin{equation*}
        \limsup_{\alpha \searrow 1} \frac{\lambda_\alpha^{i_0}}{\lambda_{\alpha}^i} \leq C < \infty \quad \text{for any } 1\leq i \neq i_0 \leq 
        n_0.
    \end{equation*}
    Therefore, the energy concentration $\mathfrak{S}$ set of rescaled sequences $w_\alpha(x) := u_\alpha(x_\alpha^{i_0})$ is empty, hence by Lemma \ref{lem4.1} we have
    \begin{equation*}
        \limsup_{\alpha \searrow 1} \Big\|\p{\tau_\alpha + \abs{\nabla_{g_\alpha}u_\alpha}^2}^{\alpha - 1}\Big\|_{C^0(B(0,1))} \leq C \limsup_{\alpha \searrow 1}\p{1 + \p{\lambda_\alpha^{i_0}}^{2 -2\alpha}} \leq C(1 + \mu_{max})
    \end{equation*}
    which yields the estimate \eqref{eq:C0 estimates} by letting $\beta_1 := C(1 + \mu_{max})$.
\end{proof}
Next we are devoted to derive some general variational formulas for the functional $\dbl{E}^\omega_\alpha$, to obtain some critical estimates of the energy of $\alpha$-$H$-surfaces on the neck domains. We adapt the idea introduced in \cite[Lemma 2.3]{li2010} and hence some reduplicative computational details are omitted. 

\begin{lemma}\label{lem:pohozaev}
    Let $(B(0,1),g_\alpha)$ be a unit disk in $\R^2$ equipped with a metric 
    $$g_\alpha = e^{\varphi_\alpha}\big((dx^1)^2 + (dx^2)^2\big)$$
    where $\varphi_\alpha(0) = 0$ and $\varphi_\alpha$ is a sequence of smooth function and $\varphi_\alpha \rightarrow \varphi$ strongly in $ C^\infty(\overline{B(0,1)})$. If $u_\alpha$ is a critical point of $\dbl{E}^\omega_\alpha(u, B(0,1))$, then for any $0 < t < 1$ there holds
\begin{align}\label{pohozaev esti 2}
    \p{1-\frac{1}{2\alpha}}&\int_{\partial B(0,t)} \p{\tau_\alpha + \abs{\nabla_{g_\alpha} u_\alpha}^2}^{\alpha - 1}\abs{\frac{\partial u_\alpha}{\partial r}}^2 ds\nonumber \\
    &\quad - \frac{1}{2\alpha} \int_{\partial B(0,t)}\p{\tau_\alpha + \abs{\nabla_{g_\alpha}u_\alpha}^2}^{\alpha - 1}\frac{1}{|x|^2}\abs{\frac{\partial u_\alpha}{\partial \theta}}^2ds\nonumber\\
    &=\p{1 - \frac{1}{\alpha}}\frac{1}{t}\int_{B(0,t)} \p{\tau_\alpha + \abs{\nabla_{g_\alpha}u_\alpha}^2}^{\alpha - 1} \abs{\nabla u_\alpha}^2 dx + O(t).
    \end{align}
    and
\begin{align}\label{pohozaev esti 3}
    \p{1-\frac{1}{2\alpha}}&\int_{\partial B(0,t)} \p{\tau_\alpha + \abs{\nabla_{g_\alpha} u_\alpha}^2}^{\alpha - 1}\abs{\nabla u_\alpha}^2 ds\nonumber \\
    &\quad - \int_{\partial B(0,t)}\p{\tau_\alpha + \abs{\nabla_{g_\alpha}u_\alpha}^2}^{\alpha - 1}\frac{1}{|x|^2}\abs{\frac{\partial u_\alpha}{\partial \theta}}^2ds\nonumber\\
    &=\p{1 - \frac{1}{\alpha}}\frac{1}{t}\int_{B(0,t)} \p{\tau_\alpha + \abs{\nabla_{g_\alpha}u_\alpha}^2}^{\alpha - 1} \abs{\nabla u_\alpha}^2 dx + O(t).
    \end{align}
\end{lemma}
\begin{proof}
    Taking a 1-parameter family of transformations group $\{\phi_s\}$ that is generated by the vector field supported in $B(0,1)\subset\R^2$, we compute
    \begin{align*}
        &\dbl{E}^\omega_\alpha(u\circ \phi_s, B(0,1))\\
        &= \int_{B(0,1)}\p{\tau_\alpha + \abs{\nabla_{g_\alpha}(u\circ\phi_s)}^2}^\alpha dV_{g_\alpha} + \int_{B(0,1)}(u\circ\phi_s)^*\omega\\
        &=\int_{B(0,1)}\p{\tau_\alpha + \sum_{i = 1}^2\abs{du\p{(\phi_s)_*(e_i(x))}}^2}^\alpha dV_{g_\alpha}\\
        &\quad\quad+ \int_{B(0,1)}\omega_{ij}(u\circ \phi_s)\nablap(u\circ \phi_s)^i \nabla(u\circ \phi_s)^j dx^1\wedge dx^2\\
        & = \int_{B(0,1)} \p{\tau_\alpha + \sum_{i = 1}^2\abs{du\p{(\phi_s)_*(e_i(\phi^{-1}_s(x)))}}^2}^\alpha J(\phi^{-1}_s) dV_{g_\alpha}\\
        &\quad \quad+ \int_{B(0,1)}\omega_{ij}(u\circ\phi_s)\p{\frac{\partial u^i}{\partial x^k}\frac{\partial \phi_s^k}{\partial z^1}\frac{\partial u^j}{\partial x^l}\frac{\partial \phi_s^l}{\partial z^2}- \frac{\partial u^j}{\partial x^k}\frac{\partial \phi_s^k}{\partial z^1}\frac{\partial u^i}{\partial x^l}\frac{\partial \phi_s^l}{\partial z^2}} dx^1\wedge dx^2\\
        &:= A + B
    \end{align*}
    where $\{e_i\}$ is a local orthonormal basis of $TB(0,1)$ and $J(\phi^{-1}_s)$ is the Jacobian of $\phi^{-1}_s$. Utilizing the first variational formula for area functional
    \begin{equation*}
        \frac{d}{ds}J(\phi_s^{-1})dV_{g_\alpha}\Big|_{s = 0} = -\diver(X)dV_{g_\alpha},
    \end{equation*}
    differentiating $\dbl{E}^\omega_\alpha$ yields
    \begin{align*}
        \frac{d}{ds}\dbl{E}^\omega_\alpha(u \circ \phi_s )\Big|_{s = 0} &= \delta\dbl{E}^\omega_\alpha(u)\p{du(X)}\\
        &= - \int_{B(0,1)} \p{\tau_\alpha + \abs{\nabla_{g_\alpha}u}^2}^\alpha \diver(X)dV_{V_{g_\alpha}}\\
        &\quad\quad + 2\alpha\sum_i\int_{B(0,1)}\p{\tau_\alpha  + \abs{\nabla_{g_\alpha}u}^2}^{\alpha - 1}\langle du(\nabla_{e_i} X), du(e_i)\rangle dV_{g_\alpha}\\
        &\quad \quad+\frac{d}{ds}B\,\Big|_{s = 0}.
    \end{align*}
    Next, we focus on
    \begin{align*}
        \frac{d}{ds}B\,\Big|_{s = 0} &= \int_{B(0,1)} \frac{\partial \omega_{ij}}{\partial y^p}\frac{\partial u^p}{\partial x^q}\frac{d \phi^q_s}{ds}\Big|_{s = 0} \nablap u^i \nabla u^j dx^1\wedge dx^2\\
        &\quad  + \int_{B(0,1)}\omega_{ij}(u)\left(\frac{\partial^2 u^i}{\partial x^p \partial x^1}\frac{d \phi^p_s}{ds}\Big|_{s=0}\frac{\partial u^j}{\partial x^2} + \frac{\partial u^i}{\partial x^k}\frac{\partial X^k}{\partial z^1}\frac{\partial u^j}{\partial x^2}+\frac{\partial u^i}{\partial x^1}\frac{\partial^2 u^j}{\partial x^p \partial x^2}\frac{d\phi^p_s}{ds}\Big|_{s= 0}\right.\\
        &\quad + \frac{\partial u^i}{\partial x^1}\frac{\partial u^j}{\partial x^l}\frac{\partial X^l}{\partial z^2} - \frac{\partial^2 u^j}{\partial x^p \partial x^1}\frac{d \phi^p_s}{ds}\Big|_{s=0}\frac{\partial u^i}{\partial x^2} - \frac{\partial u^j}{\partial x^k}\frac{\partial X^k}{\partial z^1}\frac{\partial u^i}{\partial x^2}\\
        &\quad -\left.\frac{\partial u^j}{\partial x^1}\frac{\partial^2 u^i}{\partial x^p \partial x^2}\frac{d\phi^p_s}{ds}\Big|_{s = 0} - \frac{\partial u^j}{\partial x^1}\frac{\partial u^i}{\partial x^l}\frac{\partial X^l}{\partial z^2} \right) dx^1\wedge dx^2\\
        &=\int_{B(0,1)} \frac{\partial \omega_{ij}}{\partial y^p}\frac{\partial u^p}{\partial x^q}X^q \nablap u^i \nabla u^j dx^1\wedge dx^2\\
        &\quad  + \int_{B(0,1)}\omega_{ij}(u)\left(\frac{\partial^2 u^i}{\partial x^p \partial x^1}X^p\frac{\partial u^j}{\partial x^2} + \frac{\partial u^i}{\partial x^k}\frac{\partial X^k}{\partial z^1}\frac{\partial u^j}{\partial x^2}+\frac{\partial u^i}{\partial x^1}\frac{\partial^2 u^j}{\partial x^p \partial x^2}X^p \right.\\
        &\quad  + \frac{\partial u^i}{\partial x^1}\frac{\partial u^j}{\partial x^l}\frac{\partial X^l}{\partial z^2}-\frac{\partial^2 u^j}{\partial x^p \partial x^1}X^p\frac{\partial u^i}{\partial x^2} - \frac{\partial u^j}{\partial x^k}\frac{\partial X^k}{\partial z^1}\frac{\partial u^i}{\partial x^2}\\
        &\quad \left.-\frac{\partial u^j}{\partial x^1}\frac{\partial^2 u^i}{\partial x^p \partial x^2}X^p - \frac{\partial u^j}{\partial x^1}\frac{\partial u^i}{\partial x^l}\frac{\partial X^l}{\partial z^2}\right)dx^1\wedge dx^2\\
        &: = C + D
    \end{align*}
    Rearranging terms in $D$ and integrating by parts yields
    \begin{align*}
        D &= -\int_{B(0,1)} \frac{\partial u^i}{\partial x^1}\frac{\partial}{\partial x^p}\p{\omega_{ij}(u)X^p\frac{\partial u^j}{\partial x^2}}dx \\
        &\quad + \int_{B(0,1)}\omega_{ij}(u)\p{\frac{\partial u^i}{\partial x^1}\frac{\partial u^j}{\partial x^2} - \frac{\partial u^j}{\partial x^1}\frac{\partial u^i}{\partial x^2}}\frac{\partial X^1}{\partial z^1}dx\\
        &\quad + \int_{B(0,1)}\omega_{ij}(u)\p{\frac{\partial u^i}{\partial x^1}\frac{\partial u^j}{\partial x^2} - \frac{\partial u^j}{\partial x^1}\frac{\partial u^i}{\partial x^2}}\frac{\partial X^2}{\partial z^2}dx + \int_{B(0,1)}\omega_{ij}(u)\frac{\partial^2 u^j}{\partial x^p\partial x^2}X^p\frac{\partial u^i}{\partial x^1}dx\\
        &\quad + \int_{B(0,1)} \frac{\partial u^j}{\partial x^1}\frac{\partial}{\partial x^p}\p{\omega(u)_{ij}X^p\frac{\partial u^i}{\partial x^2}}dx - \int_{B(0,1)}\omega_{ij}(u)\frac{\partial^2 u^i}{\partial x^p \partial x^2}X^p\frac{\partial u^j}{\partial x^1}dx\\
        &=  \int_{B(0,1)}\omega_{ij}(u)\nablap u^i \nabla u^j\diver(X)dx - \int_{B(0,1)}\frac{\partial\omega_{ij}}{\partial x^p}\frac{\partial u^i}{\partial x^1}\frac{\partial u^j}{\partial x^2} X^pdx\\
        &\quad - \int_{B(0,1)}\omega_{ij}(u)\frac{\partial u^i}{\partial x^1}\frac{\partial u^j}{\partial x^2} \diver(X)dx - \int_{B(0,1)}\omega_{ij}(u)\frac{\partial u^i}{\partial x^1}\frac{\partial^2 u^j}{\partial x^p \partial x^2}X^p dx\\
        &\quad + \int_{B(0,1)}\omega_{ij}(u)\frac{\partial^2 u^j}{\partial x^p\partial x^2}X^p\frac{\partial u^i}{\partial x^1}dx + \int_{B(0,1)}\frac{\partial\omega_{ij}}{\partial x^p}\frac{\partial u^i}{\partial x^2}\frac{\partial u^j}{\partial x^1} X^p dx\\
        &\quad +\int_{B(0,1)}\omega_{ij}(u)\frac{\partial u^i}{\partial x^2}\frac{\partial u^j}{\partial x^1} \diver(X)dx + \int_{B(0,1)}\omega_{ij}(u)\frac{\partial u^j}{\partial x^1}\frac{\partial^2 u^i}{\partial x^p \partial x^2}X^p dx \\
        &\quad - \int_{B(0,1)}\omega_{ij}(u)\frac{\partial^2 u^i}{\partial x^p \partial x^2}X^p\frac{\partial u^j}{\partial x^1}dx\\
        & = \int_{B(0,1)}\omega_{ij}(u)\nablap u^i \nabla u^j\diver(X)dx - \int_{B(0,1)}\omega_{ij}(u)\nablap u^i \nabla u^j\diver(X)dx\\
        &\quad + \int_{B(0,1)}\frac{\partial \omega_{ij}}{\partial x^p}\p{\frac{\partial u^i}{\partial x^2}\frac{\partial u^j}{\partial x^1} - \frac{\partial u^i}{\partial x^1}\frac{\partial u^j}{\partial x^2}}dx\\
        &= - \int_{B(0,1)}\frac{\partial \omega_{ij}}{\partial x^p} X^p \nablap u^i\nabla u^j dx
        \end{align*}
        which implies
        \begin{equation*}
            \frac{dB}{ds}\Big|_{s = 0} = 0.
        \end{equation*}
        Now if $u_\alpha$ is the critical point of $\dbl{E}^\omega_\alpha$,  for any vector field $X$ supported in unite disk $B(0,1)$ we have
        \begin{align}\label{E 21}
            2\alpha\sum_i\int_{B(0,1)}&\p{\tau_\alpha  + \abs{\nabla_{g_\alpha}u}^2}^{\alpha - 1}\langle du(\nabla_{e_i} X), du(e_i)\rangle dV_{g_\alpha}\nonumber \\
            &= \int_{B(0,1)} \p{\tau_\alpha + \abs{\nabla_{g_\alpha}u}^2}^\alpha \diver(X)dV_{V_{g_\alpha}}
        \end{align}
        To obtain \eqref{pohozaev esti 2}, we choose a vector field $X$ supported in $B_\rho$ by
        \begin{equation*}
            X = \eta(r)r\frac{\partial}{\partial r} = \eta(|x|)x^i\frac{\partial}{\partial x^i}
        \end{equation*}
        where $\eta(r)$ is defined by
        \begin{equation*}
           \eta(r) =  \left\{
            \begin{aligned}
            &1&\quad\quad &\text{if}&\,\, r\leq t^\prime,\\
            &\frac{t - r}{t - t^\prime}&\quad\quad &\text{if}&\,\, t^\prime \leq r \leq t,\\
            &0&\quad\quad &\text{if}&\,\, r\geq t,
            \end{aligned}
            \right.\quad \quad \text{for }\,0 < t^\prime < t\leq \rho < 1.
        \end{equation*}
        Plugging this vector field into \eqref{E 21}, we obtain
        \begin{align}\label{E 22}
            0= & (2 \alpha-2) \int_{B(0,t)} \eta\left(\tau_\alpha+\left|\nabla_{g_\alpha} u_\alpha\right|^2\right)^{\alpha-1}\left|\nabla_0 u_\alpha\right|^2 d x\nonumber\\
            & +\int_{B(0,t)} O(|x|)\left(\tau_\alpha+\left|\nabla_{g_\alpha} u_\alpha\right|^2\right)^{\alpha-1}\left|\nabla_0 u_\alpha\right|^2 d x \nonumber\\
            & -2 \tau_\alpha \int_{B(0,t)} \eta\left(\tau_\alpha+\left|\nabla_{g_\alpha} u_\alpha\right|^2\right)^{\alpha-1} d V_{g_\alpha}\nonumber\\
            &+\frac{\tau_\alpha}{t-t^{\prime}} \int_{B(0,t) \backslash B_{t^{\prime}}} r\left(\tau_\alpha+\left|\nabla_{g_\alpha} u\right|^2\right)^{\alpha-1} d V_{g_\alpha} \nonumber\\
            & +\frac{1}{t-t^{\prime}} \int_{B(0,t) \backslash B_{t^{\prime}}}\left(\tau_\alpha+\left|\nabla_{g_\alpha} u_\alpha\right|^2\right)^{\alpha-1}\left[\left|\nabla_0 u_\alpha\right|^2 r-2 \alpha r\left|\frac{\partial u_\alpha}{\partial r}\right|^2\right] d x \nonumber\\
            & -\int_{B(0,t)} \tau_\alpha\left(\tau_\alpha+\left|\nabla_{g_\alpha} u_\alpha\right|^2\right)^{\alpha-1} r \eta \frac{\partial \varphi}{\partial r} d V_{g_\alpha}
        \end{align}
    In equation \eqref{E 22} taking $t^\prime \nearrow t$ yields estimation \eqref{pohozaev esti 3}
    \begin{align}\label{E 23}
        &\int_{\partial B(0,t)} \p{\tau_\alpha + \abs{\nabla_{g_\alpha} u_\alpha}^2}^{\alpha - 1}\abs{\frac{\partial u_\alpha}{\partial r}}^2 ds\nonumber\\
        &\quad - \frac{1}{2\alpha} \int_{\partial B(0,t)}\p{\tau_\alpha + \abs{\nabla_{g_\alpha}u_\alpha}^2}^{\alpha - 1}\abs{\nabla u_\alpha}^2ds\nonumber\\
    &\quad=\p{1 - \frac{1}{\alpha}}\frac{1}{t}\int_{B(0,t)} \p{\tau_\alpha + \abs{\nabla_{g_\alpha}u_\alpha}^2}^{\alpha - 1} \abs{\nabla u_\alpha}^2 dx + O(t).
    \end{align}
    where we used co-area formula and Lemma \ref{bounded F}. Since under the polar coordinates the metric tensor can be written as 
    \begin{equation*}
        g_\alpha = e^{\varphi_\alpha}\p{dr^2 + r^2 d\theta^2},
    \end{equation*}
    hence
    \begin{equation}\label{eq:polar energy}
        \abs{\nabla u_\alpha}^2 = \abs{\frac{\partial u_\alpha}{\partial r}}^2 + \frac{1}{|x|^2}\abs{\frac{\partial u_\alpha}{\partial \theta}}^2.
    \end{equation}
    Therefore, \eqref{pohozaev esti 2} is obtained from above observation and \eqref{E 23}.
\end{proof}
Next, compared with previous Lemma \ref{lem:pohozaev} we proceed to derive an alternative form of the Pohozave-type identity, which directly connects the angular component of the energy function with the radial component of the energy functional.
\begin{lemma}\label{lem:second pohozaev}
    Let $(B(0,1),g_\alpha)$ be the unit disk in $\R^2$ with metric 
    $$g_\alpha = e^{\varphi_\alpha}\big((dx^1)^2 + (dx^2)^2\big)$$
    where $\varphi_\alpha \in C^\infty(B(0,1))$ and $\varphi_\alpha(0) = 0$ for $\alpha > 1$. If $u_\alpha$ is a $\alpha$-$H$-surface being a critical point of $\dbl{E}_\alpha^\omega(u, B(0,1))$, then for any $0 < t < 1$ the following holds
    \begin{align}
         \label{eq:another pohozaev}
        \int_{\partial B(0,t)} \left(\abs{\frac{\partial u_\alpha}{\partial r}}^2\right. &- \left.\frac{1}{r^2}\abs{\frac{\partial u_\alpha}{\partial \theta}}^2\right)ds\nonumber\\
        &= -\frac{2(\alpha - 1)}{t}\int_{B(0,t)}\frac{\nabla \abs{\nabla_{g_\alpha}u_\alpha}^2\nabla u_\alpha}{\tau_\alpha + \abs{\nabla_{g_\alpha}u_\alpha}^2} r \frac{\partial u_\alpha}{\partial r} dx.
    \end{align}
\end{lemma}
\begin{proof}
    Multiplying the Euler Lagrange equation \eqref{el general 1 equi} by $r \frac{\partial u_\alpha}{\partial r}$ written as polar coordinate of $B(0,1)$ and integrating over $B(0,t)$ to yield
    \begin{align}\label{eq:another pohoz eq1} 
        \int_{B(0,t)} r \frac{\partial u_\alpha}{\partial r}\Delta u_\alpha  dx = &- (\alpha - 1)\int_{B(0,t)}\frac{\nabla|\nabla_{g_\alpha} u_\alpha|^2\cdot \nabla u_\alpha}{\tau_\alpha+|\nabla_{g_\alpha} u_\alpha|^2} r \frac{\partial u_\alpha}{\partial r} dx \nonumber\\
        &+ \tau_\alpha^{\alpha -1}\int_{B(0,t)}\frac{H(u_\alpha)(\nablap u_\alpha, \nabla u_\alpha)}{\alpha \left(\tau_\alpha + |\nabla_{g_\alpha} u_\alpha|^2\right)^{\alpha - 1}} r \frac{\partial u_\alpha}{\partial r}dx.
    \end{align}
    Integration by parts to lefthand integral of \eqref{eq:another pohoz eq1} gets
    \begin{equation}\label{eq:another pohoz eq2}
        \int_{B(0,t)} r \frac{\partial u_\alpha}{\partial r}\Delta u_\alpha  dx = \int_{\partial B(0,t)} t \abs{\frac{\partial u_\alpha}{\partial r}}^2 ds - \int_{B(0,t)} \nabla\p{r \frac{\partial u_\alpha}{\partial r}}\cdot \nabla u_\alpha dx.
    \end{equation}
    The second integral of righthand of \eqref{eq:another pohoz eq2} can be further computed as
    \begin{align}\label{eq:another pohoz eq3}
        \int_{B(0,t)} \nabla\p{r \frac{\partial u_\alpha}{\partial r}}\cdot \nabla u_\alpha dx &= \sum_{i =1}^2 \int_{B(0,t)} \nabla\p{x^i \frac{\partial u_\alpha}{\partial x^i}}\cdot \nabla u_\alpha dx\nonumber\\
        &= \int_{B(0,1)} \abs{\nabla u_\alpha}^2 dx + \int_{B(0,1)} \frac{r}{2} \frac{\partial \p{\abs{\nabla u_\alpha}^2}}{\partial r} dx\nonumber\\
        &= \int_{B(0,1)} \abs{\nabla u_\alpha}^2 dx + \frac{t}{2} \int_{\partial B(0,t)} \abs{\nabla u_\alpha}^2 ds -  \int_{B(0,1)} \abs{\nabla u_\alpha}^2 dx\nonumber\\
        & = \frac{t}{2} \int_{\partial B(0,t)} \abs{\nabla u_\alpha}^2 ds.
    \end{align}
    On the other hand, we take a polar coordinate transformation, letting
     \begin{equation*}
         \frac{\partial u_\alpha}{\partial x^1} = \cos \theta \frac{\partial u_\alpha}{\partial r}  - \frac{\sin\theta}{r} \frac{\partial u_\alpha}{\partial \theta} \quad \text{and}\quad \frac{\partial u_\alpha}{\partial x^2} = \sin \theta \frac{\partial u_\alpha}{\partial r}  + \frac{\cos\theta}{r} \frac{\partial u_\alpha}{\partial \theta},
     \end{equation*}
     we can rewrite the mean curvature type vector term in \eqref{eq:another pohoz eq1} as 
     \begin{align}\label{eq:another pohoz eq4}
         &\int_{B(0,t)}\frac{H(u_\alpha)(\nablap u_\alpha, \nabla u_\alpha)}{\alpha \left(\tau_\alpha + |\nabla_{g_\alpha} u_\alpha|^2\right)^{\alpha - 1}} r \frac{\partial u_\alpha}{\partial r}dx\nonumber\\
         &= \int_{B(0,t)}\frac{H(u_\alpha)(\nablap u_\alpha, \nabla u_\alpha)\cdot(x \nabla u)}{\alpha \left(\tau_\alpha + |\nabla_{g_\alpha} u_\alpha|^2\right)^{\alpha - 1}} \nonumber\\
         & = \sum_{i,j,k = 1}^{K}\int_{B(0,t)} \frac{1}{\alpha \left(\tau_\alpha + |\nabla_{g_\alpha} u_\alpha|^2\right)^{\alpha - 1}} H^{k}_{ij} \frac{\partial u_\alpha^k}{\partial r}\p{\frac{\partial u^i}{\partial r} \frac{\partial u^j}{\partial \theta} - \frac{\partial u^j}{\partial r} \frac{\partial u^i}{\partial \theta}} dx
     \end{align}
     The antisymmetric of $H^{k}_{ij}$ in indices $i$, $j$ and $k$, see \eqref{eq:H anti-symmetric}, tells us that the above quantity vanishes identically. Hence, combining \eqref{eq:another pohoz eq2}, \eqref{eq:another pohoz eq3} and \eqref{eq:another pohoz eq4} with \eqref{eq:another pohoz eq1}, we have
     \begin{equation}\label{eq:another pohozaev 2}
         \int_{\partial B(0,t)}  \abs{\frac{\partial u_\alpha}{\partial r}}^2 - \frac{1}{2}\abs{\nabla u_\alpha}^2 ds = - \frac{\alpha - 1}{t} \int_{B(0,t)}\frac{\nabla \abs{\nabla_{g_\alpha}u_\alpha}^2\nabla u_\alpha}{\tau_\alpha + \abs{\nabla_{g_\alpha}u_\alpha}^2} r\frac{\partial u_\alpha}{\partial r} dx
     \end{equation}
     which leads to \eqref{eq:another pohozaev} keeping in mind that 
     $$|\nabla u|^2 = \abs{\frac{\partial u}{\partial r}}^2 + \frac{1}{|x|^2} \abs{\frac{\partial u}{\partial \theta}}^2.$$
\end{proof}

		\subsection{Proof of Generalized Energy Identity --- Theorem \ref{generalized ide} }\label{subsection" generalized ide}
  \
  \vskip5pt

In this subsection, our goal is to establish  the generalized energy identity for sequences of $\alpha$-$H$-surfaces (being the critical points of $\dbl{E}_\alpha^\omega$) with uniformly bounded generalized $\alpha$-energy. We will adapt the approach outlined by Ding-Tian \cite{DingTian1995} in showing the energy identity for a sequence of approximate harmonic maps with uniformly $L^2$-norm bounded tension field and Li-Wang \cite{li2010} for sequences of $\alpha$-harmonic maps. And it is important to emphasize the significance of the Pohozaev identity \eqref{pohozaev esti 2} and \eqref{pohozaev esti 3} in the proof.
	    To prove Theorem \ref{generalizd energy}, it is sufficient to focus on the simpler case of a single blow-up point, stated as below:
		\begin{theorem}\label{simple energy ide}
			Let $(B(0,1),g_\alpha)\subset \R^2$ be the unit disk in $\R^2$ equipped with sequence of conformal metric $g_\alpha = e^{\varphi_\alpha(x)}\big((dx^1)^2 + (dx^2)^2\big)$ and $g = e^{\varphi(x)}\big((dx^1)^2 + (dx^2)^2\big)$ where $\varphi_\alpha \in C^\infty(B(0,1))$, $\varphi_\alpha(0) = 0$ for $\alpha > 1$ and $\varphi_\alpha \rightarrow \varphi$ strongly in $C^\infty(\overline{B(0,1)})$ as $\alpha \searrow 1$. Let $u_\alpha \in C^\infty(B(0,1), N)$ be a sequence of $\alpha$-$H$-surfaces satisfying
			\begin{enumerate}[label=$(\mathrm{\alph*})$]
				\item $\sup_{\alpha > 1}\dbl{E}_\alpha(u_\alpha) \leq \Lambda < \infty$ and $0 < \beta_0 \leq \lim_{\alpha \searrow 1} \tau_\alpha^{\alpha - 1} \leq 1$,
				\item $u_\alpha \rightarrow u_0$ strongly in $C^\infty_{loc}\p{B(0,1)\backslash\{0\}, \R^K}$ as $\alpha \searrow 1$.
			\end{enumerate}
			Then there exists a subsequence of $u_\alpha$ still denoted by $u_\alpha$ and a nonnegative integer $n_0$ such that for any $i = 1, \dots , n_0$ there exists a sequence of points $x^i_\alpha$, positives number $\lambda^i_\alpha$ and a non-trivial $H$-sphere $w^i$ such that all following statements hold:
			\begin{enumerate}[label=$(\mathrm{\arabic*})$]
				\item $x^i_\alpha \rightarrow 0\,$ and $\,\lambda^i_\alpha \rightarrow 0$, as $\alpha \searrow 1$;
				\item $\lim_{\alpha \searrow 1}\left(\frac{r^i_\alpha}{r^j_\alpha} + \frac{r^j_\alpha}{r^i_\alpha} + \frac{|x^i_\alpha - x^j_\alpha|}{r^i_\alpha + r^j_\alpha}\right) = \infty \quad \text{for any } i\neq j;$
				\item $w^i$ is the weak limit of $u_\alpha(x_\alpha^i + \lambda^i_\alpha x)$ in $W^{1,2}_{loc}(\R^2)$
				\item \textbf{Generalized Energy Identity}:
				\begin{equation}
					\lim_{\delta \searrow 0}\lim_{\alpha \searrow 1} \dbl{E}_\alpha(u_\alpha, B(0,\delta)) = \sum_{i = 1}^{n_0} \mu_i^2E(w^i)
				\end{equation}
				where $\mu_i = \lim_{\alpha \searrow 1} (\lambda^i_\alpha)^{2- 2\alpha}$.
			\end{enumerate}
			
		\end{theorem}

  \subsubsection{Proof of Single Bubble Case for Theorem \ref{simple energy ide}}\label{subsubsection:one bubble}
  \ 
  \vskip5pt
        At the first step, we will commence by establishing Theorem \ref{simple energy ide} under the assumption of a single bubble, i.e., when $n_0 = 1$. The proof for the scenarios involving multiple bubbles will be presented in the subsequent section.
        
        To begin, since $0 \in B(0,1)$ is the only energy concentration point as stated in Theorem \ref{simple energy ide}, we can assume the only bubble $w$ is produced by sequence
		\begin{equation}\label{eq:defi of w alpha}
		w := \lim_{\alpha \searrow 1} w_\alpha(x) = \lim_{\alpha \searrow 1} u_\alpha(x_\alpha + \lambda_\alpha x) 
		\end{equation}
		where 
  $$\lambda_\alpha = \frac{1}{\max_{B(0, 1/2)}|\nabla_{g_\alpha} u_\alpha|}$$
  and $x_\alpha$ is the point where the maximum is taken on, that is, 
		$$
		|\nabla_{g_\alpha}u_\alpha(x_\alpha)| = \max_{x \in B(0, 1/2)}|\nabla_{g_\alpha}u_\alpha(x)|.
		$$
  Then, $\norm{\nabla w_\alpha}_{L^\infty} = \abs{\nabla w_\alpha(0)} =1$. So, by Lemma \ref{lem4.1}, $w_\alpha$ converges strongly in $C^\infty_{loc}(\R^2)$ to a non-trivial $H$-surface $w$ from $\R^2$ to $N$. Moreover, by Lemma \ref{blow 3} and identifying $\R^2 \cap \set{\infty} \cong \S^2$, $w$ actually extends to a $H$-sphere. 
  \begin{rmk}\label{remark,neck reduction}
      Note that for any $0 < \delta < \frac{1}{2}$, $u_\alpha$ converges to weak limit $u_0$ strongly in $C^\infty(B(0,1)\backslash B(0,\delta), N)$. For any large enough $R > 0$ and small enough $\alpha -1$ such that $\lambda_\alpha R < \delta$, we have $w_\alpha$ converges to $H$-sphere $w$ strongly in $C^\infty(B(0,R), N)$. Therefore, the workload to prove the our main Theorem \ref{simple energy ide} reduces to study to asymptotic formulation of $u_\alpha$ on the neck domain $B(0,\delta)\backslash B(x_\alpha, \lambda_\alpha R)$ when $\alpha \searrow 1$.
  \end{rmk}
  
		Under the one bubble hypothesis and by a  similar argument in Ding-Tian \cite{DingTian1995}, we claim the following:
  \claim\label{claim:ding-tian} For any $\varepsilon > 0$ there exists $\delta > 0$ and $R > 0$ such that
		\begin{equation}\label{energy on neck}
			\int_{B(x_\alpha,2t)\backslash B(x_\alpha,t)} \abs{\nabla_{g_\alpha} u_\alpha}^2 dV_{g_\alpha} \leq \varepsilon \quad \text{for any }t \in \p{\frac{\dbl{\lambda}_\alpha R}{2}, 2\delta}
		\end{equation}
		when $\alpha - 1$ is small enough.
        \begin{proof}[\textbf{Proof of Claim \ref{claim:ding-tian}}]
        We argue by contradiction, suppose that the Claim \ref{claim:ding-tian} fails, then there exists a sequence of $\alpha_j \searrow 1$ and $\lambda_{\alpha_j}^\prime \searrow 0$ satisfying $\lambda_{\alpha_j}^\prime/\lambda_{\alpha_j} \rightarrow \infty$ such that 
        \begin{equation}\label{eq:condtradiction}
        \int_{B(x_\alpha,2\lambda_{\alpha_j}^\prime)\backslash B(x_\alpha,\lambda_{\alpha_j}^\prime)} \abs{\nabla_{g_\alpha} u_\alpha}^2 dV_{g_\alpha} \geq \varepsilon
        \end{equation}
        Then, the rescaled map  $$w^\prime_{\alpha_j}(x) := u_{\alpha_j}(\lambda_{\alpha_j}^\prime x + x_{\alpha_j}) \quad \text{converges strongly in } C^\infty(\R^2\backslash\set{0,x^1,\dots,x^m}, N)$$ to some $H$-surface $w^\prime$ for some energy concentration set $\set{0,x^1,\dots,x^m}$.
        
        If $m = 0$, by \eqref{eq:condtradiction} and the fact that $\lambda_{\alpha_j}^\prime/\lambda_{\alpha_j} \rightarrow \infty$, we can conclude that $w^\prime$ is a non-constant $H$-surface that is different from $w$. This contradicts to the only one bubble assumption.

        If $m \geq 1$, then similarly to previous argument around some energy concentration point $x^i$ for $w_\alpha$ after passing to some subsequence we can also construct a sequence $\dbl{x}_{\alpha_j} \searrow x^i$ and a sequence $\dbl{\lambda}_{\alpha_j}$ such that the second rescaled map $w_{\alpha_j}^\prime(\dbl{x}_{\alpha_j} + \dbl{\lambda}_{\alpha_jx})$ converges to a non-trivial $H$-sphere $\dbl{w}^i$. Hence
        \begin{equation*}
            w_{\alpha_j}^\prime\p{\dbl{x}_{\alpha_j} + \dbl{\lambda}_{\alpha_j}x} = u_{\alpha_j}\p{x_{\alpha_j} + \lambda_{\alpha_j}^\prime\p{\dbl{x}_{\alpha_j} + \dbl{\lambda}_{\alpha_j}x}} \rightarrow \dbl{w}^i \quad \text{as } j \rightarrow \infty
        \end{equation*}
        which means $\dbl{w}^i$ is also a second non-constant $H$-sphere which contradicts to the only one bubble assumption again. In a word, the Claim \ref{claim:ding-tian} always holds.
        \end{proof}

		In the sequel, for any $0 < a < b <\infty$ and $x\in B(0,1)$ we will use the notation
		\begin{equation*}
			A(a,b,x): = \left\{ y \in\R^2\,:\,a \leq |y - x| \leq b \right\}
		\end{equation*}
		to denote the annulus centered at $x$ with inner radii $a$ and outer radii $b$.
		By small energy regularity, Lemma \ref{lem4.1}, we have
		\begin{lemma}\label{small regu N}
			With the same hypothesis as Theorem \ref{simple energy ide} and assume there is only one bubble for $u_\alpha$ as $\alpha\searrow 1$. Given any small enough $\delta > 0$, large enough $R > 0$ and small enough such that $\alpha -1 \leq \alpha_0 -1$ where $\alpha_0 - 1$ is chosen in Lemma \ref{lem4.1}. Then for any $\lambda_\alpha R < a < b \leq \delta$, we have
			\begin{equation}\label{E 25}
				\int_{A(a,b,x_\alpha)} \abs{\nabla^2_{g_\alpha} u_\alpha}\cdot|x - x_\alpha| \cdot \abs{\nabla_{g_\alpha} u_\alpha} dV_{g_\alpha} \leq C \int_{A(\frac{a}{2}, 2b, x_\alpha)} \abs{\nabla_{g_\alpha} u_\alpha}^2 dV_{g_\alpha}.
			\end{equation}
			where $C$ is a constant independent of $\alpha$ as $\alpha \searrow 1$.
		\end{lemma}
		\begin{proof}
            Since both the lefthand and righthand of \eqref{E 25} is conformally invariant, it suffices to show 
            \begin{equation}\label{E 25 conformal}
				\int_{A(a,b,x_\alpha)} \abs{\nabla^2 u_\alpha}\cdot|x - x_\alpha| \cdot \abs{\nabla u_\alpha} dx \leq C \int_{A(\frac{a}{2}, 2b, x_\alpha)} \abs{\nabla u_\alpha}^2 dx.
			\end{equation}
			Without loss of generality, we can assume $b = 2^Ia$ for some positive integer $I \in \mathbb N$.  Let $f : \R \times \S^1 \rightarrow \R^2$ be the mapping such that $r = e^{-\rho}$ and $\theta = \varphi$ where $\R\times \S^1$ equipped with product metric $f^*g = d\rho^2 + d\varphi^2$. Here, we use notation $(r,\theta)$ to represent the polar coordinates centered at $x_\alpha$ and $(\rho,\varphi)$ be coordinate of cylinder $\R\times \S^1$.
Then $f$ is a conformal map from $\R \times \S^1$ into $\R^2$. Let $v_\alpha(t,\varphi) = u_\alpha(f)(t,\varphi) = u_\alpha(e^{-t},\varphi)$ which satisfies Euler Lagrange equation \eqref{el general 1 equi} and hence fulfills Lemma \ref{lem4.1}. Thus after conformal transformation $f$ 
    the annulus $A(2^{i-1}a, 2^i a, x_\alpha)$ maps to be 
    \begin{equation}\label{eq:cylinder}
        \left[i - 1 + \log \frac{1}{a}, i + \log \frac{1}{a}\right]\times \S^1.
    \end{equation}

 Since we have assumed there is only one bubble,  in view of \eqref{energy on neck} we can apply small energy regularity Lemma \ref{lem4.1} to $v_\alpha$ on \eqref{eq:cylinder} and transform the estimates back to $u_\alpha$ to derive 
	\begin{equation}\label{eq:small regu N 1}
		|\nabla u_\alpha (x)|\cdot\abs{x -x_\alpha} \leq C\p{\int_{A(2^{i - 2}a, 2^{i + 1}a, x_\alpha)} \abs{\nabla u_\alpha}^2 dx}^{\frac{1}{2}},\quad 1\leq i \leq I,
	\end{equation}
 for any $x \in A(2^{i-1}a, 2^i a, x_\alpha)$ and constant $C$ which is independent of $1\leq i \leq I$ and $\alpha$ as $\alpha \searrow 1$. And similarly, we have
			\begin{equation}\label{eq:small regu N 2}
				|\nabla^2 u_\alpha (x)|\cdot \abs{x -x_\alpha}^2 \leq C\p{\int_{A(2^{i - 2}a, 2^{i + 1}a, x_\alpha)} \abs{\nabla u_\alpha}^2 dx}^{\frac{1}{2}},\quad 1\leq i \leq I.
			\end{equation}
			
			Then, combining \eqref{eq:small regu N 1} and \eqref{eq:small regu N 2} and taking summation with respect to $i$ from $1$ to $I$, we can get
			\begin{align}
				&\hspace{-5mm}\int_{A(a,b,x_\alpha)} \abs{\nabla^2 u_\alpha}\cdot \abs{x - x_\alpha}\cdot \abs{\nabla u_\alpha} dx \\
				&= \sum_{i = 1}^I\int_{A(2^{i -1}a , 2^ia, x_\alpha)} \abs{\nabla^2 u_\alpha}\cdot\abs{x - x_\alpha}^2 \cdot\abs{\nabla u_\alpha}\cdot\abs{x - x_\alpha}\frac{dx}{\abs{x - x_\alpha}^2}\nonumber\\
				&\leq \sum_{i = 1}^I\sup_{x\in A(2^{i-1}a, 2^ia, x_\alpha)}\abs{\nabla^2 u_\alpha}\cdot\abs{x - x_\alpha}^2\cdot\nonumber\\
				& \quad \quad \quad \sup_{x\in A(2^{i-1}a, 2^ia, x_\alpha)}\abs{\nabla u_\alpha}\cdot\abs{x - x_\alpha}\int_{A(2^{i-1}a, 2^ia, x_\alpha)}\frac{dx}{\abs{x - x_\alpha}^2}\nonumber\\
				&\leq C\sum_{i = 1}^I \p{\int_{A(2^{i - 2}a, 2^{i + 1}a, x_\alpha)} \abs{\nabla u_\alpha}^2 dx}^{\frac{1}{2}}\nonumber\\
				& = C \int_{A(\frac{a}{2}, 2b, x_\alpha)} \abs{\nabla u_\alpha}^2 dx.
			\end{align}
   This completes the proof of Lemma \ref{small regu N}.
		\end{proof}
		In the polar coordinate, the energy functional has two components, namely, the radical part and the angular part. 
		\begin{equation*}
			\int \abs{\nabla u_\alpha}^2 dx = \int \abs{\frac{\partial u_\alpha}{\partial r}}^2 dx + \int\frac{1}{|x|^2}\abs{\frac{\partial u_\alpha}{\partial \theta}}^2 dx.
		\end{equation*}
		To show the generalized energy identity stated in Theorem \ref{simple energy ide}, considering the Remark \ref{remark,neck reduction} we first establish the following energy decay of the angular component:
		\begin{lemma}\label{decay theta}
			With the same assumption as Lemma \ref{small regu N}, there holds
			\begin{equation*}
				\lim_{\delta\searrow 0}\lim_{R\rightarrow\infty}\lim_{\alpha\searrow 1} \int_{A(\lambda_\alpha R, \delta, x_\alpha)} \frac{1}{|x - x_\alpha|^2}\abs{\frac{\partial u_\alpha}{\partial \theta}}^2 dx = 0.
			\end{equation*}
			Here,  we always use the same $(r,\theta)$ to represent the polar coordinate systems centered at $x_\alpha$ as $\alpha\searrow 1$.
		\end{lemma}
		\begin{proof}
			Combining the deduction \eqref{energy on neck} of one bubble assumption and small energy regularity Lemma \ref{lem4.1} with the conformal transformation argument described in the proof of Lemma \ref{small regu N}, we have
			\begin{align}
			    \label{E 24}
				Osc_{A(t,2t,x_\alpha)}(u_\alpha) &:=\sup_{x,y \in A(t,2t,x_\alpha) } \abs{u_\alpha(x) - u_\alpha(y)}\nonumber\\
    &\leq C \norm{\nabla u_\alpha}_{L^2(A(\frac{t}{2},4t,x_\alpha))}\quad \text{for any } t\in (\lambda_\alpha R, \delta).
			\end{align}
			
			Let
			\begin{equation*}
				u^*_\alpha(t) := \frac{1}{2\pi t}\int_{\partial B(x_\alpha,t)}u_\alpha = \frac{1}{2\pi}\int_{0}^{2\pi} u_\alpha(x_\alpha + te^{i\theta}) d\theta
			\end{equation*}
			Then, using \eqref{E 24} we have 
			\begin{align}\label{E 29}
				\norm{u_\alpha(x) - u^*_\alpha(|x|)}_{L^\infty(A(\lambda_\alpha R, \delta, x_\alpha))}&\leq \sup_{\lambda_\alpha R \leq t\leq \delta} \norm{u_\alpha(x) - u^*_\alpha(|x|)}_{L^\infty(A(t, 2t, x_\alpha))}\nonumber\\
				&\leq \sup_{\lambda_\alpha R \leq t\leq \delta}Osc_{A(t,2t,x_\alpha)}(u_\alpha) \nonumber\\
				&\leq C \norm{\nabla u_\alpha}_{L^2(A(\frac{t}{2},4t,x_\alpha))} \leq C\varepsilon\quad \text{for any } t\in (\lambda_\alpha R, \delta).
			\end{align}
			Next, by integration by part and recall the Euler Lagrange equation \eqref{el general 1 equi} of $\alpha$-$H$-surface  we can estimate the energy on neck domain
			\begin{align}\label{E 26}
				&\int_{A(\lambda_\alpha R, \delta, x_\alpha)} \abs{\nabla u_\alpha}^2 dx \nonumber\\
                &= \int_{A(\lambda_\alpha R, \delta, x_\alpha)} \nabla u_\alpha\cdot \nabla\p{u_\alpha - u^*_\alpha}dx + \int_{A(\lambda_\alpha R, \delta, x_\alpha)} \nabla u_\alpha\cdot \nabla u^*_\alpha dx\nonumber\\
				& = - \int_{A(\lambda_\alpha, \delta, x_\alpha)} \Delta u_\alpha\cdot\p{u_\alpha - u^*_\alpha}dx + \int_{\partial A(\lambda_\alpha, \delta, x_\alpha)} \frac{\partial u_\alpha}{\partial r}\cdot\p{u_\alpha - u^*_\alpha} ds\nonumber\\
				&\quad + \int_{A(\lambda_\alpha R, \delta, x_\alpha)} \nabla u_\alpha \cdot\nabla u^*_\alpha dx\nonumber\\
				& = \int_{A(\lambda_\alpha R, \delta, x_\alpha)} A(u_\alpha)(\nabla u_\alpha, \nabla u_\alpha)\cdot\p{u_\alpha - u^*_\alpha}dx\nonumber\\
				&\quad + (\alpha - 1)\int_{A(\lambda_\alpha R, \delta, x_\alpha)}\frac{\nabla|\nabla_{g_\alpha} u_\alpha|^2\cdot \nabla u_\alpha}{\tau_\alpha+|\nabla_{g_\alpha} u_\alpha|^2} \cdot\p{u_\alpha - u^*_\alpha}dx\nonumber\\
				&\quad + \tau_\alpha^{\alpha - 1} \int_{A(\lambda_\alpha R, \delta, x_\alpha)} \frac{H(u_\alpha)(\nablap u_\alpha, \nabla u_\alpha)}{\alpha \left(\tau_\alpha + |\nabla_{g_\alpha} u_\alpha|^2\right)^{\alpha - 1}} \p{u_\alpha - u^*_\alpha}dx\nonumber\\
				&\quad + \int_{\partial A(\lambda_\alpha R, \delta, x_\alpha)} \frac{\partial u_\alpha}{\partial r}\cdot\p{u_\alpha - u^*_\alpha} ds + \int_{A(\lambda_\alpha R, \delta, x_\alpha)} \nabla u_\alpha\cdot \nabla u^*_\alpha dx.
			\end{align}
			In the following, we estimate every terms obtained in last equation of \eqref{E 26}.  For the last  integral of \eqref{E 26}, by Jensen's inequality we see that
			\begin{align}\label{E 27}
				\int_{A(\lambda_\alpha R, \delta, x_\alpha)}& \nabla u_\alpha \cdot \nabla u^*_\alpha dx\nonumber\\
				&= \int_{A(\lambda_\alpha R, \delta, x_\alpha)}{\frac{\partial u_\alpha}{\partial r}\cdot\frac{\partial u^*_\alpha}{\partial r}} dx \nonumber\\
				&\leq \p{\int_{A(\lambda_\alpha R, \delta, x_\alpha)}\abs{\frac{\partial u_\alpha}{\partial r}}^2dx}^{\frac{1}{2}} \p{\int_{A(\lambda_\alpha R, \delta, x_\alpha)}\abs{\frac{\partial u^*_\alpha}{\partial r}}^2dx}^{\frac{1}{2}}\nonumber\\
				& = \p{\int_{A(\lambda_\alpha R, \delta, x_\alpha)}\abs{\frac{\partial u_\alpha}{\partial r}}^2dx}^{\frac{1}{2}}\p{\int_{A(\lambda_\alpha R, \delta, x_\alpha)}\abs{\frac{1}{2\pi}\int_0^{2\pi}\frac{\partial u_\alpha}{\partial r} d\theta}^2 dx}^{\frac{1}{2}}\nonumber\\
				&\leq \int_{A(\lambda_\alpha R, \delta, x_\alpha)}\abs{\frac{\partial u_\alpha}{\partial r}}^2dx\nonumber\\
				&= \int_{A(\lambda_\alpha R, \delta, x_\alpha)}\abs{\nabla u_\alpha}^2dx - \int_{A(\lambda_\alpha R, \delta, x_\alpha)}\frac{1}{|x - x_\alpha|^2}\abs{\frac{\partial u_\alpha}{\partial \theta}}^2dx.
			\end{align}
			Next, for the boundary term in \eqref{E 26}  using trace theorem in Sobolev spaces and \eqref{E 29} we have
			\begin{align}\label{E 30}
				\int_{\partial A(\lambda_\alpha, \delta, x_\alpha)} \frac{\partial u_\alpha}{\partial r}\p{u_\alpha - u^*_\alpha} ds &\leq C \varepsilon \p{\int_{\partial A(\lambda_\alpha, \delta, x_\alpha)} \abs{\nabla u_\alpha}^2}^{\frac{1}{2}}\nonumber\\
				&\leq C\varepsilon \left(\norm{\nabla u_\alpha}_{L^2\p{A\p{\frac{1}{2}\lambda_\alpha R, 2\lambda_\alpha R, x_\alpha}\bigcup A\p{\frac{1}{2}\delta, 2\delta, x_\alpha}}}\right.\nonumber \\
				&\quad +\left.\norm{|x - x_\alpha|\cdot\nabla^2 u_\alpha}_{L^2\p{A\p{\frac{1}{2}\lambda_\alpha R, 2\lambda_\alpha R, x_\alpha}\bigcup A\p{\frac{1}{2}\delta, 2\delta, x_\alpha}}}\right)\nonumber\\
				&\leq C\varepsilon \norm{\nabla u_\alpha}_{L^2\p{A\p{\frac{1}{2}\lambda_\alpha R, 2\lambda_\alpha R, x_\alpha}\bigcup A\p{\frac{1}{2}\delta, 2\delta, x_\alpha}}} \leq  C\varepsilon
			\end{align}
			where the last inequality is obtained from the small energy regularity, Lemma \ref{lem4.1}.
			Furthermore, for the second integral in \eqref{E 26}, by Lemma \ref{small regu N}, we can estimate
			\begin{align}\label{E 28}
				&\hspace{-10mm}(\alpha - 1)\abs{\int_{A(\lambda_\alpha R, \delta, x_\alpha)}\frac{\nabla|\nabla_{g_\alpha} u_\alpha|^2\cdot \nabla u_\alpha}{\tau_\alpha+|\nabla_{g_\alpha} u_\alpha|^2} \p{u_\alpha - u^*_\alpha}dx}\nonumber\\
				&\leq 2(\alpha - 1)C\int_{A(\lambda_\alpha R, \delta, x_\alpha)} \abs{\nabla^2 u_\alpha}\cdot |x - x_\alpha|\cdot|\nabla u_\alpha| dx\nonumber\\
				&\leq 2(\alpha - 1)C \int_{A(\frac{\lambda_\alpha R}{2}, 2\delta, x_\alpha)} |\nabla u_\alpha|^2 dx.
			\end{align}
			At last, by \eqref{E 29} it is easy to see that 
			\begin{align}\label{E 31}
				&\int_{A(\lambda_\alpha R, \delta, x_\alpha)} A(u_\alpha)(\nabla u_\alpha, \nabla u_\alpha)\p{u_\alpha - u^*_\alpha}dx \nonumber\\
				&\quad + \tau_\alpha^{\alpha - 1}\int_{A(\lambda_\alpha R, \delta, x_\alpha)} \frac{H(u_\alpha)(\nablap u_\alpha, \nabla u_\alpha)}{\alpha \left(\tau_\alpha + |\nabla_{g_\alpha} u_\alpha|^2\right)^{\alpha - 1}} \cdot\p{u_\alpha - u^*_\alpha}dx\nonumber\\
				&\leq C\varepsilon \int_{A(\lambda_\alpha R, \delta, x_\alpha)} |\nabla u_\alpha|^2 dx \leq C\varepsilon
			\end{align}
			Plugging the estimates \eqref{E 27}, \eqref{E 30}, \eqref{E 28} and \eqref{E 31} to \eqref{E 26}, we can obtain
			\begin{equation*}
				\int_{A(\lambda_\alpha R, \delta, x_\alpha)}\frac{1}{|x - x_\alpha|^2}\abs{\frac{\partial u_\alpha}{\partial \theta}}^2dx \leq C\varepsilon + C (\alpha - 1) \int_{A(\frac{\lambda_\alpha R}{2}, 2\delta, x_\alpha)} |\nabla u_\alpha|^2 dx
			\end{equation*}
			which yields
			\begin{equation*}
				\lim_{\delta\searrow 0}\lim_{R\rightarrow\infty}\lim_{\alpha\searrow 1} \int_{A(\lambda_\alpha R, \delta, x_\alpha)} \frac{1}{|x - x_\alpha|^2}\abs{\frac{\partial u_\alpha}{\partial \theta}}^2 dx = 0.
			\end{equation*}
            Hence, we complete the proof of Lemma \ref{decay theta}.
		\end{proof}
		As a corollary of Lemma \ref{decay theta} and the uniformly boundedness of 
		\begin{equation*}
			\limsup_{\alpha \searrow 1} \norm{\p{\tau_\alpha + \abs{\nabla_{g_\alpha} u_\alpha}^2}^{\alpha - 1}}_{C^0(B(0,1))} \leq \beta_0 < \infty,
		\end{equation*}
		see Lemma \ref{bounded F}, we have
		\begin{coro}\label{coro 1}
			With the same hypothesis as Lemma \ref{decay theta}, there holds
			\begin{equation}
				\lim_{\delta\searrow 0}\lim_{R\rightarrow\infty}\lim_{\alpha\searrow 1} \int_{A(\lambda_\alpha R, \delta, x_\alpha)}\p{\tau_\alpha + \abs{\nabla_{g_\alpha} u_\alpha}^2}^{\alpha - 1} \frac{1}{|x - x_\alpha|^2}\abs{\frac{\partial u_\alpha}{\partial \theta}}^2 dx = 0.
			\end{equation}
		\end{coro}
        To simplify the notation, considering the Pohozaev identity \eqref{pohozaev esti 2}, for $0 < t< 1$ we define the quantities
        \begin{equation}
            \mathcal{E}_\alpha(t) = \int_{B(x_\alpha, \lambda_\alpha^t)} \p{\tau_\alpha + |\nabla_{g_\alpha} u_\alpha|^2}^{\alpha - 1} \abs{\nabla u_\alpha}^2 dx.
        \end{equation}
        Besides, for fixed $0 < t_0 < 1$ and $0<t < t_0 < 1$ we define
        \begin{equation}
            \mathcal{E}_{r,t_0,\alpha}(t) = \int_{A(\lambda_\alpha^{t_0}, \lambda_\alpha^t, x_\alpha)} \p{\tau_\alpha + |\nabla_{g_\alpha} u_\alpha|^2}^{\alpha - 1} \abs{\frac{\partial u_\alpha}{\partial r}}^2 dx
        \end{equation}
        and 
        \begin{equation}
            \mathcal{E}_{\theta,t_0,\alpha}(t) = \int_{A(\lambda_\alpha^{t_0}, \lambda_\alpha^t, x_\alpha)} \p{\tau_\alpha + |\nabla_{g_\alpha} u_\alpha|^2}^{\alpha - 1} \frac{1}{|x - x_\alpha|^2}\abs{\frac{\partial u_\alpha}{\partial \theta}}^2 dx.
        \end{equation}
        Therefore, for $t \in (0,t_0)$ the Pohozaev identity \eqref{pohozaev esti 2} can be rewritten as
        \begin{equation}
            \label{eq:pohozaev esti equiv}
            \p{1 - \frac{1}{2\alpha}}\mathcal{E}^\prime_{r,t_0, \alpha}(t) - \frac{1}{2\alpha}\mathcal{E}^\prime_{\theta,t_0, \alpha}(t) = \p{1 - \frac{1}{\alpha}}\log\p{\lambda_\alpha} \mathcal{E}_{\alpha}(t) + O(\lambda_\alpha^t \log\p{\lambda_\alpha}).
        \end{equation}
        Integrating with respect to $t$ to get
        \begin{align}\label{eq:pohozaev esti equiv 2}
            \p{1 - \frac{1}{2\alpha}}\mathcal{E}_{r,t_0, \alpha}(t) &- \frac{1}{2\alpha}\mathcal{E}_{\theta,t_0, \alpha}(t) \nonumber\\
                &= \frac{1}{2} \int_{t_0}^t \p{\frac{1}{\alpha}\log\p{\lambda_\alpha^{2(\alpha - 1)}} \mathcal{E}_{\alpha}(s) + O\p{\lambda_\alpha^s \log\p{\lambda_\alpha}}} ds
        \end{align}
        On the one hand, since the generalized $\alpha$-energy $\dbl{E}_\alpha$ is uniformly bounded and using \eqref{eq:pohozaev esti equiv} and \eqref{eq:pohozaev esti equiv 2} one can see that 
        \begin{equation*}
            \norm{\p{1 -\frac{1}{2\alpha}}\mathcal{E}_{r,t_0, \alpha}(\cdot) - \frac{1}{2\alpha} \mathcal{E}_{\theta,t_0, \alpha}(\cdot)}_{C^1([\tau, t_0])} 
        \end{equation*}
        is uniformly bounded for any $0 < \tau < t_0/2$. On the other hand, by Lemma \ref{decay theta}
        \begin{equation*}
             \norm{ \mathcal{E}_{\theta,t_0, \alpha}(\cdot)}_{C^1(\delta, t_0)} \rightarrow 0 \quad \text{as } \alpha \searrow 1  
        \end{equation*}
        for any $\tau > 0$. Therefore, we can conclude that the sequences
        \begin{equation*}
            \set{\mathcal{E}_\alpha(t)}_{\alpha \searrow 1}, \quad  \set{\mathcal{E}_{r,t_0, \alpha}(t)}_{\alpha \searrow 1}\quad \text{and} \quad \set{\mathcal{E}_{\theta,t_0, \alpha}(t)}_{\alpha \searrow 1}
        \end{equation*}
        are  compact in $C^0([\tau, t_0])$ norm for any $0 < \tau < t_0/2$, which implies there exists functions $\mathcal{E}: (0,t_0) \rightarrow \R_+$ and $\mathcal{E}_{r,t_0}: (0,t_0) \rightarrow \R_+$ such that for any $\tau > 0$
        \begin{equation*}
            \mathcal{E}_\alpha \rightarrow \mathcal{E} \quad \text{and}\quad \mathcal{E}_{r,t_0,\alpha} \rightarrow \mathcal{E}_{r,t_0} \quad \text{in } C^0([\tau,t_0]) \text{ as } \alpha \searrow 1.
        \end{equation*}
		Based on Lemma \ref{decay theta} and the construction above quantities, we can establish the following
		\begin{lemma}\label{lambda^s small}
			With the same hypothesis as Theorem \ref{simple energy ide} , for any $t\in(0,1)$, there holds
			\begin{equation}\label{eq:lambda^s small 1}
				\lim_{\alpha \searrow 1} \int_{B(x_\alpha, \lambda_\alpha^t)}\p{\tau_\alpha + \abs{\nabla_{g_\alpha} u_\alpha}^2}^{\alpha - 1}|\nabla_{g_\alpha} u_\alpha|^2 dV_{g_\alpha} = \mu^{1 - t}\Lambda
			\end{equation}
   where 
   \begin{align}
       \Lambda :&= \lim_{R\rightarrow \infty}\lim_{\alpha \searrow 1}\int_{B(x_\alpha, \lambda_\alpha R)} \abs{\nabla_{g_\alpha} u_\alpha}^{2\alpha} dV_{g_\alpha}\nonumber\\
       &= \lim_{R\rightarrow \infty}\lim_{\alpha \searrow 1} \int_{B(0,R)} \abs{\nabla_{g_\alpha} w_\alpha}^{2\alpha} \lambda_\alpha^{2 -2\alpha} dV_{g_\alpha} = \mu E(w)
   \end{align}
		\end{lemma}
    \begin{proof}
        To begin, we decompose the integral of \eqref{lambda^s small} as below
        \begin{align}\label{eq:lambda^s small 1.5}
             \int_{B(x_\alpha, \lambda_\alpha^t)}\p{\tau_\alpha + \abs{\nabla_{g_\alpha} u_\alpha}^2}^{\alpha - 1}|\nabla_{g_\alpha} u_\alpha|^2 dV_{g_\alpha} &=  \int_{B(x_\alpha, \lambda_\alpha^t)}\p{\tau_\alpha + \abs{\nabla_{g_\alpha} u_\alpha}^2}^{\alpha - 1}|\nabla u_\alpha|^2 dx\nonumber\\
             & = \mathcal{E}_{r,t_0,\alpha}(t) + \mathcal{E}_{\theta,t_0,\alpha}(t) + \mathcal{E}_{\alpha}(t_0)
        \end{align}
        By Corollary \ref{coro 1}, we know that 
        \begin{equation}\label{eq:lambda^s small 2}
            \lim_{\alpha \searrow 1} \mathcal{E}_{\theta,t_0,\alpha}(t) = 0
        \end{equation}
        for any $ 0< t_0 < 1$. Then, in \eqref{eq:pohozaev esti equiv 2} letting $\alpha \searrow 1$ we see that 
        \begin{equation}\label{eq:lambda^s small 3}
            \lim_{\alpha \searrow 1}\mathcal{E}_{r,t_0,\alpha}(t) = \mathcal{E}_{r,t_0}(t) = -\int_{t_0}^{t} \log\p{\mu} \mathcal{E}(s) ds.
        \end{equation}
        Recalling that $\mathcal{E}_{\alpha}(t) - \mathcal{E}_{\alpha}(t_0) = \mathcal{E}_{r,t_0,\alpha} - \mathcal{E}_{\theta,t_0,\alpha}$ and \eqref{eq:lambda^s small 2}, letting $\alpha \searrow 1$ we have $\mathcal{E}(t) - \mathcal{E}(t_0) = \mathcal{E}_{r,t_0}.$ Plugging this identity into \eqref{eq:lambda^s small 3}, we have that 
        \begin{equation}
            \mathcal{E}_{r,t_0}(t) = -\log\p{\mu} \int_{t_0}^{t}  \mathcal{E}_{r,t_0}(s) + \mathcal{E}(t_0) ds.
        \end{equation}
        Solving this integral equation, we see that 
        \begin{equation}\label{eq:lambda^s small 4}
            \mathcal{E}_{r,t_0}(t) = \mu^{t_0 - t}\mathcal{E}(t_0) - \mathcal{E}(t_0).
        \end{equation}
        Now, taking $\alpha \searrow 1$ in \eqref{eq:lambda^s small 1.5} and utilizing \eqref{eq:lambda^s small 4} we have 
        \begin{equation}
            \lim_{\alpha \searrow 1}\int_{B(x_\alpha, \lambda_\alpha^t)}\p{\tau_\alpha + \abs{\nabla_{g_\alpha} u_\alpha}^2}^{\alpha - 1}|\nabla u_\alpha|^2 dx = \mu^{t_0- t} \mathcal{E}(t_0)
        \end{equation}
        Therefore, to show \eqref{eq:lambda^s small 1} it suffices to prove
        \begin{equation}\label{eq:midle lambda^s}
            \lim_{t_0 \rightarrow 1} \mathcal{E}(t_0) = \Lambda.
        \end{equation}
        To see this, integrating the Pohozaev identity \eqref{pohozaev esti 2} with respect to $t$ from $\lambda_\alpha R$ to $\lambda_\alpha^{t_0}$, we get
        \begin{align}\label{eq:lambda^s small 5}
            0 \leq \mathcal{E}_\alpha(t_0) - &\int_{B(x_\alpha, \lambda_\alpha R)} \p{\tau_\alpha + \abs{\nabla_{g_\alpha} u_\alpha}^2}^{\alpha - 1} \abs{\nabla u_\alpha}^2 dx\nonumber\\
             &\leq C\int_{A(\lambda_\alpha R ,\lambda_\alpha^{t_0},x_\alpha)}\p{\tau_\alpha + \abs{\nabla_{g_\alpha} u_\alpha}^2}^{\alpha - 1} \frac{1}{|x - x_\alpha|^2}\abs{\frac{\partial u_\alpha}{\partial \theta}}^2 dx\nonumber\\
             &\quad +C \int_{\lambda_\alpha R}^{\lambda_\alpha^{t_0}} \frac{\alpha - 1}{r} dr + C(\lambda_\alpha^{t_0} - \lambda_\alpha R).
        \end{align}
        Note that for the second integral on the righthand of \eqref{eq:lambda^s small 5} we have 
        \begin{equation*}
            \lim_{t_0 \rightarrow 1}\lim_{R \rightarrow \infty}\lim_{\alpha \searrow 1}\int_{\lambda_\alpha R}^{\lambda_\alpha^{t_0}} \frac{\alpha - 1}{r} dr = \lim_{t_0 \rightarrow 1} \frac{1 - t_0}{2} \log\p{\mu} = 0.
        \end{equation*}
        Moreover, keeping in mind that \eqref{eq:lambda^s small 2} and the definition of $w_\alpha$, see \eqref{eq:defi of w alpha}, we see that 
        \begin{align*}
            \lim_{t_0 \rightarrow 1} F(t_0) &= \lim_{t_0 \rightarrow 1}\lim_{R \rightarrow \infty}\lim_{\alpha \searrow 1}  \int_{B(x_\alpha, \lambda_\alpha R)} \p{\tau_\alpha + \abs{\nabla_{g_\alpha} u_\alpha}^2}^{\alpha - 1} \abs{\nabla u_\alpha}^2 dx\nonumber\\
            &= \lim_{R \rightarrow \infty}\lim_{\alpha \searrow 1} \int_{B(0, R)} \p{\lambda_\alpha^2\tau_\alpha + \abs{\nabla_{g_\alpha} w_\alpha}^2}^{\alpha - 1} \lambda_\alpha^{2 - 2\alpha} \abs{\nabla w_\alpha}^2 dx\\
            &= \mu E(w) = \Lambda.
        \end{align*}
        Therefore, we prove the \eqref{eq:midle lambda^s} and complete the proof of Lemma \ref{lambda^s small}.
    \end{proof} 
		Now we are in a position to prove the generalized energy identity---Theorem \ref{simple energy ide} when there is only one bubble during the blow-up procedure.
		\begin{proof}[\textbf{Proof of Theorem \ref{simple energy ide} under One Bubble Assumption}]
			Integrating the Pohozaev identity \eqref{pohozaev esti 3}  with respect to $t$ over the interval $[\lambda^t_\alpha, \delta]$ for some $t \in (0,1)$, we have
			\begin{align}\label{eq:proof of 1bubble 1}
				\int_{A(\lambda_\alpha^t, \delta, x_\alpha)}&\p{\tau_\alpha + \abs{\nabla_{g_\alpha} u_\alpha}^2}^{\alpha - 1}|\nabla u_\alpha|^2 dx\nonumber \\
				&\leq C \int_{A(\lambda_\alpha^s, \delta, x_\alpha)}\p{\tau_\alpha + \abs{\nabla_{g_\alpha} u_\alpha}^2}^{\alpha - 1}\frac{1}{|x - x_\alpha|^2}\abs{\frac{\partial u_\alpha}{\partial \theta}}^2 dx\nonumber \\
				&\quad + C\int_{\lambda_\alpha^s}^\delta \frac{\alpha - 1}{r}dr + C(\delta - \lambda_\alpha^s).
			\end{align}
			By Corollary \ref{coro 1}, we have
			\begin{equation}\label{eq:proof of 1bubble 2}
				\lim_{\delta\searrow 0}\lim_{s\rightarrow 0}\lim_{\alpha \searrow 1}\int_{A(\lambda_\alpha^s, \delta, x_\alpha)}\p{\tau_\alpha + \abs{\nabla_{g_\alpha} u_\alpha}^2}^{\alpha - 1}\frac{1}{|x - x_\alpha|^2}\abs{\frac{\partial u_\alpha}{\partial \theta}}^2 dx = 0
			\end{equation}
			Moreover, by direct computations we have
			\begin{equation}\label{eq:proof of 1bubble 3}
				\lim_{\delta\searrow 0}\lim_{s\rightarrow 0}\lim_{\alpha \searrow 1}\int_{\lambda_\alpha^s}^\delta \frac{\alpha - 1}{r}dr = \lim_{s\rightarrow 0}\frac{s}{2}\log \mu = 0
			\end{equation}
			Therefore, plugging \eqref{eq:proof of 1bubble 2} and \eqref{eq:proof of 1bubble 3} into \eqref{eq:proof of 1bubble 1} we have
			\begin{equation}
				\lim_{\delta\searrow 0}\lim_{s\rightarrow 0}\lim_{\alpha \searrow 1}\int_{A(\lambda_\alpha^s, \delta, x_\alpha)}\p{\tau_\alpha + \abs{\nabla_{g_\alpha} u_\alpha}^2}^{\alpha - 1}|\nabla u_\alpha|^2 dx = 0.
			\end{equation}
			Moreover, utilizing Lemma \ref{lambda^s small} we know that
			\begin{equation*}
				\lim_{s\rightarrow 0}\lim_{\alpha \searrow 1} \int_{B(x_\alpha,\lambda_\alpha^s)}\p{\tau_\alpha + \abs{\nabla_{g_\alpha} u_\alpha}^2}^{\alpha - 1}|\nabla u_\alpha|^2 dx = \mu\Lambda = \mu^2 E(w)
			\end{equation*}
			In the end, by Lemma \ref{bounded F}, we can conclude that 
			\begin{equation*}
				\lim_{\delta\searrow 0}\lim_{\alpha \searrow 1}\int_{B(x_\alpha, \delta)}\p{\tau_\alpha + \abs{\nabla_{g_\alpha} u_\alpha}^2}^{\alpha}dV_{g_\alpha} = \mu^2 E(w).
			\end{equation*}
			which completes the proof of Theorem \ref{simple energy ide} when $n_0 = 1$.
		\end{proof}
        \subsubsection{Proof of General Case for Theorem \ref{simple energy ide}.}
        \ 
        \vskip5pt
        In this subsubsection, we employ an induction argument on the number of bubbles $n_0$ to complete the proof of Theorem \ref{simple energy ide}.
        \begin{proof}[\textbf{Proof of the General Case for Theorem \ref{simple energy ide}}]
            Since we have proved the Theorem \ref{simple energy ide} when $n_0 = 1$ in Subsubsection \ref{subsubsection:one bubble}, now suppose that the generalized energy identity asserted in Theorem \ref{simple energy ide} holds when there are $n_0 - 1$ many bubbles for sequence $u_\alpha$ as $\alpha \searrow 1$.

            Firstly, recall that the first bubble $w^1$ for $\alpha$-$H$-surfaces $u_\alpha$ are constructed by sequence
		\begin{equation}
		w^1 := \lim_{\alpha \searrow 1} w^1_\alpha(x) = \lim_{\alpha \searrow 1} u_\alpha(x_\alpha^1 + \lambda^1_\alpha x) 
		\end{equation}
		where 
  $$\lambda^1_\alpha = \frac{1}{\max_{B(0, 1/2)}|\nabla_{g_\alpha} u_\alpha|}$$
  and $x_\alpha^1$ is the point where the maximum is taken on, that is, 
		$$
		|\nabla_{g_\alpha}u_\alpha(x_\alpha)| = \max_{x \in B(0, 1/2)}|\nabla_{g_\alpha}u_\alpha(x)|.
		$$
  Then, $\norm{\nabla w^1_\alpha}_{L^\infty(\R^2)} = \abs{\nabla w^1_\alpha(0)} =1$. So, by Lemma \ref{lem4.1}, $w^1_\alpha$ converges strongly in $C^\infty_{loc}(\R^2)$ to the first non-trivial $H$-sphere $w^1$ modulo a conformal transformation from $\R^2 \cup \set{\infty}$ onto $\S^2$ and removing the singularity $\infty$, see Lemma \ref{blow 3}.

  Then, similarly, we assume that the remaining $n_0-1$ many bubbles are produced by sequences
  \begin{equation*}
      w^i : = \lim_{\alpha \searrow 1}w^i_\alpha= \lim_{\alpha \searrow 1} u_\alpha(x_\alpha^i +\lambda_\alpha^i x) \quad \text{strongly in } C^\infty_{loc}(\R^2\backslash \mathfrak{S}^i)
  \end{equation*}
  for some sequences of points $x_\alpha^i \rightarrow 0$ and $\lambda_\alpha^i \rightarrow 0$ as $\alpha \searrow 1$ satisfying the alternative \ref{A1} or \ref{A2}, $2 \leq i \leq n_0$. Here, $\mathfrak{S}^i \subset \R^2$ are finite sets consisting of energy concentration points for sequences $w_\alpha^i$ as $\alpha \searrow 1$. By our choice of first bubble, we see that 
  \begin{equation*}
      \lambda_\alpha^1 = \min_{1\leq i \leq n_0}\set{\lambda_\alpha^1, \lambda_\alpha^2, \dots, \lambda_\alpha^{n_0}}.
  \end{equation*}
  For notation simplicity, we assume that 
  \begin{equation*}
      \lambda_\alpha^{n_0} = \max_{1\leq i \leq n_0}\set{\lambda_\alpha^1, \lambda_\alpha^2, \dots, \lambda_\alpha^{n_0}}
  \end{equation*}
  and define 
  \begin{equation*}
      \dbl{\lambda}_\alpha = \lambda_\alpha^{n_0} + \frac{\sum_{i = 1}^{n_0 - 1} \abs{x_\alpha^{n_0} - x^i_\alpha}}{n_0 -1}.
  \end{equation*}
  Thanks to the choice of $\dbl{\lambda}_\alpha$ and through a complete similar argument as Claim \ref{claim:ding-tian}, we also have 
  \claim\label{claim:ding-tian 2} For any $\varepsilon > 0$ there exists $\delta > 0$ and $R > 0$ such that
		\begin{equation}\label{energy on neck multi bubble}
			\int_{B(x_\alpha^1,2t)\backslash B(x_\alpha^1,t)} \abs{\nabla_{g_\alpha} u_\alpha}^2 dV_{g_\alpha} \leq \varepsilon \quad \text{for any }t \in \p{\frac{\dbl{\lambda}_\alpha R}{2}, 2\delta}
		\end{equation}
		when $\alpha - 1$ is small enough such that $\dbl{\lambda}_\alpha R \leq \delta$.

        Then, we can apply the argument of subsubsection \ref{subsubsection:one bubble} in proving the Theorem \ref{simple energy ide} to conclude  that
        \begin{align}\label{eq:multi bubble 1}
            \lim_{\alpha \searrow 1} \int_{B(x_\alpha^1,\delta)} \p{\tau_\alpha + \abs{\nabla_{g_\alpha} u_\alpha}^2}^\alpha dV_{g_\alpha} &= \lim_{R \rightarrow \infty}\lim_{\alpha \searrow 1} \int_{B(0,R)} \p{\tau_\alpha\dbl{\lambda}_\alpha^2 + \abs{\nabla_{\dbl{g}_\alpha} \dbl{u}_\alpha}^2}^\alpha \dbl{\lambda}_\alpha^{2 - 2\alpha} dV_{\dbl{g}_\alpha}\nonumber\\
            &\quad + \Vol(B(0, \delta)) \lim_{\alpha \searrow 1} \tau_\alpha + \int_{B(0, \delta)} \abs{\nabla u_0}^2 dx,
        \end{align}
        where $\dbl{u}_\alpha(x) :=u_\alpha(x_\alpha^1 + \dbl{\lambda}_\alpha x)$, 
        $$\dbl{g}_\alpha(x) = e^{\varphi_\alpha(x_\alpha^1 + \dbl{\lambda}_\alpha x)} \p{\p{dx^1}^2 + \p{dx^2}^2}$$
        and $u_0$ is the weak limit of $u_\alpha$. Note that as a corollary of Claim \ref{claim:ding-tian 2} there exists a large $R > 0$ such that all the energy concentration points $\dbl{\mathfrak{S}}$ of $\dbl{u}_\alpha$ belongs to $B(0, R)$. Then, $\dbl{u}_\alpha$ converges to some $H$-surface $\dbl{w}$ strongly in $C^\infty_{loc}(\R^2\backslash \dbl{\mathfrak{S}})$. Then, we proceed the induction argument conditionally depending on whether $\dbl{w}$ is trivial or not.
        \begin{itemize}
            \item $\dbl{w}$ is a non-constant $H$-sphere. Then, there must holds 
            \begin{equation*}
                \limsup_{\alpha \searrow 1} \frac{\sum_{i = 1}^{n_0 - 1} \abs{x_\alpha^{n_0}- x_\alpha^i}}{(n_0 - 1)\lambda_\alpha^{n_0}} < \infty
            \end{equation*}
            Otherwise, there will exists one more bubble which is distinct from $w^i$ for $1 \leq i \leq n_0$ around $0$ contradicting the assumption of Theorem \ref{simple energy ide}. Thus, we see that $\dbl{w}$ is exactly the $w^{n_0}$ after formulating a conformal transformation and hence
            \begin{equation*}
                \lim_{\alpha \searrow 1} \p{\dbl{\lambda}_\alpha}^{2 -2 \alpha} = \lim_{\alpha \searrow 1} \p{\lambda_\alpha^{n_0}}^{2 -2\alpha} = \mu_{n_0}.
            \end{equation*}
            Furthermore, by our choice of $\dbl{\lambda}_\alpha$ we know that the later one of the alternatives \ref{A1} and \ref{A2} must hold, that is, $\lambda_\alpha^i / \lambda_\alpha^{n_0} \rightarrow 0$ as $\alpha \searrow 0$ for any $1 \leq i \leq n_0 - 1$. Observe that 
            \begin{align*}
                \dbl{u}_\alpha\p{\dbl{x}_\alpha^i + \frac{\lambda_\alpha^i}{\lambda_\alpha^{n_0}}\p{x - \dbl{x}_\alpha^i}} = w_\alpha(x)\rightarrow w^i \text{ as } \alpha \searrow 1,
            \end{align*}
            where 
            \begin{equation*}
                \dbl{x}_\alpha^i = \frac{1}{\dbl{\lambda}_\alpha - \lambda_\alpha^i} \p{x_\alpha^i - x_\alpha^i}.
            \end{equation*}
            This means $w^1, \dots, w^{n_0 - 1}$ are exactly all bubbles of $\dbl{u}_\alpha$.
             Now, we consider the functional 
            \begin{equation}
            \dbl{E}^{\omega}_{\alpha, n_0}(w) = \frac{1}{2}\int_{B(0, R)} {\p{\tau_\alpha\dbl{\lambda}_\alpha^{2} + \abs{\nabla_{g_\alpha} w}^2}^{\alpha} } dV_{g_\alpha} + \tau^{\alpha - 1}_\alpha \dbl{\lambda}_\alpha^{2\alpha -2}\int_{B(0,R)} w^* \omega.
        \end{equation}
        And we can apply the induction assumption to this functional and for sequence $\set{\dbl{u}_\alpha}$ to get 
        \begin{align}
            \lim_{\alpha \searrow 1} \int_{B(0,R)}& \p{\tau_\alpha \dbl{\lambda}_\alpha^2 + \abs{\nabla_{g_\alpha} \dbl{u}_\alpha}^2}^{\alpha - 1} \abs{\nabla_{\dbl{g}_\alpha} \dbl{u}_\alpha}^2 dV_{\dbl{g}_\alpha}\nonumber\\
            &= E(\dbl{w}, B(0, R)) + \sum_{i  = 1}^{n_0 - 1} \lim_{\alpha \searrow 1} \p{\frac{\lambda_\alpha^i}{\dbl{\lambda}_\alpha}}^{4 - 4\alpha} E(w^i).
        \end{align}
        Since 
        \begin{align*}
            \dbl{\lambda}_\alpha^{2 - 2\alpha}\lim_{\alpha \searrow 1} \int_{B(0,R)}& \p{\tau_\alpha \dbl{\lambda}_\alpha^2 + \abs{\nabla_{g_\alpha} \dbl{u}_\alpha}^2}^{\alpha - 1} \abs{\nabla_{\dbl{g}_\alpha} \dbl{u}_\alpha}^2 dV_{\dbl{g}_\alpha}\\
            &= \int_{B(x_\alpha^1, \dbl{\lambda}_\alpha R)} \p{\tau_\alpha + \abs{ \nabla_{g_\alpha} u_\alpha}^2 }^{\alpha - 1} \abs{\nabla_{g_\alpha} u_\alpha}^2 dV_{g_\alpha},
        \end{align*}
         we can conclude that 
         \begin{align}\label{eq:multi bubble 2}
             \lim_{R \rightarrow \infty} \lim_{\alpha \searrow 1} \int_{B(x_\alpha^1, \dbl{\lambda}_\alpha R)} \p{\tau_\alpha + \abs{ \nabla_{g_\alpha} u_\alpha}^2 }^{\alpha - 1} \abs{\nabla_{g_\alpha} u_\alpha}^2 dV_{g_\alpha} = \mu_{n_0}^2 E(\dbl{w}) + \sum_{i = 1}^{n_0 - 1} \mu_{i}^2 E(w^i).
         \end{align}
         Combining \eqref{eq:multi bubble 1} and \eqref{eq:multi bubble 2}, we will get 
         \begin{equation*}
             \lim_{\alpha \searrow 1} \int_{B(x_\alpha,\delta)} \p{\tau_\alpha + \abs{\nabla_{g_\alpha} u_\alpha}^2}^\alpha dV_{g_\alpha} = \Vol(B(0,\delta))\lim_{\alpha \searrow 1} \tau_\alpha + E(u_0) + \sum_{i = 1}^{n_0} \mu_{i}^2 E(w^i).
         \end{equation*}
        \item $\dbl{w}$ is constant. Then, there are at least two distinct energy concentration points for sequence $\dbl{u}_\alpha$. This means at each energy concentration point there at most have $n_0 - 1$ many bubbles. And one can apply the induction assumption and utilize a similar argument as the previous case to conclude the desired generalized  energy identity stated in Theorem \ref{simple energy ide}.
        \end{itemize}
        Therefore, whether $\dbl{w}$ is trivial or not, both cases contribute to the completion of the proof for Theorem \ref{simple energy ide}.
        \end{proof}
		\subsection{Proof of Asymptotic Behavior on Necks---Theorem \ref{analysis on neck}}
	       \ 
        \vskip5pt
		In this subsection, we will examine the convergent behaviors of necks for sequence $u_\alpha$ as $\alpha \searrow 1$ and present the proof of our second main consequence, Theorem \ref{analysis on neck}. As in the previous Subsection \ref{subsection" generalized ide}, it suffices to consider the following simple case to prove the Theorem \ref{analysis on neck}.
		\begin{theorem}\label{simple neck analysis}
			Let $(B(0,1),g_\alpha)\subset \R^2$ be the unit disk equipped with sequence of conformal metric $g_\alpha = e^{\varphi_\alpha}\big((dx^1)^2 + (dx^2)^2\big)$ and $g = e^{\varphi}\big((dx^1)^2 + (dx^2)^2\big)$ where $\varphi_\alpha \in C^\infty(B(0,1))$, $\varphi_\alpha(0) = 0$ for $\alpha > 1$ and $\varphi_\alpha \rightarrow \varphi$ strongly in $C^\infty(\overline{B(0,1)})$ as $\alpha \searrow 1$. Let $u_\alpha \in C^\infty(B(0,1), N)$ be a sequence of $\alpha$-$H$-surfaces satisfying
			\begin{enumerate}[label=$(\mathrm{\alph*})$]
				\item $\sup_{\alpha > 1}\dbl{E}_\alpha(u_\alpha) \leq \Lambda < \infty$ and $0 < \beta_0 \leq \lim_{\alpha \searrow 1} \tau_\alpha^{\alpha - 1} \leq 1$,
				\item $u_\alpha \rightarrow u$ strongly in $C^\infty_{loc}\p{B(0,1)\backslash\{0\}, \R^K}$ as $\alpha \searrow 1$.
			\end{enumerate}
			We further assume there is only one bubble $w^1: \S^2 \rightarrow N$ around $0 \in B(0,1)$ for sequence $u_\alpha$ as $\alpha \searrow 1$. Let
			\begin{equation*}
				\nu^1 = \liminf_{\alpha \searrow 1} \p{\lambda^1_\alpha}^{-\sqrt{\alpha - 1}}.
			\end{equation*}
			Then there exists a subsequence of $u_\alpha$ still denoted by $u_\alpha$,  a sequence of points $x_\alpha$ and a sequence of positive numbers $\lambda_\alpha$  such that the following statements hold:
			\begin{enumerate}[label=(\subscript{N}{\arabic*})]
				\item\label{simple neck analysis item1} when $\nu^1 = 1$, the set $u_0\p{B^M(x_1,1)} \bigcup w^1(\S^2)$ is a connected subset of $N$ where $u_0$ is the weak limit of $u_\alpha$ in $W^{1,2}(M,N)$ as $\alpha\searrow 1$;
				\item\label{simple neck analysis item2} when $\nu^1\in (1,\infty)$, the set $u_0\p{B^M(x_1,1)}$ and $w^1(\S^2)$ are connected by a geodesic $\Gamma \subset N$ with length
				\begin{equation*}
					L(\Gamma) = \sqrt{\frac{E(w^1)}{\pi}} \log \nu^1;
				\end{equation*}
				\item\label{simple neck analysis item3} when $\nu^1 = \infty$, the neck contains at least an infinite length geodesic.
			\end{enumerate}
		\end{theorem}
		\subsubsection{No Neck Property for the case \texorpdfstring{$\nu = 1$}{Lg}}\label{subsubsection: no neck}
  \ 
  \vskip5pt 
        In this subsubsection, we focus on the case where $\nu = 1$ in Theorem \ref{simple neck analysis} to demonstrate that the base map and all bubbles are directly connected.
		\begin{proof}[\textbf{Proof Theorem \ref{simple neck analysis} when $\nu = 1$}]
			Considering the Remark \ref{remark,neck reduction}, similarly to the construction of first bubble described in Subsubsection \ref{subsubsection:one bubble}, let $x_\alpha \in B(0, \delta)$ be the maximum point of $\abs{\nabla_{g_\alpha} u_\alpha}$ on $\overline{B(0, \delta)}$. Since $0$ is the only blow-up point, there must have $\lim_{\alpha \searrow 1} x_\alpha = 0$. The first bubble $w$ for $\alpha$-$H$-surfaces $u_\alpha$ is constructed by sequence
		\begin{equation}
		w := \lim_{\alpha \searrow 1} w_\alpha(x) = \lim_{\alpha \searrow 1} u_\alpha(x_\alpha + \lambda_\alpha x).
		\end{equation}
        And $1 \leq \mu \leq \nu =1$,  Theorem \ref{generalized ide} tells us that the energy identity holds, that is,
			\begin{equation}\label{eq:energy ide 1}
				\lim_{\delta\searrow 0}\lim_{R\rightarrow\infty}\lim_{\alpha \searrow 1} \int_{A(\lambda_\alpha R, \delta, x_\alpha)}\abs{\nabla u_\alpha}^2 dx = 0.
			\end{equation}
			Without loss of generality, we can assume for each $\alpha > 1$ there is a positive integer $k_\alpha$ such that $\delta = 2^{k_\alpha}\lambda_\alpha R$. For $k = 1, \dots, k_\alpha -1$ and $0 \leq t \leq \min\{k_\alpha - k, k\}$, we define 
			\begin{equation*}
				Q(k,t) = A\p{2^{k-t}\lambda_\alpha R,2^{k+t}\lambda_\alpha R,x_\alpha},
			\end{equation*}
            and
            \begin{equation*}
                \mathcal{F}_{\alpha,k}(t) = \int_{Q(k,t)} \abs{\nabla u_\alpha}^2 dx.
            \end{equation*}
			By the same estimate techniques as in Lemma \ref{decay theta}, as a consequence of \eqref{E 26} we can obtain
			\begin{align}\label{E 32}
				\int_{Q(k,t)} \abs{\nabla u_\alpha}^2 dx & \leq C\varepsilon\int_{Q(k,t)}\abs{\nabla u_\alpha}^2dx+ C(\alpha - 1)\int_{Q(k,t+1)}\abs{\nabla u_\alpha}^2dx\nonumber\\
				&\quad + \int_{\partial Q(k,t)} \frac{\partial u_\alpha}{\partial r}\p{u_\alpha - u^*_\alpha} ds + \int_{Q(k,t)} \abs{\frac{\partial u_\alpha}{\partial r}}^2dx
			\end{align}
            for any small enough $\varepsilon > 0$ that will be determined later.
			Next, we want to utilizing Pohozaev identity \eqref{eq:another pohozaev} obtained in Lemma \ref{lem:second pohozaev} to control the term 
			\begin{equation*}
				\int_{Q(k,t)} \abs{\frac{\partial u_\alpha}{\partial r}}^2dx
			\end{equation*}
            occurring in righthand of \eqref{E 32}.
			Integrating 
			\begin{align*}
				\int_{\partial B(x_\alpha,s)} \left(\abs{\frac{\partial u_\alpha}{\partial r}}^2\right. &- \left.\frac{1}{|x - x_\alpha|^2}\abs{\frac{\partial u_\alpha}{\partial \theta}}^2\right)ds\nonumber\\
        &= -\frac{2(\alpha - 1)}{s}\int_{B(x_\alpha,s)}\frac{\nabla \abs{\nabla_{g_\alpha}u_\alpha}^2\nabla u_\alpha}{\tau_\alpha + \abs{\nabla_{g_\alpha}u_\alpha}^2}\cdot |x - x_\alpha|\cdot\frac{\partial u_\alpha}{\partial r} dx.
			\end{align*}
			with respect to $s$ from $2^{k-t}\lambda_\alpha R$ to $2^{k+ t}\lambda_\alpha R$, we can get 
			\begin{align}\label{E 33}
				\int_{Q(k,t)}& \left(\abs{\frac{\partial u_\alpha}{\partial r}}^2\right. - \left.\frac{1}{\abs{x - x_\alpha}^2}\abs{\frac{\partial u_\alpha}{\partial \theta}}^2\right)ds\nonumber\\
                &= - \int_{2^{k-t}\lambda_\alpha R}^{2^{k+ t}\lambda_\alpha R}\p{\frac{2(\alpha - 1)}{s}\int_{B(x_\alpha,s)}\frac{\nabla \abs{\nabla_{g_\alpha}u_\alpha}^2\nabla u_\alpha}{\tau_\alpha + \abs{\nabla_{g_\alpha}u_\alpha}^2}\cdot |x - x_\alpha|\cdot\frac{\partial u_\alpha}{\partial r} dx} ds\nonumber\\
                &\leq C \int_{2^{k-t}\lambda_\alpha R}^{2^{k+ t}\lambda_\alpha R}\frac{2(\alpha - 1)}{s} \p{\int_{A(\lambda_\alpha R, \delta, x_\alpha)} \abs{x - x_\alpha}\cdot\abs{\nabla_{g_\alpha}^2 u_\alpha}\cdot \abs{\frac{\partial u_\alpha}{\partial r}}dx}ds\nonumber\\
                &\quad + \int_{2^{k-t}\lambda_\alpha R}^{2^{k+ t}\lambda_\alpha R}\frac{2(\alpha - 1)}{s} \p{\int_{B(x_\alpha, \lambda_\alpha R)} \abs{x - x_\alpha}\cdot\abs{\nabla_{g_\alpha}^2 u_\alpha}\cdot \abs{\frac{\partial u_\alpha}{\partial r}}dx}ds
			\end{align}
            For the first integral in the righthand of \eqref{E 33}, utilizing the energy identity \eqref{eq:energy ide 1} and together with the small energy identity Lemma \ref{lem4.1}, when $\alpha - 1$ is small we have that 
            \begin{align}\label{eq:bad term 1}
                \int_{2^{k-t}\lambda_\alpha R}^{2^{k+ t}\lambda_\alpha R}\frac{2(\alpha - 1)}{s} &\p{\int_{A(\lambda_\alpha R, \delta, x_\alpha)} \abs{x - x_\alpha}\cdot\abs{\nabla_{g_\alpha}^2 u_\alpha}\cdot \abs{\frac{\partial u_\alpha}{\partial r}}dx}ds\nonumber\\
                &\leq \int_{2^{k-t}\lambda_\alpha R}^{2^{k+ t}\lambda_\alpha R}\frac{2(\alpha - 1)}{s}\p{\int_{A(\lambda_\alpha R/2, 2\delta, x_\alpha)} \abs{\nabla u_\alpha}^2 dx } ds\nonumber\\
                &\leq C(\alpha - 1)\int_{2^{k-t}\lambda_\alpha R}^{2^{k+ t}\lambda_\alpha R} \frac{1}{s}ds\leq C (\alpha - 1)t.
            \end{align}
            For the second integral in the righthand of \eqref{E 33}, we see that 
            \begin{align}\label{eq:bad term 2}
                \int_{2^{k-t}\lambda_\alpha R}^{2^{k+ t}\lambda_\alpha R}\frac{2(\alpha - 1)}{s} &\p{\int_{B(x_\alpha, \lambda_\alpha R)} \abs{x - x_\alpha}\cdot\abs{\nabla_{g_\alpha}^2 u_\alpha}\cdot \abs{\frac{\partial u_\alpha}{\partial r}}dx}ds\nonumber\\
                &= \int_{2^{k-t}\lambda_\alpha R}^{2^{k+ t}\lambda_\alpha R}\frac{2(\alpha - 1)}{s} \p{\int_{B(0, R)} \abs{x}\cdot\abs{\nabla_{g_\alpha}^2 w_\alpha}\cdot \abs{\nabla w_\alpha}dx}ds\nonumber\\
                &\leq C(\alpha - 1)\int_{2^{k-t}\lambda_\alpha R}^{2^{k+ t}\lambda_\alpha R} \frac{1}{s}ds\leq C (\alpha - 1)t.
            \end{align}
            Here, we used that the $\alpha$-energy of $u_\alpha$ is uniformly bounded in estimates \eqref{eq:bad term 1} and \eqref{eq:bad term 2}.
			Plugging \eqref{eq:bad term 1} and \eqref{eq:bad term 2} into \eqref{E 33} and keeping in mind that \eqref{eq:polar energy}, we obtain
			\begin{align}\label{E 34}
				\int_{Q(k,t)} \abs{\frac{\partial u_\alpha}{\partial r}}^2 dx\leq \frac{1}{2} \int_{Q(k,t)}\abs{\nabla u_\alpha}^2dx + C(\alpha - 1)t
			\end{align}
			for small enough $\alpha - 1$.
			Combining inequalities \eqref{E 32} and \eqref{E 34}, we obtain
			\begin{align}\label{E 36}
				\p{\frac{1}{2} - C \varepsilon}\int_{Q(k,t)}\abs{\nabla u_\alpha}^2dx &\leq  \int_{\partial Q(k,t)} \frac{\partial u_\alpha}{\partial r}\p{u_\alpha - u^*_\alpha} ds + C(\alpha - 1)(t + 1)
			\end{align}
            where we choose $\varepsilon > 0$ such that $C \varepsilon < 1/4$.
			Furthermore, for the boundary term in \eqref{E 36} we observe that 
			\begin{align}\label{E 35}
				\int_{\partial Q(k,t)}& \frac{\partial u_\alpha}{\partial r}\cdot\p{u_\alpha - u^*_\alpha} ds \nonumber\\
				&=  \int_{\partial Q(k,t)} \sqrt{|x - x_\alpha|}\cdot\frac{\partial u_\alpha}{\partial r}\cdot\p{u_\alpha - u^*_\alpha}\cdot\frac{1}{\sqrt{|x - x_\alpha|}} ds\nonumber\\
				&\leq \frac{1}{2}\p{\int_{\partial Q(k,t)} \abs{\frac{\partial u_\alpha}{\partial r}}^2\cdot |x - x_\alpha| ds + \int_{\partial Q(k,t)} \abs{u_\alpha - u^*_\alpha}^2\cdot \frac{1}{|x - x_\alpha|} ds }\nonumber\\
				&\leq \frac{1}{2}\p{\int_{\partial Q(k,t)} \abs{\frac{\partial u_\alpha}{\partial r}}^2\cdot |x - x_\alpha| ds + \int_{\partial Q(k,t)} \abs{\frac{\partial u_\alpha}{\partial \theta}}^2\cdot \frac{1}{|x - x_\alpha|} ds}\nonumber\\
				&\leq \frac{1}{2}\int_{\partial Q(k,t)} \abs{\nabla u_\alpha}^2\cdot |x - x_\alpha| ds\nonumber\\
				& = 2^{k + t -1}\lambda_\alpha R \int_{\partial B(x_\alpha, 2^{k + t}\lambda_\alpha R)} \abs{\nabla u_\alpha}^2 ds - 2^{k -t -1}\lambda_\alpha R \int_{\partial B(x_\alpha, 2^{k-t}\lambda_\alpha R)} \abs{\nabla u_\alpha}^2 ds.
			\end{align}
			Recall the definition of $\mathcal{F}_{\alpha,k}(t)$, \eqref{E 35} tells us that
			\begin{equation}\label{E 35p}
				\int_{\partial Q(k,t)} \frac{\partial u_\alpha}{\partial r}\p{u_\alpha - u^*_\alpha} ds \leq \frac{1}{2\log 2} \mathcal{F}^\prime_{\alpha,k}(t).
			\end{equation}
            Thus, plugging the \eqref{E 35p} into \eqref{E 36}  we get
			\begin{align*}
				(1 - C \varepsilon)\mathcal{F}_{\alpha,k}(t) &\leq \frac{1}{\log 2}\mathcal{F}^\prime_{\alpha,k}(t) +  C(\alpha - 1)(t + 1)
			\end{align*}
			Let $\sigma = 1 - C\varepsilon \in (0,1)$, multiplying the both side of above inequality by $2^{-\sigma t}$ yields
			\begin{equation*}
				\p{2^{-\sigma t} F_{\alpha,k}(t)}^\prime \geq - 2^{-\sigma t}C(\alpha -1)(t + 1).
			\end{equation*}
			Integrating from $2$ to $T$ for some $T \in \mathbb N$, we have
			\begin{equation*}
				\mathcal{F}_{\alpha, k}(2) \leq C 2^{-\sigma T}\mathcal{F}_{\alpha, k}(T) + C(\alpha - 1).
			\end{equation*}
			Let $T = T_k := \min\{k, k_\alpha -k \}$, then we get
			\begin{align*}
				\int_{Q(k,2)} \abs{\nabla u_\alpha}^2 dx &\leq C2^{-\sigma T_k} \int_{A({\lambda_\alpha R}/2, 2\delta, x_\alpha)}\abs{\nabla u_\alpha}^2dx + C(\alpha - 1).
			\end{align*}
			Thus, by small energy regularity Lemma \ref{lem4.1}, we have
			\begin{align*}
				Osc_{Q(k,1)}(u_\alpha) &\leq \p{\int_{Q(k,2)} \abs{\nabla u_\alpha}^2 dx}^{\frac{1}{2}}\\
				&\leq C2^{-\frac{\sigma}{2} T_k}\p{\int_{A(\frac{\lambda_\alpha R}{2}, 2 \delta, x_\alpha)}\abs{\nabla u_\alpha}^2}^{\frac{1}{2}} + C\sqrt{\alpha -1}.
			\end{align*}
			which implies
			\begin{align}\label{neck osc}
				Osc_{A(\lambda_\alpha R,\delta,x_\alpha)}(u_\alpha) &\leq \sum_{k = 1}^{k_\alpha} Osc_{Q(k,1)}(u_\alpha)\nonumber\\
				&\leq C\p{\int_{A(\frac{\lambda_\alpha R}{2}, 2 \delta, x_\alpha)}\abs{\nabla u_\alpha}^2}^{\frac{1}{2}} + C\sqrt{\alpha - 1}\p{\log \delta - \log (\lambda_\alpha R)}
			\end{align}
			Therefore, keeping in mind that $\lim_{\alpha \searrow 1} \lambda_\alpha^{- \sqrt{\alpha - 1}} = \nu = 1$ we can conclude that
			\begin{equation*}
				\lim_{\delta\searrow 0}\lim_{R\rightarrow \infty}\lim_{\alpha\searrow 1}Osc_{A(\lambda_\alpha R ,\delta,x_\alpha)}(u_\alpha) = 0
			\end{equation*}
			which shows that the set $u_0(B(0,1))\bigcup w^1(\S^2)$ is a connected subset of $N$, as desired in part \ref{simple neck analysis item1} of Theorem \ref{simple neck analysis}.
		\end{proof}
		
		\subsubsection{Asymptotic Neck Analysis and Length Formula for \texorpdfstring{$ \nu > 1$}{Lg}}\label{section neck length}
  \ 
  \vskip5pt
        In this subsubsection, if $1 < \nu < \infty$ we demonstrate that the neck domain converges to a geodesic of finite length in $N$, allowing us to derive the formula for the length of this geodesic. And if $\nu = \infty$, we prove that the neck domain converges to a geodesic of infinite length. These cases present a higher level of complexity, necessitating the introduction of several preliminary lemmas before proving the main consequences.
 
        First, we note that
		\begin{lemma}\label{lem F 4}
			With same hypothesis as Theorem \ref{simple neck analysis} and the assumption $\nu > 1$, we have
			\begin{equation*}
				\lim_{\alpha\searrow 1} \norm{ \left(\tau_\alpha + \abs{\nabla_{g_\alpha} u_\alpha}^2\right)^{\alpha - 1}}_{C^0(B(0, 1/2)) } = \mu.
			\end{equation*}
		\end{lemma}
		\begin{proof}
			Since we have assumed there is only one bubble, there exists $x_\alpha \in B(0, 1/2)$ such that
			\begin{equation*}
				\frac{1}{\lambda_\alpha} := \max_{x \in \overline{B(0, 1/2)}} \abs{\nabla_{g_\alpha} u_\alpha (x)} = \abs{\nabla_{g_\alpha} u_\alpha (x_\alpha)} 
			\end{equation*}
			for small enough $\alpha - 1$. On the one hand, 
			\begin{equation*}
				\lim_{\alpha\searrow 1} \norm{ \left(\tau_\alpha + \abs{\nabla_{g_\alpha} u_\alpha}^2\right)^{\alpha - 1}}_{C^0(B(0, 1/2)) } \geq 
				\lim_{\alpha\searrow 1} \norm{ \nabla_{g_\alpha} u_\alpha}^{2\alpha - 2}_{C^0(B(0, 1/2)) } = \lim_{\alpha \searrow 1}\lambda_\alpha^{2-2\alpha} = \mu.
			\end{equation*}
			On the other hand, recalling we have assumed $\tau_\alpha \leq 1$ for $\alpha > 1$ we estimate
			\begin{equation*}
				\lim_{\alpha\searrow 1} \norm{ \left(\tau_\alpha + \abs{\nabla_{g_\alpha} u_\alpha}^2\right)^{\alpha - 1}}_{C^0(B(0, 1/2)) } \leq \lim_{\alpha\searrow 1} \norm{ 2 \nabla_{g_\alpha} u_\alpha}^{2\alpha - 2}_{C^0(B(0, 1/2)) } = \lim_{\alpha \searrow 1}\lambda_\alpha^{2-2\alpha} = \mu.
			\end{equation*}
		\end{proof}
		
		First, similarly to the estimates in Lemma \ref{decay theta}, we establish a more delicate decay estimates of angle component of the energy functional of $u_\alpha$ as $\alpha \searrow 1$, more precise description is following.
		\begin{prop}\label{theta zero}
			With same hypothesises as Theorem \ref{simple neck analysis}, we further assume $\nu > 1$.   Then for any sequence $t_\alpha \in [t_1 ,t_2]$ where $0 < t_1 \leq t_2 < 1$ and any $R > 0$, after choosing a subsequence,  we have
			\begin{equation}\label{E 37}
				\lim_{\alpha \searrow 1} \frac{1}{\alpha - 1
				}\int_{A(\lambda_\alpha^{t_\alpha}/R, \lambda_\alpha^{t_\alpha}R, x_\alpha)} \frac{1}{|x - x_\alpha|^2}\abs{\frac{\partial u_\alpha}{\partial \theta}}^2 dx = 0.
			\end{equation}
		\end{prop}
		\begin{proof}
            The proof splits into two steps.
            \step Firstly, we prove a weaker version of this Proposition \ref{theta zero}, that is, we show that for any positive integer $k$,  there exists a constant $C$ that is independent of $k$ such that 
			\begin{equation}
				\lim_{\alpha \searrow 1} \frac{1}{\alpha - 1
				}\int_{A(2^{-k}\lambda_\alpha^{t_\alpha}, 2^k\lambda_\alpha^{t_\alpha}, x_\alpha)} \frac{1}{|x - x_\alpha|^2}\abs{\frac{\partial u_\alpha}{\partial \theta}}^2 dx\leq C.
			\end{equation}
			Taking a small enough positive number $\gamma < \min\{t_1, 1 -t_2\}$ and $t\leq \log \lambda_\alpha^{-\gamma}/\log 2$ we can define
			\begin{equation*}
				Q(t): = A\p{2^{-t} \lambda_\alpha^{t_\alpha}, 2^{t} \lambda_\alpha^{t_\alpha},x_\alpha}.
			\end{equation*}
			Transforming the integral domain of \eqref{E 32} over $Q(t)$, we can obtain
			\begin{align}\label{E 38}
				(1 - C\varepsilon)\int_{Q(t)} \abs{\nabla u_\alpha}^2 dx & \leq C(\alpha - 1)\int_{Q(t+1)}\abs{\nabla u_\alpha}^2dx\nonumber\\
				&\quad + \int_{\partial Q(t)} \frac{\partial u_\alpha}{\partial r}\p{u_\alpha - u^*_\alpha} ds + \int_{Q(t)} \abs{\frac{\partial u_\alpha}{\partial r}}^2dx.
			\end{align}
            Next, similarly to the proof in the proof of Theorem \ref{simple neck analysis} for case $\nu  =1$ in Subsubsection \ref{subsubsection: no neck},  we want to utilizing Pohozaev identity \eqref{eq:another pohozaev 2} obtained in Lemma \ref{lem:second pohozaev} to establish a more delicate estimates of the term 
			\begin{equation*}
				\int_{Q(t)} \abs{\frac{\partial u_\alpha}{\partial r}}^2dx
			\end{equation*}
            occurring in righthand of \eqref{E 32}.
			Integrating 
			\begin{align*}
				\int_{\partial B(x_\alpha,s)}& \p{\abs{\frac{\partial u_\alpha}{\partial r}}^2 - \abs{\nabla u_\alpha}^2 }ds\nonumber\\
                &= -\frac{(\alpha - 1)}{s}\int_{B(x_\alpha,s)}\frac{\nabla \abs{\nabla_{g_\alpha}u_\alpha}^2\nabla u_\alpha}{\tau_\alpha + \abs{\nabla_{g_\alpha}u_\alpha}^2} \cdot \abs{x - x_\alpha}\cdot\frac{\partial u_\alpha}{\partial r} dx
			\end{align*}
			with respect to $s$ from $2^{-t}\lambda_\alpha^{t_\alpha} $ to $2^{ t}\lambda_\alpha^{t_\alpha} $  with respect to $t$, we can get
            \begin{align}\label{E 38 2}
                \int_{Q(t)}& \p{\abs{\frac{\partial u_\alpha}{\partial r}}^2 - \abs{\nabla u_\alpha}^2 }dx\nonumber\\
                &= - \int_{2^{-t}\lambda_\alpha^{t_\alpha}}^{2^{t}\lambda_\alpha^{t_\alpha}}\p{\frac{\alpha - 1}{s}\int_{B(x_\alpha,s)}\frac{\nabla \abs{\nabla_{g_\alpha}u_\alpha}^2\nabla u_\alpha}{\tau_\alpha + \abs{\nabla_{g_\alpha}u_\alpha}^2} \cdot \abs{x - x_\alpha}\cdot \frac{\partial u_\alpha}{\partial r} dx} ds.
            \end{align}
            Combining \eqref{E 38} and \eqref{E 38 2} we have
            \begin{align}\label{eq:Ialpha 2}
                \left(\frac{1}{2}- C\varepsilon\right)\int_{Q(t)} \abs{\nabla u_\alpha}^2 dx & \leq C(\alpha - 1)\int_{Q(t+1)}\abs{\nabla u_\alpha}^2dx\nonumber\\
				&\quad + \int_{\partial Q(t)} \frac{\partial u_\alpha}{\partial r}\p{u_\alpha - u^*_\alpha} ds + \int_{2^{-t}\lambda_\alpha^{t_\alpha}}^{2^{t}\lambda_\alpha^{t_\alpha}}\frac{\alpha - 1}{s}I_\alpha(s) ds,
            \end{align}
			where we choose small enough $\varepsilon > 0$ such that $C\varepsilon \leq 1/4$ and we  denote
			\begin{equation*}
				I_\alpha(t) := -\int_{B(x_\alpha,t)}\frac{\nabla \abs{\nabla_{g_\alpha}u_\alpha}^2\nabla u_\alpha}{\tau_\alpha + \abs{\nabla_{g_\alpha}u_\alpha}^2} \cdot \abs{x - x_\alpha}\cdot\frac{\partial u_\alpha}{\partial r} dx.
			\end{equation*}
			For any ${2^{-t}}\lambda_\alpha^{t_\alpha} \leq r \leq 2^t \lambda_\alpha^{t_\alpha}$, utilizing Lemma \ref{small regu N} we get that  
			\begin{align*}
				\abs{I(r) - I(\lambda_\alpha^{t_\alpha})} &\leq C  \int_{A(2^{-t}\lambda_\alpha^{t_\alpha},2^{t}\lambda_\alpha^{t_\alpha},x_\alpha)} \abs{\nabla^2_{g_\alpha} u_\alpha}\cdot |x - x_\alpha|\cdot \abs{\nabla u_\alpha} dx\\
				&= C\int_{{A(2^{-t - 1}\lambda_\alpha^{t_\alpha},2^{t+ 1}\lambda_\alpha^{t_\alpha}, x_\alpha)}} \abs{\nabla u_\alpha }^2 dx\\
                &\leq C\int_{A(\lambda_\alpha^{t_\alpha + \gamma}, \lambda_\alpha^{t_\alpha - \gamma}, x_\alpha)}\abs{\nabla u_\alpha }^2 dx := \eta(\gamma,\alpha),
			\end{align*}
			where in the last inequality we used the choice of $t\leq \log \lambda_\alpha^{-\gamma}/\log 2$. Next, we show that 
            $\eta(\gamma,\alpha) \rightarrow 0$ as $\alpha \searrow 1$ and $\gamma \rightarrow 0$. To see this, by transforming the integral domain from $[2^{k-t}\lambda_\alpha R,2^{k+t}\lambda_\alpha R ]$ into $[\lambda_\alpha^{t_\alpha + \gamma}, \lambda_\alpha^{t_\alpha - \gamma}]$ in \eqref{E 33}, we have 
			\begin{align}\label{eq: gamma 1}
				\int_{A(\lambda_\alpha^{t_\alpha + \gamma}, \lambda_\alpha^{t_\alpha - \gamma}, x_\alpha)} & \p{\tau_\alpha + \abs{\nabla_{g_\alpha} u_\alpha}^2}^{\alpha - 1}\abs{\frac{\partial u_\alpha}{\partial r}}^2 dx\nonumber\\
                &\leq \frac{1}{2\alpha} \int_{A(\lambda_\alpha^{t_\alpha + \gamma}, \lambda_\alpha^{t_\alpha - \gamma}, x_\alpha)}\p{\tau_\alpha + \abs{\nabla_{g_\alpha}u_\alpha}^2}^{\alpha - 1}\abs{\nabla u_\alpha}^2dx\nonumber\\
				&\quad + \frac{\alpha - 1}{\alpha}\int_{\lambda_\alpha^{t_\alpha + \gamma}}^{\lambda_\alpha^{t_\alpha - \gamma}}\frac{1}{t}\p{\int_{B(x_\alpha,t)}\p{\tau_\alpha + \abs{\nabla_{g_\alpha}u_\alpha}^2}^{\alpha - 1}\abs{\nabla u_\alpha}^2dx}dt\nonumber\\
                &\quad +  C \lambda^{2t_\alpha - 2\gamma}_\alpha\nonumber\\
				&\leq \frac{1}{2\alpha} \int_{A(\lambda_\alpha^{t_\alpha + \gamma}, \lambda_\alpha^{t_\alpha - \gamma}, x_\alpha)}\p{\tau_\alpha + \abs{\nabla_{g_\alpha}u_\alpha}^2}^{\alpha - 1}\abs{\nabla u_\alpha}^2dx\nonumber\\
				&\quad + C(\alpha -1)\gamma\log(\lambda_\alpha)+C\lambda^{2t_\alpha - 2\gamma}_\alpha.
			\end{align}
			Keeping in mind that Lemma \ref{lem F 4}, given $0<\varepsilon < 1/8$ choosing $\alpha - 1$ small enough such that
            \begin{equation*}
                \abs{\norm{\p{\tau_\alpha + \abs{\nabla_{g_\alpha}u_\alpha}^2}^{\alpha - 1}}_{C^0(B(x_\alpha, \delta))} - \mu} \leq \frac{\varepsilon}{2},
            \end{equation*}
            plugging this into \eqref{eq: gamma 1} we have
			\begin{align*}
				\eta(\gamma,\alpha)=&\int_{A(\lambda_\alpha^{t_\alpha + \gamma}, \lambda_\alpha^{t_\alpha - \gamma}, x_\alpha)} \abs{\nabla u_\alpha}^2 dx\nonumber\\
                &\leq \p{\frac{1}{2} + \varepsilon} \int_{A(\lambda_\alpha^{t_\alpha + \gamma}, \lambda_\alpha^{t_\alpha - \gamma}, x_\alpha)}\abs{\nabla u_\alpha}^2dx + C\frac{(\alpha -1)}{\mu}\gamma \log(\lambda_\alpha)\\
				&\quad + \int_{A(\lambda_\alpha^{t_\alpha + \gamma}, \lambda_\alpha^{t_\alpha - \gamma}, x_\alpha)} \frac{1}{|x - x_\alpha|^2}\abs{\frac{\partial u_\alpha}{\partial \theta}}^2 dx + \frac{C}{\mu} \lambda^{2t_\alpha -2\gamma}_\alpha,
			\end{align*}
            which implies that 
            \begin{equation}\label{eq:Ialpha 1}
                \lim_{\gamma \rightarrow 0}\lim_{\alpha \searrow 1} \abs{I(r) - I(\lambda_\alpha^{t_\alpha})} \leq  \lim_{\gamma \rightarrow 0}\lim_{\alpha \searrow 1}\eta(\gamma,\alpha) = 0,
            \end{equation}
            by recalling the Lemma \ref{decay theta}.
		  Therefore, by \eqref{eq:Ialpha 1} and \eqref{eq:Ialpha 2} we have 
			\begin{align}\label{E 39}
				\p{\frac{1}{2} - C\varepsilon}\int_{Q(t)}\abs{\nabla u_\alpha}^2 dx &\leq \int_{\partial Q(t)} \frac{\partial u_\alpha}{\partial r}\p{u_\alpha - u^*_\alpha} ds + C(\alpha - 1)\int_{Q(t+1)}\abs{\nabla u_\alpha}^2dx\nonumber\\
				&\quad + \p{I(\lambda_\alpha^{t_\alpha}) + \eta(\gamma,\alpha)} \int_{2^{-t}\lambda_\alpha^{t_\alpha}}^{2^{t}\lambda_\alpha^{t_\alpha}}\frac{\alpha - 1}{t} dt\nonumber\\
                &\leq \int_{\partial Q(t)} \frac{\partial u_\alpha}{\partial r}\p{u_\alpha - u^*_\alpha} ds + C(\alpha - 1)\eta(\gamma,\alpha)\nonumber\\
                & \quad+  2\log(2) (\alpha -1) \p{I(\lambda_\alpha^{t_\alpha}) + \eta(\gamma,\alpha)}t.
			\end{align}
			Define 
			\begin{equation*}
				\mathcal{F}_\alpha(t) : = \int_{Q(t)} \abs{\nabla u_\alpha}^2 dx,
			\end{equation*}
			then by \eqref{E 35p} we can rewrite \eqref{E 39} as
			\begin{equation*}
				(1 - C\varepsilon) \mathcal{F}_\alpha(t) \leq \frac{1}{\log 2}\mathcal{F}^{\prime}_{\alpha}(t)  + 4\log(2) (\alpha -1)I(\lambda_\alpha^{t_\alpha})t + C(\alpha - 1)\eta(\gamma,\alpha)(t+1),
			\end{equation*}
			which implies 
			\begin{align}\label{eq: inte ineq 1}
				\p{2^{-\sigma t}\mathcal{F}_\alpha(t)}^\prime &\geq - 4\log^2(2){(\alpha -1)}2^{-\sigma t}I(\lambda_\alpha^{t_\alpha})t - C(\alpha - 1)2^{-\sigma t}\eta(\gamma,\alpha)(t+1),
			\end{align}
			where $\sigma = 1 - C \varepsilon \in (0,1)$ is a constant. Letting $2^T = \lambda_\alpha^{-\gamma}$ and integrating above inequality \eqref{eq: inte ineq 1} from $k$ to $T$ yields
			\begin{align}\label{E 40}
				\mathcal{F}_\alpha(k) &\leq 2^{\sigma(k - T)} \mathcal{F}_\alpha(T) + \frac{4\log(2)}{\sigma}{(\alpha - 1)}I(\lambda_\alpha^{t_\alpha})2^{\sigma k}\int_{k}^T 2^{-\sigma t}t dt \nonumber\\ 
				&\quad +C(\alpha - 1)\eta(\gamma,\alpha) 2^{\sigma k}\int_{k}^T (t+1)2^{-\sigma t}dt\nonumber\\
				& \leq 2^{\sigma(k - T)} \mathcal{F}_\alpha(T)+ \frac{4\log(2)k}{\sigma}{(\alpha - 1)}I(\lambda_\alpha^{t_\alpha})\nonumber\\
                &\quad +C(\alpha - 1)\p{I(\lambda_\alpha^{t_\alpha}) + \eta(\gamma,\alpha)(k + 1)},
			\end{align}
            where we used the estimates
            \begin{equation*}
                \int_{k}^T 2^{-\sigma t}t dt \leq \frac{k}{\sigma\log (2)}2^{-\sigma k} + \p{\frac{1}{\sigma \log (2)}}^2 2^{-\sigma k}.
            \end{equation*}
			On the other hand, utilizing the Pohozaev identity \eqref{eq:another pohozaev}, we obtain
			\begin{align}\label{E 41}
				\int_{Q(k)}\p{\abs{\frac{\partial u_\alpha}{\partial r}}^2 - \frac{1}{|x - x_\alpha|^2}\abs{\frac{\partial u_\alpha}{\partial \theta}}^2} dx&= 2 \int_{2^{-k}\lambda_\alpha^{t_\alpha}}^{2^{k}\lambda_\alpha^{t_\alpha}}\frac{\alpha - 1}{t} I_\alpha(t) dt\nonumber\\
                &\geq 4\log (2) k(\alpha - 1)(I_\alpha(\lambda_\alpha^{t_\alpha}) - \eta(\gamma,\alpha)).
			\end{align}
			Next, subtracting  \eqref{E 40} by \eqref{E 41} yields
			\begin{align}\label{E 42}
				2\int_{Q(k)} \frac{1}{|x - x_\alpha|^2}\abs{\frac{\partial u_\alpha}{\partial \theta}}^2 dx 
				&\leq 2^{\sigma k}\lambda_\alpha^{\gamma \sigma} \mathcal{F}_\alpha(T) + (\alpha - 1)4\log(2)I_\alpha(\lambda_\alpha^{t_\alpha})\p{\frac{1}{\sigma} - 1}k \nonumber\\
				&\quad + C(\alpha - 1)I(\lambda_\alpha^{t_\alpha}) + C(\alpha - 1)\eta(\gamma,\alpha)k.
			\end{align}
			Since $\nu = \lim_{\alpha \searrow 1} \lambda_\alpha^{- \sqrt{\alpha - 1}} > 1$, one have
			\begin{equation*}
				\lambda_\alpha^{\gamma \sigma } = o\p{(\alpha - 1)^m} \quad \text{ as } \alpha \searrow 1,
			\end{equation*}
			for any positive integer $m > 0$. Then in \eqref{E 42}, taking $\alpha\searrow 1$ first, then $\varepsilon \rightarrow 0$ and $\gamma \rightarrow 0$, yields
			\begin{equation}\label{E 43}
				\lim_{\alpha \searrow 1} \frac{1}{\alpha - 1
				}\int_{A(2^{-k}\lambda_\alpha^{t_\alpha}, 2^k\lambda_\alpha^{t_\alpha}, x_\alpha)} \frac{1}{|x - x_\alpha|^2}\abs{\frac{\partial u_\alpha}{\partial \theta}}^2 dx \leq C \lim_{\alpha \searrow 1} I(\lambda_\alpha^{t_\alpha}) \leq C,
			\end{equation}
			for some universal constant $C > 0$ independent of $k$.
   
            \step Next, we prove the assertion of Proposition \ref{theta zero}, that is, for any $k \in \mathbb{N}$
			\begin{equation*}
				\lim_{\alpha \searrow 1} \frac{1}{\alpha - 1
				}\int_{A(2^{-k}\lambda_\alpha^{t_\alpha}, 2^k\lambda_\alpha^{t_\alpha}, x_\alpha)} \frac{1}{|x - x_\alpha|^2}\abs{\frac{\partial u_\alpha}{\partial \theta}}^2 dx= 0.
			\end{equation*}
			Utilizing Fubini's theorem we rewrite  \eqref{E 43} as
			\begin{align*}
				\frac{1}{\alpha - 1
				}\int_{A(2^{-k}\lambda_\alpha^{t_\alpha}, 2^k\lambda_\alpha^{t_\alpha}, x_\alpha)} &\frac{1}{|x - x_\alpha|^2}\abs{\frac{\partial u_\alpha}{\partial \theta}}^2 dx\\
				&= \int_{2^{-k}\lambda_\alpha^{t_\alpha}}^{2^{k}\lambda_\alpha^{t_\alpha}}\p{\frac{1}{\alpha - 1}\int_{0}^{2\pi} \abs{\frac{\partial u_\alpha}{\partial \theta}(r,\theta)}^2 d\theta} \frac{dr}{|x - x_\alpha|} \leq C,
			\end{align*}
			thus given any small $\varepsilon > 0$ there will always exist a large enough positive integer $k_0$, which is independent of $\alpha > 1$, and $L_\alpha \in [2^{k_0}, 2^{k_0 + 1}]$ such that
			\begin{equation}\label{eq:theta zero eq1}
				\frac{1}{\alpha - 1}\int_{0}^{2\pi} \abs{\frac{\partial u_\alpha}{\partial \theta}(L_\alpha\lambda_\alpha^{t_\alpha},\theta)}^2 d\theta < \varepsilon
			\end{equation}
			and 
			\begin{equation}\label{eq:theta zero eq2}
				\frac{1}{\alpha - 1}\int_{0}^{2\pi} \abs{\frac{\partial u_\alpha}{\partial \theta}\p{\frac{1}{L_\alpha}\lambda_\alpha^{t_\alpha},\theta}}^2 d\theta < \varepsilon.
			\end{equation}
			From these two estimates \eqref{eq:theta zero eq1} and \eqref{eq:theta zero eq2},  we can obtain a more delicate estimate of \eqref{E 35}
			\begin{align}\label{E 44}
				\int_{\partial A(\frac{1}{L_\alpha}\lambda_\alpha^{t_\alpha}, L_\alpha\lambda_\alpha^{t_\alpha}, x_\alpha )}& \frac{\partial u_\alpha}{\partial r}\p{u_\alpha - u^*_\alpha} ds\nonumber\\
				&\leq \p{\int_{\partial A(\frac{1}{L_\alpha}\lambda_\alpha^{t_\alpha}, L_\alpha\lambda_\alpha^{t_\alpha} , x_\alpha )} |x - x_\alpha|\cdot\abs{\frac{\partial u_\alpha}{\partial r}}^2 ds}^{\frac{1}{2}}\cdot\nonumber\\
                &\quad \quad \p{\int_{\partial A(\frac{1}{L_\alpha}\lambda_\alpha^{t_\alpha}, L_\alpha\lambda_\alpha^{t_\alpha} , x_\alpha)} \abs{\frac{\partial u_\alpha}{\partial \theta}}^2 ds}^{\frac{1}{2}}\nonumber\\
				&\leq \sqrt{(\alpha - 1)\varepsilon} \p{\int_{\partial A(\frac{1}{L_\alpha}\lambda_\alpha^{t_\alpha}, L_\alpha\lambda_\alpha^{t_\alpha} , x_\alpha )} |x - x_\alpha|\cdot\abs{\frac{\partial u_\alpha}{\partial r}}^2 ds}^{\frac{1}{2}}.
			\end{align}
			Moreover, using \eqref{pohozaev esti 2} and Corollary \ref{coro 1}  we have
			\begin{align*}
				\int_{\partial A(\frac{1}{L_\alpha}\lambda_\alpha^{t_\alpha}, L_\alpha\lambda_\alpha^{t_\alpha} , x_\alpha )}& |x - x_\alpha|\cdot\abs{\frac{\partial u_\alpha}{\partial r}}^2 ds\\
				&\leq C \int_{\partial A(\frac{1}{L_\alpha}\lambda_\alpha^{t_\alpha}, L_\alpha\lambda_\alpha^{t_\alpha} , x_\alpha )} \frac{1}{|x - x_a|}\abs{\frac{\partial u_\alpha}{\partial \theta}}^2 ds \\
				& + C(\alpha  - 1) + C \lambda_\alpha^{t_\alpha},
			\end{align*}
			which implies 
			\begin{equation}\label{E 45}
				\int_{\partial A(\frac{1}{L_\alpha}\lambda_\alpha^{t_\alpha}, L_\alpha\lambda_\alpha^{t_\alpha}, x_\alpha )} \frac{\partial u_\alpha}{\partial r}\p{u_\alpha - u^*_\alpha} ds \leq C\sqrt{\varepsilon}(\alpha - 1).
			\end{equation}
			Applying \eqref{E 45} to \eqref{E 39} , we can obtain a more refined estimate
			\begin{align}\label{E 46}
				(1 - C\varepsilon)\int_{{ A(\frac{1}{L_\alpha}\lambda_\alpha^{t_\alpha}, L_\alpha\lambda_\alpha^{t_\alpha}, x_\alpha )}}\abs{\nabla u_\alpha}^2 dx &\leq C\sqrt{\varepsilon}(\alpha - 1) + C(\alpha - 1)\eta(\gamma,\alpha)\p{\log(L_\alpha) + 1}\nonumber\\
                & \quad+  4\log(L_\alpha)t (\alpha -1)I(\lambda_\alpha^{t_\alpha}).
			\end{align}
			Similar to \eqref{E 41}, we have 
			\begin{align}\label{E 47}
				\int_{A(\frac{1}{L_\alpha}\lambda_\alpha^{t_\alpha}, L_\alpha\lambda_\alpha^{t_\alpha}, x_\alpha)}&\p{\abs{\frac{\partial u_\alpha}{\partial r}}^2 - \frac{1}{|x - x_\alpha|^2}\abs{\frac{\partial u_\alpha}{\partial \theta}}^2} dx \nonumber\\
				&\geq {4(\alpha - 1)\log\p{L_\alpha}} I(\lambda_\alpha^{t_\alpha}) - C(\alpha - 1)\eta(\gamma,\alpha)\log{(L_\alpha)}.
			\end{align}
			Subtracting \eqref{E 46} by \eqref{E 47} yields
			\begin{align*}
				2 \int_{A(\frac{1}{L_\alpha}\lambda_\alpha^{t_\alpha}, L_\alpha\lambda_\alpha^{t_\alpha}, x_\alpha)}& \frac{1}{|x - x_\alpha|^2}\abs{\frac{\partial u_\alpha}{\partial \theta}}^2 dx\\ 
                &\leq C\sqrt{\varepsilon}(\alpha - 1) + \p{1 - \frac{1}{1 - C\varepsilon}}{4(\alpha - 1)\log\p{L_\alpha}} I(\lambda_\alpha^{t_\alpha})\\
                &\quad + C(\alpha - 1)\eta(\gamma,\alpha)\p{\log(L_\alpha) + 1}.
			\end{align*}
			Thus, by the choice of $\log(L_\alpha) \in [\log(2)k_0, \log(2)(k_0 +1)]$ and the fact \eqref{eq:Ialpha 1},  taking $\alpha\searrow 1$ firstly and then letting $\varepsilon \searrow 0$  will yield the assertion of Proposition \ref{theta zero} immediately.
		\end{proof}
		As a corollary, we have the following observation which will be used later. 
		\begin{coro}\label{theta sup}
			With the same hypothesis as Proposition \ref{theta zero}. For any fixed $R > 0$ and $0 < t_1 < t_2 < 1$, we have
			\begin{equation*}
				\lim_{\alpha \searrow 1} \sup_{t \in [t_1, t_2]} \frac{1}{\alpha - 1
				}\int_{A\p{\frac{1}{R}\lambda_\alpha^{t_\alpha}, \lambda_\alpha^{t_\alpha} R, x_\alpha}} \frac{1}{|x - x_\alpha|^2}\abs{\frac{\partial u_\alpha}{\partial \theta}}^2 dx = 0.
			\end{equation*}
		\end{coro}
		\begin{proof}
			We prove by contradiction. If the assertion fails, then after choosing a subsequence there exists $\varepsilon > 0$ and  $t_{\alpha_k} \rightarrow t_0$ for some $t_0 \in [t_1, t_2]$ such that 
			\begin{equation*}
				\frac{1}{\alpha - 1
				}\int_{A(\frac{1}{R}\lambda_{\alpha_k}^{t_{\alpha_k}}, \lambda_{\alpha_k}^{t_{\alpha_k}} R, x_{\alpha_\alpha})} \frac{1}{|x - x_{\alpha_k}|^2}\abs{\frac{\partial u_\alpha}{\partial \theta}}^2 dx \geq \varepsilon_0
			\end{equation*}
			However, Proposition \ref{theta zero} tells us that 
			\begin{equation*}
				\lim_{\alpha \searrow 1} \frac{1}{\alpha - 1
				}\int_{A(\frac{1}{R}\lambda_\alpha^{t_\alpha}, \lambda_\alpha^{t_\alpha} R, x_\alpha)} \frac{1}{|x - x_\alpha|^2}\abs{\frac{\partial u_\alpha}{\partial \theta}}^2 dx = 0.
			\end{equation*}
			for any sequence $\{t_{\alpha}\}_{\alpha\searrow 1} \subset [t_1, t_2]$. This is a contradiction. 
		\end{proof}
		Note that by Lemma \ref{lambda^s small}, we find that  for any $0< t_1 \leq t \leq t_2 < 1$
		\begin{align*}
			Osc_{\partial B(x_\alpha, \lambda^{t}_\alpha) } (u_\alpha) &\leq C \p{\int_{A(\frac{1}{2}\lambda^{t}_\alpha, 2\lambda^{t}_\alpha, x_\alpha)}|\nabla u_\alpha|^2 dx }^{\frac{1}{2}}\\
			&\leq C \int_{A(\frac{1}{2}\lambda^{t}_\alpha, 2\lambda^{t}_\alpha, x_\alpha)}\p{\tau_\alpha + \abs{\nabla_{g_\alpha} u_\alpha}^2}^{\alpha - 1}|\nabla u_\alpha|^2 dx\quad \rightarrow 0\quad \text{as } \alpha \searrow 1.
		\end{align*}
		which implies the $u_\alpha\p{\partial B(x_\alpha, \lambda_\alpha^{t_\alpha})}$ converges to some point of $N$ as $\alpha \searrow 1$.
		\begin{prop}\label{prop 2}
			With same hypothesises as Theorem \ref{simple neck analysis}, we further assume $\nu > 1$.   Then for any sequence $t_\alpha \in [t_1 ,t_2]$ where $0 < t_1 \leq t_2 < 1$ and any $R > 0$, after choosing a subsequence,  we have
			\begin{equation*}
				\frac{1}{\sqrt{\alpha - 1}}\Big(u_\alpha\p{x_\alpha + \lambda_\alpha^{t_\alpha}x} - u\big(x_\alpha + (\lambda_\alpha^{t_\alpha},0)\big)\Big) \rightarrow \Vec{a}\log{|x|}
			\end{equation*}
			strongly in $C^2\p{A\p{\frac{1}{R}, R, 0},\R^{K}}$ for any $R > 0$ and  any integer $k\in \mathbb{N}$, here  
   $$y = \lim_{\alpha \searrow 1} u_\alpha\p{\partial B(x_\alpha, \lambda_\alpha^{t_\alpha})}$$  
   and 
   $$ \Vec{a} \in T_yN \subset T_y\R^{K} \cong \R^{K}$$ 
   is a vector in $\R^{K}$ satisfying
			\begin{equation*}
				|\Vec{a}| = \mu^{1 - \lim_{\alpha \searrow 1} t_\alpha} \sqrt{\frac{E(w^1)}{\pi}}.
			\end{equation*}
		\end{prop}
		\begin{proof}
			Let
			\begin{equation*}
				u^\prime_\alpha(x) := u_\alpha(x_\alpha + \lambda_\alpha^{t_\alpha} x) \quad \text{and}\quad  v_\alpha(x) := \frac{1}{\sqrt{\alpha - 1}}\Big(u_\alpha\p{x_\alpha + \lambda_\alpha^{t_\alpha}x} - u\big(x_\alpha + (\lambda_\alpha^{t_\alpha},0)\big)\Big).
			\end{equation*}
			By \eqref{E 40} and small energy regularity Lemma \ref{lem4.1}, recalling $\lambda_\alpha^{\gamma} = o\p{(\alpha - 1)^m}$ for all $\gamma > 0$ and $m\in \mathbb{N}$ we have 
			\begin{align*}
				\norm{\nabla u^\prime_\alpha}_{C^0(A(2^{-k}, 2^k,0))} + \norm{\nabla^2 u^\prime_\alpha}_{C^0(A(2^{-k}, 2^k,0))} \leq C(k)\sqrt{\alpha - 1}
			\end{align*}
			which further implies
			\begin{equation*}
				\norm{\nabla v_\alpha}_{C^0(A(2^{-k}, 2^k,0))} + \norm{\nabla^2 v_\alpha}_{C^0(A(2^{-k}, 2^k,0))} \leq C(k),
			\end{equation*}
            for some constant $C(k)$ depending on $k$.
			Since $v_\alpha(1,0) = 0$, the above estimate implies 
			\begin{equation*}
				\norm{ v_\alpha}_{C^0(A(2^{-k}, 2^k,0))}\leq C(k).
			\end{equation*}
			By the Euler-Lagrange equation \eqref{el general 1 equi} of $u_\alpha$, one can check that $v_\alpha$ satisfies the following equation
			\begin{align*}
			    \Delta v_\alpha &+ \sqrt{\alpha - 1}A(\nabla v_\alpha, \nabla v_\alpha) + (\alpha - 1)O(|\nabla^2 v_\alpha|)\\
                &= \sqrt{\alpha - 1}H(\nablap v_\alpha,\nabla v_\alpha) + \sqrt{\alpha - 1}o(1)
			\end{align*}
            where $o(1)$ tends to 0  as $\alpha \searrow 1$.
			By the compactness of PDE's theory, there exists a subsequence of $v_\alpha$, which is still denoted by $v_\alpha$, such that
			\begin{equation*}
				v_\alpha \rightarrow v_0 \quad \text{in }\, C^2_{loc}\p{\R^2\backslash \{0\}}
			\end{equation*}
			where $v_0$ is a harmonic function on $\R^2$. Moreover, by Proposition \ref{theta zero}, the angel component energy of $v_0$ vanishes, that is, $v_0(x) = v_0(|x|)$.  Thus, $v_0$ is a fundamental solution of Laplacian equation over $\R^2$, without loss of generality we can write $v_0$ as
			\begin{equation*}
				v_0 = \Vec{a}\log{r} = (a_1,\dots,a_K)\log{r}, \quad \text{for some vector } \Vec{a} \in T_{y}N \subset \R^K 
			\end{equation*}
			From \eqref{pohozaev esti 2}, we know that $v_\alpha$ satisfies
			\begin{align*}
				\int_{\partial B(0,t)} &\p{\tau_\alpha + \abs{\nabla_{g_\alpha} u_\alpha}^2}^{\alpha - 1}\abs{\nabla v_\alpha}^2 ds\\
				&= \frac{2\alpha}{(2\alpha - 1)} \int_{\partial B(0,t)}\p{\tau_\alpha + \abs{\nabla_{g_\alpha}u_\alpha}^2}^{\alpha - 1}\frac{1}{|x|^2}\abs{\frac{\partial v_\alpha}{\partial \theta}}^2ds\nonumber\\
				&\quad+\frac{2}{(2\alpha - 1)t}\int_{B(0,t)} \p{\tau_\alpha + \abs{\nabla_{g_\alpha}u_\alpha}^2}^{\alpha - 1} \abs{\nabla u_\alpha}^2 dx + \frac{O(t)}{\alpha - 1},
			\end{align*}
			which implies
			\begin{align}\label{eq:lenght a 1}
				&\lim_{\alpha \searrow 1}\int_{
					A(\lambda_\alpha^{t_\alpha}, 2\lambda_\alpha^{t_\alpha}, x_\alpha)} \p{\tau_\alpha + \abs{\nabla_{g_\alpha} u_\alpha}^2}^{\alpha - 1}\abs{\nabla v_\alpha}^2 dx\nonumber\\
				&= \lim_{\alpha \searrow 1}\frac{2\alpha}{(2\alpha - 1)} \int_{A(\lambda_\alpha^{t_\alpha}, 2\lambda_\alpha^{t_\alpha}, x_\alpha)}\p{\tau_\alpha + \abs{\nabla_{g_\alpha}u_\alpha}^2}^{\alpha - 1}\frac{1}{|x|^2}\abs{\frac{\partial v_\alpha}{\partial \theta}}^2dx\nonumber\\
				&\quad+\lim_{\alpha \searrow 1}\frac{2}{(2\alpha - 1)}\int_{\lambda_\alpha^{t_\alpha}}^{2\lambda_\alpha^{t_\alpha}}\frac{1}{t}\p{\int_{B(0,t)} \p{\tau_\alpha + \abs{\nabla_{g_\alpha}u_\alpha}^2}^{\alpha - 1} \abs{\nabla u_\alpha}^2 dx} dt\nonumber\\
                &\quad + \lim_{\alpha \searrow 1} \frac{O(\lambda_\alpha^{t_\alpha})}{\alpha - 1}\nonumber\\
				& = 2 \log 2 \mu^{1 - \lim_{\alpha \searrow 1}t_\alpha}\Lambda.
			\end{align}
			Here, we used Lemma \ref{lambda^s small} and Proposition \ref{theta zero}. On the other hand, we observe that
			\begin{align}\label{eq:lenght a 2}
				&\lim_{\alpha \searrow 1} \int_{
					A(\lambda_\alpha^{t_\alpha}, 2\lambda_\alpha^{t_\alpha}, x_\alpha)} \p{\tau_\alpha + \abs{\nabla_{g_\alpha} u_\alpha}^2}^{\alpha - 1}\abs{\nabla v_\alpha}^2 dx\nonumber \\
				&\quad \quad \quad \quad = \lim_{\alpha \searrow 1} \int_{A(1,2,0)} \p{\tau_\alpha + \abs{\nabla_{g_\alpha}v_\alpha}^2 \frac{\alpha - 1}{\lambda_\alpha^{2t_\alpha}} }^{\alpha - 1} \abs{\nabla v_\alpha }^2 dx\nonumber\\
				&\quad \quad \quad \quad = 2\pi \log 2 |\Vec{a}|^2 \mu^{\lim_{\alpha\searrow 1}t_\alpha}.
			\end{align}
			Therefore, combining above two identities \eqref{eq:lenght a 1} and \eqref{eq:lenght a 2} we have
			\begin{equation*}
				\abs{\Vec{a}}^2 = \frac{\Lambda}{\pi}\mu^{1 - 2\lim_{\alpha \searrow 1} t_\alpha}.
			\end{equation*}
			This completes the proof the Proposition \ref{prop 2}.
		\end{proof}
		As a corollary of above Proposition \ref{prop 2}, we can obtain the following result.
		\begin{coro}\label{coro 2}
			Under the same assumption of Proposition \ref{prop 2}, the following holds
			\begin{enumerate}[label=(\arabic*)]
				\item\label{coro 2 item 1} For the radical direction, we have
				\begin{equation*}
					\int_{\lambda_\alpha^t}^{2\lambda_\alpha^t} \frac{1}{\sqrt{\alpha - 1}}\abs{\frac{\partial u_\alpha}{\partial r}} dr \rightarrow \log 2 \mu^{1-t}\sqrt{\frac{E(w)}{\pi}} \quad \text{in }C^0([t_1,t_2]),
				\end{equation*}
				and 
				\begin{equation*}
					\frac{1}{\sqrt{\alpha - 1}}\p{r\abs{\frac{\partial u_\alpha}{\partial r}}}(\lambda_\alpha^t,\theta) \rightarrow \mu^{1-t}\sqrt{\frac{E(w)}{\pi}}\quad \text{in }C^0([t_1,t_2]);
				\end{equation*}
				\item\label{coro 2 item 2} For the angular direction, we have 
				\begin{equation*}
					\int_{0}^{2\pi} \frac{1}{\sqrt{\alpha - 1}}\p{\frac{1}{r} \abs{\frac{\partial u_\alpha}{\partial\theta}}}(\lambda_\alpha^t,\theta) d\theta \rightarrow 0\quad \text{in }C^0([t_1,t_2])
				\end{equation*}
				and 
				\begin{equation*}
					\frac{1}{\sqrt{\alpha - 1}}\p{\frac{1}{r} \abs{\frac{\partial u_\alpha}{\partial\theta}}}(\lambda_\alpha^t,\theta) \rightarrow 0 \quad \text{in }C^0([t_1,t_2])
				\end{equation*}
			\end{enumerate}
		\end{coro}
		\begin{proof}
			We only prove the angular direction case \ref{coro 2 item 2},  the other statements in \ref{coro 2 item 1} can be argued similarly and the proof can be found in \cite{li2010,JostLiuZhu2019}. For the case \ref{coro 2 item 2}, it suffices to show the second assertion
			\begin{equation*}
				\frac{1}{\sqrt{\alpha - 1}}\p{ \frac{1}{r}\abs{\frac{\partial u_\alpha}{\partial\theta}}}(\lambda_\alpha^t,\theta) \rightarrow 0 \quad \text{in }C^0([t_1,t_2]),
			\end{equation*}
			since the first one of \ref{coro 2 item 2} is a direct corollary of the second one. By contradiction, if it fails, then there exists a sequence $t_\alpha \in [t_1,t_2]$ and $\theta_\alpha \in [0,2\pi]$ such that
			\begin{equation}\label{contradiction}
				\abs{\frac{1}{\sqrt{\alpha - 1}}\p{\frac{1}{r}\abs{\frac{\partial u_\alpha}{\partial\theta}}}(\lambda_\alpha^{t_\alpha},\theta_\alpha)} \geq \varepsilon_0 > 0
			\end{equation}
			for some $\varepsilon_0 > 0$. But, Proposition \ref{prop 2} tells us that for any $\theta \in [0,2\pi]$
			\begin{equation*}
				\frac{1}{\sqrt{\alpha - 1}}\p{\frac{1}{r}\frac{\partial u_\alpha}{\partial \theta}}(\lambda_\alpha^{t_\alpha},\theta) \rightarrow 0\quad \text{in }\, C^2 \text{ as } \alpha \searrow 1
			\end{equation*}
			which contradicts to \eqref{contradiction} by the compactness of $\theta_\alpha \in [0,2\pi]$ modulo some subsequences. Thus, we complete the proof of Corollary \ref{coro 2}.
		\end{proof}
		Next, we show the necks of $\alpha$-$H$-surfaces $u_\alpha$ converges to a geodesic, that is, the base map $u_0$ and single bubble $w$ are connected by some geodesic. To this end, we define the following curve 
		\begin{equation*}
			\gamma_\alpha(r) : =\frac{1}{2\pi} \int_{0}^{2\pi} u_\alpha(r,\theta) d\theta\,:\,[\lambda_\alpha^{t_2}, \lambda_\alpha^{t_1}] \rightarrow \R^{K}
		\end{equation*}
		where $(r,\theta)$ is the polar coordinate around $x_\alpha$. We denote the image of $\gamma_\alpha$ by $\Gamma_\alpha \subset N$. For  convenience, we use the following notation
		\begin{equation*}
			\dot{\gamma}_\alpha := \frac{d \gamma_\alpha}{dr},\quad \ddot{\gamma}_\alpha := \frac{d^2 \gamma_\alpha}{dr^2}.
		\end{equation*}
		We directly compute that
		\begin{align}\label{curve gamma}
			\ddot{\gamma}_\alpha &= \frac{1}{2\pi}\int_{0}^{2\pi} \frac{\partial^2 u_\alpha}{\partial r^2} d\theta \nonumber\\
			&=  \frac{1}{2\pi}\int_{0}^{2\pi} \frac{\partial^2 u_\alpha}{\partial r^2} + \frac{1}{r}\frac{\partial u_\alpha}{\partial r} + \frac{1}{r^2}\frac{\partial^2 u_\alpha}{\partial \theta^2}d\theta - \frac{1}{2\pi}\int_{0}^{2\pi}  \frac{1}{r}\frac{\partial u_\alpha}{\partial r} d\theta\nonumber\\
			& = \frac{1}{2\pi}\int_{0}^{2\pi}\Delta u_\alpha d\theta - \frac{1}{2\pi}\int_{0}^{2\pi}  \frac{1}{r}\frac{\partial u_\alpha}{\partial r} d\theta\nonumber\\
			& = - \frac{1}{2\pi}\int_{0}^{2\pi} A(u_\alpha)\left(\nabla u_\alpha, \nabla u_\alpha\right)d\theta\nonumber\\
            &\quad - \frac{ (\alpha - 1)}{2\pi}\int_{0}^{2\pi}\frac{\nabla|\nabla_{g_\alpha} u_\alpha|^2\cdot \nabla u_\alpha}{\tau_\alpha +|\nabla_{g_\alpha} u_\alpha|^2} d\theta\nonumber\\
			&\quad +\frac{\tau_\alpha^{\alpha -1}}{2\pi}\int_{0}^{2\pi} \frac{H(u_\alpha)(\nablap u_\alpha, \nabla u_\alpha)}{\alpha \left(\tau_\alpha + |\nabla_{g_\alpha} u_\alpha|^2\right)^{\alpha - 1}}d\theta - \frac{\dot{\gamma}_\alpha}{r}.
		\end{align}
		We use $h_\alpha$ to denote the induced metric upon $\Gamma_\alpha$ in $\R^K$ and  $A_{\Gamma_\alpha}$ to denote the second fundamental form restricted on $\Gamma_\alpha$. Equipped with these notations, we have
		\begin{lemma}\label{ays lemma}
			For any $\lambda_\alpha^{t_\alpha} \in [\lambda_\alpha^{t_2}, \lambda_\alpha^{t_1}]$, after choosing a subsequence, there holds
			\begin{align}
				\dot{\gamma}_\alpha(\lambda_\alpha^{t_\alpha}) &= \frac{\sqrt{\alpha - 1}}{\lambda_\alpha^{t_\alpha}}(\Vec{a} + o(1)),\nonumber\\
				h_\alpha\left(\frac{d}{dr},\frac{d}{dr}\right)& = \abs{\dot{\gamma}_\alpha}^2 = \frac{\alpha - 1}{\lambda_\alpha^{2t_\alpha}}\p{\abs{\Vec{a}}^2 + o(1)} \nonumber\\
				A_{\Gamma_\alpha}(\nabla \gamma_\alpha, \nabla\gamma_\alpha)& = \frac{\alpha - 1}{\lambda_\alpha^{2t_\alpha}}\p{A(y)(\Vec{a},\Vec{a}) + o(1)},
			\end{align}
            where $\Vec{a}$ and $y$ are constructed in Proposition \ref{prop 2} and $o(1) \rightarrow 0$ as $\alpha \searrow 1$.
			Moreover, for any $t\in [t_1, t_2]$, there exists a positive constant $C > 0$, such that 
			\begin{equation*}
				\norm{A_{\Gamma_\alpha}}_{h_\alpha}(\lambda_\alpha^{t}) \leq C.
			\end{equation*}
		\end{lemma}
		\begin{proof}
			For any $\lambda_\alpha^{t_\alpha} \in  [\lambda_\alpha^{t_2}, \lambda_\alpha^{t_1}]$, by Proposition \ref{prop 2}, we have 
			\begin{equation*}
				\frac{1}{\sqrt{\alpha - 1}}\Big(u_\alpha(x_\alpha + \lambda_\alpha^{t_\alpha}x) - u\big(x_\alpha + (\lambda_\alpha^{t_\alpha},0)\big)\Big) \rightarrow \Vec{a}\log{|x|}, \quad \text{as } \alpha \searrow 1 
			\end{equation*}
			where $ \Vec{a} \in T_yN \subset T_y\R^{K} = \R^{K}$ is a vector in $\R^{K}$ satisfying
			\begin{equation*}
				|\Vec{a}| = \mu^{1 - \lim_{\alpha \searrow 1} t_\alpha} \sqrt{\frac{E(w^1)}{\pi}}
			\end{equation*}
			and  $y = \lim_{\alpha \searrow 1} u_\alpha(\partial B(x_\alpha, \lambda_\alpha^{t_\alpha})) = \lim_{\alpha \searrow 1} u_\alpha(x_\alpha + \lambda_\alpha^{t_\alpha}e^{i\theta})$. Then we have
			\begin{equation*}
				\dot{\gamma}_\alpha(\lambda_\alpha^{t_\alpha}) = \frac{1}{2\pi}\int_{0}^{2\pi} \frac{\partial u_\alpha}{\partial r}(\lambda_\alpha^{t\alpha},\theta)d\theta = \frac{\sqrt{\alpha - 1}}{\lambda_\alpha^{t_\alpha}}(\Vec{a} + o(1))
			\end{equation*}
			and hence
			\begin{equation*}
				h_\alpha\left(\frac{d}{dr},\frac{d}{dr}\right) = \abs{\dot{\gamma}_\alpha}^2 = \frac{\alpha - 1}{\lambda_\alpha^{t_\alpha}}\p{\abs{\Vec{a}}^2 + o(1)}
			\end{equation*}
			where $o(1) \rightarrow 0$ as $\alpha\searrow 1$. Let
			\begin{equation*}
				G_\alpha = - \ddot{\gamma}_\alpha - \frac{\dot{\gamma}_\alpha}{r}
			\end{equation*}
			and by  equation\eqref{curve gamma}, Proposition \ref{prop 2}, Corollary \ref{coro 2} and the assumption $0<\beta_0 \leq \lambda_\alpha^{\alpha - 1}\leq 1$ we can further compute that
			\begin{align*}
				G_\alpha(\lambda_\alpha^{t_\alpha}) &=  \frac{1}{2\pi}\int_{0}^{2\pi} A(u_\alpha)\left(\nabla u_\alpha, \nabla u_\alpha\right)d\theta
				+ \frac{ (\alpha - 1)}{2\pi}\int_{0}^{2\pi}\frac{\nabla|\nabla_{g_\alpha} u_\alpha|^2\cdot \nabla u_\alpha}{\tau_\alpha+|\nabla_{g_\alpha} u_\alpha|^2} d\theta\nonumber\\
				&\quad -\frac{1}{2\alpha\pi}\tau_\alpha^{\alpha -1}\int_{0}^{2\pi} \frac{H(u_\alpha)(\nablap u_\alpha, \nabla u_\alpha)}{ \left(\tau_\alpha+ |\nabla_{g_\alpha} u_\alpha|^2\right)^{\alpha - 1}}d\theta\\
				&=\frac{\alpha - 1}{\lambda_\alpha^{2t_\alpha}}\p{\frac{1}{2\pi}\int_{0}^{2\pi}A(y)(\Vec{a},\Vec{a})d\theta + o(1)} + (\alpha - 1)\int_0^{2\pi}O\p{\abs{\nabla^2_{g_\alpha }u_\alpha}}d\theta\\
				&\quad + \frac{1}{2\pi}\int_{0}^{2\pi} O\p{\abs{H(u_\alpha)(\nablap u_\alpha, \nabla u_\alpha)}} d\theta\\
				& = \frac{\alpha - 1}{\lambda_\alpha^{2t_\alpha}}\p{\frac{1}{2\pi}\int_{0}^{2\pi}A(y)(\Vec{a},\Vec{a})d\theta +\sqrt{\alpha - 1}\int_0^{2\pi}O\p{\frac{\lambda_\alpha^{2t_\alpha}\abs{\nabla^2_{g_\alpha }u_\alpha}}{\sqrt{\alpha - 1}}}d\theta +  o(1)}\\
				&\quad + \frac{\alpha - 1}{2\pi\lambda_\alpha^{2t_\alpha}}\int_{0}^{2\pi} \frac{\lambda_\alpha^{2t_\alpha}}{\alpha - 1} O\p{\abs{\frac{\partial u_\alpha}{\partial r}}\frac{1}{\lambda_{\alpha}^{t_\alpha}}\abs{\frac{\partial u_\alpha}{\partial \theta}}} d\theta\\
				& =  \frac{\alpha - 1}{\lambda_\alpha^{2t_\alpha}}\p{\frac{1}{2\pi}\int_{0}^{2\pi}A(y)(\Vec{a},\Vec{a})d\theta +\sqrt{\alpha - 1}\int_0^{2\pi}O\p{\abs{\nabla^2_{g_\alpha}v_\alpha}}d\theta + o(1)}\\
				&\quad + \frac{\alpha - 1}{\lambda_\alpha^{2t_\alpha}} \int_{0}^{2\pi} O\p{\abs{\frac{\partial v_\alpha}{\partial r}}(\lambda_\alpha^{t_\alpha},\theta)\cdot \p{\frac{1}{r}\abs{\frac{\partial v_\alpha}{\partial \theta}}}(\lambda_\alpha^{t_\alpha},\theta)}d\theta\\
				& = \frac{\alpha - 1}{\lambda_\alpha^{2t_\alpha}}\p{A(y)(\Vec{a},\Vec{a}) +O\p{\sqrt{\alpha - 1}} + o(1)}\\
				& = \frac{\alpha - 1}{\lambda_\alpha^{2t_\alpha}}\p{A(y)(\Vec{a},\Vec{a}) + o(1)}.
			\end{align*}
			Because $\langle A(y)(\Vec{a}, \Vec{a}), \Vec{a}\rangle = 0$, we have
			\begin{align*}
				-A_{\Gamma_\alpha}(\nabla \gamma_\alpha,\nabla \gamma_\alpha) &= \ddot{\gamma}_\alpha - \frac{\langle \ddot{\gamma}_\alpha, \dot{\gamma}_\alpha \rangle}{\abs{\dot{\gamma}_\alpha}^2} \dot{\gamma}_\alpha = - G_\alpha + \frac{\langle G_\alpha, \dot{\gamma}_\alpha \rangle}{\abs{\dot{\gamma}_\alpha}^2} \dot{\gamma}_\alpha\\
				& = - \frac{\alpha - 1}{\lambda_\alpha^{2t_\alpha}}\p{A(y)(\Vec{a},\Vec{a}) + o(1)}.
			\end{align*}
			which implies $\norm{A_{\Gamma_\alpha}}_{h_\alpha}(\lambda_\alpha^{t_\alpha}) < \infty$. Since $t_\alpha \in [t_1, t_2]$ is a arbitrary sequence, by a contradiction argument similar to Corollary \ref{coro 2} we will obtain that for any $t \in [t_1, t_2]$ 
			\begin{equation*}
				\norm{A_{\Gamma_\alpha}}_{h_\alpha}(\lambda_\alpha^{t}) \leq C,
			\end{equation*}
			which completes the proof of Lemma \ref{ays lemma}.
		\end{proof}
		\begin{lemma}\label{converge to geodesic}
			After choosing a subsequence, the sequence of curves $\Gamma_\alpha \subset \R^K$,  which is defined by $\gamma_\alpha$ and parametrized by its arc length, converges to geodesic a $\gamma : [0, L] \rightarrow (N,h)$ for some $L \in \R_+$, that is, $\gamma$ satisfies the following equation
			\begin{equation*}
				\frac{d^2 \gamma}{ds^2} + A(\gamma)\p{\frac{d \gamma}{ds}, \frac{d \gamma}{ds}} = 0.
			\end{equation*}
		\end{lemma}
		\begin{proof}
			Let $s$ be the arc length parameter of $\gamma_\alpha(t)$ with $s(\lambda_\alpha^{t_\alpha}) = 0$  for $t_\alpha \in [t_1, t_2]$ and 
			\begin{equation*}
				y_\alpha = \gamma_\alpha(\lambda_\alpha^{t_\alpha}) = \frac{1}{2\pi}\int_0^{2\pi}u_\alpha(\lambda_\alpha^{t_\alpha},\theta)d\theta.
			\end{equation*}
			We know that the sequence $\{\gamma_\alpha(\lambda_\alpha^{t_\alpha})\} = \{y_\alpha\}$ is convergent and $\gamma_\alpha(s)$ satisfies equation
			\begin{equation}\label{eq: equation gamma alpha}
				\frac{d^2\gamma_\alpha}{ds^2} + A_{\Gamma_\alpha}(\gamma_\alpha)\left(\frac{d \gamma_\alpha}{ds},\frac{d \gamma_\alpha}{ds}\right) = 0
			\end{equation}
			for $\alpha > 1$. Then, by the uniformly boundedness of $\norm{A_{\Gamma_\alpha}}_{h_\alpha}(\lambda_\alpha^{t})$ obtained in Lemma \ref{ays lemma}  as $\alpha \searrow 1$, $\gamma_\alpha(s)$ converges locally  in $C^1([0,s_1], \R^K)$ to a vector valued function for some small $s_1 > 0$, denoted by $\gamma(s)$ which also parameterized by arc length. To show $\gamma$ is a geodesic, that is, to show  $\gamma(s)$ solves
			\begin{equation*}
				\frac{d^2 \gamma}{ds^2} + A(\gamma)\p{\frac{d \gamma}{ds}, \frac{d \gamma}{ds}} = 0,
			\end{equation*}
			it suffices to show 
			\begin{equation}\label{contradiction 2}
				A_{\Gamma_\alpha}(\gamma_\alpha)\left(\frac{d \gamma_\alpha}{ds},\frac{d \gamma_\alpha}{ds}\right) \longrightarrow A(\gamma)\p{\frac{d \gamma}{ds}, \frac{d \gamma}{ds}}
			\end{equation}
			strongly in $C^0([0,s_1], \R^K)$ for some small enough $s_1 > 0$. By contradiction, if not, then for any arbitrary small  $s_1$,  there always exists a subsequence of $\{u_\alpha\}$ still denoted by $\{u_\alpha\}$ and a sequence of $\lambda_\alpha^{t^\prime_\alpha}$ such that 
			\begin{equation}\label{eq:s prime 1}
				s_\alpha^{\prime} := s(\lambda_\alpha^{t^\prime_\alpha}) = \int_{\lambda_\alpha^{t_\alpha}}^{\lambda_\alpha^{t^\prime_\alpha}} \abs{\dot{\omega}_\alpha(r)}dr \rightarrow s^\prime \in (0,s_1)
			\end{equation}
			and 
			\begin{equation}\label{eq:s prime 2}
				\abs{ A_{\Gamma_\alpha}(\gamma_\alpha)\left(\frac{d \gamma_\alpha}{ds},\frac{d \gamma_\alpha}{ds}\right) - A(\gamma)\p{\frac{d \gamma}{ds}, \frac{d \gamma}{ds}}}_{s = s_\alpha^\prime} > \varepsilon_0 > 0,
			\end{equation}
			for some $\varepsilon_0  > 0$. Furthermore, we can choose small enough $s_1$  and small enough $\alpha - 1$ such that $t_\alpha^\prime \in [\frac{t_1}{2}, t_2]$. In fact, without loss of generality, we assume there exists an integer $T_\alpha$ such that
			\begin{equation*}
				\lambda_{\alpha}^{\frac{t_1}{2}} = 2^{T_\alpha}\lambda_\alpha^{t_\alpha}
			\end{equation*}
			where $T_\alpha \rightarrow \infty$ as $\alpha \searrow 1$. Utilizing Corollary \ref{coro 2}, when $\alpha - 1$ is small enough we have
			\begin{align*}
				\int_{\lambda_\alpha^{t_\alpha}}^{\lambda_\alpha^{\frac{t_1}{2}}}\abs{\dot{\gamma_\alpha}(r)} dr &= \sum_{k = 1}^{T_\alpha} \int_{2^{k - 1}\lambda_\alpha^{t_\alpha}}^{2^k\lambda_\alpha^{t_\alpha}}\abs{\dot{\gamma_\alpha}(r)} dr\\
                &\geq T_\alpha \sqrt{\alpha - 1}\p{\log 2 \sqrt{\frac{E((w^1)}{\pi}} + o(1)}\\
				&\geq C\left(t_\alpha - \frac{t_1}{2}\right) \log{\lambda_\alpha^{-\sqrt{ \alpha - 1 }}}\\
				&\geq C \frac{t_1}{2}\log \nu > 0.
			\end{align*}
			Thus, if we let $s_1 \leq C \frac{t_1}{2} \log{\nu}$, we can make $t_\alpha^\prime \in [\frac{t_1}{2}, t_2]$ when $\alpha - 1$ is small enough. Therefore, we can apply Proposition \ref{prop 2} and Lemma \ref{ays lemma} to yield that 
			\begin{equation}\label{E 48}
				\frac{d \gamma_\alpha}{ds}(s^\prime_\alpha) = \frac{\dot{\gamma}_\alpha(\lambda_\alpha^{t^\prime_{\alpha}})}{|\dot{\gamma}_\alpha(\lambda_\alpha^{t^\prime_{\alpha}})|} \longrightarrow \frac{d \gamma }{d s}(s^\prime)\quad \text{as } \alpha \searrow 1
			\end{equation}
			which furthermore implies that
			\begin{align*}
				\left. A_{\Gamma_\alpha}(\gamma_\alpha)\left(\frac{d \gamma_\alpha}{ds},\frac{d \gamma_\alpha}{ds}\right)\right|_{s = s^\prime_\alpha} &= \left.\frac{1}{\abs{\dot{\gamma}_\alpha\p{\lambda_\alpha^{t^\prime_{\alpha}}}}^2}A_{\Gamma_\alpha}(\gamma_\alpha)\left( \dot{\gamma}_\alpha(r), \dot{\gamma}_\alpha(r)\right)\right|_{r = \lambda_\alpha^{t_\alpha^\prime}}\\
				&\quad \longrightarrow \left.A(\gamma)\p{\frac{d \gamma}{ds}, \frac{d \gamma}{ds}}\right|_{s = s^\prime}\quad \text{as  } \alpha \searrow 1.
			\end{align*}
			This contradicts to the choice of $s_\alpha^\prime$ asserted in \eqref{eq:s prime 1} and \eqref{eq:s prime 2}, hence \eqref{contradiction 2} holds.  Therefore, from the equation \eqref{eq: equation gamma alpha} of $\gamma_\alpha$ and the convergence properties induced from Lemma \ref{ays lemma}:
			\begin{equation*}
				\frac{d \gamma_\alpha}{ds}(r) = \frac{\dot{\gamma}_\alpha(r)}{|\dot{\gamma}_\alpha(r)|} \longrightarrow \frac{d \gamma }{d s}(s)\quad \text{as } \alpha \searrow 1
			\end{equation*}
			and \eqref{contradiction 2} we will get
			\begin{equation*}
				\frac{d \gamma}{ds}(s) - \frac{d \gamma }{ds}(0) = - \int_{0}^s A(\gamma)\p{\frac{d \gamma}{ds}, \frac{d \gamma}{ds}} ds \quad \text{for all } s\in [0,s_1]
			\end{equation*}
			which is the integral formation of geodesic equation. Thus, $\gamma(s)$ is a geodesic which completes  the proof of Lemma \ref{converge to geodesic}.
		\end{proof}
		Now, we are in a position to prove the remaining cases of main Theorem \ref{simple neck analysis}
		\begin{proof}[\textbf{Proof of the Theorem \ref{simple neck analysis} for \bm{$\nu > 1$}}]
			Without loss of generality, we assume 
			\begin{equation*}
				k_\alpha = \frac{t_1 - t_2}{\log 2}\log \lambda_\alpha
			\end{equation*}
			is an integer, that is equivalent to $\lambda_\alpha^{t_1} = 2^{k_\alpha}\lambda_\alpha^{t_2}$,   which tends to infinity as $\alpha \searrow 1$.
			
			\noindent\textbf{Case 1:} \noindent We first consider the case $\nu = \infty$. 
            
            For any $0 < t_1 < t_2 < 1$, by Corollary \ref{coro 2} there holds
			\begin{equation*}
				L\left(\Gamma_\alpha|_{A(2^k\lambda_\alpha^{t_2}, 2^{k + 1}\lambda_\alpha^{t_2}, x_\alpha)}\right) := \int_{2^k\lambda_\alpha^{t_2}}^{2^{k + 1}\lambda_\alpha^{t_2}} |\dot{\gamma}_\alpha (r)| dr \geq \sqrt{\alpha - 1}\p{\log 2\sqrt{\frac{E(w^1)}{\pi}} + o(1) }.
			\end{equation*}
			Then, we can estimate
			\begin{equation*}
				L(\Gamma_\alpha) \geq C k_\alpha \sqrt{\alpha - 1}\p{\log 2\sqrt{\frac{E(w^1)}{\pi}} + o(1) }\geq C \log{\lambda_\alpha^{- \sqrt{\alpha - 1}}} \rightarrow \infty\quad \text{as } \alpha \searrow 1.
			\end{equation*}
			which means in this case the length $L(\Gamma)$ of $\gamma(s)$ is infinite.
			
			\noindent\textbf{Case 2:}  Now, we consider the case $1 < \nu < \infty$.
            
            Note that $1 < \nu < \infty$ implies $\mu = 1$ which is equivalent to the energy identity shown in Theorem \ref{generalized ide}, that is 
			\begin{equation*}
				\lim_{\delta\searrow 0} \lim_{R\rightarrow \infty} \lim_{\alpha \searrow 1} \int_{A(\lambda_\alpha R, \delta, x_\alpha)} \abs{\nabla u_\alpha}^2 dx = 0.
			\end{equation*}
			Then, we can use same estimates as the proof for the case $\nu = 1$, just replacing $\delta$ by $\lambda_\alpha^t$ for any $0 < t_1 \leq t\leq t_2 < 1$, to obtain
			\begin{align}\label{theorem 4.6 osc 1}
				Osc_{A(\lambda_\alpha R,\lambda_\alpha^t,x_\alpha)}(u_\alpha) &\leq C\p{\int_{A(\frac{\lambda_\alpha R}{2}, 2 \delta, x_\alpha)}\abs{\nabla u_\alpha}^2}^{\frac{1}{2}} \nonumber\\
                &+ C\sqrt{\alpha - 1}\p{t - 1}\log{\lambda_\alpha} - C\sqrt{\alpha - 1}\log R\nonumber\\
				&\quad\quad \longrightarrow 0 \quad \text{letting }\alpha\searrow 1\quad \text{then }R \rightarrow \infty, \, \delta\searrow 0 \text{ and }t \rightarrow 1.
			\end{align}
			And similarly,  replacing $\lambda_\alpha R$ by $\lambda_\alpha^t$ for any $0 < t_1 \leq t\leq t_2 < 1$, we obtain
			\begin{align}\label{theorem 4.6 osc 2}
				Osc_{A(\lambda_\alpha^t,\delta,x_\alpha)}(u_\alpha) &\leq C\p{\int_{A(\frac{\lambda_\alpha R}{2}, 2 \delta, x_\alpha)}\abs{\nabla u_\alpha}^2}^{\frac{1}{2}} + C\sqrt{\alpha - 1}\p{\log{\delta} - t\log{\lambda_\alpha} }\nonumber\\
				&\quad\quad \longrightarrow 0 \quad \text{letting }\alpha\searrow 1\quad \text{then }R\rightarrow \infty,\, \delta\searrow 0 \text{ and }t \rightarrow 0.
			\end{align}
			Also, by Corollary \ref{coro 2},  we have
			\begin{equation*}
				L\left(\Gamma_\alpha|_{A(2^k\lambda_\alpha^{t_2}, 2^{k + 1}\lambda_\alpha^{t_2}, x_\alpha)}\right)  = \sqrt{\alpha - 1}\p{\log 2\sqrt{\frac{E(w)}{\pi}} + o(1) },
			\end{equation*}
			which implies 
			\begin{equation*}
				L(\Gamma) = \lim_{\alpha\searrow 1} k_\alpha\sqrt{\alpha - 1}\p{\log 2\sqrt{\frac{E(w)}{\pi}} + o(1) } = (t_2 - t_1)\log{\nu}\sqrt{\frac{E(w)}{\pi}}.
			\end{equation*}
			Now, letting $t_1 \rightarrow 0$ and $t_2 \rightarrow 1$ and keeping in mind that \eqref{theorem 4.6 osc 1}, \eqref{theorem 4.6 osc 2} and Lemma \ref{converge to geodesic}, we know that the neck converges to  a geodesic of length
			\begin{equation*}
				L = \log{\nu}\sqrt{\frac{E(w)}{\pi}}
			\end{equation*}
			which complete the proof of Theorem \ref{simple neck analysis} for the case $\nu > 1$. 
		\end{proof}

		\subsection{Energy Identity for \texorpdfstring{$\alpha$-$H$}{Lg}-surfaces with Bounded Morse Index}\label{subsec morse}
        \ 
        \vskip5pt
		In this subsection, we prove another main consequence --- Theorem \ref{thm convergence}. Before giving the detailed proof, some lemmas are needed.
  
        Let $s$ be the arc length parameter of $\gamma_\alpha(r)$ such that $s(\lambda_\alpha^{t}) = 0$ for some fixed  $0 < t < 1$.  Then as a corollary of the proof of  Lemma \ref{converge to geodesic} and Theorem \ref{analysis on neck}, we have the following result
		\begin{lemma} \label{sec index lem 1}
			Let $\{u_\alpha\}_{\alpha \searrow 1}$ be a sequence satisfying hypothesis of Theorem \ref{simple neck analysis}. Assume that the limiting neck of $\{u_\alpha\}$ is a geodesic of infinite length. Then, for any given $l > 0$ and $\theta \in [0,2\pi]$, $u_\alpha(s,\theta) = u_\alpha(x_\alpha + s(\cos{\theta}, \sin{\theta}))$ converges to $\gamma$ in $C^1([0,l])$. Furthermore, we have 
			\begin{equation}\label{E 49}
				\norm{\frac{r(s)}{\sqrt{\alpha - 1}} \abs{\frac{\partial s}{\partial r}} - \mu^{1 - t} \sqrt{\frac{E(w^1)}{\pi}} }_{C^0([0,l])} \longrightarrow 0\quad \text{as }\, \alpha \searrow 1.
			\end{equation}
			where $r(s)$ is the inverse of the arc length parameter $s(r)$ with $s(\lambda_\alpha^t) = 0$.
		\end{lemma}
		\begin{proof}
			Since the limiting neck of $u_\alpha$ converges to a geodesic of infinite length, we can choose a  real number $\iota \in \R$ such that 
			\begin{equation*}
				s\p{\lambda_\alpha^{t^\iota_\alpha}} = l
			\end{equation*}
			By Corollary \ref{coro 2}, we can estimate
			\begin{align}\label{eq: bound T alpha}
				l = \int_{\lambda_\alpha^{t^\iota_\alpha}}^{\lambda_\alpha^t} \abs{\frac{d \gamma_\alpha(r)}{dr}} dr &= \int_{\lambda_\alpha^{t^\iota_\alpha}}^{\lambda_\alpha^t} \abs{\frac{1}{2\pi}\int_{0}^{2\pi} \frac{d u_\alpha (r,\theta) }{dr} d\theta }dr\nonumber\\
				&\geq C \int_{\lambda_\alpha^{t^\iota_\alpha}}^{\lambda_\alpha^t} \frac{\sqrt{\alpha - 1}}{r} dr = C(t^\iota_\alpha - t)\log \p{\lambda_\alpha^{-\sqrt{\alpha - 1}}} 
			\end{align}
			But, notice that 
			\begin{equation*}
				\log \p{\lambda_\alpha^{-\sqrt{\alpha - 1}}}   \rightarrow \infty \quad \text{as }\, \alpha \searrow 1,
			\end{equation*}
			there must holds $t_\alpha^\iota \rightarrow t$ as $\alpha \searrow 1$. We prove the Lemma \ref{sec index lem 1} by contradiction, suppose that $u_\alpha(s, \theta)$ does not converge to $\gamma$ in $C^1([0,l])$, then, after choosing a subsequence, there exists $\varepsilon_0 > 0$ and a sequence $\{s_\alpha \}_{\alpha \searrow 1} \subset [0,l]$ such that 
			\begin{equation}\label{eq:contra 4.5.1}
				\sup_{\theta \in [0,2\pi]} \abs{ \frac{\partial u_\alpha}{\partial s}(s_\alpha, \theta) - \frac{d \gamma_\alpha(s_\alpha)}{ds} } > \varepsilon_0
			\end{equation}
			Write $s(\lambda_\alpha^{\tilde{t}_\alpha}) = s_\alpha$, then $\tilde{t}_\alpha \in [t, t_\alpha^\iota]$ which implies $\Tilde{t}_\alpha \rightarrow t$. Thus, applying Corollary \ref{coro 2} yields 
			\begin{equation*}
				\frac{\lambda_\alpha^{{\tilde{t}_\alpha}}}{\sqrt{\alpha - 1}}\abs{\frac{\partial u_\alpha}{\partial r}\p{\lambda_\alpha^{{\tilde{t}_\alpha}}, \theta} - \frac{d \gamma_\alpha}{dr}\p{\lambda_\alpha^{{\tilde{t}_\alpha}}}} \longrightarrow 0 \quad \text{as }\, \alpha \searrow 1.
			\end{equation*}
			Note that by Corollary \ref{coro 2}
			\begin{equation*}
				\abs{\frac{d s}{dr}\p{\lambda_\alpha^{{\tilde{t}_\alpha}}}} = \abs{\frac{d \gamma_\alpha}{dr}\p{\lambda_\alpha^{{\tilde{t}_\alpha}}}} \geq \frac{C \lambda_\alpha^{{\tilde{t}_\alpha}}}{\sqrt{ \alpha - 1 }},
			\end{equation*}
			which implies
			\begin{align*}
				\abs{ \frac{\partial u_\alpha}{\partial s}(s_\alpha, \theta) - \frac{d \gamma_\alpha(s_\alpha)}{ds} } &= \abs{\frac{d r}{d s}}\cdot \abs{ \frac{\partial u_\alpha}{\partial r}(s_\alpha, \theta) - \frac{d \gamma_\alpha(s_\alpha)}{dr} }\\
				&\leq \frac{\lambda_\alpha^{{\tilde{t}_\alpha}}}{\sqrt{\alpha - 1}}\abs{\frac{\partial u_\alpha}{\partial r}(\lambda_\alpha^{{\tilde{t}_\alpha}}, \theta) - \frac{d \gamma_\alpha}{dr}(\lambda_\alpha^{{\tilde{t}_\alpha}})} \longrightarrow 0 \quad \text{as }\, \alpha \searrow 1.
			\end{align*}
			This is a contradiction to \eqref{eq:contra 4.5.1}, thus we obtain the converges of first derivatives 
			\begin{equation*}
				\norm{\frac{\partial u_\alpha}{\partial s}(s_\alpha, \theta) - \frac{d \gamma_\alpha(s_\alpha)}{ds}}_{C^0([0,l])}
			\end{equation*}
			The $C^0$ converges of $\gamma_\alpha$ can be obtained by a similar argument using Proposition \ref{prop 2}. Next, the convergence \eqref{E 49} is the direct result of  radical part of Corollary \ref{coro 2}.
		\end{proof}
		Let us recall the definition of stability of a geodesic in Riemannian manifold.
		\begin{defi}
			A geodesic $\gamma(s) : [0,l]\rightarrow N$ is \textit{unstable} if its index form is not  non-negative definite, that is, there exists $V_0 \in \mathscr{V}_\gamma$ such that
			\begin{equation*}
				I_{\gamma}(V_0,V_0) = \int_0^l \langle \nabla_{\gamma^\prime} V_0 ,\nabla_{\gamma^\prime} V_0 \rangle - R(V_0, \gamma^\prime, V_0, \gamma^\prime) ds < 0 
			\end{equation*}
			where $R$ is the Riemann curvature tensor on $N$ and $\mathscr{V}_\gamma$ is the vector space formed by vector fields $V$ along $\gamma$  which are piecewise differentiable and vanish at the end points of $\gamma$, that is, $V(0) = V(l) = 0$.
		\end{defi}
		The following Lemma \ref{sec index lem 1} is of vital importance in the proof of Theorem \ref{thm convergence}
		\begin{lemma}\label{sec index lem 2}
			Let $\{u_\alpha\}_{\alpha \searrow 1}$ be a  sequence satisfying hypothesis of Theorem \ref{thm convergence}. If the necks of $\{u_\alpha\}$ converges to an unstable geodesic $\gamma(s) : [0,l]\rightarrow N$ parameterized by arc length, then for  small enough $\alpha - 1$, there exists a vector field $V_\alpha$ along $u_\alpha$ on $N$, that is, $V_\alpha \in u_\alpha^*(TN)$, which vanishes outside of $A(\lambda_\alpha^{t_\alpha^\iota}, \lambda_\alpha^t,x_\alpha)$, such that the second variation of $E^\omega_\alpha$ acting on $V_\alpha$ is strictly negative, i.e.
			\begin{equation*}
				\delta^2 E^\omega_\alpha(V_\alpha, V_\alpha) < 0.
			\end{equation*}
		\end{lemma}
		\begin{proof}
			By the assumption of Lemma \ref{sec index lem 1}, $\gamma(s) : [0,l]\rightarrow N$ is an unstable geodesic, then there exists a vector field $V_0 \in \mathscr{V}_\gamma$  such that 
			\begin{equation*}
				I_{\gamma}(V_0, V_0) < 0.
			\end{equation*}
			Recall that $\mathcal{P}$ be the projection from $T\R^K$ onto $TN$, more precisely for $y\in N$ $\mathcal{P}_y$ is the orthogonal projection from $T_y\R^K = \R^K$ onto $T_yN \subset T_y\R^K$ . Then we define $V_\alpha$ as
			\begin{equation*}
				V_\alpha(u_\alpha(s,\theta)) = V_\alpha\Big( u\big(x_\alpha + r(s)(\cos{\theta},\sin{\theta})\big)\Big) := \mathcal{P}_{u_\alpha(s,\theta)}\big(V_0(s)\big),
			\end{equation*}
			where $r(s)$ is the inverse function of arc parameter of $\gamma_\alpha(r)$ with $s(\lambda_\alpha^{t}) = 0$. Here, we put $V_0(s)$ as a vector in $\R^K$ which is identified with $T_{u_\alpha(s,\theta)}\R^K$, thus the expression $\mathcal{P}_{u_\alpha(s,\theta)}\big(V_0(s)\big)$ is well defined. Then, $V_\alpha$ is a piecewise smooth vector field along $u_\alpha$ which also vanishes outside of $A( \lambda_\alpha^{t_\alpha^\iota}, \lambda_\alpha^t, x_\alpha)$. From Lemma \ref{sec index lem 1}, for fixed $\theta \in [0,2\pi]$ $V_\alpha(u_\alpha(s,\theta))$ converges to $V_0(\gamma(s))$ in $C^1([0,l])$.  Next, we will compute $\delta^2 E_\alpha^\omega(V_\alpha,V_\alpha)$ and judge its negativity by showing that
			\begin{equation*}
				\lim_{\alpha \searrow 1} \frac{1}{\sqrt{\alpha - 1}} \delta^2 E_\alpha^\omega(V_\alpha,V_\alpha) = 4 \pi \mu \sqrt{\frac{E(w^1)}{\pi}} I_\gamma(V_0, V_0).
			\end{equation*}
			The right hand of identity is negative by our assumption which implies the conclusion of Lemma \ref{sec index lem 2}. 
			
			To this end, we first split the computation into some different parts using second variation formula obtained in Lemma \ref{variation formula} and the conformal coordinates of $M$
			\begin{align*}
				&\delta^2 E^\omega_\alpha(u_\alpha)(V_\alpha,V_\alpha)\\
				&= \alpha \int_{A(\lambda_\alpha^{t_\alpha^\iota}, \lambda_\alpha^t,x_\alpha)} \p{\tau_\alpha + \abs{\nabla_{g_\alpha} u_\alpha}^2}^{\alpha - 1} \Big( \inner{ \nabla V_\alpha, \nabla V_\alpha} - R\p{V_\alpha, \nabla u_\alpha, V_\alpha, \nabla u_\alpha} \Big) d x\\
				& \quad + 2\alpha(\alpha - 1)\int_{A(\lambda_\alpha^{t_\alpha^\iota}, \lambda_\alpha^t,x_\alpha)} \p{\tau_\alpha + \abs{\nabla_{g_\alpha} u_\alpha}^2}^{\alpha - 1}\langle \nabla u_\alpha, \nabla V_\alpha\rangle^2 dx\\
				&\quad + 2 \int_{A(\lambda_\alpha^{t_\alpha^\iota}, \lambda_\alpha^t,x_\alpha)} \inner{ H(\nablap u_\alpha, \nabla V_\alpha), V_\alpha} dx \\
                & \quad + \int_{A(\lambda_\alpha^{t_\alpha^\iota}, \lambda_\alpha^t,x_\alpha)} \inner{ (\nabla_{V_\alpha}H)(\nablap u_\alpha, \nabla u_\alpha), V_\alpha} dx \\
				& = \alpha \int_{A(\lambda_\alpha^{t_\alpha^\iota}, \lambda_\alpha^t,x_\alpha)} \p{\tau_\alpha + \abs{\nabla_{g_\alpha} u_\alpha}^2}^{\alpha - 1} \Big( \inner{ \nabla_{\frac{\partial u_\alpha}{\partial r}} V_\alpha, \nabla_{\frac{\partial u_\alpha}{\partial r}} V_\alpha }\\
                &\hspace{200pt}-R\p{V_\alpha, \nabla_{\frac{\partial u_\alpha}{\partial r}} u_\alpha, V_\alpha, \nabla_{\frac{\partial u_\alpha}{\partial r}} u_\alpha} \Big) d x\\
				& \quad + \alpha \int_{A(\lambda_\alpha^{t_\alpha^\iota}, \lambda_\alpha^t,x_\alpha)} \p{\tau_\alpha + \abs{\nabla_{g_\alpha} u_\alpha}^2}^{\alpha - 1} \frac{1}{r^2} \Big( \inner{ \nabla_{\frac{\partial u_\alpha}{\partial \theta}} V_\alpha, \nabla_{\frac{\partial u_\alpha}{\partial \theta}} V_\alpha } \\
                &\hspace{200pt}-R\p{V_\alpha, \nabla_{\frac{\partial u_\alpha}{\partial \theta}} u_\alpha, V_\alpha,\nabla_{\frac{\partial u_\alpha}{\partial \theta}} u_\alpha} \Big) dx\\
				& \quad+ 2\alpha(\alpha - 1)\int_{A(\lambda_\alpha^{t_\alpha^\iota}, \lambda_\alpha^t,x_\alpha)} \p{\tau_\alpha + \abs{\nabla_{g_\alpha} u_\alpha}^2}^{\alpha - 2}\inner{ \nabla u_\alpha, \nabla V_\alpha}^2 dx\\
				&\quad +  \int_{A(\lambda_\alpha^{t_\alpha^\iota}, \lambda_\alpha^t,x_\alpha)} 2\inner{ H(\nablap u_\alpha, \nabla V_\alpha), V_\alpha} +  \inner{ (\nabla_{V_\alpha}H)(\nablap u_\alpha, \nabla u_\alpha), V_\alpha} dx\\
				&:= I_1 + I_2 + I_3 + I_4
			\end{align*}
			where $I_i$ represents the $i$-th integral of above identity. First, we consider $I_1$ and observe that
			\begin{align*}
				\mu^{t} \longleftarrow \sup_{{A( \lambda_\alpha^{t_\alpha^\iota}, \lambda_\alpha^t, x_\alpha)}} \left(C \frac{1}{\lambda_\alpha^{2t_\alpha^\iota}}\right)^{\alpha - 1}&=\sup_{{A(\lambda_\alpha^{t_\alpha^\iota}, \lambda_\alpha^t,x_\alpha)}}\p{\tau_\alpha + \abs{\nabla_{g_\alpha} u_\alpha}^2}^{\alpha - 1}\\
				&\leq  \sup_{{A( \lambda_\alpha^{t_\alpha^\iota}, \lambda_\alpha^t, x_\alpha)}} \left(\tau_\alpha + C \frac{1}{\lambda_\alpha^{2t_\alpha^\iota}}\right)^{\alpha - 1} \longrightarrow \mu^t \quad \text{as }\, \alpha \searrow 1.
			\end{align*}
			since in Lemma \ref{sec index lem 1} we have concluded that $t_\alpha^\iota \rightarrow t$ as $\alpha\searrow 1$. Utilizing \eqref{E 49}, we can estimate $I_1$ as below
			\begin{align*}
				&\lim_{\alpha\searrow 1} \frac{I_1}{\sqrt{\alpha - 1}}\\
                &= \lim_{\alpha\searrow 1} \frac{2\alpha}{\sqrt{\alpha - 1}} \int_{A(\lambda_\alpha^{t_\alpha^\iota}, \lambda_\alpha^t,x_\alpha)} \p{\tau_\alpha + \abs{\nabla_{g_\alpha} u_\alpha}^2}^{\alpha - 1} \Big( \inner{ \nabla_{\frac{\partial u_\alpha}{\partial r}} V_\alpha, \nabla_{\frac{\partial u_\alpha}{\partial r}} V_\alpha } \\
				&\hspace{50mm} -R\p{V_\alpha, \nabla_{\frac{\partial u_\alpha}{\partial r}} u_\alpha, V_\alpha, \nabla_{\frac{\partial u_\alpha}{\partial r}} u_\alpha} \Big) d x\\
				&=\lim_{\alpha\searrow 1}\int_0^{2\pi}\int_0^l \Big( \inner{ \nabla_{\frac{\partial u_\alpha}{\partial r}} V_\alpha, \nabla_{\frac{\partial u_\alpha}{\partial r}} V_\alpha }\\
				& \hspace{25mm}  -R\p{V_\alpha,\nabla_{\frac{\partial u_\alpha}{\partial r}} u_\alpha, \nabla_{\frac{\partial u_\alpha}{\partial r}} u_\alpha, V_\alpha} \Big) \p{1 + \abs{\nabla_{g_\alpha} u_\alpha}^2}^{\alpha - 1}\frac{\abs{\frac{\partial s}{\partial r}} r(s)}{\sqrt{\alpha - 1}} ds d\theta\\
				& = \mu \sqrt{\frac{E(w^1)}{\pi}} \lim_{\alpha\searrow 1} \int_0^{2\pi} \int_0^l \Big( \inner{ \nabla_{\frac{\partial u_\alpha}{\partial r}} V_\alpha, \nabla_{\frac{\partial u_\alpha}{\partial r}} V_\alpha } -R\p{V_\alpha, \nabla_{\frac{\partial u_\alpha}{\partial r}} u_\alpha, V_\alpha, \nabla_{\frac{\partial u_\alpha}{\partial r}} u_\alpha} \Big) ds d\theta\\
				& = 2\pi \mu \sqrt{\frac{E(w^1)}{\pi}} I(V_0,V_0).
			\end{align*}
			Here, we used the fact that  $V_\alpha(u_\alpha(s,\theta))$ converges to $V_0(\gamma(s))$ in $C^1([0,l])$ and $u_\alpha(s,\theta)$  converges to $\gamma$ in $C^1([0,l])$ for fixed $\theta \in [0,2\pi]$, see Lemma \ref{sec index lem 1}. Before calculating $I_2$, we note that 
			\begin{align}\label{E 50}
				\nabla_{\frac{\partial u_\alpha}{\partial \theta}} V_\alpha &=  \mathcal{P}_{u_\alpha(s,\theta)} \left(\frac{\partial V_\alpha}{\partial \theta}\right) = \mathcal{P}_{u_\alpha(s,\theta)}\p{\frac{\partial}{\partial \theta}\p{\mathcal{P}_{u_\alpha(s,\theta)}(V_0)}}\nonumber\\
				&= \mathcal{P}_{u_\alpha(s,\theta)}\p{\frac{\partial}{\partial \theta}\p{\mathcal{P}_{u_\alpha(s,\theta)}}(V_0)}
			\end{align}
			where $\frac{\partial V_\alpha}{\partial \theta}$ is taken in $\R^K$. This implies 
			\begin{equation}\label{E 50 1}
				\abs{\nabla_{\frac{\partial u_\alpha}{\partial \theta}} V_\alpha} \leq C_l \abs{\frac{\partial u_\alpha}{\partial \theta}}
			\end{equation}
            for some constant $C_l$ depending on $l$.
			Given $R > 0$, we take 
			\begin{equation*}
				T_\alpha = \left[ \frac{\log {\lambda_\alpha^{|t - t_\alpha^\iota|}}}{\log R}\right] + 1
			\end{equation*}
			and by choice of $T_\alpha$, one can see that 
			\begin{equation*}
            A\p{\lambda_\alpha^{t_\alpha^\iota},\lambda_\alpha^t, x_\alpha} \subset \bigcup_{i = 1}^{T_\alpha} A\p{R^{i - 1}\lambda_\alpha^{t_\alpha^\iota}, R^i\lambda_\alpha^{t_\alpha^\iota}, x_\alpha}.
			\end{equation*}
			Then, we can compute $I_2$ 
			\begin{align*}
				\lim_{\alpha \searrow 1} \frac{I_2}{\sqrt{\alpha - 1}} &= \lim_{\alpha \searrow 1} \int_{A(\lambda_\alpha^{t_\alpha^\iota},\lambda_\alpha^t, x_\alpha)}  \p{\tau_\alpha + \abs{\nabla_{g_\alpha} u_\alpha}^2}^{\alpha - 1} \Big( \inner{ \nabla_{\frac{\partial u_\alpha}{\partial \theta}} V_\alpha, \nabla_{\frac{\partial u_\alpha}{\partial \theta}} V_\alpha }\\
				&\hspace{40mm}-R\p{V_\alpha, \nabla_{\frac{\partial u_\alpha}{\partial \theta}} u_\alpha, V_\alpha, \nabla_{\frac{\partial u_\alpha}{\partial \theta}} u_\alpha} \Big) \frac{dr}{r} d\theta\\
				&\leq \lim_{\alpha \searrow 1} \frac{C}{\sqrt{\alpha - 1}} \int_{A(\lambda_\alpha^{t_\alpha^\iota},\lambda_\alpha^t, x_\alpha)} \frac{1}{|x - x_\alpha|^2}\abs{\frac{\partial u_\alpha}{\partial \theta}}^2 dx\\
				& \leq \lim_{\alpha \searrow 1} \frac{C}{\sqrt{\alpha - 1}} \sum_{i = 1}^{T_\alpha} \int_{A(R^{i - 1}\lambda_\alpha^{t_\alpha^\iota}, R^i\lambda_\alpha^{t_\alpha^\iota}, x_\alpha)} \frac{1}{|x - x_\alpha|^2}\abs{\frac{\partial u_\alpha}{\partial \theta}}^2 dx\\
				&\leq \lim_{\alpha \searrow 1} C T_\alpha\sqrt{\alpha - 1}\sup_{\tau \in [t-\varepsilon, t+\varepsilon]}\frac{1}{\alpha - 1} \int_{A{\frac{1}{R} \lambda_\alpha^\tau, \lambda_\alpha R^\tau}}\frac{1}{|x - x_\alpha|^2}\abs{\frac{\partial u_\alpha}{\partial \theta}}^2 dx.
			\end{align*}
			By the choice of $T_\alpha$, we have 
			\begin{equation*}
				\lim_{\alpha \searrow 1} T_\alpha \sqrt{\alpha - 1} \leq C(R)\lim_{\alpha \searrow 1}\abs{t - t_\alpha^\iota} \log{\lambda_{\alpha}^{- \sqrt{\alpha - 1}}} + \sqrt{\alpha - 1}
			\end{equation*}
			which is bounded by \eqref{eq: bound T alpha}. Thus, by Corollary \ref{theta sup} we can conclude
			\begin{equation*}
				\lim_{\alpha \searrow 1} \frac{I_2}{\sqrt{\alpha - 1}} = 0.
			\end{equation*}
			For the third integral $I_3$, by Cauchy-Schwartz inequality we can straightforward estimate that
			\begin{align*}
				\frac{I_3}{\sqrt{\alpha - 1}} &= 2\alpha \sqrt{\alpha - 1} \int_{A(\lambda_\alpha^{t_\alpha^\iota},\lambda_\alpha^t, x_\alpha)} \p{\tau_\alpha + \abs{\nabla_{g_\alpha} u_\alpha}^2}^{\alpha - 2}\langle \nabla u_\alpha, \nabla V_\alpha\rangle^2 dx\\
				&\leq 2\alpha \sqrt{\alpha - 1} \int_{A(\lambda_\alpha^{t_\alpha^\iota},\lambda_\alpha^t, x_\alpha)} \p{\tau_\alpha + \abs{\nabla_{g_\alpha} u_\alpha}^2}^{\alpha - 2} \abs{\nabla u_\alpha}^2 \abs{\nabla V_\alpha}^2 dx \\
				&\leq C \sqrt{\alpha - 1} \int_{A(\lambda_\alpha^{t_\alpha^\iota},\lambda_\alpha^t, x_\alpha)} \p{\tau_\alpha + \abs{\nabla_{g_\alpha} u_\alpha}^2}^{\alpha}\\
				&\leq C\sqrt{\alpha - 1} \longrightarrow 0 \quad \text{as }\, \alpha \searrow 1.
			\end{align*}
			Here, similar to derivation of \eqref{E 50} and \eqref{E 50 1},  we used the estimate
			\begin{equation}\label{E 51}
				\abs{\nabla V_\alpha} \leq C_l \abs{\nabla u_\alpha}.
			\end{equation}
		  At last, for the fourth integral $I_4$, we split $I_4$ into $I_{41}$ an $I_{42}$ and consider $I_{41}$ firstly
			\begin{align*}
				&\frac{I_{41}}{\sqrt{\alpha - 1}}\\
				&:=  \frac{1}{\sqrt{\alpha - 1}} \int_{A(\lambda_\alpha^{t_\alpha^\iota}, \lambda_\alpha^t,  x_\alpha)} 2\inner{ H(\nablap u_\alpha, \nabla V_\alpha), V_\alpha}  dx \\
				& = \frac{1}{\sqrt{\alpha - 1}} \int_{A(\lambda_\alpha^{t_\alpha^\iota}, \lambda_\alpha^t,  x_\alpha)} 2H^k_{ij} \nablap u_\alpha^i \nabla V_\alpha^j V_\alpha^k dx\\
				& \leq \frac{2}{\sqrt{\alpha - 1}} \norm{H}_{L^\infty(N)} \norm{V}_{C^0(N)}\sum_{i,j}^K\int_{A(\lambda_\alpha^{t_\alpha^\iota}, \lambda_\alpha^t,  x_\alpha)}\frac{1}{|x - x_\alpha|}\abs{\frac{\partial u^i_\alpha}{\partial r} \frac{\partial V^j_\alpha}{\partial \theta} - \frac{\partial V_\alpha^j}{\partial r}\frac{\partial u_\alpha^i}{\partial \theta}} dx\\
				& \leq C \frac{1}{\sqrt{\alpha - 1}}\int_{A(\lambda_\alpha^{t_\alpha^\iota}, \lambda_\alpha^t,  x_\alpha)} \frac{1}{|x - x_\alpha|}\p{ \abs{\frac{\partial V_\alpha}{\partial \theta}} \abs{\frac{\partial u_\alpha}{\partial r}} + \abs{\frac{\partial u_\alpha}{\partial \theta}} \abs{\frac{\partial V_\alpha}{\partial r}} } d x\\
				& \leq  C \frac{1}{\sqrt{\alpha - 1}}\int_{A(\lambda_\alpha^{t_\alpha^\iota}, \lambda_\alpha^t,  x_\alpha)} \frac{1}{|x - x_\alpha|} \abs{\frac{\partial u_\alpha}{\partial \theta}} \abs{\frac{\partial u_\alpha}{\partial r}} dx\\
				& \leq C \int_{A(\lambda_\alpha^{t_\alpha^\iota}, \lambda_\alpha^t,  x_\alpha)} \frac{1}{|x - x_\alpha|^2} \abs{\frac{\partial u_\alpha}{\partial \theta}} dx\\
				&= o(1) \sqrt{\alpha - 1} \int_{\lambda_\alpha^{t_\alpha^\iota}}^{\lambda_\alpha^t} \frac{1}{r} dr\\
				&= -(t_\alpha^\iota - t)\sqrt{\alpha - 1} \log{\lambda_\alpha} \,o(1) = o(1)
			\end{align*}
			Here, we have used the estimate \eqref{E 51} and Corollary  \ref{coro 2}. For the second part $I_{42}$ of $I_4$ we have similarly computations
			\begin{align*}
				\frac{I_{42}}{\sqrt{\alpha - 1}} &= \int_{A(\lambda_\alpha^{t_\alpha^\iota}, \lambda_\alpha^t,x_\alpha)} \inner{ (\nabla_{V_\alpha}H)(\nablap u_\alpha, \nabla u_\alpha), V_\alpha} dx\\
				&=\frac{1}{\sqrt{\alpha - 1}}\int_{A(\lambda_\alpha^{t_\alpha^\iota}, \lambda_\alpha^t,x_\alpha)} \p{\frac{\partial H^k_{ij}}{\partial y^l } + \frac{\partial H_{jl}^k}{\partial y^i} + \frac{\partial H^k_{il}}{\partial y^l}} \nablap u_\alpha^i \nabla u_\alpha^j V_\alpha^l V_\alpha^k dx\\
				&\leq C \frac{1}{\sqrt{\alpha - 1}} \norm{\nabla H}_{L^\infty(N)} \norm{V_0}_{L^\infty(N)}^2\cdot \\
				&\quad \int_{A(\lambda_\alpha^{t_\alpha^\iota}, \lambda_\alpha^t,  x_\alpha)} \frac{1}{|x - x_\alpha|}\p{ \abs{\frac{\partial u_\alpha}{\partial \theta}} \abs{\frac{\partial u_\alpha}{\partial r}} + \abs{\frac{\partial u_\alpha}{\partial \theta}} \abs{\frac{\partial u_\alpha}{\partial r}} } d x \\
				& = o(1)
			\end{align*}
			Combining all computations for $I_1$, $I_2$, $I_3$ and $I_4 = I_{41} + I_{42}$, we conclude that 
			\begin{equation*}
				\lim_{\alpha \searrow 1} \frac{1}{\sqrt{\alpha - 1}} \delta^2 E_\alpha^\omega(V_\alpha,V_\alpha) = 4 \pi \mu \sqrt{\frac{E(w^1)}{\pi}} I_\gamma(V_0, V_0).
			\end{equation*}
			which implies the conclusion of Lemma \ref{sec index lem 2}.
		\end{proof}
		Under the assumption of Theorem \ref{thm convergence}, $(N,h)$ has finite fundamental group.  Then, the Gromov's estimates \cite{Gromov1978} (See also \cite[Corollary 3.3.5]{moore2017}) on the length of geodesic $\gamma$ and its Morse index hold:
		\begin{equation*}
			\mathrm{Length}(\gamma) \leq C_0 \p{\mathrm{Ind}(\gamma) + 1} \leq C \p{C_{I} + 1}
		\end{equation*}
		for some universal constant  $C_0  > 1$ and $C_I$ is the uniformly Morse index upper bound of $\gamma$. Thus. any geodesic $\gamma \subset (N,h)$ with length
		\begin{equation*}
			\mathrm{Length}(\gamma) >  C_0
		\end{equation*}
		is unstable. Therefore, if we choose $l > C_0$ , then any geodesic $\gamma(s): [0,l] \rightarrow N$  with arc length parameter  is  unstable and under the assumption that $\gamma$ is of infinite length we can apply above Lemma \ref{sec index lem 2} successively to finish the proof of our Theorem \ref{thm convergence}. More precisely, we have:
		\begin{proof}[\textbf{Proof of Theorem \ref{thm convergence}}]
			By contradiction, suppose that  the neck of $u_\alpha$ converges to a geodesic $\gamma$ with infinite length. Then, we take $l > C_0$ and $u_\alpha(s,\theta)$ converges to an unstable geodesic $\gamma : [0,l]\rightarrow N$ in $C^1([0,l])$ for any $\theta \in [0,2\pi]$ by Lemma \ref{sec index lem 1}. Using same notation as Lemma \ref{sec index lem 1} and Lemma \ref{sec index lem 2}, let $s(\lambda_\alpha^t) = 0$ and $\delta_1 > 0$ such that  $s(\lambda_\alpha^{t + \delta_1}) = l$. Since $t_\alpha^\iota \rightarrow t$ as $\alpha \searrow 1$, for arbitrary small $\varepsilon > 0$ when $\alpha$ tends to $1$ close enough there holds  that $\abs{ t^\iota - t} < \varepsilon$. Therefore, by Lemma \ref{sec index lem 2} there exists a vector fields $V_\alpha^1$ along $u_\alpha$, that is vanishes outside $A(\lambda_\alpha^{t + \delta_1}, \lambda_\alpha^t,x_\alpha)$, such that 
			\begin{equation*}
				\delta^2 E^\omega_\alpha(V^1_\alpha,V^1_\alpha) < 0 \quad \text{for all }\, \alpha \leq \alpha_1
			\end{equation*}
			for some small enough $\alpha_1 - 1$. Since the limiting neck is a geodesic with infinite length, then we replace $t$ by $t + \delta_1$ and apply Lemma \ref{sec index lem 2} again on $A(\lambda_\alpha^{t+\delta_1 + \delta_2}, \lambda_\alpha^{t+\delta_1}, x_\alpha)$ for some $\delta_2 > 0$ to find  second vector field $V^2_\alpha$, that vanishes outside $A(\lambda_\alpha^{t+\delta_1 + \delta_2}, \lambda_\alpha^{t+\delta_1}, x_\alpha)$, such that 
			\begin{equation*}
				\delta^2 E^\omega_\alpha(V^2_\alpha,V^2_\alpha) < 0 \quad \text{for all }\, \alpha \leq \alpha_2
			\end{equation*}
            for some small enough $\alpha_2 - 1 \leq \alpha_1 - 1$.
			This process can keep going continuously and  for any integer $L > 0$ we can construct a collection of vector fields $\{V^1_\alpha, V^2_\alpha, \dots , V^L_\alpha\}$ and small  $\alpha_L -1$ satisfying
			\begin{equation*}
				\delta^2 E^\omega_\alpha(V^i_\alpha,V^i_\alpha) < 0 \quad \text{for all }\, \alpha \leq \alpha_L \,\text{ and } \, 1 \leq i \leq L.
			\end{equation*}
			Since the support of $V^i_\alpha$ are disjoint each other, $V^1_\alpha, V^2_\alpha, \dots , V^L_\alpha$ are linearly independent which implies that 
			\begin{equation*}
				\mathrm{Ind}_{E^\omega_\alpha}(u_\alpha) \geq L, \quad \text{ for any } L\geq 0.
			\end{equation*}
			Thus, $\mathrm{Ind}_{E^\omega_\alpha}(u_\alpha) \rightarrow \infty$ as $\alpha \searrow 1$ which contradicts to the uniformly bounded assumptions of $\mathrm{Ind}_{E^\omega_\alpha}(u_\alpha)$. We conclude that the limiting necks of $\{u_\alpha \}$ consists of  geodesics with finite length and by Theorem \ref{analysis on neck} we know that the energy identity holds, completing the proof of Theorem \ref{thm convergence}.
		\end{proof}
	
		\vskip2cm
		\section{Existence of \texorpdfstring{$H$}{Lg}-Sphere of Bounded Morse Index}\label{section 5}
		\vskip10pt
		In this section, we  prove our main results using the convergence schemes developed in Section \ref{section 4} and combining with the existence results obtained in Section \ref{sec: 3 non-constant u alpha}. In the following, $\varepsilon_0$ is a uniform constant depending on the geometries of $N$ and mean curvature vectors $\lambda H$, which is assumed to be the minimum of the constants appearing in the previous results.
  
		\subsection{Existence of Minimizing \texorpdfstring{{$H$}}{Lg}-Surfaces.} \label{section 5.0}
		\
		\vskip5pt
		In the first instance, let us consider a relatively simpler scenario where $ \| \omega \| _ { L^ \infty (N) } < 1 $. In such a case, the surgery construction appeared in \cite {sacks1981existence} can be applied in the $H$-surface setting to rule out the occurrence of bubbles.
		\begin{proof}[\textbf{Proof of Theorem \ref{thm: homotopy1}}]
			Since $E^\omega_\alpha$ satisfies the Palais-Smale condition, see Lemma \ref{P-S condition} and  Corollary \ref{coro palais}, and by the upper bound of $\| \omega \| _ { L^ \infty (N) } < 1 $, we can take a minimizing map $u_\alpha:M\rightarrow N$  for $E^\omega_\alpha$ in a fixed non-trivial homotopy class in $W^{1,2\alpha}(M,N)$ with 
			\begin{equation*}
				E_\alpha(u_\alpha) \leq C E^\omega_\alpha(u_\alpha) \leq C \p{1 + B^2}^\alpha + C \norm{\omega}_{L^\infty(N)} B^2 \mathrm{Vol}(M)
			\end{equation*}
			where $B = \max_{x \in M} |\nabla u(x)|$ and $u$ is a smooth map in that homotopy class.
			By Theorem \ref{thm convergence}, we can choose a subsequence, which still denoted by $u_\alpha$, such that $u_\alpha\rightarrow u$ in $C^2(M-\{x_1,\cdots, x_l\}, N)$ for some $l \in \mathbb N$  and $u:M\rightarrow N$ is a $H$-surface. Next we prove that there is actually no energy concentration point for $\{u_\alpha\}$, that is, $u_\alpha\rightarrow u$ in $C^2(M,N)$.
			
			Take a small ball centering at $x_i$ in $M$ of radius $\rho$ where $\rho$ is small enough such that $x_j\notin B(x_i,\rho)$ for $1\leq j\neq i \leq l$ and will be determined more precisely later. Let $\varphi(r)$ be a smooth function which is 1 on $r\geq 1$ and 0 on $r\leq \frac{1}{2}$ and $\exp:TN\rightarrow N$ be the exponential map on $N$. Then we can define
			\begin{equation}\label{eq; construct u hat}
				\hat{u}_\alpha(x) = \left\{
				\begin{aligned}
					&u(x)\quad &\text{if\quad}& 0 \leq |x| \leq \frac{\rho}{2}\\
					&\exp_{u(x)}\left(\varphi\left(\frac{|x|}{\rho}\right)\exp^{-1}_{u(x)}\circ\, u_\alpha(x)\right) \quad &\text{if\quad}& \frac{\rho}{2} < |x| < \rho\\
					&u_\alpha(x)\quad &\text{if\quad}& |x| \geq \rho
				\end{aligned}
				\right.
			\end{equation}
			which agrees with $u_\alpha$ near the boundary of $B(x_i,\rho)$ and with $u$ near the center $x_i$.
			Then 
			\begin{equation*}
				u_\alpha\rightarrow u\quad  \text{in}\quad  C^2\left(\supp \varphi \left(\frac{|x|}{\rho}\right)\cap B(x_i,\rho),N\right),
			\end{equation*}
			and we have $\hat{u}_\alpha\rightarrow u$ in $C^2(B(x_i,\rho), N)$ which implies
			\begin{align}
				\label{E 56}
				\lim_{\alpha\searrow 1} E^\omega_\alpha(\hat{u}_\alpha,B(x_i,\rho)) -\frac{1}{2}\mathrm{Vol}(M) &= \frac{1}{2}\lim_{\alpha\searrow 1} \int_{B(x_i,\rho)} \p{1 + |\nabla \hat{u}_\alpha|^2}^\alpha - 1 dV_g \nonumber\\
				&\quad+ \lim_{\alpha\searrow 1} \int_{B(x_i,\rho)} (\hat{u}_\alpha)^*\omega\nonumber\\
				&=  E(u, {B(x_i,\rho)}) + \int_{B(x_i,\rho)} u^*\omega
			\end{align}
			By assumption $\pi_2(N) = 0$, which implies that every $u_1,\, u_2 \in C^0(B(x_i,\rho),N)$ with $u_1|_{\partial B(x_i,\rho)} = u_2|_{\partial B(x_i,\rho)}$ are homotopic,  $u_\alpha$ and $\hat{u}_\alpha$ are homotopic. Since $u_\alpha$ is a minimizing map for $E^\omega_\alpha$ in its homotopy class, we have
			\begin{equation*}
				E^\omega_\alpha(u_\alpha, {B(x_i,\rho)})\leq E^\omega_\alpha(\hat{u}_\alpha,{B(x_i,\rho)})
			\end{equation*}
			Applying \eqref{E 56} we get 
			\begin{align*}
				\limsup_{\alpha \rightarrow 1} E(u_\alpha, \, &{B(x_i,\rho)}) + \frac{1}{2}\mathrm{Vol}(M)\\
				&\leq
				\limsup_{\alpha \rightarrow 1} E^\omega_\alpha(u_\alpha, {B(x_i,\rho)}) - \liminf_{\alpha\searrow 1}\int_{B(x_i,\rho)} (u_\alpha)^*\omega\\
				&\leq  \limsup_{\alpha \rightarrow 1} E_\alpha^\omega(\hat{u}_\alpha, {B(x_i,\rho)}) - \liminf_{\alpha\searrow 1}\int_{B(x_i,\rho)} (u_\alpha)^*\omega\\
				&= E^\omega(u,B(x_i,\rho)) + \frac{1}{2}\mathrm{Vol}(M) - \liminf_{\alpha\searrow 1}\int_{B(x_i,\rho)} (u_\alpha)^*\omega\\
				& \leq C \pi\rho^2\|\nabla u\|_\infty +  \norm{\omega}_{L^\infty(N)} \liminf_{\alpha\searrow 1} E(u_\alpha, B(x_i,\rho)) + \mathrm{Vol}(M).
			\end{align*}
			If we initially choose $\rho$ small enough such that
			\begin{equation*}
				C\pi\rho^2\|\nabla u\|_\infty \leq \frac{\varepsilon_0}{2(1 - \norm{\omega}_{L^\infty(N)})}
			\end{equation*}
			and keeping in mind that $ \norm{\omega}_{L^\infty(N)} < 1 $, then we can utilize small energy regularity Lemma \ref{lem4.1} to conclude that $u_{\alpha}$ converges to $u$ in $C^2(B(x_i, \rho), N)$, given that $E(u_{\alpha}, B(x_i, \rho)) < \varepsilon_0^2$ for $\alpha$ sufficiently close to 1. Therefore, an induction argument tells us that the convergence can be extended over the points $\{x_1,\cdots,x_l\}$ and hence we can conclude $u_\alpha \rightarrow u$ in $C^2(M,N)$. Since $u_\alpha$ minimizes $E^\omega_\alpha$, $u$ must minimize the $E^\omega$ in the same homotopy class.
		\end{proof}
		Now, we are in a position to prove the Theorem \ref{thm: homotopy2}.
		\begin{proof}[\textbf{Proof of Theorem \ref{thm: homotopy2}}]
			Let $\mathscr{C}$ be the set of free homotopy classes containing minimizing $H$-sphere and $G$ be the subgroup of $\pi_2(N)$ generated by the elements of $\mathscr C$. If $G \neq \pi_2(N)$, then there exists a homotopy class that does not contain the minimizing $H$-sphere. Let 
			\begin{equation*}
				\mathcal{C}_1 = \left\{u \in C^1(\S^2, N)\,:\, \text{the corresponding free homotopy class  }[u] \not \in G\right\}
			\end{equation*}
			then by Corollary \ref{coro palais} we can find a sequence of maps $\{u_{\alpha}\}_{\alpha > 1}$ which are minimizers for each $E^\omega_{\alpha}$ in $\mathcal{C}_1$. By a similarly argument as proof of Theorem \ref{thm: homotopy1}, there is a constant $C$ such that $E_{\alpha}(u_{\alpha})\leq C$. Then by small energy regularity Lemma \ref{lem4.1} either there is a subsequence converges strongly in $C^2$ to a non-constant $H$-sphere $u : \S^2 \rightarrow N$ such that $u \in \mathcal{C}_1$ and $E^\omega(u) = \inf_{v \in \mathcal{C}_1} E(v)$, or there exists some energy concentration point, saying $x_1 \in \S^2$.
			
			Let us consider the second case. Pick a small disk $B(x_1,r)$ near blow-up point $x_1$, by energy gap Lemma \ref{blow 2}, there exists a $\varepsilon_0 > 0$ such that $E(u_{\alpha}) \geq \varepsilon_0^2$ provided that $\alpha - 1$ is small enough. Then, we define
			\begin{align*}
				s_{\alpha}(x) &= \left\{
				\begin{aligned}
					&u_{\alpha}(x), \quad &\, &\text{when }x \in \S^2 \backslash B(x_1,r)\\
					&\hat{u}_{\alpha}(x) \quad &\, &\text{when }x \in  B(x_1,r)
				\end{aligned}
				\right.\\
				w_{\alpha}(x) &= \left\{
				\begin{aligned}
					&\hat{u}_{\alpha} \circ f(x), \quad &\, &\text{when }x \in \S^2 \backslash B(x_1,r)\\
					&u_{\alpha}(x) \quad &\, &\text{when }x \in  B(x_1,r)
				\end{aligned}
				\right.
			\end{align*}
			where $\hat{u}_{\alpha}$ is constructed as \eqref{eq; construct u hat} and $f : \S^2 \backslash B(x_1,r) \rightarrow B(x_1,r)$ is the conformal reflection preserving the boundary $\partial B(x_1,r)$ fixed. Thus, $s_{\alpha}$ agrees with $u_{\alpha}$ outside $B(x_1,r)$ while $w_{\alpha}$ agrees with $u_{\alpha}$ inside $B(x_1,r)$. Next, by conformality of $f$, we have 
			\begin{align*}
				\lim_{\alpha \searrow 1} E_{\alpha}^\omega(s_{\alpha}) &= \lim_{\alpha \searrow 1}E_{\alpha}^\omega(u_{\alpha}, \S^2 \backslash B(x_1,r)) + E(u, B(x_1,r)) + \frac{1}{2}\mathrm{Vol}(B(x_1,r)),\\
				\lim_{\alpha \searrow 1} E_{\alpha}^\omega(w_{\alpha}) &= \lim_{\alpha \searrow 1}E_{\alpha}^\omega(u_{\alpha},  B(x_1,r)) + E(u, B(x_1,r)) + \frac{1}{2}\mathrm{Vol}(B(x_1,r)).
			\end{align*}
			Therefore, we can choose small enough $r > 0$ and small enough $\alpha - 1$ such that 
			\begin{align*}
				E_{\alpha}^\omega(s_{\alpha}) &\leq E_{\alpha}^\omega\left(u_{\alpha},\S^2 \backslash B(x_1,r)\right) + \frac{\delta}{3}\\
				E_{\alpha}^\omega(w_{\alpha}) &\leq E_{\alpha}^\omega \left(u_{\alpha}, B(x_1,r)\right) + \frac{\delta}{3}
			\end{align*}
			for some $\delta > 0$. Let $[s_{\alpha}]$ and $[w_{\alpha}]$ be the free homotopy classes of $s_{\alpha}$ and $w_{\alpha}$, respectively. Then, $[u_\alpha] \subset [s_\alpha] + [w_\alpha]$ and we can conclude that
			\begin{equation*}
				\inf_{v \in [s_{\alpha}]} E_{\alpha}^\omega(v) +  \inf_{v \in [w_{\alpha}]} E_{\alpha}^\omega(v) \leq E_{\alpha}^\omega(s_{\alpha}) + E_{\alpha}^\omega(w_{\alpha}) < E_{\alpha}^\omega(u_{\alpha}) + \frac{2\delta}{3} = \inf_{v \in \mathcal{C}_1} E_{\alpha}^\omega(v) + \frac{2\delta}{3},
			\end{equation*}
			which implies that
			\begin{equation*}
				\inf_{v \in [s_{\alpha}]} E_{\alpha}^\omega(v) \leq \inf_{v \in \mathcal{C}_1} E_{\alpha}^\omega(u_{\alpha}) - \frac{\varepsilon_0^2}{4} \quad \text{and}\quad \inf_{v \in [w_{\alpha}]} E_{\alpha}^\omega(v) \leq \inf_{v \in \mathcal{C}_1} E_{\alpha}^\omega(u_{\alpha}) - \frac{\varepsilon_0^2}{4},
			\end{equation*}
			where we used the Proposition \ref{non constant limit} to conclude that $E(s_\alpha) \geq \varepsilon_0^2$ and $E(w_\alpha) \geq \varepsilon_0^2$. Thus, $[u_\alpha] \neq [s_\alpha]$ and $[u_\alpha] \neq [w_\alpha]$, in particular, $[s_\alpha]$ and $[w_\alpha]$ are both non-trivial. Furthermore, by the choice of $\mathcal{C}_1$, the free homotopy classes $[s_{\alpha}]$ and $[w_{\alpha}]$ must belong to $G$ which further implies $[u_{\alpha}] \in G$ for $[u_\alpha] \subset [s_\alpha] + [w_\alpha]$. This contradicts to the choice of $u_{\alpha}$, hence the conclusion of Theorem \ref{thm: homotopy2} holds.
		\end{proof}
  
		\subsection{Existence of \texorpdfstring{{$H$}}{Lg}-Sphere of Bounded Morse Index for Generic Choice of \texorpdfstring{{$\lambda \in \R_+$}}{Lg}} \label{section 5.1}
		\ 
		\vskip5pt
		
		In this subsection, we complete the proof of main Theorem \ref{main theorem 1}.
		Before presenting into the detailed proofs, we would like to demonstrate that it is possible to modify the values of a finite number of points on a non-constant $H$-sphere $u: \S^2 \rightarrow N$ without affecting its Morse index, see \cite[Lemma in Section 4]{Micallef-Moore-1988} for the setting of $\alpha$-harmonic maps and Gulliver-Lawson \cite[Proposition 1.9]{Gulliver-Lawson1984} for more general consequences.
		\begin{lemma}\label{lem vanishes}
			Let $m$ be the Morse index of a non-constant $H$-sphere $u:\S^2 \rightarrow N$. For any finite points $\{x_1, x_2, \cdots, x_l\}$ in $N$, there exists a $m$-dimensional linear subspace $\mathscr{V}$ of $\Gamma(u^*TN)$ such that
			\begin{enumerate}
				\item \label{lem vanishes item 1} The index form
				\begin{align*}
					I_u(V,V)&= \int_{\S^2} \Big( \langle \nabla V, \nabla V \rangle - R(V,\nabla u, \nabla u, V) \Big) d V_g + 2 \int_M \inner{ H(\nablap u, \nabla V), V} dV_g\\
					& \quad + \int_{\S^2}\inner{(\nabla_VH)(\nablap u,\nabla u),V}dV_g,\quad \text{for } V \in \mathscr{V},
				\end{align*}
				of $u$ is negative definite on $\mathscr{V}$.
				\item \label{lem vanishes item 2} Given any $V \in \mathscr{V}$, $V$ vanishes in some neighborhood of $x_i$, for every $1 \leq i \leq l$.
			\end{enumerate}
		\end{lemma}
		\begin{proof}
			By assumption, we can find a $m$-dimensional linear subspace $\mathscr{V}_0$ of $\Gamma(u^*TN)$ on which the index form $I_u$ is negative definite. Then, choose a small enough $\rho > 0$ such that $B(x_i,\rho) \cap B(x_j,\rho) = \emptyset$ for $1\leq i\neq j\leq l$ and the distant function $r_i : B(x_i,\rho)\backslash\{x_i\} \rightarrow \R_{+}$ is smooth. Take $0 < \varepsilon < \min\{\rho, 1\}$ and a series of piecewise smooth functions $\varphi_i(r_i):\S^2 \rightarrow [0,1]$ can be defined as following
			\begin{equation*}
				\varphi_i(r_i) =
				\left\{
				\begin{aligned}
					&0 \quad &\text{if } 0 \leq r_i < \varepsilon^2,\\
					& \frac{2\log{\varepsilon} - \log {r_i}}{\log{\varepsilon}} \quad &\text{if } \varepsilon^2 \leq r_i \leq \varepsilon ,\\
					& 1 \quad &\text{otherwise}.
				\end{aligned}
				\right.
			\end{equation*}
			Then, a straightforward computation yields that 
			\begin{equation*}
				\int_{\S^2} \abs{\nabla \varphi_i} dV_g \leq C \int_0^{2\pi}\int_{\varepsilon^2}^\varepsilon \frac{1}{|\log {\varepsilon}|} dr d \theta \leq \frac{ C\varepsilon}{|\log {\varepsilon}|} \quad \text{and} \quad \int_{\S^2} \abs{\nabla \varphi_i}^2 dV_g \leq  \frac{C}{|\log {\varepsilon}|}.
			\end{equation*}
			Let $\varphi = \Pi_{i = 1}^l \varphi_i$ and for each $V \in \mathscr{V}$ we estimate that 
			\begin{align*}
				I_u(\varphi V,\varphi V) &=  \int_{\S^2}\p{ \inner{\nabla(\varphi V), \nabla(\varphi V)} - \varphi^2 R(V,\nabla u, \nabla u, V) } dV_g\\
				& \quad + 2 \int_{\S^2} \inner{ H(\nablap u, \nabla (\varphi V)), \varphi V} dV_g + \int_{\S^2}\inner{(\nabla_{\varphi V}H)(\nablap u,\nabla u), \varphi V}dV_g\\
				& \leq\int_{\S^2} \varphi^2 \Big( \langle \nabla V, \nabla V \rangle - R(V,\nabla u, \nabla u, V) \Big) d V_g \\
				& \quad  + 2 \int_{\S^2} \varphi^2 \inner{ H(\nablap u, \nabla V), V} dx \\
				&\quad + \int_{\S^2}\varphi^2\inner{(\nabla_V H)(\nablap u,\nabla u),V}dV_g + \frac{C \varepsilon}{|\log {\varepsilon}|} \sup_{x \in \S^2}\p{|V(x)| + |\nabla V(x)|}\\
				& \quad + \frac{C}{|\log \varepsilon|} \sup_{x \in \S^2}|V(x)|^2 + \frac{C \sqrt{E(u)} \sup_{x \in \S^2}|V(x)|^2 }{|\log \varepsilon|^{\frac{1}{2}}}.
			\end{align*}
			Noting that by the construction of $\varphi$ which equals to $1$ away from $x_i$ and vanishes in each small neighborhood of $x_i$, thus we must have 
			\begin{equation*}
				I_u(\varphi V,\varphi V) < 0 \quad \text{for each } V \in \mathscr{V_0}
			\end{equation*}
			as long as $\varepsilon > 0$ is small enough. And $\mathscr{V} := \varphi \mathscr{V}_0$ is a $m$-dimensional linear space of $\Gamma(u^*TN)$ satisfying the desired conclusions \eqref{lem vanishes item 1} and \eqref{lem vanishes item 2} of Lemma \ref{lem vanishes}.
		\end{proof}
		We are now ready to complete the proof for the main Theorem \ref{main theorem 1}. As per Theorem \ref{thm convergence}, we only need to establish the upper bound for the Morse index of $H$-spheres mentioned in Theorem \ref{main theorem 1}. We adapt the convergence scheme established in \cite[Theorem 2 in Section 4]{Micallef-Moore-1988} where the authors proved a similar Morse index upper bound for sequences of $\alpha$-harmonic maps.
		\begin{proof}[\textbf{Proof of Theorem \ref{main theorem 1}}]
			By Corollary \ref{coro:section 3 summary}, for almost every $\lambda \in (0,\infty)$, we can find a sequence of non-constant critical points $\{u_{\alpha_j}\}_{j \in \mathbb{N}}$ for $E^{\lambda\omega}_{\alpha_j}$ such that the Morse index of $E^{\lambda\omega}_{{\alpha_j}}$ at $u_{\alpha_j}$ bounded from above by $k-2$ and the ${\alpha_j}$-energy of $u_{\alpha_j}$ is uniformly bounded as $j \rightarrow \infty$. Then, by Theorem \ref{thm convergence}, after passing to a subsequence, $u_{\alpha_j}$ converges strongly in $C^2(\S^2,N)$ except a finite many singular points $\{x_1, x_2, \cdots , x_l\}$ to a smooth $\lambda H$-sphere $u:\S^2 \rightarrow N$. In the following part of proof Theorem \ref{main theorem 1}, without ambiguity we simply write $H$ to denote $\lambda H$ and $\omega$ to denote $\lambda \omega$ for notation simplicity. Then, it suffices to show that the Morse index of limit map $u$ is at most $k-2$ if $l = 0$, or to show that the Morse index of bubbles is no more than $k-2$ if $l \geq 1$. To this end, we split the proof into two steps.
			\ 
			\vskip5pt
			\step \label{theorem 1 step 1} The Morse index of weak limit $u$ is at most $k-2$.
			\ 
			\vskip5pt
			In this step, it suffices to show $\mathrm{Ind}_{E^\omega}(u) : = m \leq k-2$. Note that our argument is non-vacuous, since in viewing of Proposition \ref{non constant limit} the weak limit $u$ is always non-constant.
			
			By previous Lemma \ref{lem vanishes}, there exists $m$ linearly independent vector fields $V_1, V_2, \dots$ $ , V_m$ such that the index form of $u$ is negative definite on subspace $\mathrm{Span}\{V_1, V_2, \dots , V_m\}$ and $V_1, V_2, \dots , V_m \in u^*TN$ vanish in neighborhoods of singular points $\{x_1, x_2, \cdots , x_l\}$. We then consider the commutative diagram of vector bundles
			\[
			\begin{tikzcd}
				& u^*TN \arrow{r} \arrow{d} & \Pi_2^* TN \arrow{d} \arrow{r} & TN \arrow{d} \\
				& \S^2 \arrow{r}{(i,u)} & \S^2\times N \arrow{r}{\Pi_2} & N
			\end{tikzcd}
			\]
			where $\Pi_2$ is the projection to the second variable and $i : \S^2\rightarrow \S^2$ is the identity map. Then we extend $V_1,V_2, \dots ,V_m$ to smooth vector fields $\dbl{V}_1,\dbl{V}_2, \dots, \dbl{V}_m \in \Pi_2^*TN$ that supported in a tubular neighborhood of $(i,u)(\S^2)$, and set 
			\begin{equation*}
				W^l_{\alpha_j} = (i,u_{\alpha_j})^*(\dbl{V}_l), \quad \text{for } 1 \leq l \leq m.
			\end{equation*}
			The sequence $W_{\alpha_j}^l$ can be regarded as a map from $\S^2$ to the tangent bundle of $N$ such that $W^l_{\alpha_j}(p) \in T_{u_{\alpha_j}(p)}N$, then
			\begin{equation}\label{E 57}
				W_{\alpha_j}^l = (i,u_{\alpha_j})^*(\dbl{V}_l) \longrightarrow (i,u)^*(\dbl{V}_l) = V_l,\quad \text{in }C^2\left(\S^2,u_{\alpha_j}^*(TN)\right).
			\end{equation}
			Now we apply the second variation formula of $E^\omega_{\alpha_j}$ to study the asymptotic properties of Morse index $\mathrm{Ind}_{E^\omega_{\alpha_j}}(u_{\alpha_j})$ as ${\alpha_j} \searrow 1$. By Corollary \ref{coro hessian}, we have 
			\begin{align*}
				&\delta^2 E_{\alpha_j}^\omega(u_{\alpha_j})(W_{\alpha_j}^p,W_{\alpha_j}^q)\\
				&= {\alpha_j} \int_{\S^2} \p{1 + \abs{\nabla u_{\alpha_j}}^2}^{{\alpha_j} - 1} \Big( \inner{ \nabla W_{\alpha_j}^p, \nabla W_{\alpha_j}^q } - R\p{W_{\alpha_j}^p,\nabla u_{\alpha_j},  W_{\alpha_j}^q, \nabla u_{\alpha_j}} \Big) d V_g\\
				& \quad + 2{\alpha_j}({\alpha_j} - 1)\int_{\S^2} \p{1 + \abs{\nabla u_{\alpha_j}}^2}^{{\alpha_j} - 2}\inner{ \nabla u_{\alpha_j}, \nabla W_{\alpha_j}^p } \inner{ \nabla u_{\alpha_j}, \nabla W_{\alpha_j}^q }  dV_g\\
				&\quad + \int_{\S^2} \p{ \inner{ H(\nablap u_{\alpha_j}, \nabla W_{\alpha_j}^p), W_{\alpha_j}^q} + \inner{ H(\nablap u_{\alpha_j}, \nabla W_{\alpha_j}^q), W_{\alpha_j}^p} }  dV_g\\
				&\quad + \frac{1}{2}\int_{\S^2} \inner{\p{\nabla_{W^p_{\alpha_j}}H}(\nablap u_{\alpha_j},\nabla u_{\alpha_j}), W^q_{\alpha_j}}dV_g\\
				&\quad + \frac{1}{2}\int_{\S^2} \inner{\p{\nabla_{W^q_{\alpha_j}}H}(\nablap u_{\alpha_j},\nabla u_{\alpha_j}), W^p_{\alpha_j}}dV_g
			\end{align*}
			for $1\leq p , q \leq m$.  By the choice of $W^l_{\alpha_j}$ and \eqref{E 57}, we can estimate
			\begin{align*}
				2{\alpha_j}({\alpha_j} - 1)&\int_{\S^2} \p{1 + \abs{\nabla u_{\alpha_j}}^2}^{{\alpha_j} - 2}\abs{\inner{ \nabla u_{\alpha_j}, \nabla W_{\alpha_j}^p } \inner{ \nabla u_{\alpha_j}, \nabla W_{\alpha_j}^q } } dV_g\\
				&\quad \leq 2{\alpha_j}({\alpha_j} - 1)\int_{\S^2} \p{1 + \abs{\nabla u_{\alpha_j}}^2}^{{\alpha_j} - 1} \norm{\nabla W^p_{\alpha_j}}_{L^\infty(\S^2)} \norm{\nabla W^q_{\alpha_j}}_{L^\infty(\S^2)} dV_g\\
				& \quad  \leq C ({\alpha_j} - 1)\quad \longrightarrow 0 \quad \text{as } {\alpha_j} \searrow 1,
			\end{align*}
			which implies that
			\begin{align*}
				\delta^2E_{\alpha_j}^\omega(u_{\alpha_j})(W_{\alpha_j}^p\,,W_{\alpha_j}^q) \, \longrightarrow \, & \int_{\S^2}  \Big( \inner{ \nabla V_p\,, \nabla V_q } - R(V_p,\nabla u, V_q, \nabla u) \Big) d V_g\\
				&\quad + \int_{\S^2} \Big( \inner{ H(\nablap u, \nabla V_p), V_q } + \inner{ H(\nablap u, \nabla V_q), V_p } \Big)  dV_g\\
				&\quad + \frac{1}{2}\int_{\S^2} \inner{(\nabla_{V^p}H)(\nablap u,\nabla u), V^q }dV_g\\
				&\quad + \frac{1}{2}\int_{\S^2} \inner{(\nabla_{V^q}H)(\nablap u,\nabla u), V^p }dV_g\\
				& = I_u(V_p,V_q).
			\end{align*}
			Here, $I_u$ is the index form of $E^\omega$ at $u$. Therefore, when ${\alpha_j} - 1$ is small enough, $\delta^2 E^\omega_{\alpha_j}(u_{\alpha_j})$ is also negatively definite on $\mathrm{span}\{W^1_{\alpha_j},W^2_{\alpha_j}, \dots , W^m_{\alpha_j}\}$, that is, $m = \mathrm{Ind}_{E^\omega}(u) \leq k-2$, as desired.
			\ 
			\vskip5pt
			\step The Morse index of bubbles is less than $k-2$ if the number of energy concentration points $l \geq 1$.
			\ 
			\vskip5pt
			By energy gap Lemma \ref{blow 2}, we have 
			\begin{equation*}
				E(u_{\alpha_j}) \geq \varepsilon_0^2,\quad \text{for all } j\in \mathbb{N}, 
			\end{equation*}
			which means the set of singular points is non-empty, say $x_1$ is an energy concentration point. Thus, by Theorem \ref{thm convergence}, there exists a non-constant $H$-sphere $v : \S^2 \rightarrow N$ obtained by rescaling the sequence $v_{\alpha_j}(x) = u_{\alpha_j}(x_{\alpha_j} + \lambda_{\alpha_j} x) $ where 
			$$\frac{1}{\lambda_{\alpha_j}} = \max_{x \in B(x_1, r_0)} |\nabla u_{\alpha_j}|$$
			for some small $r_0 > 0$ and $x_{\alpha_j}$ is the point such that the maximum is take on. Then $v_{\alpha_j} : B(0,\lambda_{\alpha_j}^{-1}r_0) \rightarrow N$ is a critical points of a new functional
			\begin{equation*}
				\dbl{E}^\omega_{\alpha_j}(u_{\alpha_j}) = \frac{1}{2}\int_{B(0,r_0 \lambda_{\alpha_j}^{-1})} \p{\lambda_{\alpha_j}^2  + |\nabla v_{\alpha_j}|^2}^{{\alpha_j}} dV_{g_{\alpha_j}}+ \lambda_{\alpha_j}^{2{\alpha_j} - 2} \int_{B(0,r_0 \lambda_{\alpha_j}^{-1})}(v_{\alpha_j})^*\omega
			\end{equation*}
			where 
			$$g_{\alpha_j} := e^{\varphi(x_{\alpha_j} + \lambda_{\alpha_j} x)}\left((dx^1)^2 + (dx^2)^2\right)$$
			converges to the Euclidean metric on $\R^2$ as ${\alpha_j}\searrow 1$. Near the energy concentration point $x_1$, we note that
			\begin{equation*}
				\mathrm{Ind}_{E^\omega_{\alpha_j}}(u_{\alpha_j}) = \mathrm{Ind}_{\dbl{E}^\omega_{\alpha_j}}(v_{\alpha_j}) \quad \text{for}\quad  E^\omega_{\alpha_j}(u_{\alpha_j}) = \lambda_{\alpha_j}^{2 - 2{\alpha_j}}\dbl{E}^\omega_{\alpha_j}(v_{\alpha_j}).
			\end{equation*}
			As in Step \ref{theorem 1 step 1}, it follows from the Lemma \ref{lem vanishes} that there exists $m$ linearly independent vector fields $V_1, V_2, \dots , V_m$ on $\S^2$ such that the Morse index of $v$ is negative definite on subspace $\mathrm{Span}\{V_1, V_2, \dots , V_m\}$ and vanish in neighborhoods of the branch points of $v$ and vanish around the infinite point $\infty \in \S^2$. Then we extend $V_1,V_2, \dots ,V_m$ to smooth vector fields $\dbl{V}_1,\dbl{V}_2, \dots, \dbl{V}_m \in \Pi_2^*TN$ that supported in a tubular neighborhood of $(i,v)(\S^2)$, and set 
			\begin{equation*}
				\dbl{W}^l_{\alpha_j} = (i,v_{\alpha_j})^*(\dbl{V}_l), \quad \text{for } 1 \leq l \leq m.
			\end{equation*}
			Similar to the computation of Lemma \ref{variation formula}, we obtain the second variation of the functional $\dbl{E}^\omega_{\alpha_j}$
			\begin{align*}
				&\delta^2 \dbl{E}_{\alpha_j}^\omega(v_{\alpha_j})(\dbl{W}_{\alpha_j}^p,\dbl{W}_{\alpha_j}^q)\\
				&= {\alpha_j} \int_{ B(0,\lambda_{\alpha_j}^{-1}r_0)} \p{\lambda_{\alpha_j}^2 + \abs{\nabla v_{\alpha_j}}^2}^{{\alpha_j} - 1} \Big( \inner{ \nabla \dbl{W}_{\alpha_j}^p, \nabla \dbl{W}_{\alpha_j}^q }\\
				&\quad \quad \quad \quad\quad \quad\quad \quad\quad \quad\quad \quad\quad \quad\quad \quad\quad \quad\quad \quad - R\p{\dbl{W}_{\alpha_j}^p,\nabla v_{\alpha_j}, \dbl{W}_{\alpha_j}^q, \nabla v_{\alpha_j}} \Big) d V_{g_{\alpha_j}}\\
				& \quad + 2{\alpha_j}({\alpha_j} - 1)\int_{ B(0,\lambda_{\alpha_j}^{-1}r_0)} \p{\lambda_{\alpha_j}^2 + \abs{\nabla v_{\alpha_j}}^2}^{{\alpha_j} - 2}\inner{ \nabla v_{\alpha_j}, \nabla \dbl{W}_{\alpha_j}^p } \inner{ \nabla v_{\alpha_j}, \nabla \dbl{W}_{\alpha_j}^q}  dV_{g_{\alpha_j}}\\
				&\quad + \lambda_{\alpha_j}^{2{\alpha_j} - 2}\int_{ B(0,\lambda_{\alpha_j}^{-1}r_0)} \Big( \inner{ H(\nablap v_{\alpha_j}, \nabla \dbl{W}_{\alpha_j}^p), \dbl{W}_{\alpha_j}^q} + \inner{ H(\nablap v_{\alpha_j}, \nabla \dbl{W}_{\alpha_j}^q), \dbl{W}_{\alpha_j}^p } \Big)  dV_{g_{\alpha_j}}\\
				&\quad + \frac{ \lambda_{\alpha_j}^{2{\alpha_j} - 2}}{2}\int_{ B(0,\lambda_{\alpha_j}^{-1}r_0)}\inner{\p{\nabla_{\dbl{W}^p_{\alpha_j}}H}(\nablap v_{\alpha_j},\nabla v_{\alpha_j}),\dbl{W}^q_{\alpha_j}}dV_{g_{\alpha_j}}\\
				&\quad +  \frac{ \lambda_{\alpha_j}^{2{\alpha_j} - 2}}{2}\int_{ B(0,\lambda_{\alpha_j}^{-1}r_0)}\inner{\p{\nabla_{\dbl{W}^q_{\alpha_j}}H}(\nablap v_{\alpha_j},\nabla v_{\alpha_j}),\dbl{W}^p_{\alpha_j}}dV_{g_{\alpha_j}}
			\end{align*}
			and observe that 
			\begin{align*}
				2{\alpha_j}&({\alpha_j} - 1)\int_{ B(0,\lambda_{\alpha_j}^{-1}r_0)} \p{\lambda_{\alpha_j}^2 + \abs{\nabla v_{\alpha_j}}^2}^{{\alpha_j} - 2}\abs{\inner{ \nabla v_{\alpha_j}, \nabla \dbl{W}_{\alpha_j}^p } \inner{e \nabla u_{\alpha_j}, \nabla \dbl{W}_{\alpha_j}^q} } dV_{g_{\alpha_j}}\\
				& \leq 4{\alpha_j}({\alpha_j} - 1)\int_{ B(0,\lambda_{\alpha_j}^{-1}r_0)} \p{1 + \abs{\nabla v_{\alpha_j}}^2}^{{\alpha_j} - 1} \norm{\nabla \dbl{W}^i_{\alpha_j}}_{L^\infty(\S^2)} \norm{\nabla \dbl{W}^j_{\alpha_j}}_{L^\infty(\S^2)} dV_{g_{\alpha_j}}\\
				&   \leq C ({\alpha_j} - 1)\quad \longrightarrow \,\, 0 \quad \text{as } {\alpha_j} \searrow 1.
			\end{align*}
			Because $\max_{x \in B(0,\lambda_{\alpha_j}^{-1}r_0)}\abs{\nabla v_{\alpha_j}} = 1$ and $\lambda_{\alpha_j} ^{2 - 2{\alpha_j}} \rightarrow \mu = 1$ as ${\alpha_j} \searrow 1$, see Theorem \ref{thm convergence}, we conclude that 
			\begin{align*}
				\delta^2 \dbl{E}^\omega_{\alpha_j}(v_{\alpha_j})\quad \longrightarrow \quad & \int_{\S^2}  \Big( \langle \nabla \dbl{V}_i\,, \nabla \dbl{V}_j\rangle - R(\dbl{V}_i,\nabla v, \dbl{V}_j, \nabla v) \Big) dV_g\\
				&\quad + \int_{\S^2}\Big( \langle H(\nablap v, \nabla \dbl{V}_i), \dbl{V}_j\rangle + \langle H(\nablap v, \nabla \dbl{V}_j), \dbl{V}_i\rangle \Big)  dV_g\\
				&\quad + \frac{1}{2}\int_{\S^2} \inner{\p{\nabla_{\dbl{V}^p} H}(\nablap v,\nabla v), \dbl{V}^q }dV_g\\
				&\quad + \frac{1}{2}\int_{\S^2} \inner{\p{\nabla_{\dbl{V}^q} H}(\nablap v,\nabla v), \dbl{V}^p }dV_g\\
				& = I_v(V_p,V_q),
			\end{align*}
			where we used the conformally invariance of index form $I_v$ to change the integral domain from $\R^2$ to $\S^2$.
			Therefore, when ${\alpha_j} - 1$ is small enough, $\delta^2 \dbl{E}^\omega_{\alpha_j}(v_{\alpha_j})$ is negatively definite on 
			$$\mathrm{span}\set{\dbl{W}^1_{\alpha_j},\dbl{W}^2_{\alpha_j}, \dots , \dbl{W}^m_{\alpha_j}},$$
			that is, $m = \mathrm{Ind}_{\dbl{E}^\omega}(v) \leq k-2$. Therefore, we complete the proof of the main Theorem \ref{main theorem 1}.
		\end{proof}
		\subsection{Existence of \texorpdfstring{$H$}{Lg}-Sphere  under Ricci Curvature Assumption When \texorpdfstring{{$\mathrm{dim}(N)$}}{Lg} \texorpdfstring{{$ = 3$}}{Lg}}\label{section 5.2}
		\ 
		\vskip5pt
		
		In this subsection, we prove the part \eqref{main theorem 2 part 1} of Theorem \ref{main theorem 2}, more precisely, we aim to show that there exists a $H$-sphere in $N$ for every choice of prescribed mean curvature $H$ satisfying \eqref{eq:intro ricci condi} with Morse index at most $1$. To this end, we first combine the Ricci curvature condition with Morse index estimates to obtain an energy bound for $H$-spheres which is uniformly for $H$, see Proposition \ref{prop:index comparison} and Proposition \ref{prop:uniform energy bound} and then we can pass the limit to obtain the desired $H$-sphere for every $H$ sastisfying \eqref{eq:intro ricci condi}.
		
		In order to get the uniformly energy bound, we aim to build an index comparison for variable prescribed mean curvature $H$, see \cite[Theorem 1.1]{Ejiri-Micallef2008} for the case of minimal surfaces and \cite[Proposition 5.1]{Cheng2020ExistenceOC} within the context of CMC surfaces.
		
		Before stating the detailed results, we recall some  fundamental concepts about complex vector bundle and introduce our notations which inherits from \cite[Section 2]{Ejiri-Micallef2008} and \cite[Section 5]{Cheng2020ExistenceOC}. 
		
		Let $u: \S^2 \rightarrow N$ be a $H$-sphere and denote the pull back bundle $u^*TN$ simply by $E$. We represent the Riemannian metric on $TN$ by $\inner{\cdot,\cdot}$, which can be complex bi-linearly extended to $T_{\mathbb{C}}N:=TN\otimes_{\R}\mathbb{C}$. The Levi-Civita connection $\nabla$ on $TN$ is also extended complex linearly to $T_{\mathbb{C}}N$ and is compatible with the Hermitian metric $\inner{\cdot,\overline{\cdot}}$ on $T_{\mathbb{C}}N$. Here, we use the same notations, $ \inner {\cdot, \cdot } $ and $ \nabla $, regardless of whether they are defined on $ TN $ or $ T_ { \mathbb C } N $. When the dimension of $N$ is three, it is worth noting that the mean curvature type vector field $H \in \Gamma(\wedge^2(N) \otimes TN)$ can be identified with a function defined on $N$. Therefore, the Euler-Lagrange equation \eqref{eq H-surface intro} of $H$-sphere is written as \eqref{eq: H-surface N = 3} and the corresponding second variation formula becomes
		\begin{align}\label{eq:second vari N=3}
			\delta^2 E^\omega(u)(V,V) &= \int_{\S^2} \inner{\nabla V,\nabla V} - R(V,\nabla u,V,\nabla u) dV_g\nonumber \\
			& \quad + 2\int_{\S^2} H\inner{\nabla_{\partial_{x^1}}V\wedge u_{x^2} + u_{x^1}\wedge \nabla_{\partial_{x^2}}V,V} dx^1dx^2\nonumber\\
			&\quad +2\int_{\S^2} (\nabla_V H)\inner{u_{x^1}\wedge u_{x^2},V} dx^1dx^2
		\end{align}
		in this scenario. Then, let $z = x^1 + \sqrt{-1}x^2$ be a local complex coordinate of $\S^2$, we can rewrite the conformal $H$-sphere equations \eqref{eq: H-surface N = 3} and \eqref{eq: H surface 2} as 
		\begin{align}
			\nabla_{\partial_{\bar{z}}} u_z &= \sqrt{-1}H(u_{\Bar{z}}\wedge u_z),\label{eq 16}\\
			\inner{u_z,u_z} &= 0.\label{eq 17}
		\end{align}
		The solution $u : \S^2\rightarrow N$ to \eqref{eq 16} and \eqref{eq 17} is a branched immersion and at each branch point $p$ using coordinate $z = x^1 + \sqrt{-1}x^2$ $u_z$ can be locally represented as 
		\begin{equation*}
			u_z = z^{b_p} V
		\end{equation*}
		where $b_p \in \mathbb{N}$ is the order of branching at $p$ and $V$ is a local section of $E\otimes\mathbb{C} := \bm E$ with $V(p) \neq 0$. This allows us to define the ramified tangent bundle $\xi$ on $\S^2$, that is, $\xi$ is the tangent bundle of $\S^2$ twisted at the branch points by the amount equal to $b_p$  such that $E = \xi\oplus \nu$ where $\nu$ is the normal bundle of $\S^2$ in $N$. The complex structure on $\S^2$ gives $\xi$ the structure of a complex line bundle and induces the splitting 
		$$\xi_{\mathbb{C}}: = \xi\otimes_{\mathbb{R}}\mathbb{C} = \xi^{1,0}\oplus \xi^{0,1}$$
		where the fibres of $\xi^{1,0}$ and $\xi^{0,1}$ are locally spanned by $u_z$ and $u_{\Bar{z}}$ away from the branch points of $u$. The connection $ \nabla $ on $ E $ gives rise to metric compatible connections $ \nabla ^ \top $ on $ \xi $ and $ \nabla ^ { \perp } $ on $ \nu $.
		
		Inspired by the construction of compared bi-linear functional described in \cite[Theorem 2.1]{Ejiri-Micallef2008} and \cite[Equation (5.4)]{Cheng2020ExistenceOC}, we define the following bi-linear form to compared with $\delta^2 E^\omega(u)$. For any $s \in \Gamma(\nu)$ we let
		\begin{equation*}
			B_{\omega}(u)(s,s) =\int_{\S^2}|\nabla f|^2 - {|\nabla u|^2}\Big( {|H|^2} + \frac{\mathrm{Ric}(\bm n, \bm n)}{2} - |\nabla H| \Big)f^2 dV_{g}
		\end{equation*}
		where $\bm{n}\in \Gamma(\nu)$ is the unit normal section of $\nu$, $f = \inner{s,\bm{n}} \in C^\infty(\S^2,\R)$, $\mathrm{Ric}_h(N)$ (in the following abbreviated as $\mathrm{Ric}$) is the Ricci curvature tensor of target manifold $(N,h)$. The index of $B_{\omega}(u)$ is naturally defined to be the maximum of dimension of the linear subspace of $\Gamma(\nu)$ on which $B_{\omega}(u)$ is negative definite. The following proposition is the first main result in this subsection.
		\begin{prop}\label{prop:index comparison}
			Let $(N,h)$ be a 3-dimensional Riemannian manifold, then for non-constant solution $u$ to~\eqref{eq 16} and~\eqref{eq 17}, the index of $B_{\omega}(u)$ is no more than $\mathrm{Ind}_{E^\omega}(u)$.
		\end{prop}
		Before proving the Proposition \ref{prop:index comparison}, we modify computations in \cite[Theorem 2.1]{Ejiri-Micallef2008} and \cite[Lemma 5.2]{Cheng2020ExistenceOC} to obtain the following relations between $\delta^2 E^\omega(u)$ and $B_{\omega}(u)$:
		\begin{lemma}\label{lem:index computations}
			For any $\sigma\in \Gamma(\xi)$ and $s \in \Gamma(\nu)$, we define $\eta \in \Gamma\big(\wedge^{1,0}(\S^2)\otimes\xi^{0,1}\big)$ by 
			\begin{equation*}
				\eta = \p{(\nabla_{\partial_z} s)^{0,1} + \nabla^\top_{\partial_z}\sigma^{0,1}}dz.
			\end{equation*}
			Then
			\begin{equation}\label{eq 18}
				\delta^2 E^{\omega}(u)(s + \sigma, s+\sigma) \leq B_{\omega}(u)(s, s) + 4\int_{\S^2} |\eta|^2  dV_g.
			\end{equation}
		\end{lemma}
		\begin{proof}
			First,  recalling  Lemma \ref{lem vanishes} it suffices to prove~\eqref{eq 18} in the case where $s$ and $\sigma$ are supported away from the set  of branch points $\mathcal{B}$. And, the Riemann uniformization theorem enables us to simplify our computations by assuming that they are always performed in a isothermal coordinate, denoted as $(x^1, x^2)$, with a metric of $ds^2 = \kappa^2((dx^1)^2 + (dx^2)^2)$. To begin, keeping in mind that
			\begin{align*}
				|\nabla_{\partial_{x^1}} v|^2 + |\nabla_{\partial_{x^2}} v|^2 = 4 |\nabla_{\partial_z} v|^2, \quad  u_{x^1}\wedge u_{x^2} = 2\sqrt{-1} u_{\bar{z}}\wedge u_{z}
			\end{align*}
			and
			\begin{equation*}
				\nabla_{\partial_{x^1}}v\wedge u_{x^2} + u_{x^1}\wedge \nabla_{\partial_{x^2}} v = 4\im (u_{z}\wedge \nabla_{\partial_{\bar{z}}} v) = 4\re \p{-\sqrt{-1}u_{z}\wedge \nabla_{\partial_{\bar{z}}} v}, 
			\end{equation*}
			we rewrite the second variational formula \eqref{eq:second vari N=3} of $E^\omega$  by representing $\delta^2 E^\omega(u)(v, v)$ in terms of complex coordinates as follows:
			\begin{align}\label{eq:complex second variation}
				\delta^2 E^\omega(u)(v, v) &= 4\int_{\S^2} |\nabla_{\partial_z} v|^2 - \langle R(v, u_{z})u_{\bar{z}}, v \rangle dx^1  dx^2\nonumber\\
				&\quad +8\int_{\S^2} \re\inner{ \nabla_{\partial_z}v, \overline{\sqrt{-1}H(u_z\wedge v)} } dx^1 dx^2\nonumber\\
				&\quad + 4 \int_{\S^2}\im\inner{v,(\nabla_v H)(u_z\wedge u_{\bar{z}})} dx^1dx^2
			\end{align}
			where $v:=s+\sigma$.  Taking derivative to the identity $\inner{s,u_{\bar{z}}} = 0$ and utilizing equation \eqref{eq 16} to yields
			\begin{equation*}
				\inner{\nabla_{\partial_z}s + \sqrt{-1}H(u_{z}\wedge s),u_{\bar{z}}}  =\inner{\nabla_{\partial_z}s + \sqrt{-1}H(u_{z}\wedge s),\overline{u_{z}}} = 0
			\end{equation*}
			which is equivalent to
			\begin{equation}\label{eq 1,0 s}
				(\nabla_{\partial_z}s)^{1,0} + \sqrt{-1}H(u_{z} \wedge s) = 0
			\end{equation}
			since clearly 
			\begin{equation*}
				\inner{(\nabla_{\partial_z}s)^{1,0} + \sqrt{-1}H(u_{z} \wedge s), u_z} = 0 \quad \text{and} \quad  \inner{(\nabla_{\partial_z}s)^{1,0} + \sqrt{-1}H(u_{z}\wedge s), \bm{n}} = 0.
			\end{equation*}
			Similarly, from identity $\inner{\sigma^{0,1}, u_{\bar{z}}} = 0$ and equation \eqref{eq 16}, we can obtain
			\begin{equation}\label{eq sigma 0,1}
				(\nabla_{\partial_z} \sigma^{0, 1})^{\perp} + \sqrt{-1}H(u_z, \sigma^{0, 1}) =0.
			\end{equation}
			Then, we decompose $v = s + \sigma^{1, 0} + \sigma^{0, 1}$, and keeping in mind that $u_z \wedge \sigma^{1,0} = 0$, \eqref{eq 1,0 s} and \eqref{eq sigma 0,1}, we can rewrite the integrand in the second line of\eqref{eq:complex second variation} as 
			\begin{align*}
				\re&\inner{ \nabla_{\partial_z}v, \overline{\sqrt{-1}H(u_z\wedge v)} }\\
                &= \re\inner{ \nabla_{\partial_z}v, \overline{\sqrt{-1}H(u_z\wedge \sigma^{0, 1})} } + \re\inner{ \nabla_{\partial_z}v, \overline{\sqrt{-1}H(u_z\wedge s)} }\\
				&=- \re\inner{ (\nabla_{\partial_z}v)^{\perp}, \overline{(\nabla_{\partial_z} \sigma^{0, 1})^{\perp}} } - \re\inner{ (\nabla_{\partial_z}v)^{\top}, \overline{(\nabla_{\partial_z} s)^{1, 0}} }.
			\end{align*}
			Combining the above calculation with the first term in the first line of~\eqref{eq:complex second variation} and writing it as $|\nabla_{\partial_z} v|^2 = |(\nabla_{\partial_z} v)^{\perp}|^2 + |(\nabla_{\partial_z} v)^{\top}|^2$, we get
			\begin{align} \label{eq:complex second vari}
				|\nabla_{\partial_z} v|^2 +  2 \re\inner{ \nabla_{\partial_z}v, \overline{\sqrt{-1}H(u_z \wedge v)} } \nonumber
				&= \Big(|(\nabla_{\partial_z} v)^{\perp}|^2 -2 \re\inner{ (\nabla_{\partial_z}v)^{\perp}, \overline{(\nabla_{\partial_z} \sigma^{0, 1})^{\perp}} }\Big)\nonumber\\
				&\quad  +  \Big(|(\nabla_{\partial_z} v)^{\top}|^2 -2 \re\inner{ (\nabla_{\partial_z}v)^{\top}, \overline{(\nabla_{\partial_z} s)^{1, 0}} } \Big)\nonumber \\
				&= \Big( |(\nabla_{\partial_z} v)^{\perp} - (\nabla_{\partial_z} \sigma^{0, 1})^{\perp}|^2 - |(\nabla_{\partial_z} \sigma^{0, 1})^{\perp}|^2\Big) \nonumber\\
				&\quad+ \Big(|(\nabla_{\partial_z} v)^{\top} - (\nabla_{\partial_z} s)^{1, 0}|^2 - |(\nabla_{\partial_z} s)^{1, 0}|^2\Big).
			\end{align}
			We split
			\begin{equation*}
				(\nabla_{\partial_z} v)^{\perp} = (\nabla_{\partial_z} s)^{\perp} + (\nabla_{\partial_z} \sigma^{1, 0})^{\perp} + (\nabla_{\partial_z} \sigma^{0, 1})^{\perp}
			\end{equation*}
			for the normal component, and recalling $\eta = (\nabla_{\partial_z} s)^{0,1} + \nabla^\top_{\partial_z}\sigma^{0,1}$, for the tangent component, we have
			\begin{align*}
				(\nabla_{\partial_z} v)^{\top} &= (\nabla_{\partial_z} s)^{1, 0} + (\nabla_{\partial_z} s)^{0, 1} + (\nabla_{\partial_z} \sigma^{1, 0})^{\top} + (\nabla_{\partial_z} \sigma^{0, 1})^{\top}\\
				&=(\nabla_{\partial_z} s)^{1, 0} + \eta + (\nabla_{\partial_z} \sigma^{1, 0})^{\top} .
			\end{align*}
			Then plugging these terms into \eqref{eq:complex second vari} and expanding the square terms yield
			\begin{align}\label{eq 22}
				|\nabla_{\partial_z} v|^2 &+ 2 \re\inner{ \nabla_{\partial_z}v, \overline{\sqrt{-1}H(u_z \wedge v)} }\nonumber\\
				&= |(\nabla_{\partial_z} s)^{\perp} + (\nabla_{\partial_z} \sigma^{1, 0})^{\perp} |^2 - |(\nabla_{\partial_z} \sigma^{0, 1})^{\perp}|^2 \nonumber \\
				&\quad + | (\nabla_{\partial_z} s)^{0, 1} + (\nabla_{\partial_z} \sigma^{0, 1})^{\top} + (\nabla_{\partial_z} \sigma^{1, 0})^{\top} |^2 - |(\nabla_{\partial_z} s)^{1, 0}|^2\nonumber\\
				&=|(\nabla_{\partial_z} s)^{\perp}|^2 + |(\nabla_{\partial_z} \sigma^{1, 0})^{\perp}|^2 + 2\re\langle (\nabla_{\partial_z} s)^{\perp}, \overline{\nabla_{\partial_z} \sigma^{1, 0}} \rangle\nonumber\\
				&\quad - |(\nabla_{\partial_z} \sigma^{0, 1})^{\perp}|^2 + \frac{\mu^2}{2}|\eta|^2 + |(\nabla_{\partial_z} \sigma^{1, 0})^{\top}|^2 - |(\nabla_{\partial_z} s)^{1, 0}|^2.
			\end{align}
			Next we combine the following identities into \eqref{eq 22}
			\begin{align*}
				|(\nabla_{\partial_z} \sigma^{1, 0})^{\perp}|^2  + |(\nabla_{\partial_z} \sigma^{1, 0})^{\top}|^2 &= |\nabla_{\partial_z} \sigma^{1, 0}|^2\\
				\re\inner{ (\nabla_{\partial_z} s)^{\perp}, \overline{\nabla_{\partial_z} \sigma^{1, 0}} } &= \re\inner{ \nabla_{\partial_z} s, \overline{\nabla_{\partial_z} \sigma^{1, 0}} } - \re\inner{ (\nabla_{\partial_z} s)^{1, 0}, \overline{\nabla_{\partial_z} \sigma^{1, 0}} }\\
				|(\nabla_{\partial_z} s)^{1, 0}|^2 &= |(\nabla_{\partial_z} s)^{\top}|^2 - |(\nabla_{\partial_z} s)^{0, 1}|^2
			\end{align*}
			and plugging the result into second variation formula \eqref{eq:complex second variation} to get
			\begin{align}\label{eq:complex second vari 2}
				\frac{1}{4}\delta^2 E^\omega(u)(v,v)&= \frac{1}{2}\int_{\S^2}|\eta|^2dV_g +\int_{\S^2} |(\nabla_{\partial_z} s)^{\perp}|^2  - |(\nabla_{\partial_z} s)^{\top}|^2 + |(\nabla_{\partial_z} s)^{0, 1}|^2  dx^1dx^2\nonumber\\
				&\quad + \int_{\S^2} |(\nabla_{\partial_z} \sigma^{1, 0})|^2 - |(\nabla_{\partial_z} \sigma^{0, 1})^{\perp}|^2 dx^1dx^2\nonumber\\
				&\quad+ 2\int_{\S^2} \re\inner{ \nabla_{\partial_z} s, \overline{\nabla_{\partial_z} \sigma^{1, 0}} } - \re\inner{ (\nabla_{\partial_z} s)^{1, 0}, \overline{\nabla_{\partial_z} \sigma^{1, 0}} } dx^1dx^2\nonumber\\
				&\quad + \int_{\S^2}\im\inner{v,(\nabla_v H)(u_z\wedge u_{\bar{z}})} -  \langle R(v, u_{z})u_{\bar{z}}, v \rangle dx^1dx^2.
			\end{align}
			Then we consider the integral in \eqref{eq:complex second vari 2} term by term. First, we compute the first two integrands in the second line of \eqref{eq:complex second vari 2}. Integration by parts gives
			\begin{align}\label{eq: integral by parts 1}
				\int_{\S^2} \abs{\nabla_{\partial_z} \sigma^{1, 0}}^2  &- |(\nabla_{\partial_z} \sigma^{0, 1})^{\perp}|^2 dx^1 dx^2\nonumber\\
                &= \int_{\S^2} |(\nabla_{\partial_z}\sigma^{0, 1})^\top|^2 + \inner{ R(u_z, u_{\bar{z}})\sigma^{1,0}, \sigma^{0, 1} } dx^1 dx^2.
			\end{align}
			Similarly, for the first integrand in the third line of integral \eqref{eq:complex second vari 2}, we integrate by parts again to get
			\begin{align}\label{eq: integral by parts 2}
				2&\int_{\S^2} \re\inner{ \nabla_{\partial_z} s, \nabla_{\partial_{\bar{z}}}\sigma^{0, 1} }  dx^1 dx^2 \nonumber\\
				&=\int_{\S^2}-2\re\inner{ s, R(u_z, u_{\bar{z}})\sigma^{0, 1} } + 2\re\inner{ \nabla_{\partial_z} s, \nabla_{\partial_{\bar{z}}}\sigma^{1, 0} } dx^1dx^2  \nonumber\\
				&= \int_{\S^2}-2\re\inner{ s, R(u_z, u_{\bar{z}})\sigma^{0, 1} } + 2\re\inner{ (\nabla_{\partial_z} s)^{0, 1}, (\nabla_{\partial_{\bar{z}}}\sigma^{1, 0})^{\top} } dx^1  dx^2\nonumber\\
				&\quad  + \int_{\S^2}2\re\inner{ (\nabla_{\partial_z} s)^{\perp}, (\nabla_{\partial_{\bar{z}}}\sigma^{1, 0})^{\perp} } dx^1  dx^2.
			\end{align}
			Inserting equations \eqref{eq: integral by parts 1} and \eqref{eq: integral by parts 2} into \eqref{eq:complex second vari 2}, we get 
			\begin{align}\label{eq:complex second vari 3}
				\frac{1}{4}\delta^2 E^\omega(u)(v,v)&= \int_{\S^2}|\eta|^2 dV_{g} + \int_{\S^2}|(\nabla_{\partial_z} s)^{\perp}|^2  - |(\nabla_{\partial_z} s)^{\top}|^2 dx^1dx^2 \nonumber \\
				&\quad + \int_{\S^2} \inner{ R(\sigma^{1, 0}, u_{\bar{z}})u_{z}, \sigma^{0, 1} } -2\re\inner{ s, R(u_{z}, u_{\bar{z}})\sigma^{0, 1} } dx^1dx^2 \nonumber\\
				&\quad+ \int_{\S^2}2\re\inner{ \nabla_{\partial_z} s, (\nabla_{\partial_{\bar{z}}}\sigma^{1, 0})^{\perp} } - 2\re\inner{(\nabla_{\partial_z} s)^{1, 0}, \overline{\nabla_{\partial_z} \sigma^{1, 0}} }  dx^1dx^2\nonumber\\
				& \quad+ \int_{\S^2}\im\inner{v,(\nabla_v H)(u_z\wedge u_{\bar{z}})} -  \langle R(v, u_{z})u_{\bar{z}}, v \rangle dx^1dx^2.
			\end{align}
			Here we used the integral identity
			\begin{equation*}
				\frac{\mu^2}{2}|\eta|^2 = |(\nabla_{\partial_z} s)^{0, 1}|^2 + |(\nabla_{\partial_z} \sigma^{0, 1})^{\top}|^2 + 2 \re\inner{ (\nabla_{\partial_z} s)^{0, 1}, (\nabla_{\partial_{\bar{z}}}\sigma^{1, 0})^{\top} }.
			\end{equation*}
			By~\eqref{eq 1,0 s} and~\eqref{eq sigma 0,1}, utilizing integration by parts we consider the integral in the third line of \eqref{eq:complex second vari 3} to see
			\begin{align}\label{eq:complex second vari 4}
				&\int_{\S^2}2\re\inner{ \nabla_{\partial_z} s, (\nabla_{\partial_{\bar{z}}}\sigma^{1, 0})^{\perp} } - 2\re\inner{ (\nabla_{\partial_z} s)^{1, 0}, \overline{\nabla_{\partial_z} \sigma^{1, 0}} }  dx^1dx^2\nonumber\\
				& =\int_{\S^2}2\re\inner{ \nabla_{\partial_z} s,\overline{-\sqrt{-1}H(u_z \wedge \sigma^{0,1})}} \nonumber\\
				&\quad \quad \quad - 2\re\inner{ (\overline{-\sqrt{-1}H(u_z\wedge s)}, \nabla_{\partial_z} \sigma^{1, 0} }  dx^1dx^2\nonumber\\
				& = \int_{\S^2}2H \re \partial_z \inner{{s,\overline{-\sqrt{-1}u_z \wedge \sigma^{0,1}}}} dx^1dx^2\nonumber\\
				& = -\int_{\S^2} \re \inner{{s,\overline{-\sqrt{-1} (\nabla_{\sigma}H) (u_z \wedge u_{\bar{z}})}} } dx^1dx^2.
			\end{align} 
			For the term about derivatives of $H$ in \eqref{eq:complex second vari 3},  we use the ant-symmetric of wedge product to get that 
			\begin{align}\label{eq:complex second vari 5}
				\int_{\S^2}\im\inner{{v,\overline{(\nabla_v H)(u_z \wedge u_{\bar{z}})}}}dx^1dx^2 &= \int_{\S^2}\im\inner{{s,\overline{(\nabla_s H)(u_z\wedge u_{\bar{z}})}}} dx^1dx^2 \nonumber\\
				&\quad + \int_{\S^2} \re \inner{ {s,\overline{-\sqrt{-1} (\nabla_{\sigma}H) (u_z \wedge u_{\bar{z}})}}} dx^1dx^2.
			\end{align}
			Furthermore, conjunction the second line in \eqref{eq:complex second vari 3} with the Riemann curvature tensor term in the final line of equation \eqref{eq:complex second vari 3} results in
			\begin{align}\label{eq:complex second vari 6}
				&\int_{\S^2} \inner{ R(\sigma^{1, 0}, u_{\bar{z}})u_{z}, \sigma^{0, 1} }-2\re\inner{ s, R(u_{z}, u_{\bar{z}})\sigma^{0, 1} } - \inner{ R(v, u_z)u_{\bar{z}}, v } dx^1dx^2 \nonumber\\
				&= -\int_{\S^2} \inner{ R(s, u_z)u_{\bar{z}}, s } dx^1dx^2.
			\end{align}
			Therefore, by substituting \eqref{eq:complex second vari 4}, \eqref{eq:complex second vari 5}, \eqref{eq:complex second vari 6} into \eqref{eq:complex second vari 3}, we obtain
			\begin{align}\label{eq:index compare}
				\delta^2 E^\omega(u)(v, v) &=\ 4\int_{\S^2} |\eta|^2 dV_{g} + 4\int_{\S^2} |(\nabla_{\partial_z} s)^{\perp}|^2 - |(\nabla_{\partial_z} s)^{\top}|^2dx^1dx^2\nonumber\\
				&\quad  + 4\int_{\S^2}\im\inner{s,\overline{\sqrt{-1}(\nabla_s H)(u_z,u_{\bar{z}})}} -  \inner{ R(s, u_z)u_{\bar{z}}, s } dx^1dx^2 \nonumber\\
				&\quad \leq 4\int_{\S^2} |\eta|^2 dV_{g} + \int_{\S^2}|\nabla f|^2 - {|\nabla u|^2}\left( {|H|^2} + \frac{\mathrm{Ric}(\bm n, \bm n)}{2} - |\nabla H| \right)f^2 dV_{g}
			\end{align}
			where $s = f\bm n$,
			$$\langle R(s, u_z)u_{\bar{z}}, s \rangle =  \frac{\mu^2 |\nabla u|^2}{8}f^2 \mathrm{Ric}(\bm n, \bm n)$$
			$$
			|(\nabla_{\partial_z} s)^{\top}|^2 \geq |(\nabla_{\partial_z} s)^{1, 0}|^2 = |H|^2f^2 |u_z|^2 = \frac{\mu^2 |\nabla u|^2}{4}|H|^2f^2,
			$$
			and 
			$$
			\im\inner{s,\overline{\sqrt{-1}(\nabla_s H)(u_z\wedge u_{\bar{z}})}} \leq \frac{1}{4}|\nabla H|\cdot |\nabla u|^2 f^2,
			$$
			which gives the inequality~\eqref{eq 18} as asserted.
		\end{proof}
		
		We are now ready to give the proof of Proposition~\ref{prop:index comparison}.
		
		\begin{proof}[\textbf{Proof of Proposition~\ref{prop:index comparison}}]
			Based on Lemma~\ref{lem:index computations}, the task at hand can be accomplished by finding the solution to the following equation for $\sigma^{0,1} \in \Gamma(\xi^{0,1})$
			\begin{equation}\label{eq: eta}
				(\nabla_{\partial_z}\sigma^{0, 1})^\top dz
				= -(\nabla_{\partial_z} s)^{0, 1} dz.
			\end{equation}
			But from the proof of genus zero part of \cite[Theorem 1]{Ejiri-Micallef2008} or \cite[Proposition 5.3]{Cheng2020ExistenceOC}, for each $s \in \Gamma(\nu)$ there exists a solution to \eqref{eq: eta}.
			Then, pick any linearly independent sections $s_1, \dots, s_d$ of $\nu$ such that $B_{\omega}(u)$ is negative definite on their $\mathscr{V} : = \mathrm{Span}\,\{s_1, \dots, s_d\}$. For each $ 1 \leq i \leq d$ we choose a solution $\sigma_i^{0, 1} \in \Gamma(\xi^{0, 1})$ to~\eqref{eq: eta} with $s$ placed by $s_i$ and define 
			$$\sigma_i := \sigma_i^{0, 1} + \overline{\sigma_i^{0, 1}}.$$
			Next, we define a linear map $T: V \to \Gamma(E)$ by assigning each $s_i$ to $s_i + \sigma_i$. By substituting $s + \sigma$ with $T(s)$ in Lemma ~\ref{lem:index computations}, it can be inferred that
			\begin{equation*}
				\delta^2 E^\omega(u)(T(s), T(s)) \leq B_{\omega}(u)(s, s) < 0 \text{ for all }s \in V.
			\end{equation*}
			The fact that $T$ is injective leads to the conclusion that the index of $B_{\omega}(u)$ is no greater than $\mathrm{Ind}_{E^\omega}(u)$.
		\end{proof}
		
		Using the index comparison mentioned in Proposition \ref{prop:index comparison}, in conjunction with the standard conformal balancing argument (as described in \cite{Li-Yau1982} and also see \cite[Proposition 5.3]{Cheng2020ExistenceOC}), we are able to derive a uniform energy bound. The main result is the following.
		\begin{prop}\label{prop:uniform energy bound}
			Suppose 
			\begin{equation}\label{eq: curvature condi}
				{|H|^2} h + \frac{\mathrm{Ric}_h}{2} - |\nabla H| h > C_0 h,
			\end{equation}
			and let $u:\S^2 \rightarrow N$ be a solution to~\eqref{eq 16} and~\eqref{eq 17} with Morse index at most 1. Then
			\begin{equation}\label{eq:uniform energy bound}
				E(u) \leq \frac{C_0}{8\pi}.
			\end{equation}
		\end{prop}
		\begin{proof}
			To establish this proposition, it is necessary to assume that $u$ is non-constant. By assumption and Proposition~\ref{prop:index comparison}, we see that $B_{\omega}(u)$ also has Morse index at most $1$. Moreover, let $f$ be a constant function in $B_{\omega}(u)(f, f)$, it follows that $B_{\omega}(u)$ admits  index exactly one, which further implies
			\begin{equation}\label{eq: positive B_omega}
				B_{\omega}(u)(f, f) \geq 0,
			\end{equation}
			for any non-constant $f$. Suppose $\varphi > 0$ is a smallest eigenfunction for the elliptic operator 
			$$-\Delta - \left({|H|^2} + \frac{\mathrm{Ric}(\bm n, \bm n)}{2} - |\nabla H|\right)|\nabla u|^2.$$
			A mapping degree argument described in \cite[pp. 274]{Li-Yau1982} tells us that there exists a conformal map $F: \S^2 \rightarrow \S^2 \subset \R^3$ such that 
			$$
			\int_{\S^2} (F^i) \varphi dV_g = 0, \text{ for }i = 1, 2, 3.
			$$
			which implies $F^i$ can not be constant. Here $\set{x^1, x^2, x^3}$ is the coordinates of the standard embedding $\S^2 \hookrightarrow \R^3$ and we use the abbreviation $F^i$ as $x^i(F)$ to simplify the notation. Hence by~\eqref{eq: positive B_omega} we have
			$$
			\int_{\S^2}|\nabla (F^i)|^2 - |\nabla u|^2 \left( {|H|^2} + \frac{\mathrm{Ric}(\bm n, \bm n)}{2} - |\nabla H| \right)(F^i)^2 dV_g \geq 0.
			$$
			Rearranging the above inequality and taking summation to conclude
			\begin{align*}
				C_0\int_{\S^2}{|\nabla u|^2} dV_g &= C_0\int_{\S^2}{|\nabla u|^2}\sum_{i = 1}^3 (F^i)^2 dV_g \\
				&\leq \sum_{i = 1}^3 \int_{\S^2} |\nabla (F^i)|^2 dV_g = \sum_{i = 1}^3\int_{\S^2} |\nabla x^i|^2 dV_g = 8\pi,
			\end{align*}
			where the inequality is obtained by  the conformally invariance of the energy. This gives \eqref{eq:uniform energy bound}.
		\end{proof}
		\begin{proof}[\textbf{Proof of Part \eqref{main theorem 2 part 1} Theorem~\ref{main theorem 2}}]
			Firstly, by Theorem~\ref{main theorem 1} there exists a strictly increasing sequence $\lambda_j \nearrow\lambda$ such that we can find a sequence of corresponding non-constant $H$-spheres $u_j$ with prescribed mean curvature $\lambda_j H$ and Morse index at most 1. By compactness of $N$ and assumption \eqref{eq:intro ricci condi}, there exists $C_0 > 0$ such that~\eqref{eq: curvature condi} holds, hence
			\begin{equation*}
				E(u_j) \leq \frac{C_0}{8\pi}.
			\end{equation*}
			By the energy gap Lemma \ref{blow 2}, there exists $\varepsilon_0 > 0$ such that 
			\begin{equation}\label{eq:lower bound gap}
				E(u_j) \geq \varepsilon_0^2.
			\end{equation}
			By a similar estimate as Lemma \ref{lem4.1} and Theorem \ref{thm convergence}, we also have an alternative: either, after passing to a subsequnece, $u_j$ converges strongly to a $H$-sphere $u$ with $E(u) \geq \varepsilon_0^2$, or the the energy $E(u_j)$ of sequence $u_j$ concentrates at some points, in which case we can also obtain a $H$-sphere $v$ satisfying $E(v) \geq \varepsilon_0^2$ by a rescaling argument and applying Lemma~\ref{blow 3}. In both cases, the index upper bound is established exactly as in the proof of corresponding part of Theorem~\ref{main theorem 1} and the Ricci curvature condition \eqref{eq: curvature condi} implies that the Morse index of a non-constant $H$-sphere is positive. In a word, we complete the proof of Theorem~\ref{main theorem 2}.
		\end{proof}
		\subsection{Existence of \texorpdfstring{$H$}{Lg}-Sphere  under Isotropic Curvature Assumption When \texorpdfstring{{$\mathrm{dim}(N)\geq 4$}}{Lg}} \label{section 5.3}
		\ 
		\vskip5pt
		
		In this subsection, we prove the second assertion of Theorem \ref{main theorem 2}. We first recall that an element $z\in T_pN\otimes \mathbb{C}$ is called isotropic if $\inner{z,z} = 0$ for complex linearly extended metric $\inner{\cdot,\cdot}$ from $h$ defined on $T_pN$ and a complex linear subspace $Z \subset T_pN\otimes \mathbb{C}$ is called \textit{totally isotropic} if $\inner{z,z} = 0$ for any $z \in Z$. The Riemannian manifold $(N,h)$ is said that has positive isotropic curvature if the complexified sectional curvature $R$ satisfies
		\begin{equation*}
			\mathcal{K}(\sigma) := \frac{R(z,w,\bar{z},\bar{w})}{|z \wedge w|^2} > 0
		\end{equation*}
		whenever $\sigma \subset T_pN\otimes \mathbb{C} $ is a totally isotropic two plane at $p \in N$. Moreover, recall $\xi^{1,0}$ and $\xi^{0,1}$ are locally spanned by $u_z$ and $u_{\bar{z}}$ which are isotropic line bundles within $\mathbf{E}$, let $\nu\otimes \mathbb{C} : = \nu_{\mathbb C}$ be the complexified normal bundle of $\xi_{\mathbb C}$ in $\mathbf{E}$. Since $\S^2$ can be viewed as an 1-dimensional complex manifold, by \cite[Theorem 5.1]{Atiyah} there exists a unique holomorphic structure on $ \nu_{\mathbb C}$ such that $\nabla^{\prime\prime} = \bar{\partial}$ where
		\begin{equation*}
			\nabla^\prime : \wedge^{0,0}(\nu_{\mathbb C}) \longrightarrow \wedge^{1,0}(\nu_{\mathbb C}),\quad \nabla^{\prime\prime} : \wedge^{0,0}(\nu_{\mathbb C}) \longrightarrow \wedge^{0,1}(\nu_{\mathbb C})
		\end{equation*}
		are two component of complex linear extended connection $\nabla^\perp$ on $\nu_{\mathbb C}$. Equipped with above notions, we have: 
		\begin{proof}[\textbf{Proof of Part \eqref{main theorem 2 part 2} of Main Theorem \ref{main theorem 2}}]
			We will actually prove a stronger assertion:
			\ 
			\vskip5pt
			\claim\label{section 5.4 claim} Let $(N,h)$ be an $n$-dimensional Riemannian manifold with isotropic curvature satisfying \eqref{eq:intro isotropic}. Then any non-constant conformal $H$-sphere has Morse index at least $[(n-2)/2]$.\vspace{1ex}
			\ 
			\vskip5pt
			\begin{proof}[\textbf{Proof of Claim \ref{section 5.4 claim}}]
				By Grothendick's theorem, see \cite{Grothendieck1957}, which says that any holomorphic vector bundle over $\S^2$ can be represented as a direct sum of holomorphic line bundles, we can decompose $\nu_{\mathbb C}$ as 
				\begin{equation*}
					\nu_{\mathbb C} = L_1 \oplus L_2 \oplus \cdots \oplus L_{n-2}
				\end{equation*}
				which is unique up to a permutation of the order for $L_i$. So after changing the order of $L_i$, we can assume that
				\begin{equation*}
					\bm{c}_1(L_1)\geq \bm{c}_1(L_2)\geq \cdots \geq \bm{c}_1(L_{n-2})
				\end{equation*}
				where $\bm{c}_1(L_i)$ is the first Chern class of $L_i$ evaluated on the fundamental class of $\S^2$. The Levi-Civita connection $\nabla$ preserves the Riemannian metric parallelly, resulting in a complex linearly extended bi-linear form denoted as $\inner{\cdot,\cdot}:\nu_{\mathbb C} \times\nu_{\mathbb C} \rightarrow \mathbb{C}$. This bi-linear form is holomorphic and establishes a holomorphic isomorphism between $\nu_{\mathbb C}$ and its dual $\nu_{\mathbb C}^*$. Thus, by the invairance of Chern class, we have 
				\begin{equation*}
					\bm{c}_1(L_i) + \bm{c}_1(L_{n-i-1}) = 0.
				\end{equation*}
				Let $V_i$ be a meromorphic section of $L_i$ for $1 \leq i \leq n-2$. If $\inner{V_i,V_j} \not \equiv 0$, then we must have $\bm{c}_1(L_i) + \bm{c}_1(L_{j}) =  0$, that is, $\inner{V_i,V_j} \equiv 0$ provided that $\bm{c}_1(L_i) + \bm{c}_1(L_{j}) \neq 0$ for any section $V_i \in \Gamma(L_i)$ and $V_j \in \Gamma(L_j)$. Denote $\bm N_0$ the direct sum of the line bundle that has zero first Chern class and $\bm N_+$($\bm N_-$) be the direct sum of the line bundle that has positive (negative) first Chern class. Then $\bm N_+$ is an isotropic sub-bundle of $\nu_{\mathbb C}$ and $\inner{V_0,V_+} = 0$ for each $V_0 \in  \Gamma(\bm{N_0})$ and $V_+ \in \Gamma(\bm N_+)$. It follows from the Riemann-Roch theorem that 
				\begin{equation*}
					\mathrm{dim}_{\mathbb{C}}(\mathscr{O}(L_i)) = \left\{
					\begin{aligned}
						&\bm{c}_1(L_i) + 1,\quad &\text{if } \bm{c}_1(L_i) \geq 0,\\
						&0 \quad &\text{if } \bm{c}_1(L_i) < 0,
					\end{aligned}
					\right.
				\end{equation*}
				where $\mathscr{O}(L_i)$ is the set of holomorphic section of $L_i$. For any holomorphic sections $W_i,W_j$ of $\bm{N}_0$, since $\inner{V_i,V_j}$ is a holomorphic function on $\S^2$, $\inner{W_i,W_j}$ is constant  which means that one can choose $W_i$ such that $\{W_i\}_{1 \leq i \leq \mathrm{dim}(N_0)}$ to be an othornormal basis at each fibre of $\bm N_0$ respect to the bi-linear form $\inner{\cdot,\cdot}$. Let $\mathscr{O}$ be the complex linear space of holomorphic sections of $\nu_{\mathbb C}$ spanned by the holomorphic isotropic sections 
				$$\set{W_1 + \sqrt{-1} W_2, W_3 + \sqrt{-1} W_4, \cdots , W_{2m-1} + \sqrt{-1}W_{2m} }$$
				where $m = [\mathrm{dim}_{\mathbb{C}}(\bm N_0)/2] \in \mathbb{N}$ together with the holomorphic sections of $\bm N_+$. Then, we conclude that $\dim_{\mathbb{C}}(\mathscr{O}) \geq [(n-2)/2]$. Note that by the choice of $\nu_{\mathbb{C}}$, $u_z$ is linearly independent with elements in $\mathscr{O}$, that is, $u_z$ and $V$ can span a totally isotropic two plane in $\bm E$ for any $V \in \mathscr O$. Thus, by complex form of second variation formula \eqref{eq:complex second variation} for $E^\omega$,  for any $V \in \mathscr{O}$ we have 
				\begin{align*}
					\delta^2 E^\omega(u)(V, V) &= 4\int_{\S^2} - \langle R(V, u_{z})u_{\bar{z}}, V \rangle dV_g \\
					&\quad + 4 \int_{\S^2}\re\inner{V,\overline{\sqrt{-1}(\nabla_V H)(u_z,u_{\bar{z}})}} dV_g < 0,
				\end{align*}
				provided that \eqref{eq:intro isotropic} holds. Therefore, the Morse index of $u$ is greater than or equal to $[(n-2)/2]$.
			\end{proof}
			The proof of part \eqref{main theorem 2 part 2} of Theorem \ref{main theorem 2} follows directly from Claim \ref{section 5.4 claim}.
		\end{proof}
		
		\vskip2cm
		\bibliographystyle{amsalpha}
		\bibliography{references}
		\providecommand{\MR}{\relax\ifhmode\unskip\space\fi MR }
		\providecommand{\MRhref}[2]{
			\href{http://www.ams.org/mathscinet-getitem?mr=#1}{#2}
		}
	\end{document}